\documentclass[leqno,11pt]{amsart}
\usepackage[top=1.25in, bottom=1.25in, left=1in, right=1in]{geometry}
\usepackage{amssymb, amsmath,amsmath,latexsym,amssymb,amsfonts,amsbsy, amsthm,esint,dsfont,bm}
\usepackage{hyperref}
\usepackage{color}
\usepackage{mathrsfs}
\usepackage{graphicx}
\definecolor{pku}{RGB}{139,0,18}

\let\pa=\partial

\let\b=\beta
\let\g=\gamma
\let\d=\delta
\let\z=\zeta

\let\r=\rho
\let\s=\sigma
\let\f=\frac

\let\th=\theta
\let\p=\psi
\let\om=\omega
\let\G= \Gamma
\let\D=\Delta

\let\Om=\Omega

\let\e=\varepsilon
\let\pa=\partial
\let\va=\varphi

\let\na=\nabla
\let\P = \Psi

\def\di{\mathrm{div}\,}

\newcommand{\beq}{\begin{equation}}
\newcommand{\eeq}{\end{equation}}
\newcommand{\beqo}{\begin{equation*}}
\newcommand{\eeqo}{\end{equation*}}
\newcommand{\ben}{\begin{eqnarray}}
\newcommand{\een}{\end{eqnarray}}
\newcommand{\beno}{\begin{eqnarray*}}
\newcommand{\eeno}{\end{eqnarray*}}

\newtheorem{thm}{Theorem}[section]
\newtheorem{lem}{Lemma}[section]

\newtheorem{prop}{Proposition}[section]

\newtheorem{theorem}{Theorem}[section]

\newtheorem{proposition}[theorem]{Proposition}
\newtheorem{corol}[theorem]{Corollary}

\theoremstyle{remark}
\newtheorem{step}{Step}
\newtheorem{case}{Case}
\newtheorem{rmk}{Remark}[section]

\newcommand{\CD}{\mathcal{D}}
\newcommand{\CK}{\mathcal{K}}
\newcommand{\CR}{\mathcal{R}}
\newcommand{\CS}{\mathcal{S}}

\newcommand{\CT}{\mathcal{T}}
\newcommand{\CP}{\mathcal{P}}
\newcommand{\CH}{\mathcal{H}}

\newcommand{\BR}{\mathbb{R}}
\newcommand{\BZ}{\mathbb{Z}}
\newcommand{\BT}{\mathbb{T}}

\numberwithin{equation}{section}

\begin{document}
\title{Interface Dynamics in a Two-phase Tumor Growth Model}

\author{Inwon Kim}
\address{University of California, Los Angeles\\
Box 951555, Los Angeles, CA 90095, USA}
\email{ikim@math.ucla.edu}
\author{Jiajun Tong}
\address{University of California, Los Angeles\\
Box 951555, Los Angeles, CA 90095, USA}
\email{jiajun@math.ucla.edu}

\date{\today}

\begin{abstract}

We study a tumor growth model in two space dimensions, where proliferation of the tumor cells leads to expansion of the tumor domain and migration of surrounding normal tissues into the exterior vacuum. The model features two moving interfaces separating the tumor, the normal tissue, and the exterior vacuum. We prove local-in-time existence and uniqueness of strong solutions for their evolution starting from a nearly radial initial configuration.
It is assumed that the tumor has lower mobility than the normal tissue,
which is in line with the well-known Saffman-Taylor condition in viscous fingering.

\end{abstract}
\maketitle

\setcounter{tocdepth}{1}
\tableofcontents

\section{Introduction}
\label{sec: problem formulation}
In this paper, we study free boundary dynamics arising in a model of avascular tumor growth which is adapted from \cite{LLP}.

\subsection{A two-species model of tumor growth}
Consider two species of cells in $\mathbb{R}^2$, one being actively growing tumor cell and the other being inactive normal cell.
Spatial densities of tumor and normal cells, each denoted by $m$ and $n$, satisfy
\begin{align}
\pa_t m -\di(\mu m\nabla p) =&\; mG(p),\label{eqn: m equation}\\
\pa_t n -\di(\nu n\nabla p) =&\; 0,\label{eqn: n equation}\\
m+n\leq &\;1.\label{eqn: total cell density no greater than 1}
\end{align}
Here $\mu,\nu>0$ denote mobilities of the tumor and normal cells. $p$ is the pressure generated by the cells, serving as a Lagrange multiplier for the constraint $m+n\leq 1$.
It satisfies
\begin{align}
\label{eqn: equation for p}-\di((\mu m+\nu n)\nabla p) = mG(p)&\quad \mbox{ if }m+n=1,\\
p = 0&\quad \mbox{ if }m+n<1.\label{eqn: equation for p outside}
\end{align}
In \eqref{eqn: m equation} and \eqref{eqn: equation for p}, $G(p)$ represents pressure-dependent  proliferation rate of the tumor cell.
In the spirit of \cite{LLP}, we assume that
\begin{enumerate}
  \item $G\in C^1[0,+\infty)$.
  \item $G(p)$ is decreasing.
  \item $G(0)>0$ and $G(p_M) = 0$ for some $p_M>0$.
\end{enumerate}
In short, \eqref{eqn: m equation}-\eqref{eqn: equation for p outside} models the scenario where the tumor keeps growing and where two species of cells migrate with different mobilities, according to the Darcy's law \cite{byrne1996modelling}, under the pressure they generate together.

Mathematical analysis of strongly-coupled competitive systems such as  \eqref{eqn: m equation}-\eqref{eqn: equation for p outside} can be challenging \cite{kim2018nonlinear,bubba2019hele,gwiazda2019two,bertsch1985interacting,bertsch2010free,carrillo2018splitting}.
To the best of our knowledge, existing analyses of such problems are carried out either in one space dimension or with equal mobility of the two species. In contrast, it is suggested in \cite{LLP} that the cells moving with different mobilities is an important feature of the model \eqref{eqn: m equation}-\eqref{eqn: equation for p outside}. Indeed, the numerical results in \cite{LLP} show that when $\mu<\nu$, certain radially symmetric solution is stable, while when $\mu>\nu$ a Saffman-Taylor type instability \cite{SaffmanTaylor} can occur.

\subsection{A free boundary problem }
In this paper, we study \eqref{eqn: m equation}-\eqref{eqn: equation for p outside} with the restriction that $m$ and $n$ are segregated and fully saturated in their regions.
Namely, we assume that $m = \chi_\Omega$ and $n = \chi_{\tilde{\Omega}\backslash\Omega}$,  where $\Omega\subset\subset \tilde{\Omega}$ are two time-varying bounded domains.
This gives rise to a free boundary problem that concerns dynamics of both $\g:=\partial\Omega$ and $\tilde{\g}:=\partial\tilde{\Omega}$.

First, the equation for $p$ reduces to
\begin{align}
\label{eqn: equation for p simplified case}-\di((\mu \chi_\Omega+\nu \chi_{\tilde{\Om}\backslash \Om})\nabla p) = \chi_\Om G(p)&\quad \mbox{in }\tilde{\Omega},\quad p|_{\partial \tilde{\Om}} = 0,\\
p = 0&\quad \mbox{in }\tilde{\Om}^c.
\label{eqn: p boundary condition}
\end{align}
%
%
Then the motion law of the free boundaries are given as follows.
From \eqref{eqn: m equation}, we may derive the normal velocity for $\g$:
\beq
V_{n,\gamma} = -\mu\f{\pa p}{\pa \s_\Om}.
\label{eqn: normal motion of gamma}
\eeq
Here $\s_\Om$ denotes the outward normal of $\g$ with respect to $\Om$. Similarly, the normal speed of $\tilde{\g}$ is given by \eqref{eqn: n equation}:
\beq
V_{n,\tilde{\g}} = -\nu\f{\pa p}{\pa \s_{\tilde{\Om}}}, 
\label{eqn: normal motion of gamma tilde}
\eeq
where $\s_{\tilde{\Om}}$ denotes  outward normal of $\tilde{\g}$ with respect to $\tilde{\Om}$.


Our main result is the local-in-time well-posedness of  the free boundary problem \eqref{eqn: equation for p simplified case}-\eqref{eqn: normal motion of gamma tilde}. Inspired by the numerical results in \cite{LLP}, we assume $\mu<\nu$ for the well-posedness.
Interestingly, we will illustrate later that even with this assumption, instabilities may still occur along $\g$ without further geometric assumptions on $\Om$ and $\tilde{\Om}$ (see Remark \ref{rmk: bad case due to geometry}).
We thus need to restrict ourselves to the case where the initial configuration is nearly radial (see Figure \ref{Figure: tumor geometry}).
More precise statement of our main results can be found in Theorem \ref{thm: local well-posedness} and Theorem \ref{thm: uniqueness} in Section \ref{sec: main results}.

\subsection{Related works and our approach}

The evolution of the inner interface $\g$ is similar to the 2-D Muskat problem \cite{muskat1934two,bear2013dynamics} with viscosity jump \cite{siegel2004global,matioc2018well}, which is concerned with a close-to-flat interface between two fluids driven by the Darcy's law.
In the case when the more viscous fluid is pushed towards the less viscous one, \cite{siegel2004global} establishes global well-posedness for small initial data; in the opposite case, ill-posedness is shown.  With generalized Rayleigh-Taylor condition \cite{escher2011generalized}, \cite{matioc2018well} formulates similar result on the well-posedness in a more general setup allowing density-viscosity jumps.
Note that these rigorous results agree very well with \cite{SaffmanTaylor} and the aforementioned numerical results in \cite{LLP}.
They are obtained by 
exploring the inherent parabolicity in the interface motion with complex analysis \cite{siegel2004global} and functional analytic \cite{matioc2018well} approaches.
However, it is not clear if these approaches can be directly applied here as our model involves a geometry-dependent source term, whose support touches $\g$. 

Notably there is a lot more literature concerning the Muskat problem with density jump \cite{ambrose2004well,cordoba2007contour,constantin2012global,constantin2016muskat,cordoba2018global} or density-viscosity jumps \cite{cordoba2011,cordoba2013porous,gancedo2019muskat,cheng2016well,matioc2018well}.
In both of these cases, the smoothing mechanism is essentially provided by the fact that a heavier fluid sits below a lighter one, where the gravity naturally damps the oscillation of the interface.
In contrast, the smoothing mechanism is much less explicit when there is only jump in the viscosity across the interface \cite{siegel2004global,matioc2018well}.

Motion of the outer interface $\tilde{\g}$ is reminiscent of the free boundary arising in the one-phase Hele-Shaw problem \cite{richardson1972hele}, where a blob of fluid is injected into a Hele-Shaw cell or a porous medium and expands according to the Darcy's law.
Despite its similarity with the Muskat problem in some aspects, it admits a few other treatments.
We direct the readers to \cite{elliott1981variational,dibenedetto1984ill,howison1992complex,escher1997classical,kim2003uniqueness,cheng2012global} and the references therein.
Once again, in our problem, the presence of the source term depending nonlocally on $\tilde{\g}$ and $\g$ may hinder direct applications of these approaches.

In this paper, we study the dynamics of both interfaces $\g$ and $\tilde{\g}$ in a unified framework, adapted from the study of contour equations in the Muskat problem \cite{cordoba2013porous,gancedo2019muskat}.
We first reduce \eqref{eqn: equation for p simplified case}-\eqref{eqn: normal motion of gamma tilde}, which involves an elliptic equation for $p$ in a time-varying domain, partially into contour equations for the interface configurations and quantities along them;
see \eqref{eqn: equation for inner interface new notation}, \eqref{eqn: equation for outer interface new notation}, \eqref{eqn: equation for derivative of jump of potential final form} and \eqref{eqn: equation for derivative of potential on outer interface}.
A key step in this reduction is to represent the transporting velocity over $\tilde{\Om}$ as a sum of three parts, which arise from the discontinuity of the cell mobilities across $\g$, the zero Dirichlet boundary condition of $p$ along $\tilde{\g}$, and the source term in $\Om$, respectively; see \eqref{eqn: velocity in Omega} and also \eqref{eqn: representation of phi}.
Then by linearizing these contour equations around radially symmetric configurations, we show their parabolic nature under suitable conditions (c.f., Section \ref{sec: scheme of the proof}).
In particular, the interfaces can smooth themselves according to a fractional-heat-type equation with source terms.
After deriving good estimates for these source terms, we prove well-posedness of the interface motion by a fixed-point argument.
Smallness of the geometric deviation of $\g$ and $\tilde{\g}$ from radially symmetric configurations helps close the estimates needed in this argument.
See Section \ref{sec: equations of the problem} for more details.

\subsection{Difficulties arising from the source term}
This problem features a geometry-dependent source term $\chi_\Om G(p)$ in \eqref{eqn: equation for p simplified case} that is supported up to the inner interface.
It may be tempting to think of it as an innocuous regular term, but in fact, it changes the dynamics in a crucial way compared to the related problems discussed above.

Firstly, on the technical level, the source term seems to prevent the complex analysis approach in \cite{siegel2004global} from being applied here. Secondly,  the parabolicity of $\g$ relies on the fact that the cell with lower mobility is displacing the other species, i.e., $(\mu-\nu) \f{\pa p}{\pa \s_\Om}|_\g > 0$ (c.f., \eqref{eqn: normal motion of gamma} and Remark \ref{rmk: localized analysis}), in line with the classic Saffman-Taylor condition \cite{SaffmanTaylor}.
Since $\chi_\Om G(p)$ depends on the domain geometry, one can manufacture such $\Om$ and $\tilde{\Om}$, so that the tumor is pushed by the normal tissue along some part of $\g$ under the assumption $\mu<\nu$.
This is possible even and both $\g$ and $\tilde{\g}$ are required to be graphs of functions over $\BT$ in the polar coordinate; see Remark \ref{rmk: bad case due to geometry}.
In this sense, simply assuming $\mu<\nu$ is not enough for proving well-posedness, and it is reasonable to additionally require that $\g$ and $\tilde{\g}$ are close to concentric circles (see Figure \ref{Figure: tumor geometry}).
Then characterization of the parabolicity of $\g$ is based on a good understanding of $p$.
In Section \ref{sec: estimate for pressure in circular geometry}, we apply elliptic regularity theory to justify that given the domain geometry close to a radially symmetric one, the corresponding $p$ should not be far away from a radially solution.
That would be sufficient to guarantee 
parabolicity in the motion of $\g$ as $\mu<\nu$. 
Furthermore, these elliptic estimates together with the results in Section \ref{sec: gradient estimate of growth potential} and Section  \ref{sec: estimates for singular integral operators} will help justify that such parabolicity can be characterized by a fractional heat operator with exponent $\f12$, which plays a central role in our analysis.
See Section \ref{sec: scheme of the proof} and Section \ref{sec: local well-posedness}.

The source term also poses new difficulty in studying global well-posedness and stability properties near the radially symmteric solutions. 
Indeed, as the tumor grows larger, the pressure becomes more sensitive to the interface geometry.
We demonstrate this by a scaling argument.
Suppose at given time $T>0$, $\Om$ and $\tilde{\Omega}$ are close to two concentric discs, and $\tilde{\Om}$ has radius of order $R\gg 1$.
Define $p_R(x,t) := p(Rx,R(t-T))$ and let $\tilde{\Omega}_R$ and $\Omega_R$ denote the corresponding dilated version of $\tilde{\Om}$ and $\Omega$ according to the scaling.
Then \eqref{eqn: equation for p simplified case} becomes
\beq
-\di((\mu \chi_{\Om_R}+\nu \chi_{\tilde{\Om}_R\backslash \Om_R})\nabla p_R) = \chi_{\Om_R} R^2G(p_R),
\eeq
with zero boundary data on $\partial\tilde{\Om}_R$, while the boundary motion laws \eqref{eqn: normal motion of gamma} and \eqref{eqn: normal motion of gamma tilde} remain the same.
In this new problem, the proliferation rate $R^2 G(\cdot)$ can have a large magnitude where $p_R$ is small and it is sensitive to the pressure.
This results in concentration of the source term near the inner interface and a steep growth of $p_R$ there.
On the other hand, the total mass of the normal tissue is preserved due to \eqref{eqn: n equation}, and thus $\tilde{\Om}_R\backslash\Om_R$ is extremely thin as $R\gg 1$.
So in the rescaled problem the source term is close to both the inner and outer interfaces.
It is then conceivable that
$p_R$ will be highly sensitive to the domain geometry
in the sense that even when the domain is pretty close to being radial, $p_R$ may be highly oscillatory and far from being radially symmetric.
Therefore, because of the source term, nonlinear stability of the interface configurations around radially symmetric ones becomes a much more subtle issue when it comes to long time asymptotics.

\subsection{Acknowledgement}
This work is partially supported by National Science Foundation under Award DMS-1900804.

\section{Interface Motion in an Almost Radially Symmetric Geometry}
\label{sec: equations of the problem}

In this section, we will derive equations for the moving interfaces $\g$ and $\tilde{\g}$ in the case when they are close to concentric circles. Our main result will be established in terms of these equations.
Parabolicity of these equations will be revealed, which plays a key role in proving the well-posedness.

\subsection{Problem reformulation}
Define a potential $\va$ to be
\beq
\va : = \mu p\mbox{ in }\Omega,\quad \va := \nu p\mbox{ in }\Om^c.
\label{eqn: def of potential varphi}
\eeq
So $\va$ solves
\beqo
-\Delta \va = G(p)\chi_\Om \quad \mbox{in } \tilde{\Om}\backslash \g, \quad \va|_{\tilde{\g}} = 0,
\eeqo
and $\va\equiv 0$ on $\tilde{\Om}^c$.
When $\mu\not = \nu$, $\va$ has discontinuity across $\g$, denoted by
\beqo
[\va]_\g(x) := \va|_{\g,\Om}(x) - \va|_{\g, \Om^c}(x),\quad x\in\g.
\eeqo
 \eqref{eqn: normal motion of gamma} and \eqref{eqn: normal motion of gamma tilde} yield that each cell phase is transported by the velocity field $u = -\na \va$.
It has discontinuity across $\g$ in the tangential component, but not in its normal component.

Let $\G$ denote the fundamental solution of the Laplace's equation in $\BR^2$,
\beqo
\G(x) := -\f{1}{2\pi}\ln |x|.
\eeqo
%
%
%
Let $\CD_\g$ denote the double layer potential operator associated with $\g$.
Namely, with a boundary potential $\p$ defined on $\g$, we define $\CD_\g \p$ on $\BR^2$ to be
\beq
\CD_\g \p(x) :=\int_{\g} \s_{y}\cdot \na_y(\G(x-y)) \p(y)\,dy.
\label{eqn: def of double layer potential}
\eeq
Note that here the gradient is taken with respect to $y$.
It is well-known that for $\g$ and $\p$ sufficiently smooth, say $C^{1,\alpha}(\BT)$, $[\CD_\g \p]_\gamma = -\p$. 
Then $\va$ admits the following representation
\beq
\va = -\CD_\g[\va]-\CD_{\tilde{\g}}\phi +\G*g\quad \mbox{in }\tilde{\Om}\backslash\g.
\label{eqn: representation of phi}
\eeq
where $\phi$  is some boundary potential defined along $\tilde{\g}$ to be determined in order for the boundary condition $\va|_{\tilde{\g}} = 0$,
and where
\beq
g = G(p)\chi_\Om = G(\mu^{-1}\va)\chi_\Om\geq 0.
\label{eqn: def of source term g}
\eeq

Assume $C^{1,\alpha}(\BT)$-regularity of $\g$ and $[\va]$.
Then the representation \eqref{eqn: representation of phi} along $\g$ takes the average of $\va$ on two sides of $\g$, i.e.,
\beqo
\left.\left(-\CD_\g[\va]-\CD_{\tilde{\g}}\phi +\G*g \right)\right|_\g= \f12(\va|_{\g,\Om}+\va|_{\g, \Om^c}) = \f{\mu+\nu}{2}p = \f{\mu+\nu}{2(\mu-\nu)}[\va].
\eeqo
This implies
\beq
[\va] = 2A(- \CD_\g[\va]-\CD_{\tilde{\g}}\phi +\G*g)|_\g,
\label{eqn: equation for jump of phi}
\eeq
where
\beqo
A = \f{\mu-\nu}{\mu+\nu}.
\eeqo
On the other hand, the zero Dirichlet boundary condition of $\va$ along $\tilde{\g}$ requires that
\beqo
\lim_{x\to \tilde{\g}(\th)\atop x\in \tilde{\Om}}(-\CD_{\tilde{\g}}\phi)(x) = (\CD_\g[\va]-\G*g)|_{\tilde{\g}(\th)}.
\eeqo
Assuming $C^{1,\alpha}(\BT)$-regularity of $\tilde{\g}$ and $\phi$, by the property of the double layer potential, $\phi$ should solve
\beqo
-(\CD_{\tilde{\g}}\phi)|_{\tilde{\g}}+\f12\phi = (\CD_{\g}[\va]-\G*g)|_{\tilde{\g}}
\eeqo
along $\tilde{\g}$, i.e.,
\beq
\phi = 2(\CD_{\tilde{\g}}\phi+\CD_{\g}[\va]-\G*g)|_{\tilde{\g}}.
\label{eqn: equation for boundary potential on outer interface}
\eeq

Finally, \eqref{eqn: normal motion of gamma} and \eqref{eqn: normal motion of gamma tilde} become
\beq
V_{n,\g} = -\f{\pa \va}{\pa \s_{\Om}},\quad V_{n,\tilde{\g}} = -\f{\pa \va}{\pa \s_{\tilde{\Om}}}. 
\label{eqn: normal motion of two interfaces}
\eeq
\eqref{eqn: representation of phi}-
\eqref{eqn: normal motion of two interfaces} readily form a closed system.

\begin{figure}
\centering
\includegraphics[width=0.3\textwidth]{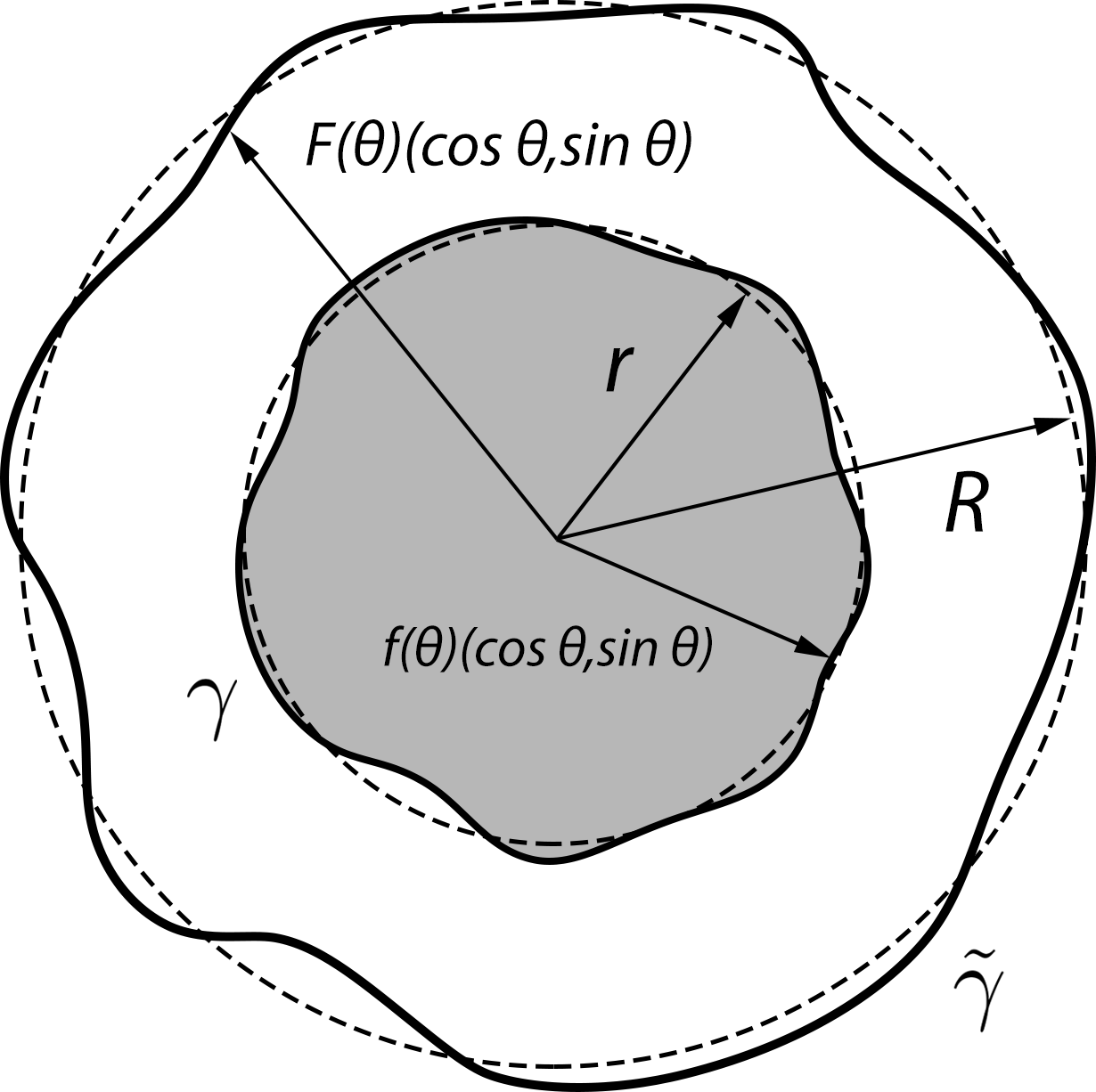}
\caption{An illustration of the geometry. The grey region represents the domain of the tumor cells, while the white region surrounding it is occupied by the normal cells.
The solid curves $\g$ and $\tilde{\g}$ are moving boundaries of the domains. The dashed circles indicate that $\g$ and $\tilde{\g}$ are close to two concentric circles with radii $r$ and $R$, respectively.
$\g$ and $\tilde{\g}$ are parameterized in the polar coordinate as functions of $\th\in \BT = \mathbb{R}/ (2\pi\BZ)$.}
\label{Figure: tumor geometry}
\end{figure}

\subsection{Derivation of contour equations}
We consider the case when $\g$ and $\tilde{\g}$ are close to two concentric circles centered at the origin, with some radii $r<R$, respectively.
See Figure \ref{Figure: tumor geometry}.
We parameterize $\g$ and $\tilde{\g}$ using the polar coordinate,
\begin{align}
\g(\th,t) = &\;f(\th,t)(\cos \th, \sin\th),\label{eqn: parameterization of gamma}\\
\tilde{\g}(\th,t) = &\;F(\th,t)(\cos\th,\sin\th),
\end{align}
where $\th\in \BT = \BR/(2\pi \BZ) = [-\pi,\pi)$.
Then $[\va]$ and $\phi$ can be naturally understood as functions of $\th\in \BT$.
Next we shall derive equations for $\g$ and $\tilde{\g}$ (or equivalently, for $f$ and $F$).



Note that $\s_\Om (\th) = -\g'(\th)^\perp/|\g'(\th)|$, where $v^\perp$ denote a vector $v\in \BR^2$ rotated counter-clockwise by $\pi/2$.
By \eqref{eqn: def of double layer potential}, all $x\in\tilde{\Om}\backslash \g$,
\beq
\CD_\g [\va] = \f{1}{2\pi}\int_\BT \f{(x-\g(\th'))\cdot(-\g'(\th'))^\perp}{|x-\g(\th')|^2}[\va](\th')\,d\th'.
\eeq
By assuming $[\va]\in C^1(\BT)$,
\beq
\begin{split}
\na \CD_\g [\va] 
=&\;\f{1}{2\pi}\int_\BT \f{\pa}{\pa \th'}\left(- \f{(x-\g(\th'))^\perp}{|x-\g(\th')|^2}\right)[\va](\th')\,d\th'\\
=&\;\f{1}{2\pi}\int_\BT \f{(x-\g(\th'))^\perp}{|x-\g(\th')|^2}[\va]'(\th')\,d\th'.
\end{split}
\label{eqn: gradient of double layer away from the boundary}
\eeq
which is a Birkhoff-Rott-type integral \cite{majda2002vorticity}.
%
Hence, by \eqref{eqn: representation of phi}, for $x\in \tilde{\Om}\backslash \g$,
\beq
u(x)= -\na \va(x) = \f{1}{2\pi}\int_\BT \f{(x-\g(\th'))^\perp}{|x-\g(\th')|^2}[\va]'(\th')\,d\th' + \f{1}{2\pi}\int_\BT \f{(x-\tilde{\g}(\th'))^\perp}{|x-\tilde{\g}(\th')|^2}\phi'(\th')\,d\th' -\na(\G*g).
\label{eqn: velocity in Omega}
\eeq
On the other hand, by \eqref{eqn: normal motion of two interfaces} and \eqref{eqn: parameterization of gamma}, 
\beq
\pa_t f = u(\g(\th))\cdot \s_{\Om}(\th)\cdot \f{|\g'(\th)|}{f(\th)} = -\f{1}{f}\cdot u(\g(\th))\cdot \g'(\th)^\perp.
\label{eqn: time derivative of f}
\eeq
Although $u(\g(\th))$ here should be understood as the limit of \eqref{eqn: velocity in Omega} when letting $x\rightarrow \g(\th)$ from the inside of $\g$, it is safe to simply take $x=\g(\th)$ since the normal component of $u$ does not have discontinuity across $\g$.
Define
\begin{align}
\CK_\g \p :=&\;\f{1}{2\pi}\mathrm{p.v.}\int_\BT\f{\g(\th)-\g(\th')}{|\g(\th)-\g(\th')|^2}\cdot \p(\th')\,d\th',
\label{eqn: def of operator K}\\
\mathcal{K}_{\g,\tilde{\g}}\p(\th) := &\; \f{1}{2\pi}\int_\BT \f{ \g(\th)-\tilde{\g}(\th')}{|\g(\th)-\tilde{\g}(\th')|^2}\cdot \p(\th')\,d\th'.
\label{eqn: def of integral operators for a pair of interfaces}
\end{align}
Let $\mathcal{K}_{\tilde{\g},\g}\p(\th)$ be defined symmetrically by interchanging $\g$ and $\tilde{\g}$ in \eqref{eqn: def of integral operators for a pair of interfaces}.
Thanks to \eqref{eqn: gradient of double layer away from the boundary} and
\eqref{eqn: velocity in Omega}, \eqref{eqn: time derivative of f} can be rewritten as
\beq
\pa_t f = -\f{1}{f}\g'(\th)\cdot \CK_\g [\va]'-\f{1}{f}\g'(\th)\cdot \CK_{\g,\tilde{\g}}\phi'+\f{1}{f} \na(\G*g)|_{\g}\cdot \g'(\th)^\perp.
\label{eqn: equation for inner interface new notation}
\eeq
Similarly,
\beq
\pa_t F = -\f{1}{F}\tilde{\g}'(\th)\cdot \CK_{\tilde{\g}} \phi'- \f{1}{F}\tilde{\g}'(\th)\cdot \CK_{\tilde{\g},\g}[\va]'+\f{1}{F} \na(\G*g)|_{\tilde{\g}}\cdot \tilde{\g}'(\th)^\perp.
\label{eqn: equation for outer interface new notation}
\eeq
These equations are coupled with initial conditions
\beq
f(t= 0) = f_0(\th),\quad F(t= 0) = F_0(\th).
\label{eqn: initial condition}
\eeq

For future use, we introduce
\beq
h(\th,t) = \f{f(\th,t)}{r}-1,\quad H(\th,t) = \f{F(\th,t)}{R}-1.
\label{eqn: def of h and H}
\eeq
They are relative deviations of $\g$ and $\tilde{\g}$ from radially symmetric configurations.

\subsection{Main results}
\label{sec: main results}
We first introduce $W^{k-\f1p,p}(\BT)$-space for $k\in \BZ_+$ and $p\in (1,\infty)$ \cite[\S\,2.12.2]{triebel2010theory}.
Let $\{e^{-t(-\D)^{1/2}}\}_{t\geq 0}$ denote the Poisson semi-group on $\BT$ with generator $-(-\D)^{1/2}$.
For $f \in L^p(\BT)$, let
\beq
\|f\|_{\dot{W}^{k-\f1p,p}(\BT)}: = \left\|e^{-t(-\D)^{1/2}}f\right\|_{L^p_{[0,\infty)}\dot{W}^{k,p}(\BT)}.
\label{eqn: def of W s p semi norm}
\eeq
We say $f\in W^{k-\f1p,p}(\BT)$ if and only if $f\in L^p(\BT)$ such that $\|f\|_{\dot{W}^{k-\f1p,p}(\BT)}<+\infty$.

Our main results are as follows.

\begin{thm}
\label{thm: local well-posedness}
Suppose $0<\mu<\nu$.
Let $G$ satisfy the assumptions in Section \ref{sec: problem formulation}.
Suppose $f_0,F_0 \in W^{2-\f1p,p}(\BT)$ for some $p\in (2,\infty)$.
Let
\beq
r = \f{1}{2\pi} \int_{\BT} f_0(\th)\,d\th,\quad
R = \f{1}{2\pi} \int_{\BT} F_0(\th)\,d\th.
\eeq
With $p_*$ be defined by \eqref{eqn: equation for p star}, let $c_*$ and $\tilde{c}_*$ be negative constants
\beq
c_* = -\f{1}{2\pi r}\int_{B_r} G(p_*(X))\,dX,\quad \tilde{c}_* = \f{r}{R}c_*,
\label{eqn: def of c characteristic normal derivative}
\eeq
which corresponds to negative speeds of concentric circular interfaces with radii $r$ and $R$ respectively (see e.g., \eqref{eqn: length of grad p star}).
Take $\d$ such that
\beq
\f{R-r}{100 R}\leq \d\leq \f{R-r}{10R},
\label{eqn: condition on delta}
\eeq
Define $h_0$ and $H_0$ as in \eqref{eqn: def of h and H}.

Suppose $h_0$ and $H_0$ satisfy that, with $\alpha = 1-\f2p$ and for some $\e >0$,
\beq
M:= \d^{-1}(\|h_0\|_{L^\infty(\BT)}+\|H_0\|_{L^\infty(\BT)})
+
\d^{\alpha - \e}\left(\|h_0\|_{\dot{W}^{2-\f1p,p}(\BT)}+\|H_0\|_{\dot{W}^{2-\f1p,p}(\BT)}\right)\leq M_*,
\label{eqn: smallness of fractional Sobolev norm of initial data}
\eeq
where $M_*$ is a small constant depending on $p$, $\e$, $\mu$, $\nu$, $R/|\tilde{c}_*|$, $G$ and $\d R^2$, but not directly on $\d$.
Then there exists $T>0$ depending on the above quantities and additionally on $\d$,
such that  the system \eqref{eqn: equation for inner interface new notation}-\eqref{eqn: initial condition} admits a strong solution
\beq
f,F\in C_{[0,T]}C^{1,\alpha}(\BT) \cap L^p_{[0,T]}W^{2,p}(\BT),
\eeq
with $\pa_t f, \pa_t F \in C_{[0,T]}C^{\alpha''}(\BT)$ for any $\alpha''<\min\{\f14,\alpha\}$.
The solution satisfies that, with $h$ and $H$ defined in \eqref{eqn: def of h and H},
\beq
\d^{-1}(\|h\|_{C_{[0,T]}L^\infty} +\|H\|_{C_{[0,T]}L^\infty} ) +\d^{\alpha-\e}\left(\|h\|_{C_{[0,T]}\dot{C}^{1,\alpha}} +\|H\|_{C_{[0,T]}\dot{C}^{1,\alpha}}\right)
\leq  C(p, G)M,
\label{eqn: bounds for the local solution}
\eeq
\beq
\d^{\alpha-\e}\left(\|\pa_t h\|_{L^p_{[0,T]}\dot{W}^{1,p}}+\|\pa_t H\|_{L^p_{[0,T]}\dot{W}^{1,p}}\right)
\leq C(p,\mu, \nu, G)M,
\label{eqn: bounds for the local solution time derivative}
\eeq
and
\beq
\d^{\alpha-\e}\left(\|h\|_{L^p_{[0,T]}\dot{W}^{2,p}} +\|H\|_{L^p_{[0,T]}\dot{W}^{2,p}}\right) \leq  C(p,\mu, \nu,R/|\tilde{c}_*|, G)M.
\label{eqn: bounds for the local solution L^p W^2,p}
\eeq

\end{thm}

\begin{rmk}
In the claim $\pa_t f, \pa_t F \in C_{[0,T]}C^{\alpha''}(\BT)$ $(\alpha''<\min\{\f14,\alpha\})$, we did not pursue the optimal range of the H\"{o}lder exponent $\alpha''$.
\end{rmk}
\begin{rmk}
We use $\d$ to characterize the relative thinness of the gap between $\g$ and $\tilde{\g}$.
Note that requiring $\d^{-1}(\|h_0\|_{L^\infty}+\|H_0\|_{L^\infty})\ll 1$ in \eqref{eqn: smallness of fractional Sobolev norm of initial data} seems very natural, as otherwise the two interfaces may touch or cross each other.
It is worthwhile to remark that the right hand side of \eqref{eqn: smallness of fractional Sobolev norm of initial data} does not deteriorate as $\d$ becomes smaller, in the sense that if all the model parameters and $R$ are fixed and we let $r\to R$ (so that $\d\to 0$), then the right hand side does not decrease to $0$.
Though $\d$ also shows up on the right hand side in the form of $\d R^2$, 
it will be clear later (see \eqref{eqn: derivation of space-time estimate for the source terms in the fixed point argument} in the proof of Theorem \ref{thm: local well-posedness}) that $M_*$ increases as $\d R^2$ decreases. 

In contrast, the smallness of $T$ has to depend on $\d$ directly: when $\d\ll 1$, we may need $T\ll 1$.
\end{rmk}
\begin{rmk}
In the 2-D Muskat problem, $\dot{W}^{1,\infty}$ and $\dot{H}^{3/2}$ are considered to be critical and scaling-invariant semi-norms \cite{gancedo2019muskat}.
Although our problem does not admit any scaling law, considering its similarity with the Muskat problem,
it seems to be the best thing one can do to prove well-posedness with initial data being small in $W^{1,\infty}(\BT)$- or $H^{3/2}(\BT)$-norms.
We note that in Theorem \ref{thm: local well-posedness}, the condition \eqref{eqn: smallness of fractional Sobolev norm of initial data} on the initial data is proposed in the way that, by interpolation, $C^{1,\b'}$-semi-norms of $h_0$ and $H_0$ are small for some $\b'>0$ depending on $p$ and $\e$ (see \eqref{eqn: def of beta'} and \eqref{eqn: L inf in time bound of Holder norm for proposed solution}).
%
In other words, although we are not able to prove well-posedness of our problem with smallness in the ``critical" spaces, partly because of the source term, we manage to do that in all the ``sub-critical" cases, which can be arbitrarily close to the ``critical" one --- note that $p>2$ and $\e>0$ are arbitrary.
\end{rmk}

Thanks to the estimates for the local solution, one can apply Theorem \ref{thm: local well-posedness} iteratively and show that local solutions exist for an arbitrary time period $\tilde{T}>0$ as long as $h_0$ and $H_0$ are correspondingly sufficiently small.
\begin{corol}\label{cor: long time small solution}
Under the assumptions of Theorem \ref{thm: local well-posedness}, for any $\tilde{T}>0$, if $h_0, H_0\in W^{2-\f1p,p}(\BT)$ satisfy $M\ll 1$, where the smallness depends on $p$, $\e$, $\mu$, $\nu$, $G$, $r$, $R$ and $\tilde{T}$, the local strong solution exists up to time $\tilde{T}$.
\end{corol}

Uniqueness of local solutions can be shown if $G$ is more regular.

\begin{thm}\label{thm: uniqueness}
Under the assumptions of Theorem \ref{thm: local well-posedness}, if in addition, $G\in C^{1,1}[0,+\infty)$, then the solution is unique.
\end{thm}

\subsection{Parabolic nature of the interface motion and scheme of the proof}
\label{sec: scheme of the proof}
To elucidate the hidden parabolicity of \eqref{eqn: equation for inner interface new notation}-\eqref{eqn: initial condition}, we linearize it around the radially symmetric configurations.

It is convenient to first derive equations for $[\va]'$ and $\phi'$ by taking derivative in \eqref{eqn: equation for jump of phi} and \eqref{eqn: equation for boundary potential on outer interface}.
Assuming $\g,[\va]\in C^1(\BT)$, we have
\beq
\f{d}{d\th}(\CD_\g [\va])|_{\g} = -\g'(\th)^\perp\cdot\mathcal{K}_\g[\va]'. 
\label{eqn: derivative of double layer potential}
\eeq
Indeed, by integration by parts,
\beq
\begin{split}
(\CD_\g [\va])|_{\g} 
= &\;-\f{1}{2\pi} \mathrm{p.v.}\int_{\BT} \pa_{\th'}[\mathrm{arg}((\g(\th)-\g(\th'))_1+i(\g(\th)-\g(\th'))_2)] \cdot[\va](\th')\,d\th'\\
= &\;-\f{1}{2\pi}\cdot \pi [\va](\th)+\f{1}{2\pi}\int_{\BT} \mathrm{arg}((\g(\th)-\g(\th'))_1+i(\g(\th)-\g(\th'))_2) \cdot[\va]'(\th')\,d\th'.
\end{split}
\eeq
Here the argument is defined such that its values at $\th=\pm \pi$ coincide.
In the last equality, we need the assumption $\g\in C^1(\BT)$.
Hence, using the fact that $[\va]\in C^1(\BT)$,
\beq
\begin{split}
\f{d}{d\th}(\CD_\g [\va])|_{\g}
= &\;-\f{1}{2}[\va]'+\f{1}{2\pi} \f{d}{d\th}\int_{\BT} \mathrm{arg}((\g(\th)-\g(\th'))_1+i(\g(\th)-\g(\th'))_2) \cdot[\va]'(\th')\,d\th'\\
= &\;\f{1}{2\pi} \mathrm{p.v.}\int_{\BT} \f{d}{d\th}\left( \mathrm{arg}((\g(\th)-\g(\th'))_1+i(\g(\th)-\g(\th'))_2) \right) \cdot[\va]'(\th')\,d\th',
\end{split}
\eeq
This justifies \eqref{eqn: derivative of double layer potential}. Next let
\beq
e_r := (\cos \th,\sin\th),\quad e_\th := (-\sin\th, \cos\th).
\label{eqn: def of e_r e_th}
\eeq
Then $[\va]'$ and $\phi'$ satisfy
\begin{align}
[\va]'
= &\;2A \left((f'(\th) e_r+ f(\th) e_\th)\cdot \na(\G*g)|_{\g}+\g'(\th)^\perp\cdot \CK_\g [\va]'+\g'(\th)^\perp\cdot \CK_{\g,\tilde{\g}} \phi'\right),
\label{eqn: equation for derivative of jump of potential final form}\\
\phi' = &\;-2\left((F'(\th) e_r+ F(\th) e_\th)\cdot \na(\G*g)|_{\tilde{\g}}
+\tilde{\g}'(\th)^\perp\cdot \CK_{\tilde{\g}} \phi'
+\tilde{\g}'(\th)^\perp\cdot \CK_{\tilde{\g},\g} [\va]'\right).
\label{eqn: equation for derivative of potential on outer interface}
\end{align}

Now we shall linearize the equations \eqref{eqn: equation for inner interface new notation},
\eqref{eqn: equation for outer interface new notation},
\eqref{eqn: equation for derivative of jump of potential final form}
and \eqref{eqn: equation for derivative of potential on outer interface} around the radially symmetric configurations, i.e., $f \equiv r$, $F \equiv R$, and  $[\va]' = \phi' \equiv 0$.
The following discussion is only formal and gives an overview of the analysis carried out in the rest of the paper.
Let us begin by collecting a few facts that will be justified in later sections.
\begin{itemize}
  \item It will be clear in Section \ref{sec: gradient estimate of growth potential} and Section \ref{sec: bounds for boundary potentials} that
  \beq
  e_r\cdot \na(\G*g)|_{\g}\approx c_*\quad \mbox{ and }\quad e_r\cdot \na(\G*g)|_{\tilde{\g}} \approx \tilde{c}_* :=\f{c_* r}{R}.
  \eeq
  Here $c_*$ and $\tilde{c}_*$ are constants defined in \eqref{eqn: def of c characteristic normal derivative}.
  \item Let $\CH$ be the Hilbert transform on $\BT$ \cite{grafakos2008classical}, i.e.,
      \beq
      \CH f(\th) := \f{1}{2\pi}\mathrm{p.v.}\int_\BT \cot\left(\f{\th-\th'}{2}\right)f(\th')\,d\th'.
      \eeq
  Then in Section \ref{sec: estimates for singular integral operators} we shall show
  \beq
  \g'\cdot \CK_\g \approx \f12\CH \quad\mbox{ and }\quad \tilde{\g}'\cdot \CK_{\tilde{\g}} \approx \f12\CH.
  \eeq

  \item Define $\CS$ to be a smoothing operator on $\BT$ with a Poisson kernel,
  \beq
  \CS \p(\th) = \f{1}{2\pi}P_{\f{r}{R}}*\p(\th) = \f{1}{2\pi}\int_{\BT}\f{1-\left(\f{r}{R}\right)^2}{1+\left(\f{r}{R}\right)^2-2\left(\f{r}{R}\right)\cos\xi}\p(\th-\xi)\,d\xi.
  \label{eqn: def of operator S}
  \eeq
  The notation $P_{\f{r}{R}}$ will be introduced in Section \ref{sec: estimates for interaction operators}.
  Then in Section \ref{sec: estimates for interaction operators} we shall see
  \begin{align}
  &\g'\cdot \CK_{\g,\tilde{\g}}\phi'\approx \f12\CH \CS\phi',&
  &\g'^\perp \cdot \CK_{\g,\tilde{\g}}\phi'\approx \f12\CS\phi',\\
  &\tilde{\g}'\cdot \CK_{\tilde{\g},\g}[\va]'\approx \f12\CH\CS[\va]',&
  &\tilde{\g}'^\perp \cdot \CK_{\tilde{\g},\g}[\va]'\approx -\f12\CS[\va]'.
  \end{align}
  \item The remaining terms in \eqref{eqn: equation for inner interface new notation},
\eqref{eqn: equation for outer interface new notation},
\eqref{eqn: equation for derivative of jump of potential final form}
and \eqref{eqn: equation for derivative of potential on outer interface} and the error made above are considered to be smaller or more regular, which will be omitted at this moment.
\end{itemize}
Putting these facts together, the linearized system 
can be written as
\begin{align}
\pa_t f+c_* = &\;-\f{1}{2r}\CH ([\va]'+\CS\phi'),
\label{eqn: linearized equation for inner interface}\\
\pa_t F+\f{c_*r}{R} = &\;-\f{1}{2R}\CH (\phi'+\CS[\va]').
\label{eqn: linearized equation for outer interface}\\
[\va]'
= &\; 2Ac_*f'+A\CS\phi',
\label{eqn: linearized equation for derivative of jump of potential}\\
\phi' = &\;-\f{2c_*r}{R}F'+\CS [\va]'.
\label{eqn: linearized equation for derivative of potential on outer interface}
\end{align}
See Section \ref{sec: bounds for boundary potentials} and Section \ref{sec: local well-posedness} for the complete equations.

Combining \eqref{eqn: linearized equation for derivative of jump of potential} and  \eqref{eqn: linearized equation for inner interface}, we obtain
\beq
\pa_t f+c_* =  -\f{Ac_*}{r}(-\D)^{1/2}f-\f{1+A}{2r}\CH \CS\phi'.
\label{eqn: simple linearized f equation}
\eeq
\eqref{eqn: simple linearized f equation} is a fractional heat equation only when $Ac_*>0$.
Note that the last term in \eqref{eqn: simple linearized f equation} and all those omitted ones are supposed to be small or regular source terms.
Since $c_*<0$, it is natural to believe that the motion of $\g$ can be well-posed only when $A<0$, i.e, $\mu<\nu$.


Similarly, by combining \eqref{eqn: linearized equation for outer interface} with \eqref{eqn: linearized equation for derivative of potential on outer interface}, 
\beq
\pa_t F+\f{c_*r}{R} = \f{c_*r}{R^2}(-\D)^{1/2}F-\f{1}{R}\CH \CS[\va]'.
\label{eqn: simple linearized F big equation}
\eeq
Note that it shows the smoothing of the outer interface not to depend on $A$, but only on the fact that $\f{c_*r^2}{R}<0$.

\begin{rmk}
\label{rmk: localized analysis}
The above formal derivation may be localized as long as the interfaces are locally graphs. By doing so we may be able to show that  the local parabolicity condition for the motion of $\g$ is $(\mu-\nu) \f{\pa p}{\pa \s_\Om}|_\g > 0$, while it is $\f{\pa p}{\pa \s_{\tilde{\Om}}}|_{\tilde{\g}} < 0$ for the motion of $\tilde{\g}$.
The former condition implies that when the less mobile cells are locally pushing the other one, we expect well-posedness in the motion of that local segment of $\g$.
This is in the same spirit as the Saffman-Taylor condition \cite{SaffmanTaylor} (see also the condition for well-posedness in \cite{siegel2004global}), and it is formulated in a more general setting in \cite{matioc2018well}.
The parabolicity condition $\f{\pa p}{\pa \s_{\tilde{\Om}}}|_{\tilde{\g}} < 0$ indicates that $\tilde{\g}$ may stay regular when it is pushed towards the vacuum, but otherwise it may lose regularity.
This fact echoes with many well-posedness and ill-posedness results on a variety of free boundary problems arising in, for instance, one-phase Hele-Shaw problems \cite{elliott1981variational,dibenedetto1984ill,howison1992complex,escher1997classical,kim2003uniqueness,cheng2012global} and porous medium equations \cite{caffarelli1980regularity,caffarelli1987lipschitz,caffarelli1990c1,vazquez2007porous,kim2018porous}.

In our problem, under the assumption of the almost radial symmetry, the parabolicity condition $(\mu-\nu)c_*>0$ derived for \eqref{eqn: simple linearized f equation} is an approximation of $(\mu-\nu) \f{\pa p}{\pa \s_\Om}|_\g > 0$, while the condition $\f{c_*r^2}{R}<0$ corresponding to \eqref{eqn: simple linearized F big equation} is an approximation of $\f{\pa p}{\pa \s_{\tilde{\Om}}}|_{\tilde{\g}} < 0$.
\end{rmk}

\begin{rmk}\label{rmk: bad case due to geometry}

\begin{figure}
\centering
\includegraphics[width=0.3\textwidth]{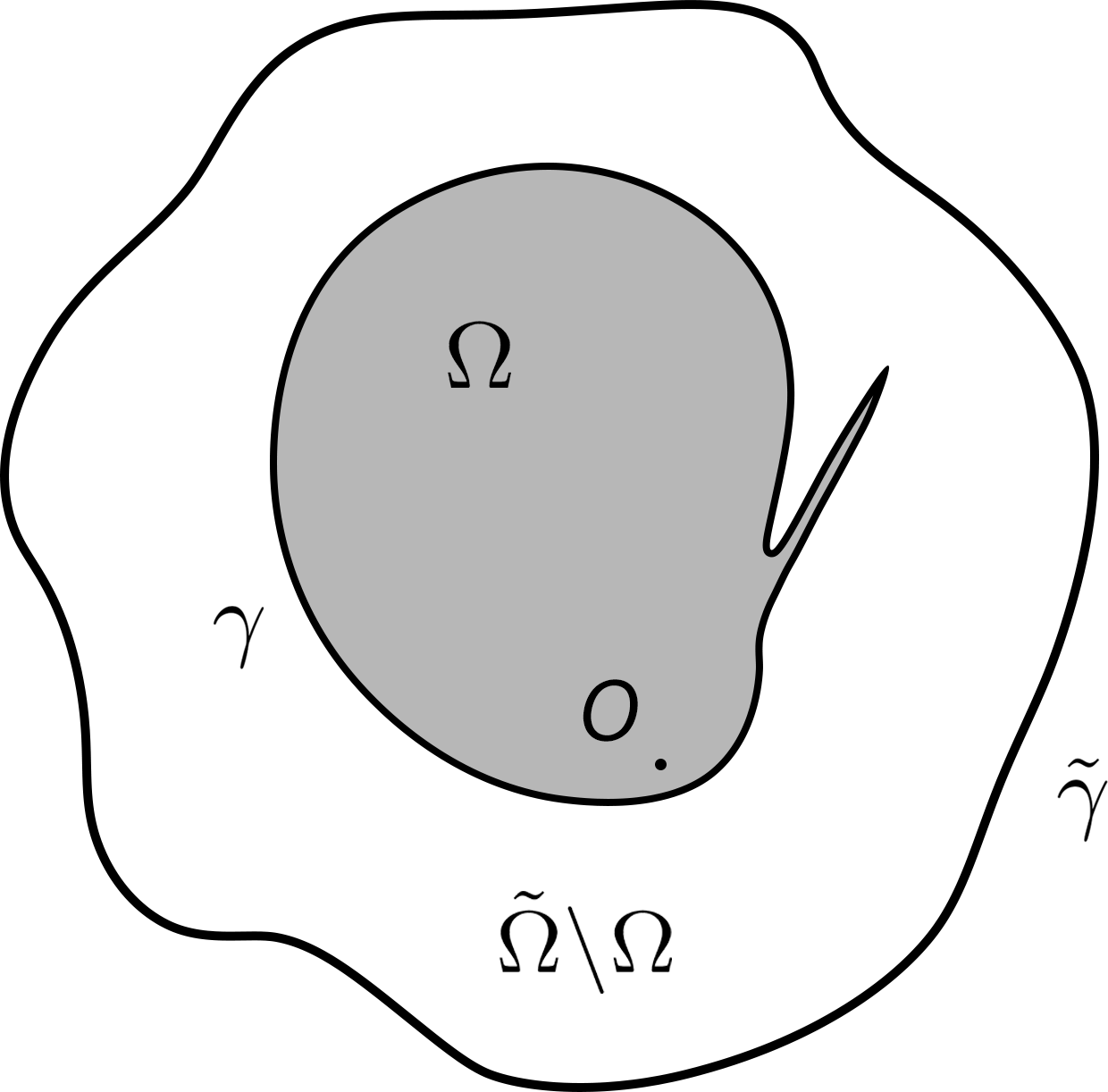}
\caption{A possible example exhibiting ill-posedness of motion of $\g$ when $\mu<\nu$.
Here $\Om$ consists of a big chuck and a thin branch; the latter is expected to move towards right.
Along a part of $\g$, more mobile normal cells are pushing less mobile tumor cells, i.e., $(\mu-\nu) \f{\pa p}{\pa \s_\Om}|_\g < 0$, making the local evolution of $\g$ ill-posed.
Note that both $\Om$ and $\tilde{\Om}$ are star-shaped with respect to the origin, denoted by $O$, so $\g$ and $\tilde{\g}$ are still graphs of functions of $\th$ in the polar coordinate at this moment.}
\label{Figure: tumor geometry bad case}
\end{figure}
From the above discussion, we can tell that $\mu<\nu$ is not sufficient for the parabolicity of the motion of $\g$, since the domain geometry determines how $\g$ moves in a nontrivial way.
Even if both $\Om$ and $\tilde{\Om}$ are assumed to be star-shaped with respect to the same point, which means $\g$ and $\tilde{\g}$ can be realized as graphs of functions of $\th$ in the polar coordinate, we can still manufacture such domain so that the parabolicity fails along some portion of $\g$.
A possible example is shown in Figure \ref{Figure: tumor geometry bad case}, where both $\Om$ and $\tilde{\Om}$ are star-shaped with respect to the origin, denoted by $O$ in the figure.
The tumor domain $\Om$ consists of a big chunk and a thin branch, where the branch is so thin that it does not significantly affect $p$.
Then it is conceivable that the thin branch will be pushed towards right under the expansion of the big chunk.
So along the part of $\g$ where the thin branch faces the main body of $\Om$, the more mobile normal cells are pushing the less mobile tumor cells (since $\mu<\nu$), which potentially gives rise to ill-posedness of the motion of $\g$ locally.

Given this, in order to guarantee well-posedness of the motion of $\g$, it is then reasonable to assume $\g$ and $\tilde{\g}$ are close to concentric circles, in which case the tumor cells should be always pushing the normal ones.
\end{rmk}

The parabolicity of \eqref{eqn: simple linearized f equation} and \eqref{eqn: simple linearized F big equation}
is sufficient to prove existence of local solutions in Section \ref{sec: local well-posedness}, and then uniqueness in Section \ref{sec: uniqueness}.
The proof of the local existence uses two layers of fixed-point arguments.
We sketch it as follows.
\begin{enumerate}
  \item Fix a pair of interface dynamics $f$ and $F$.
  \item First we need to solve for $[\va]'$ and $\phi'$ associated with the domain defined by $f$ and $F$.
      To do that, in Section \ref{sec: bounds for boundary potentials}, we apply a fixed-point argument to static equations \eqref{eqn: introducing remainder in va equation} and \eqref{eqn: introducing remainder in phi equation} (or equivalently, \eqref{eqn: equation for derivative of jump of potential final form} and \eqref{eqn: equation for derivative of potential on outer interface}) with the variable $([\va]',\phi')$.
      In this argument, we need estimates for the remainder terms that are omitted in \eqref{eqn: linearized equation for derivative of jump of potential} and \eqref{eqn: linearized equation for derivative of potential on outer interface}, which turn out to be small.
  \item Once $[\va]'$ and $\phi'$ are well-defined and their estimates are derived, we use them to bound $\CH \CS\phi'$ and $\CH \CS[\va]'$ in \eqref{eqn: simple linearized f equation} and \eqref{eqn: simple linearized F big equation} as well as all the remainder terms omitted there (see \eqref{eqn: backbone equation for h} and \eqref{eqn: backbone equation for H big} for the complete equations).
      They altogether will be put as the source terms in some fractional heat equations similar to \eqref{eqn: simple linearized f equation} and \eqref{eqn: simple linearized F big equation} in order to construct a new pair of interface dynamics, $\tilde{f}$ and $\tilde{F}$.
      See \eqref{eqn: equation for h dag}-\eqref{eqn: decomposition of solution}.
      We then show in Section \ref{sec: local well-posedness} that the map $(f,F)\mapsto (\tilde{f},\tilde{F})$ has a fixed-point, which is a local solution.
  \item In this process, bounds for all the remainder terms will rely on estimates derived in Sections \ref{sec: estimate for pressure in circular geometry}-\ref{sec: estimates for interaction operators}.
      See Section \ref{sec: organization} for what are exactly covered in them.
\end{enumerate}
The proof of the uniqueness boils down to showing that $[\va]'$, $\phi'$ and all the remainder terms above depend in a Lipschitz manner on the interface configurations. Indeed, what we prove is a stability-type estimate for $f$ and $F$ based on that of the fractional heat equation.
We carry out this idea in Section \ref{sec: uniqueness} with a twist in order to slightly reduce complexity of the proof.

\subsection{Organization of the paper}
\label{sec: organization}
In Section \ref{sec: estimate for pressure in circular geometry}, we first study the pressure $p$ in an almost radially symmetric geometry by elliptic regularity theory.
In Section \ref{sec: gradient estimate of growth potential}, we derive estimates concerning gradients of the growth potential $\G*g$ (c.f., \eqref{eqn: representation of phi} and \eqref{eqn: def of source term g}) restricted to inner and outer interfaces.
Section \ref{sec: estimates for singular integral operators} is devoted to proving estimates for singular integral operators $\CK_\g$ and $\CK_{\tilde{\g}}$, while Section \ref{sec: estimates for interaction operators} establishes estimates for integral operators $\CK_{\g,\tilde{\g}}$ and $\CK_{\tilde{\g},\g}$.
Section \ref{sec: bounds for boundary potentials} shows well-definedness of $[\va]$ and $\phi$ as well as their estimates.
Finally, we prove existence of the local solution in Section \ref{sec: local well-posedness}, and uniqueness in Section \ref{sec: uniqueness}.
Some auxiliary estimates and non-essential lengthy proofs are collected in Appendices.

\section{Pressure in an Almost Radially Symmetric Geometry}
\label{sec: estimate for pressure in circular geometry}

In this section, we focus on the elliptic equation \eqref{eqn: equation for p simplified case} and \eqref{eqn: p boundary condition} for the pressure $p$ in $\tilde{\Om}$.  The goal is to quantify the fact that if $\Om$ and $\tilde{\Om}$ are close to two concentric discs then $p$ should be almost radially symmetric.

\subsection{Geometric preliminaries}
First we introduce a diffeomorphism to transform the physical domain into a reference domain that is perfectly radially symmetric.
Given $\d$ satisfying \eqref{eqn: condition on delta}, define a cut-off function $\eta_\d\in C_0^\infty([0,+\infty))$, such that $\eta_\d\in [0,1]$ is only supported on $[1-2\d,1+2\d]$, $\eta_\d=1$ on $[1-\d,1+\d]$, and for some universal constant $C$,
\beq
\d|\eta_\d'|+\d^2 |\eta_\d''|\leq C.
\eeq

Let $X=(\r\cos\om,\r\sin\om)\in \BR^2$ be a point in the reference coordinate, with $\r = |X|$.
Define
\beq
x(X) = \left[1+h(\om)\eta_\d\left(\f{\r}{r}\right)+H(\om)\eta_\d\left(\f{\r}{R}\right)\right]X=: \zeta(X)X.
\label{eqn: change of variables to the reference coordinate}
\eeq
where $h$ and $H$ are given in \eqref{eqn: def of h and H}.
In other words, $x$ deforms the reference domain in the radial direction only in annuli around $\pa B_r$ and $\pa B_R$.
It depends only on $\g$ in the annulus $B_{r(1+2\d)}\backslash B_{r(1-2\d)}$, and only on $\tilde{\g}$ in $B_{R}\backslash B_{R(1-2\d)}$;
$x(X) = X$ elsewhere.
We may also write $\zeta(X)$ as $\zeta(\r,\om)$.
We know that $x(X)$ is a diffeomorphism from $\BR^2$ to itself provided that $\zeta(\r,\om)\r$ is strictly increasing in $\r$ for all $\om\in\BT$.
This is true if
oscillations of $\g$ and $\tilde{\g}$ in the radial direction are small with respect to the gap between them, i.e.,
\beq
\d^{-1}(\|h\|_{L^\infty(\BT)}+\|H\|_{L^\infty(\BT)})\ll 1.
\label{eqn: smallness of oscillation}
\eeq
Under this assumption, it is clear that $x(X)$ maps $B_r$, $B_R$, $\pa B_r$ and $\pa B_R$ to $\Om$, $\tilde{\Om}$, $\g$ and $\tilde{\g}$, respectively.
We denote its inverse to be $X(x)$.

\subsection{Pressure in the reference coordinate}
Define
\beq
\tilde{p}(X) := p(x(X)).
\label{eqn: pressure in the reference coordinate}
\eeq
By \eqref{eqn: equation for p simplified case}, $\tilde{p}$ in the $X$-coordinate satisfies
\beq
-\f{\pa X_k}{\pa x_i}\na_{X_k}\left(a\f{\pa X_j}{\pa x_i}\na_{X_j} \tilde{p}\right) = G(\tilde{p})\chi_{B_r}\quad \mbox{in }B_R,\quad \tilde{p}|_{\pa B_R} = 0.
\label{eqn: equation for p tilde}
\eeq
Here the summation convention applies to repeated indices.
We also used the notations
\beq
a (X)= \mu \chi_{B_r}(X)+\nu \chi_{B_R\backslash B_r}(X)
\label{eqn: def of a}
\eeq
and
\beq
\f{\pa X_k}{\pa x_i} = \left(\f{\pa X}{\pa x}\right)_{ki}=\left[\left(\f{\pa x}{\pa X}\right)^{-1}\right]_{ki},
\eeq
which are both functions in $X$.
We may write $a = a(\r)$.

In order to show $\tilde{p}$ is almost radially symmetric, we shall compare it with a radially symmetric solution $p_*$ defined as follows.
\begin{lem}
\label{lem: estimate for p_*}
Let $p_*$ be the $H^1$-weak solution of
\beq
-\na_{X_i}\left(a\na_{X_i} p_*\right) = G(p_*)\chi_{B_r}\quad \mbox{in }B_R,\quad p_*|_{\pa B_R} = 0.
\label{eqn: equation for p star}
\eeq
Then
\begin{enumerate}
\item $p_*$ is radially symmetric, i.e., $p_* = p_*(\r)$, and $p_*\in W^{1,\infty}(B_R)$.
\item $p_*\in [0,p_M]$ and $p_*$ is decreasing in $\r$.
\item In $\overline{B}_R\backslash B_r$,
\beq
p_*(\r) = -\ln\left(\f{\r}{R}\right)\cdot\f{1}{2\pi\nu}\int_{B_r}G(p_*)\,dx.
\label{eqn: representation of p_* on the annulus}
\eeq

\item For $\r\in [0,r]$, 
\beq
\int_{B_\r}G(p_*)\,dx\leq C\r^2 \min\left\{ 1, \mu^{1/2} r^{-1}\right\},
\label{eqn: bound for the integral of G(p_*)}
\eeq
where $C$ only depends on $G$.

\item For $\r \in [0,r^-]$,
\beq
|\na p_*|(\r)\leq C\min\left\{\mu^{-1}\r, \mu^{-1/2}\right\}.
\label{eqn: bound on grad p_* for inner domain}
\eeq
For $\r\in[r^+,R]$,
\beq
|\na p_*|(\r)\leq C\r^{-1}\min\left\{\nu^{-1}r^2, \mu^{1/2}\nu^{-1}r \right\}.
\label{eqn: bound on grad p_* for outer domain}
\eeq
Here the constants $C$ only depend on $G$.
Note that $\na p_*$ has discontinuity across $\pa B_r$, so we use $|\na p_*|(r^{\pm})$ to distinguish the gradients taken from two sides of $\pa B_r$.
\end{enumerate}
\begin{proof}
The radial symmetry of $p_*$ can be justified by a symmetrization argument in the variational formulation of \eqref{eqn: equation for p star}.
$W^{1,\infty}$-regularity of $p_*$ follows from \cite{li2000gradient}.
The fact that $p_*\in [0,p_M]$ and monotonicity of $p_*$ follows from the maximum principle.
\eqref{eqn: representation of p_* on the annulus} is obvious since $p_*$ is harmonic in $\overline{B}_R\backslash B_r$.

The first bounds in \eqref{eqn: bound for the integral of G(p_*)}-\eqref{eqn: bound on grad p_* for outer domain} follow from the trivial fact $|G|\leq C$ and
\beq
|\na p_*|(\r) = |\pa_\r p_*(\r)| = \f{1}{2\pi a(\r) \r}\int_{B_\r\cap B_r} G(p_*)\,dx. 
\label{eqn: length of grad p star}
\eeq

To show the second bounds in \eqref{eqn: bound on grad p_* for inner domain} and \eqref{eqn: bound on grad p_* for outer domain}, define $\mathcal{G}$ to be the anti-derivative of $G$ with $\mathcal{G}(0) = 0$.
Obviously, $\mathcal{G}\geq 0$ on $[0,p_M]$, attaining its maximum at $p_M$. 
Since in the polar coordinate, $p_*$ solves $-\mu\pa_\r(\r \pa_\r p_*) = \r G(p_*)$ on $[0,r)$, by multiplying with $\r^{-1}\pa_\r p_*$,
\beq
\mu \r^{-1}|\pa_\r p_*|^2+\mu \pa_\r p_*\pa^2_\r p_* + G(p_*)\pa_\r p_* = 0.
\eeq
Taking integral in $\r$ from $0$ to $\tau\in [0,r^-]$ yields
\beq
\mu \int_0^\tau\r^{-1}|\pa_\r p_*|^2\,d\r+\f{\mu}{2} |\pa_\r p_*(\tau)|^2 + \mathcal{G}(p_*(\tau)) = \mathcal{G}(p_*(0)).
\eeq
Hence,
\beq
\|\pa_\r p_*\|_{L^\infty(B_r)}^2\leq 2\mu^{-1}\mathcal{G}(p_M).
\eeq
By the nature of discontinuity of $\pa_\r p_*$ across $\pa B_r$, $a(\r) \pa_\r p_*$ is continuous at $\r = r$.
Hence, for $\r\in [r^+,R]$, $\pa_\r p_*(\r) = \f{\mu r}{\nu \rho}\pa_\r p_*|_{\r = r_-}$.
This gives the second bound in \eqref{eqn: bound on grad p_* for outer domain}. 
%
Finally, the second bound in \eqref{eqn: bound for the integral of G(p_*)} follow from \eqref{eqn: bound on grad p_* for outer domain}, \eqref{eqn: length of grad p star} and the fact that $G(p_*(\r))$ is increasing in $\r$.
\end{proof}
\end{lem}

In order to derive a bound for $(\tilde{p}-p_*)$, we need estimates concerning $x(X)$ and its inverse. 
Denote
\begin{align}
m_0 :=&\; \d^{-1}\|h\|_{L^\infty(\BT)}+\|h'\|_{L^\infty(\BT)},\label{eqn: bound assumption on the Lipschitz norm of h}\\
M_0 :=&\; \d^{-1}\|H\|_{L^\infty(\BT)}+\|H'\|_{L^\infty(\BT)}.
\label{eqn: bound assumption on the Lipschitz norm of H}
\end{align}

\begin{lem}\label{lem: L inf estimate of jacobi}
Suppose $h,H\in W^{1,\infty}(\BT)$ satisfy that $m_0+M_0\ll 1$.
%
%
Then
\begin{align}
\left\|\f{\pa X}{\pa x}-Id\right\|_{L^\infty(B_{r(1+2\d)}\backslash B_{r(1-2\d)})}\leq &\; Cm_0,
\label{eqn: L inf bound for small deformation in inner disc}\\
\left\|\f{\pa X}{\pa x}-Id\right\|_{L^\infty(B_R\backslash B_{R(1-2\d)})}\leq &\;
CM_0,
\label{eqn: L inf bound for small deformation}
\end{align}
and
\begin{align}
\left\|\na_{X_k}\f{\pa X_k}{\pa x_i}\right\|_{L^\infty(B_{r(1+2\d)}\backslash B_{r(1-2\d)})}\leq &\;C(\d r)^{-1}m_0,
\label{eqn: L inf for divergence of jacobi in inner disc}\\
\left\|\na_{X_k}\f{\pa X_k}{\pa x_i}\right\|_{L^\infty(B_R\backslash B_{R(1-2\d)})}\leq &\;C(\d R)^{-1}M_0.
\label{eqn: L inf for divergence of jacobi}
\end{align}
The constants $C$ are all universal.
\begin{proof}
The proof is a straightforward calculation.
By \eqref{eqn: change of variables to the reference coordinate}, 
\beq
\f{\pa x}{\pa X} = \zeta \cdot Id + X\otimes \na \zeta.
\eeq
Its inverse is given by
\beq
\begin{split}
\f{\pa X}{\pa x} = &\;(\zeta^2 + \zeta \r\pa_\r \zeta)^{-1}((\zeta+\r\pa_\r \zeta) Id - X\otimes \na \zeta)\\
= &\; \zeta^{-1} Id - (\zeta^2 + \zeta \r\pa_\r \zeta)^{-1} X\otimes \na \zeta.
\end{split}
\label{eqn: inverse of jacobi matrix}
\eeq
On the other hand, since
$\na_{X_k}(\f{\pa X_k}{\pa x_i}\cdot \f{\pa x_i}{\pa X_j}) = \na_{X_k} \d_{kj} = 0$,
we deduce that
\beq
\begin{split}
&\;\na_{X_k}\left(\f{\pa X_k}{\pa x_l}\right)\\
=&\;-\f{\pa X_j}{\pa x_l}\f{\pa X_k}{\pa x_i}\cdot \na_{X_k}\left(\f{\pa x_i}{\pa X_j}\right)\\
=&\;-\f{\pa X_j}{\pa x_l}\cdot
(\zeta^2 + \zeta \r\pa_\r \zeta)^{-1}((\zeta+\r\pa_\r\zeta) \d_{ki}- X_k(\na \zeta)_i)\cdot  \na_{X_k}(\zeta\d_{ij} + X_i \nabla_{X_j} \zeta)\\
=&\;-\f{\pa X_j}{\pa x_l}
(\zeta^2 + \zeta \r \pa_\r \zeta)^{-1}\na_{X_j}(\zeta^2 + \zeta \r \pa_\r\zeta).
\end{split}
\label{eqn: div of jacobi matrix}
\eeq

By \eqref{eqn: change of variables to the reference coordinate}, 
\begin{align}
\zeta - 1 = &\;h\eta_\d\left(\f{\r}{r}\right)+ H\eta_\d\left(\f{\r}{R}\right),\label{eqn: estimate of small deformation}\\
\r\pa_\r \zeta = &\;h(\om)\cdot \f{\r}{r}\eta_\d'\left(\f{\r}{r}\right) +H(\om)\cdot \f{\r}{R}\eta_\d'\left(\f{\r}{R}\right).
\label{eqn: calculation of rho partial rho zeta}
\end{align}
Thanks to the smallness of $m_0$ and $M_0$, 
\beq
|\z-1|+|(\zeta^2 + \zeta \r \pa_\r \zeta)-1|\ll 1.
\label{eqn: denominator of jacobi is close to 1}
\eeq
Hence, by the last line in \eqref{eqn: inverse of jacobi matrix},
\beq
\left|\f{\pa X}{\pa x}-Id\right| \leq C(|1-\z|+\r|\na \z|), 
\label{eqn: bound for deformation tensor}
\eeq
We calculate
\beq
\begin{split}
\na \z
= &\;\left[h(\om)\cdot \f{1}{r}\eta_\d'\left(\f{\r}{r}\right) +H(\om)\cdot \f{1}{R}\eta_\d'\left(\f{\r}{R}\right) \right] e_r\\
&\;+\left[h'(\om)\cdot \eta_\d\left(\f{\r}{r}\right)+H'(\om)\cdot \eta_\d\left(\f{\r}{R}\right)\right]\r^{-1} e_\theta,
\end{split}
\label{eqn: grad zeta}
\eeq
where $e_r$ and $e_\th$ are defined in \eqref{eqn: def of e_r e_th}.
Then \eqref{eqn: L inf bound for small deformation in inner disc} and \eqref{eqn: L inf bound for small deformation} follow easily. 

Similarly, \eqref{eqn: div of jacobi matrix} implies that
\beq
\left|\na_{X_k}\left(\f{\pa X_k}{\pa x_l}\right)\right|
\leq C|\na (\z^2+\z\r\pa_\r \z)|
\leq C(|\na\z|+|\na(\r\pa_\r\z)|).
\label{eqn: pointwise bound for divergence of jacobi}
\eeq
Then \eqref{eqn: L inf for divergence of jacobi in inner disc} and \eqref{eqn: L inf for divergence of jacobi} follow from \eqref{eqn: grad zeta} and the calculation
\beq
\begin{split}
\na (\r\pa_\r \zeta) = &\;\left[h(\om)\cdot \f{\r}{r^2}\eta_\d''\left(\f{\r}{r}\right) +H(\om)\cdot \f{\r}{R^2}\eta_\d''\left(\f{\r}{R}\right)\right]\cdot e_r\\
&\;+\left[h(\om)\cdot \f{1}{r}\eta_\d'\left(\f{\r}{r}\right) +H(\om)\cdot \f{1}{R}\eta_\d'\left(\f{\r}{R}\right)\right]\cdot e_r\\
&\;+\left[h'(\om)\cdot \f{\r}{r}\eta_\d'\left(\f{\r}{r}\right) +H'(\om)\cdot \f{\r}{R}\eta_\d'\left(\f{\r}{R}\right)\right]\cdot \r^{-1}e_\th.
\end{split}
\label{eqn: grad of radial derivative of zeta}
\eeq
\end{proof}
\end{lem}

By \eqref{eqn: equation for p tilde} and \eqref{eqn: equation for p star}, $(\tilde{p}-p_*)$ solves
\beq
\begin{split}
-\na_{X_k}&\left(a\f{\pa X_k}{\pa x_i}\f{\pa X_j}{\pa x_i}\na_{X_j} (\tilde{p}-p_*)\right) +c(\tilde{p}-p_*)\\
&\;=\na_{X_k}\left[a\left(\f{\pa X_k}{\pa x_i}\f{\pa X_j}{\pa x_i}-\d_{kj}\right)\na_{X_j}p_*\right]-\na_{X_k}\f{\pa X_k}{\pa x_i}\cdot a\f{\pa X_j}{\pa x_i}\na_{X_j}\tilde{p},
\end{split}
\label{eqn: equation for the difference between p tilde p star}
\eeq
in the reference coordinate with boundary condition $(\tilde{p}-p_*)|_{\pa B_R} = 0$.
Here
\beq
c(X) := -\f{G(\tilde{p})-G(p_*)}{\tilde{p}-p_*}\chi_{B_r}\geq 0 
\eeq
due to the assumptions on $G$.
Then we can prove stability of the pressure with respect to the domain geometry around the radially symmetric case.

\begin{lem}\label{lem: estimates for difference between tilde p and p_*}
Under the assumptions of Lemma \ref{lem: L inf estimate of jacobi},
\beq
\|\na(\tilde{p}-p_*)\|_{L^2(B_R)}
\leq C(m_0+M_0) (\d R^2)^{1/2},
\eeq
where $C = C(\mu,\nu,G)$.
\begin{proof}
We take inner product of $(\tilde{p}-p_*)$ and \eqref{eqn: equation for the difference between p tilde p star} and integrate by parts,
\beq
\begin{split}
&\;\int_{B_R} a\left|\f{\pa X_j}{\pa x_i}\na_{X_j}(\tilde{p}-p_*)\right|^2\,dX
+\int_{B_r}c|\tilde{p}-p_*|^2\,dX\\
= &\;-\int_{B_R}\na_{X_k}(\tilde{p}-p_*)\cdot a\left(\f{\pa X_k}{\pa x_i}\f{\pa X_j}{\pa x_i}-\d_{kj}\right)\na_{X_j} p_*\,dX\\
&\;-\int_{B_R}(\tilde{p}-p_*)\na_{X_k}\f{\pa X_k}{\pa x_i}\cdot a\f{\pa X_j}{\pa x_i}[\na_{X_j}(\tilde{p}-p_*)+\na_{X_j}p_*]\,dX.
\end{split}
\eeq
By the definition of $a$ in \eqref{eqn: def of a}, the assumptions on $G$, Lemma \ref{lem: L inf estimate of jacobi} and H\"{older}'s inequality,
\beq
\begin{split}
&\;\|\na(\tilde{p}-p_*)\|_{L^2(B_R)}^2\\
\leq &\;C[m_0(\d r^2)^{1/2}+M_0(\d R^2)^{1/2}]\|\na(\tilde{p}-p_*)\|_{L^2(B_R)}
\|\na p_*\|_{L^\infty(B_R)}
\\
&\;+C(\d r)^{-1}m_0\cdot \|\tilde{p}-p_*\|_{L^2(B_{r(1+2\d)}\backslash B_{r(1-2\d)})}
\|\na(\tilde{p}-p_*)\|_{L^2(B_R)}\\
&\;+C(\d R)^{-1}M_0\cdot \|\tilde{p}-p_*\|_{L^2(B_R\backslash B_{R(1-2\d)})}
\|\na(\tilde{p}-p_*)\|_{L^2(B_R)}\\
&\;+C(\d r)^{-1}m_0\cdot \|\tilde{p}-p_*\|_{L^2(B_{r(1+2\d)}\backslash B_{r(1-2\d)})}\cdot ( \d r^2)^{1/2}
\|\na p_*\|_{L^\infty(B_R)}\\
&\;+C(\d R)^{-1}M_0\cdot \|\tilde{p}-p_*\|_{L^2(B_R\backslash B_{R(1-2\d)})}
\cdot (\d R^2)^{1/2}
\|\na p_*\|_{L^\infty(B_R)},
\end{split}
\label{eqn: H^1 estimate for the difference crude form}
\eeq
where $C = C(\mu,\nu,G)$.
We proceed in two different cases.
\setcounter{case}{0}
\begin{case}
If $ R/2\leq r< R$, by \eqref{eqn: condition on delta} and Poincar\'{e} inequality on thin domains,
\beq
\begin{split}
\|\tilde{p}-p_*\|_{L^2(B_R\backslash B_{r(1-2\d)})}\leq &\;C(R-r(1-2\d))\|\na (\tilde{p}-p_*)\|_{L^2(B_R)}\\
\leq &\;C(\d r)\|\na (\tilde{p}-p_*)\|_{L^2(B_R)}. 
\end{split}
\label{eqn: Poincare inequality on thin domains}
\eeq
%
Combining this with \eqref{eqn: H^1 estimate for the difference crude form} yields
\beq
\begin{split}
&\;\|\na(\tilde{p}-p_*)\|_{L^2(B_R)}^2\\
\leq &\;C(m_0+M_0)(\d R^2)^{1/2}\|\na(\tilde{p}-p_*)\|_{L^2(B_R)}\|\na p_*\|_{L^\infty(B_R)}\\
&\;+C(m_0+M_0) \|\na(\tilde{p}-p_*)\|_{L^2(B_R)}^2. 
\end{split}
\label{eqn: H^1 estimate for the difference}
\eeq
By Young's inequality, smallness of $m_0$ and $M_0$ and Lemma \ref{lem: estimate for p_*}, the desired estimate follows.
\end{case}
\begin{case}
If otherwise $r< R/2$, by \eqref{eqn: condition on delta}, $\d \geq C$ for some universal constant $C>0$.
We shall first derive a bound for $\|\tilde{p}-p_*\|_{L^\infty(B_R)}$.

Recall that $p$ solves \eqref{eqn: equation for p simplified case} and \eqref{eqn: p boundary condition}.
Taking inner product of \eqref{eqn: equation for p simplified case} and $p$, we find that 
\beq
\|\na p\|_{L^2(\tilde{\Om})}^2 \leq C\int_\Om G(p)p\,dx\leq C|\Om|\leq Cr^2,
\eeq
where $C = C(\mu,\nu,G)$.
Hence, by Lemma \ref{lem: L inf estimate of jacobi}, in the reference coordinate,
\beq
\left\|\f{\pa X_j}{\pa x_i}\na_{X_j}\tilde{p}\right\|_{L^2(B_R)}\leq Cr.
\label{eqn: bound for H^1 norm of p tilde}
\eeq

Now consider \eqref{eqn: equation for the difference between p tilde p star}. By boundedness of weak solutions \cite[Theorem 8.16]{gilbarg2015elliptic},
\beq
\begin{split}
&\;\|\tilde{p}-p_*\|_{L^\infty(B_R)}\\
\leq &\; C\left(R^{1/2}\left\|a\left(\f{\pa X_k}{\pa x_i}\f{\pa X_j}{\pa x_i}-\d_{kj}\right)\na_{X_j}p_*\right\|_{L^4(B_R)}+R\left\|\na_{X_k}\f{\pa X_k}{\pa x_i}\cdot a\f{\pa X_j}{\pa x_i}\na_{X_j}\tilde{p}\right\|_{L^2(B_R)}\right).
\end{split}
\eeq
Applying Lemma \ref{lem: estimate for p_*}, Lemma \ref{lem: L inf estimate of jacobi}, \eqref{eqn: bound for H^1 norm of p tilde} and the fact $\d\geq C$,
\beq
\begin{split}
&\;\|\tilde{p}-p_*\|_{L^\infty(B_R)}\\
\leq &\; CR^{1/2} (m_0 (\d r^2)^{1/4}+ M_0 (\d R^2)^{1/4}) \|\na p_*\|_{L^\infty(B_R)}\\
&\;+CR(m_0(\d r)^{-1}+M_0(\d R)^{-1})\cdot r\\
\leq 
&\;C(m_0+M_0)(\d R^2)^{1/2},
\end{split}
\label{eqn: L^inf bound of tilde p - p* when R is way larger than r}
\eeq
where $C = C(\mu,\nu,G)$.

With this estimate and Lemma \ref{lem: estimate for p_*}, \eqref{eqn: H^1 estimate for the difference crude form} becomes
\beq
\begin{split}
&\;\|\na(\tilde{p}-p_*)\|_{L^2(B_R)}^2\\
\leq &\;C(m_0+M_0)(\d R^2)^{1/2}\|\na(\tilde{p}-p_*)\|_{L^2(B_R)}\|\na p_*\|_{L^\infty(B_R)}\\
&\;+C(m_0+ M_0)\|\tilde{p}-p_*\|_{L^\infty(B_R)} \|\na(\tilde{p}-p_*)\|_{L^2(B_R)}\\
&\;+C (m_0+ M_0) \|\tilde{p}-p_*\|_{L^\infty(B_R)}\cdot (\d R^2)^{1/2}\|\na p_*\|_{L^\infty(B_R)}\\
\leq &\;C(m_0+M_0)(\d R^2)^{1/2}\|\na(\tilde{p}-p_*)\|_{L^2(B_R)}\\
&\;+C (m_0+ M_0) ^2  \d R^2.
%
\end{split}
\eeq
Then the desired estimate follows from Young's inequality.
\end{case}
\end{proof}
\begin{rmk}\label{rmk: delta R^2 is bounded}
The above estimate involves $\d R^2$.
If $\d^{-1}(\|h\|_{L^\infty}+\|H\|_{L^\infty})\ll 1$, by \eqref{eqn: condition on delta}, there exist universal constants $0<c_1<c_2$, such that
\beqo
c_1|\tilde{\Om}_0\backslash \Om_0|\leq \d R^2 \leq c_2|\tilde{\Om}_0\backslash \Om_0|.
\eeqo
It is noteworthy that $|\tilde{\Om}_t\backslash \Om_t|$ is constant in time provided that $\g$ and $\tilde{\g}$ have sufficient regularity.
This is because the transporting velocity field $-\na \va$ in $\tilde{\Om}_t\backslash \Om_t$ is divergence-free.
\end{rmk}
\end{lem}

\subsection{More stability results}
\label{sec: stability of interface velocity}
For later use, further stability results are presented here for the interface velocities and the pressure, with respect to the interface configurations.

Fix $0<r<R$ and take $\d$ as in \eqref{eqn: condition on delta}.
Given two pairs of interface configurations $(\g_{1},\tilde{\g}_{1})$ and $(\g_{2},\tilde{\g}_{2})$, let $(h_{1},H_{1}),(h_{2},H_{2})$ be defined as in \eqref{eqn: parameterization of gamma}-\eqref{eqn: def of h and H}.
As in \eqref{eqn: bound assumption on the Lipschitz norm of h} and \eqref{eqn: bound assumption on the Lipschitz norm of H}, 
we define $m_{0,i}$ and $M_{0,i}$ 
that correspond to $h_i$ and $H_i$ $(i=1,2)$.
We additionally introduce for some $\alpha \in (0,1)$,
\begin{align}
m_{\alpha,i} := &\; \d^{-1}\|h_i\|_{L^\infty}+\d^{\alpha}\|h_i'\|_{\dot{C}^\alpha},\label{eqn: def of m_alpha}\\
M_{\alpha,i} := &\; \d^{-1}\|H_i\|_{L^\infty}+\d^{\alpha}\|H_i'\|_{\dot{C}^\alpha}.\label{eqn: def of M_alpha}
\end{align}
Also denote
\begin{align}
\D m_{0} :=&\;\d^{-1}\|h_1-h_2\|_{L^\infty(\BT)}+\|h_1'-h_2'\|_{L^\infty(\BT)},\label{eqn: def of normalized W 1 inf difference of two h}\\
\D M_{0} :=&\;\d^{-1}\|H_1-H_2\|_{L^\infty(\BT)}+\|H_1'-H_2'\|_{L^\infty(\BT)},\label{eqn: def of normalized W 1 inf difference of two H large}\\
\D m_{\alpha} :=&\;\d^{-1}\|h_1-h_2\|_{L^\infty(\BT)}+\d^{\alpha}\|h_1'-h_2'\|_{\dot{C}^\alpha(\BT)},
\label{eqn: def of normalized C 1 alpha difference of two h}\\
\D M_{\alpha} :=&\;\d^{-1}\|H_1-H_2\|_{L^\infty(\BT)}+\d^{\alpha}\|H_1'-H_2'\|_{\dot{C}^\alpha(\BT)}.
\label{eqn: def of normalized C 1 alpha difference of two H large}
\end{align}

Then we can show
\begin{lem}\label{lem: stability of the interface velocities}
Suppose $(h_{1},H_{1}),(h_{2},H_{2})\in C^{1,\alpha}(\BT)\times C^{1,\alpha}(\BT)$ for some $\alpha\in (0,\f14)$, 
satisfying that for $i = 1,2$, $m_{\alpha,i}+M_{\alpha,i}\ll 1$.
Then
\beq
\|\pa_t h_1-\pa_t h_2\|_{C^\alpha(\BT)}+\|\pa_t H_1-\pa_t H_2\|_{C^\alpha(\BT)}
\leq C_*(\D m_\alpha +\D M_\alpha),
\label{eqn: C alpha difference of time derivative of h and H}
\eeq
where $C_* = C_*(\alpha,\mu,\nu, r, R, G)$.
Here $\pa_t h_i$ and $\pa_t H_i$ are the interface velocities in the radial direction, normalized by $r$ and $R$ respectively (see \eqref{eqn: time derivative of f}.)
\end{lem}

Let $p_i$ $(i =1,2)$ denote the pressure solving \eqref{eqn: equation for p simplified case} and \eqref{eqn: p boundary condition} on the physical domain that is determined by $\g_i$ and $\tilde{\g}_i$, while $\tilde{p}_i$ denotes its pull back into the reference coordinate as in \eqref{eqn: pressure in the reference coordinate}.
An important intermediate result in proving Lemma \ref{lem: stability of the interface velocities} is the following lemma on $C^{1,\alpha}$-bound for $(\tilde{p}_1-\tilde{p}_2)$, which will be also used when proving uniqueness of the local solution in Section \ref{sec: uniqueness}.

\begin{lem}
\label{lem: C 1 alpha bounds for two different pressures in reference coordinate}
Under the assumption of Lemma \ref{lem: stability of the interface velocities},
\beq
\|\tilde{p}_1-\tilde{p}_2\|_{L^\infty(B_R)}\leq  C_*(\D m_0+\D M_0),\label{eqn: L inf difference of two pressures}
\eeq
and
\beq
\|\tilde{p}_1-\tilde{p}_2\|_{C^{1,\alpha}(\overline{B_r})}+\|\tilde{p}_1-\tilde{p}_2\|_{C^{1,\alpha}(\overline{B_R\backslash B_r})}
\leq C_*(\D m_\alpha +\D M_\alpha),
\label{eqn: Holder estimate for difference of two pressures}
\eeq
where $C_* = C_*(\alpha,\mu,\nu, r, R, G)$.
\end{lem}

Their proofs involve lengthy calculation, while they are relatively independent from the rest of the paper.
So we leave them to Appendix \ref{sec: proof of stability lemma}.

\section{Gradient Estimates for $\G* g$ along Interfaces}
\label{sec: gradient estimate of growth potential}

In this section, we shall derive estimates concerning $e_r\cdot \na (\G*g)$ and $e_\th\cdot\na (\G*g)$ along $\g$ and $\tilde{\g}$, where $e_r = (\cos \th, \sin \th)$ and $e_\th = (-\sin \th,\cos \th)$.
Aiming at greater generality, instead of working with $g$ defined in \eqref{eqn: def of source term g}, here we shall assume $g := g_0(X(x))$ for some $g_0$ defined in the reference coordinate and supported on $\overline{B_{(1+4\d)r}}$, where $X(x)$ is the inverse of $x(X)$ defined by \eqref{eqn: change of variables to the reference coordinate}.
We remark that the support is a slightly larger than the one corresponding to \eqref{eqn: def of source term g} ($\overline{B_r}$ in that case).
The motivation for this will be clear in Section \ref{sec: uniqueness}.
Also note that $\overline{B_{(1+4\d)r}}\subset B_{(1-2\d)R}$.

\subsection{Preliminaries}
We introduce Poisson kernel $P$ on the 2-D unit disc and its conjugate $Q$:
\begin{align}
P(s,\xi) = &\;\f{1-s^2}{1+s^2-2s\cos\xi},
\label{eqn: def of P}\\
Q(s,\xi) = &\;\f{2s\sin\xi}{1+s^2 -2s\cos\xi}.
\label{eqn: def of Q}
\end{align}
Elementary estimates for them as well as their derivatives are collected in Lemma \ref{lem: properties of Poisson kernel}.
Define
\begin{align}
K(s, \xi): =&\; \f{2s^2\sin\xi}{1+s^2-2s\cos\xi} = sQ(s,\xi),
\label{eqn: def of function K}\\
J(s,\xi) := &\;\f{2(s\cos\xi-1)s}{1+s^2-2s\cos\xi} = -s(1+P(s,\xi)).
\label{eqn: def of function J}
\end{align}
See \eqref{eqn: formula for tangent derivative} and \eqref{eqn: formula for normal derivative} for the motivation of defining these kernels.
They have the following properties.

\begin{lem}\label{lem: difference of kernels}
Let $z_i\in [0,2]$ $(i = 1,2,3,4)$.
Suppose for some $w\in [0,2]$ and $\xi\in \BT$, $|z_i-w|\leq c(|\xi|+|1-w|)$. Here $c$ is some universal small constant, whose smallness will be clear in the proof.
Then
\beq
|K(z_i,\xi)|
\leq  \f{C|z_i|} {(1+w^2-2w\cos \xi)^{1/2}},
\label{eqn: bound on K}
\eeq
\beq
\left|\f{\pa K}{\pa s}(z_i,\xi)\right|+\left|\f{\pa K}{\pa \xi}(z_i,\xi)\right|
\leq  \f{C} {1+w^2-2w\cos \xi},
\label{eqn: bound on derivative of K}
\eeq
\beq
|K(z_1,\xi)-K(z_2,\xi)|
\leq  \f{C|z_1-z_2|} {1+w^2-2w\cos \xi},
\eeq
\beq
\left|\f{\pa K}{\pa s}(z_1,\xi)-\f{\pa K}{\pa s}(z_2,\xi) \right|+\left|\f{\pa K}{\pa \xi}(z_1,\xi)-\f{\pa K}{\pa \xi}(z_2,\xi) \right|\leq \f{C|z_1-z_2|}{(1+w^2-2w\cos\xi)^{3/2}},
\label{eqn: difference of K derivatives}
\eeq
and
\beq
\begin{split}
&\;\left|\f{\pa K}{\pa s}(z_1,\xi)-\f{\pa K}{\pa s}(z_2,\xi) -\f{\pa K}{\pa s}(z_3,\xi)+\f{\pa K}{\pa s}(z_4,\xi) \right|\\
\leq &\; \f{C|z_1-z_2-z_3+z_4|}{(1+w^2-2w\cos\xi)^{3/2}}+\f{C(|z_1-z_2|+|z_3-z_4|)(|z_1-z_3|+|z_2-z_4|)}{(1+w^2-2w\cos\xi)^{2}}.
\end{split}
\label{eqn: double difference of K derivatives}
\eeq
Here $C$ are all universal constants.
These estimates also hold if $K$ is replaced by $J$.
\begin{proof}

We derive that
\beq
\left|\f{1+z_i^2 -2z_i\cos\xi}{1+w^2-2w\cos\xi}-1\right|
\leq \f{|z_i-w|+2|w-\cos\xi|}{1+w^2-2w\cos\xi}|z_i-w|.
\label{eqn: ratio of two kernels}
\eeq
When $c$ is suitably small, the right hand side is bounded by $\f12$.
This implies that $(1+z_i^2 -2z_i\cos\xi)$ are comparable with $(1+w^2 -2w\cos\xi)$, and thus they are comparable with each other.
%
Then \eqref{eqn: bound on K} and \eqref{eqn: bound on derivative of K} follow from Lemma \ref{lem: properties of Poisson kernel} and the assumption $z_i \in [0,2]$.
Using the same facts, we can also derive that
\beq
\begin{split}
|K(z_1,\xi)- K(z_2,\xi)|=&\;
\left|\f{2\sin\xi (z_1-z_2) [z_1(1-z_2\cos\xi)+z_2(1-z_1\cos\xi)]} {(1+z_1^2-2z_1\cos\xi)(1+z_2^2-2z_2\cos\xi)}\right|\\
\leq &\;
\f{C |z_1-z_2|} {1+w^2-2w\cos \xi}.
\end{split}
\label{eqn: difference between K}
\eeq
Moreover, by Lemma \ref{lem: properties of Poisson kernel},
\begin{align}
\f{\pa K}{\pa s} = &\; Q+s\pa_s Q = Q+\f{2s\sin \xi(1-s^2)}{(1+s^2-2s\cos\xi)^2} = Q(1+P),\\
\f{\pa K}{\pa \xi} = &\; s\pa_\xi Q = s^2 \pa_s P = \f{K}{\tan \xi}-QK.
\end{align}
Then \eqref{eqn: difference of K derivatives} and \eqref{eqn: double difference of K derivatives} follow from
\begin{align}
P(z_1,\xi)-P(z_2,\xi) = &\; 2(z_1-z_2)\cdot \frac{(1-z_1)(1-z_2)-(1-\cos\xi)(1+z_1z_2)}{(1+z_1^2-2z_1\cos \xi)(1+z_2^2-2z_2\cos \xi)},\\
Q(z_1,\xi)-Q(z_2,\xi) = &\; 2(z_1-z_2)\cdot \frac{\sin \xi ((1-z_1)+z_1(1-z_2))}{(1+z_1^2-2z_1\cos \xi)(1+z_2^2-2z_2\cos \xi)},
\end{align}
and Lemma \ref{lem: properties of Poisson kernel} by a direct calculation as in \eqref{eqn: difference between K}.

The estimates for $J$ can be justified similarly. Indeed,
\beq
J(z_1,\xi)-J(z_2,\xi)= 2(z_1-z_2)\cdot \frac{z_1z_2\sin^2\xi-(1-z_1\cos \xi)(1-z_2\cos\xi)}{(1+z_1^2 -2z_1\cos\xi)(1+z_2^2 -2z_2\cos\xi)},
\eeq
and
\begin{align}
\f{\pa J}{\pa s} = &\; -(1+P)-s\pa_s P = -1-P-\f{Q}{\tan\xi}+Q^2,\\
\f{\pa J}{\pa \xi} = &\; -s\pa_\xi P = s^2 \pa_s Q = PK.
\end{align}
\end{proof}
\end{lem}

Suppose the inner interface $\g$ and the outer interface $\tilde{\g}$ are defined by $h$ and $H$ through \eqref{eqn: parameterization of gamma}-\eqref{eqn: def of h and H}, respectively.
Let $\eta_\d$ be defined as in the beginning of Section \ref{sec: estimate for pressure in circular geometry}.
With $\r = rw$, let
\begin{align}
\tilde{b}(w,\th,\xi):= &\;\f{w(1+h(\th+\xi)\eta_\d(w))}{1+h(\th)},
\label{eqn: def of b tilde}\\
\tilde{B}(w,\th,\xi):= &\;\f{r}{R}\cdot \f{w(1+h(\th+\xi)\eta_\d(w))}{1+H(\th)}.
\label{eqn: def of big b tilde}
\end{align}
Additionally, we define
\begin{align}
b(w,\th):= &\;\tilde{b}(w,\th,0) = \f{w(1+h(\th)\eta_\d(w))}{1+h(\th)},\label{eqn: def of b}\\
B(w,\th):= &\;\tilde{B}(w,\th,0) = \f{r}{R}\cdot \f{w(1+h(\th)\eta_\d(w))}{1+H(\th)}.\label{eqn: def of b big}
\end{align}
The motivation of introducing these quantities will be clear later in \eqref{eqn: formula for tangent derivative} and \eqref{eqn: formula for normal derivative}.
In what follows, we will work with several different configurations of interfaces, determined by $h_i$ and $H_i$ $(i = 1,2)$, respectively.
We define the corresponding quantities $\tilde{b}_i$, $\tilde{B}_i$, $b_i$ and $B_i$ as above, with $h$ and $H$ replaced by $h_i$ and $H_i$.

Recall that $m_{0,i}$ and $M_{0,i}$ are defined in \eqref{eqn: bound assumption on the Lipschitz norm of h} and \eqref{eqn: bound assumption on the Lipschitz norm of H},
while $\D m_0$ and $\D M_0$ are defined in \eqref{eqn: def of normalized W 1 inf difference of two h} and \eqref{eqn: def of normalized W 1 inf difference of two H large}.
It is straightforward to show that
\begin{lem}\label{lem: estimates involving b}
Suppose $h_i,H_i\in W^{1,\infty}(\BT)$ $(i = 1,2)$, with $m_{0,i}+M_{0,i}\ll 1$.
Then with $C$ being universal constants, for all $w\in [0,1+4\d]$ and $\xi\in \BT$,
\begin{align}
|\tilde{b}_i-b_i|\leq &\;C|\eta_\d||\xi|\|h_i'\|_{L^\infty},\\
|\tilde{b}_1-\tilde{b}_2|\leq 
&\;C(|\eta_\d||\xi|+\d|1-\eta_\d|)\D m_0\leq C(|\xi|+|1-w|)\D m_0,\label{eqn: difference between b tilde}\\
|b_1-b_2|\leq &\;C|1-\eta_\d|\|h_1-h_2\|_{L^\infty}\leq C |1-w|\D m_0,\label{eqn: difference between b}\\
|\tilde{b}_1-b_1-\tilde{b}_2+b_2|\leq &\;
C|\eta_\d||\xi|\D m_0,
\end{align}
\begin{align}
|\tilde{B}_i-B_i|\leq &\;\f{Cr}{R}|\eta_\d||\xi|\|h_i'\|_{L^\infty},\\
|\tilde{B}_1-\tilde{B}_2|+|B_1-B_2|
\leq &\;\f{Cr}{R}(|\eta_\d|\|h_1-h_2\|_{L^\infty}+\|H_1-H_2\|_{L^\infty})\leq \f{Cr\d}{R}(\D m_0+\D M_0),\label{eqn: difference between b big}\\
|\tilde{B}_1-B_1-\tilde{B}_2+B_2|\leq &\;\f{Cr}{R}|\eta_\d||\xi|( \|h_1'-h_2'\|_{L^\infty}+\|h_2'\|_{L^\infty}\|H_1-H_2\|_{L^\infty}),
\end{align}
\begin{align}
\left|\f{\pa \tilde{B}_i}{\pa\th}-\f{\pa B_i}{\pa\th}\right|
\leq &\; \f{Cr}{R}
|\eta_\d|(\|h_i'\|_{L^\infty}+\|H_i'\|_{L^\infty}\|h_i'\|_{L^\infty}|\xi|)\leq \f{Cr}{R}|\eta_\d|\|h_i'\|_{L^\infty},
\label{eqn: difference between theta derivatives of b big}\\
\left|\f{\pa B_1}{\pa\th}-\f{\pa B_2}{\pa\th}\right|
\leq &\; \f{Cr}{R}(\D m_0+\D M_0),
\end{align}
and
\beq
\begin{split}
&\;\left|\f{\pa \tilde{B}_1}{\pa\th}-\f{\pa B_1}{\pa\th}-\f{\pa \tilde{B}_2}{\pa\th}+\f{\pa B_2}{\pa\th}\right|\\
\leq&\; \f{Cr}{R} |\eta_\d|(\|h_1'-h_2'\|_{L^\infty}+\|h_2'\|_{L^\infty}
\|H_1-H_2\|_{L^\infty}+\|H_1'-H_2'\|_{L^\infty}\|h_1'\|_{L^\infty}|\xi|).
\end{split}
\label{eqn: double difference between theta derivatives of b big}
\eeq
If in addition, $h_i\in C^{1,\beta}(\BT)$ for some $\beta\in(0,1)$, then
\begin{align}
\left|\f{\pa \tilde{b}_i}{\pa\th}-\f{\pa b_i}{\pa\th}\right|
\leq &\; C
|\eta_\d|(\|h_i'\|_{\dot{C}^\beta}|\xi|^\beta+\|h_i'\|_{L^\infty}^2|\xi|)\leq C|\eta_\d|\|h_i'\|_{\dot{C}^\beta}|\xi|^\beta,\\
\left|\f{\pa b_1}{\pa\th}-\f{\pa b_2}{\pa\th}\right|
\leq &\; C
|1-\eta_\d|(\|h_1'-h_2'\|_{L^\infty}+\|h_2'\|_{L^\infty}\|h_1-h_2\|_{L^\infty})\leq C
|1-\eta_\d|\D m_0,
\label{eqn: difference between theta derivatives of b}
\end{align}
and
\beq
\left|\f{\pa \tilde{b}_1}{\pa\th}-\f{\pa b_1}{\pa\th}-\f{\pa \tilde{b}_2}{\pa\th}+\f{\pa b_2}{\pa\th}\right|
\leq C|\eta_\d||\xi|^\b(\|h_1'-h_2'\|_{\dot{C}^\b}+\|h_2'\|_{\dot{C}^\b}
\|h_1-h_2\|_{L^\infty}).
\label{eqn: double difference between theta derivatives of b}
\eeq
Here all the constants $C$ are universal.

\begin{proof}
These estimates follow directly from \eqref{eqn: def of b tilde}-\eqref{eqn: def of b big} and 
\begin{align}
\f{\pa\tilde{b}_i}{\pa \th} = &\; \f{wh_i'(\th+\xi)\eta_\d(w)(1+h_i(\th))-h_i'(\th)w(1+h_i(\th+\xi)\eta_\d(w))}{(1+h_i(\th))^2},\\
\f{\pa b_i}{\pa\th} = &\;\f{wh_i'(\th)(\eta_\d(w)-1)}{(1+h_i(\th))^2},\label{eqn: theta derivative of b}\\
%
%
%
%
%
%
\f{\pa\tilde{B}_i}{\pa \th} = &\;\f{r}{R}\cdot \f{wh_i'(\th+\xi)\eta_\d(w)(1+H_i(\th))-H_i'(\th)w(1+h_i(\th+\xi)\eta_\d(w))}{(1+H_i(\th))^2},\\
\f{\pa B_i}{\pa\th} = &\;\f{r}{R}\cdot \f{wh_i'(\th)\eta_\d(w)(1+H_i(\th))-H_i'(\th)w(1+h_i(\th)\eta_\d(w))}{(1+H_i(\th))^2}.\label{eqn: theta derivative of big b}
\end{align}
We omit the details.
\end{proof}
\begin{rmk}\label{rmk: tilde b and b satisfy assumptions of Lemma 4.1}
Taking $h_1 = H_1 = 0$ (or $h_2 = H_2 = 0$), we find by \eqref{eqn: difference between b tilde}, \eqref{eqn: difference between b} and \eqref{eqn: difference between b big} that
\begin{align}
|\tilde{b}_i-w|+|b_i-w|\leq &\;C(|\xi|+|1-w|)m_{0,i},\\
\left|\tilde{B}_i-\f{rw}{R}\right|+\left|B_i-\f{rw}{R}\right|\leq &\; \f{Cr \d }{R}(m_{0,i}+M_{0,i})\leq C\left(|\xi|+\left|1-\f{rw}{R}\right|\right)(m_{0,i}+M_{0,i}).
\end{align}
Here we used the fact that $|1-\f{rw}{R}|\geq C\d$ for all $w\in [0,1+4\d]$ (c.f.\;\eqref{eqn: condition on delta}).
If $m_{0,i}+M_{0,i}$ is assumed to be suitably small, $\tilde{b}_i(w,\th,\xi)$ and $b_i(w,\th)$ satisfy the assumption of Lemma \ref{lem: difference of kernels}, while $\tilde{B}_i(w,\th,\xi)$ and $B_i(w,\th)$ satisfy the assumption of Lemma \ref{lem: difference of kernels} with $w$ there replaced by $\f{rw}{R}$.
\end{rmk}
\end{lem}

\subsection{Estimates along $\g$}
Let $x = f(\th)(\cos\th,\sin\th)\in \g$.
With abuse of notations, let $y = x((\r\cos(\th+\xi), \r\sin(\th+\xi)))$ be an arbitrary point in $\BR^2$, where the map $x$ is defined in \eqref{eqn: change of variables to the reference coordinate}.
Then
\beq
\begin{split}
e_\th\cdot \na(\G*g) = &\;\f{1}{2\pi}\int_{\tilde{\Om}} \f{(y-x)\cdot e_\th}{|x-y|^2}g_0(X(y))\,dy\\
=&\;\f{1}{2\pi}\int_{\BT}d\xi \int_0^{r(1+4\d)} \f{|y|\sin\xi \cdot g_0(\r, \th+\xi)}{f(\th)^2+|y|^2-2|y|f(\th)\cos\xi}\cdot\f{\pa |y|}{\pa \r}|y|\,d\r\\
=&\;\f{1}{4\pi}\int_{\BT}d\xi \int_0^{r(1+4\d)} \f{2\left(\f{|y|}{f(\th)}\right)^2\sin\xi }{1+\left(\f{|y|}{f(\th)}\right)^2-2\f{|y|}{f(\th)}\cos\xi}\cdot\f{\pa |y|}{\pa \r}g_0(\r, \th+\xi)\,d\r.
\end{split}
\label{eqn: formula for tangent derivative crude form}
\eeq
For $w\in [0,1+4\d]$, $|y| = |y(\r,\th+\xi)| =  rw[1+h(\th+\xi)\eta_\d(w)]$.
Note that the third term in \eqref{eqn: change of variables to the reference coordinate} does not show up since $\r = rw \leq R(1-2\d)$.
Then \eqref{eqn: formula for tangent derivative crude form} becomes
\beq
(e_\th\cdot \na(\G*g) )_{\g(\th)}
=\f{r}{4\pi}\int_{\BT}d\xi \int_0^{1+4\d} K(\tilde{b},\xi)\cdot\f{\pa |y|}{\pa \r}g_0(rw, \th+\xi)\,dw.
\label{eqn: formula for tangent derivative}
\eeq
Similarly, 
\beq
(e_r\cdot \na(\G*g))_{\g(\th)} 
= \f{r}{4\pi}\int_{\BT}d\xi\int_{0}^{1+4\d} J(\tilde{b},\xi)\cdot \f{\pa|y|}{\pa \r}g_0(rw, \th+\xi)\,dw.
\label{eqn: formula for normal derivative}
\eeq

We first show

\begin{lem}\label{lem: L inf estimate of difference same source}
Suppose for $i = 1,2$, $h_i\in W^{1,\infty}(\BT)$ such that $m_{0,i}\ll 1$.
Let $\D m_0$ be defined in \eqref{eqn: def of normalized W 1 inf difference of two h}.
Let $x_i(X)$ be the map \eqref{eqn: change of variables to the reference coordinate} determined by $h_i$ ($H$ is irrelevant in this context, and one may take $H = 0$ in \eqref{eqn: change of variables to the reference coordinate} without loss of generality.)
Let $X_i(x)$ be its inverse.
Define $g_i = g_0(X_i(x))$.
Then
\beq
\|(e_\th\cdot \na(\G*g_1) )_{\g_1(\th)}-(e_\th\cdot \na(\G*g_2) )_{\g_2(\th)}\|_{L^\infty(\BT)}
\leq  Cr\d |\ln \d| \D m_0\|g_0\|_{L^\infty},
\label{eqn: L inf difference of tangential gradient}
\eeq
where $C$ is a universal constant.

In addition, $\|(e_r\cdot \na(\G*g_1) )_{\g_1(\th)}-(e_r\cdot \na(\G*g_2) )_{\g_2(\th)}\|_{L^\infty(\BT)}$ satisfies an identical estimate.
\begin{proof}
Let $y_i = x_i(\r,\th+\xi)$, with $|y_i| = \r [1+h_i(\th+\xi)\eta_\d(\r/r)]$.
We calculate
\beq
\f{\pa |y_i|}{\pa \r}(\r,\th+\xi)-1
= h_i(\th+\xi)(\eta_\d(w)+w\eta_\d'(w)).
\label{eqn: ingredient 1}
\eeq

By Lemma \ref{lem: difference of kernels}, Lemma \ref{lem: estimates involving b} and Remark \ref{rmk: tilde b and b satisfy assumptions of Lemma 4.1},
\beq
\begin{split}
&\;\left|\int_{\BT}d\xi \int_{0}^{1+4\d} (K(\tilde{b}_1,\xi)-K(\tilde{b}_2,\xi))\cdot \f{\pa |y_1|}{\pa \r}g_0(rw,\th+\xi)\,dw\right|\\
\leq &\;C\D m_0\|g_0\|_{L^\infty}\int_{\BT}d\xi \int_0^{1+4\d}\f{|\eta_\d||\xi|+|1-\eta_\d|\d } {1+w^2-2w\cos \xi}\,dw\\
\leq &\;C\d |\ln \d| \D m_0\|g_0\|_{L^\infty}.
\end{split}
\eeq
On the other hand, by \eqref{eqn: ingredient 1},
\beq
\begin{split}
&\;\left|\int_{\BT}d\xi \int_{0}^{1+4\d} K(\tilde{b}_2,\xi)\cdot \left(\f{\pa |y_1|}{\pa \r}-\f{\pa |y_2|}{\pa \r}\right)g_0(rw,\th+\xi)\,dw\right|\\
\leq &\; C\int_{\BT}d\xi \int_{1-2\d}^{1+2\d} \f{1}{|1-w|+|\xi|} \cdot \|h_1-h_2\|_{L^\infty}\d^{-1}\|g_0\|_{L^\infty}\,dw \\
\leq &\; C\d |\ln \d|\D m_0\|g_0\|_{L^\infty}.
\end{split}
\eeq
Combining these estimates with \eqref{eqn: formula for tangent derivative} yields \eqref{eqn: L inf difference of tangential gradient}.
The estimate concerning $(e_r\cdot \na(\G*g_i) )_{\g_i(\th)}$ can be justified in the same way.
\end{proof}
\end{lem}

\begin{lem}\label{lem: L^inf estimate of derivative of growth potential inner}
Let $h\in W^{1,\infty}(\BT)$ such that $m_{0}\ll 1$, which defines the map $x$ in \eqref{eqn: change of variables to the reference coordinate} and $g = g_0(X(x))$.
Then
\beq
\begin{split}
&\;\|(e_\th\cdot \na(\G*g))_{\g(\th)} \|_{L^\infty(\BT)}+
\|(e_r\cdot \na(\G*g))_{\g(\th)}-c_{g_0} \|_{L^\infty(\BT)}\\
\leq &\;Cr( m_0 \d |\ln \d| \|g_0\|_{L^\infty(B_{(1+4\d)r})}+ \|e_\th\cdot \na g_0 \|_{L^2(B_{r(1+4\d)})}),
\end{split}
\label{eqn: L inf bound of derivative of potential inner}
\eeq
where $C$ is a universal constant and
\beq
c_{g_0} := -\f{1}{2\pi r}\int_{B_r}g_0(X)\,dX.
\label{eqn: def of c characteristic normal derivative inner interface}
\eeq

\begin{proof}
We first derive an $L^\infty$-estimate of
\beq
(e_\th\cdot \na(\G*g_0) )_{\pa B_r}
=\f{r}{4\pi}\int_{\BT}d\xi \int_0^{1+4\d} K(w,\xi)g_0(rw, \th+\xi)\,dw,
\eeq
which corresponds to the case $h = 0$.
Define $\bar{g}_0(rw) = (2\pi)^{-1}\int_\BT g_0(rw,\xi)\,d\xi$.
Since $K(w,\cdot)$ is an odd kernel, by H\"{o}lder's inequality and Sobolev embedding,
\beq
\begin{split}
&\;|(e_\th\cdot \na(\G*g_0) )_{\pa B_r}|\\
= &\; \f{r}{4\pi}\left|\int_{\BT}d\xi \int_0^{1+4\d} K(w,\xi)(g_0(rw, \th+\xi)-\bar{g}_0(rw))\,dw\right|\\
\leq &\; Cr \int_0^{1+4\d}\left\|\f{1}{|1-w|+|\xi|}\right\|_{L^1_\xi(\BT)}\|g_0(rw, \cdot)-\bar{g}_0(rw)\|_{L^\infty_\xi(\BT)}\,dw\\
\leq &\; Cr \int_0^{1+4\d}(1+|\ln|1-w||)\|\pa_\th g_0(rw, \cdot)\|_{L^2(\BT)}\,dw\\
\leq &\; Cr \|1+|\ln|1-w||\|_{L^2([0,1+4\d])} \left(\int_0^{1+4\d} r\|e_\th\cdot \na g_0 \|_{L^2(\pa B_{rw})}^2\,dw\right)^{1/2}\\
\leq &\; Cr \|e_\th\cdot \na g_0 \|_{L^2(B_{r(1+4\d)})}.
\end{split}
\label{eqn: bound tangent derivative of potential for a perfect circle inner}
\eeq
%
Now we take in Lemma \ref{lem: L inf estimate of difference same source} that $h_1= h$ and $h_2 = 0$, and derive
\beq
\begin{split}
&\;\|(e_\th\cdot \na(\G*g))_{\g(\th)} \|_{L^\infty(\BT)}\\
\leq &\;\|(e_\th\cdot \na(\G*g) )_{\g(\th)}-(e_\th\cdot \na(\G*g_0) )_{\pa B_r}\|_{L^\infty(\BT)}+\|(e_\th\cdot \na(\G*g_0) )_{\pa B_r}\|_{L^\infty(\BT)}\\
\leq &\;Cr\d |\ln \d| m_0 \|g_0\|_{L^\infty}+Cr \|e_\th\cdot \na g_0 \|_{L^2(B_{r(1+4\d)})}.
\end{split}
\label{eqn: bounding L inf of tangential derivative of growth potential at inner interface}
\eeq

Next we study
\beq
\begin{split}
&\;(e_r\cdot \na(\G*g_0) )_{\pa B_r}\\
=&\;\f{r}{4\pi}\int_{\BT}d\xi \int_0^{1+4\d} J(w,\xi)(g_0(rw, \th+\xi)-\bar{g}_0(rw))\,dw\\
&\; + \f{r}{4\pi} \int_0^{1+4\d} \int_{\BT}d\xi \,J(w,\xi)\bar{g}_0(rw)\,dw.
\end{split}
\eeq
The first term can be bounded exactly as in \eqref{eqn: bound tangent derivative of potential for a perfect circle inner}.
We use the definition of $J$ in \eqref{eqn: def of function J} to simplify the second term as
\beq
\f{r}{4\pi} \int_0^{1+4\d} \int_{\BT}d\xi \,J(w,\xi)\bar{g}_0(rw)\,dw = -r \int_0^1 w\bar{g}_0(rw)\,dw = -\f{1}{2\pi r}\int_{B_r} g_0(X)\,dX.
\eeq
Then the desired estimate follows.
\end{proof}
\end{lem}

Next we derive $W^{1,p}$-estimates for $(e_\th\cdot \na(\G*g))_{\g(\th)}$ and $(e_r\cdot \na(\G*g))_{\g(\th)}$.

\begin{lem}\label{lem: W 1p estimate of tangent derivative}
Assume $h_1,h_2\in C^{1,\beta}(\mathbb{T})$ for some $\beta\in (0,1)$, such that $m_{0,i} \ll 1$.
Let $\D m_0$ be defined in \eqref{eqn: def of normalized W 1 inf difference of two h}, and let $g_i(x) = g_0(X_i(x))$.
Then for all $p\in [2,\infty)$,
\beq
\begin{split}
&\;\|(e_\th\cdot \na(\G*g_1))_{\g_1(\th)}
-(e_\th\cdot \na(\G*g_2))_{\g_2(\th)}\|_{\dot{W}^{1,p}(\BT)}\\
\leq
&\;Cr\|g_0\|_{L^\infty(B_{(1+4\d)r})}\left[(1+ \d^\b(\|h_1'\|_{\dot{C}^\b}+\|h_2'\|_{\dot{C}^\b}))\D m_0 +\d^\b\|h_1'-h_2'\|_{\dot{C}^\b}\right]\\
&\; +Cr\D m_0\|e_\th\cdot \na g_0\|_{L^2(B_{(1+4\d)r})}.
\label{eqn: W 1p estimate for difference of tangent derivative}
\end{split}
\eeq
where $C = C(p,\b)$.
\begin{proof}

Let $y_i = x_i(\r,\th+\xi)$.
We take $\theta$-derivative in \eqref{eqn: formula for tangent derivative}.
\beq
\begin{split}
&\;
\f{d}{d\th}(e_\th\cdot \na(\G*g_i))_{\g_i(\th)}\\
=&\;
\f{r}{4\pi}\int_{\BT}d\xi \int_0^{1+4\d} dw\,
\f{\pa}{\pa\th}\left[K(\tilde{b}_i,\xi)\cdot \f{\pa |y_i|}{\pa \r}g_0(rw,\xi+\th)\right]\\
= &\;
\f{r}{4\pi}\int_{\BT}d\xi \int_0^{1+4\d} dw\, \left[\f{\pa K}{\pa s}(\tilde{b}_i, \xi)\f{\pa \tilde{b}_i}{\pa \th}-\f{\pa K}{\pa s}(b_i,\xi)\f{\pa b_i}{\pa \th}\right]\left[\f{\pa |y_i|}{\pa \r}g_0\right]_{(rw,\xi+\th)}\\
&\;+\f{r}{4\pi}\int_{\BT}d\xi \int_0^{1+4\d} dw\, \f{\pa K}{\pa s}(b_i,\xi)\f{\pa b_i}{\pa \th}\left(\left[\f{\pa |y_i|}{\pa \r}g_0\right]_{(rw,\xi+\th)}-\left[\f{\pa |y_i|}{\pa \r}g_0\right]_{(rw,\th)}\right)\\
&\;-\f{r}{4\pi}\int_{\BT}d\xi \int_0^{1+4\d} dw\,
 \left[\f{\pa K}{\pa s}(\tilde{b}_i,\xi)\f{\pa \tilde{b}_i}{\pa \xi}+\f{\pa K}{\pa \xi}(\tilde{b}_i,\xi)\right]\left(\left[\f{\pa |y_i|}{\pa \r}g_0\right]_{(rw,\xi+\th)} -\left[\f{\pa |y_i|}{\pa \r}g_0\right]_{(rw,\th)}\right)\\
=:&\;J_{\th,1}^{(i)}+J_{\th,2}^{(i)}+J_{\th,3}^{(i)}.
\label{eqn: decomposition of derivative of tangent derivative}
\end{split}
\eeq

Here we exchanged the integral with the $\th$-derivative, which can be justified rigorously by a limiting  argument.
In $J_{\th,2}^{(i)}$, an extra term is inserted without changing its value, since $\pa_s K(b_i,\xi)$ is odd in $\xi$.
When deriving $J_{\th,3}^{(i)}$, we used the fact that
\beq
\f{\pa}{\pa\th}\left[\f{\pa |y_i|}{\pa \r}g_0\right]_{(rw,\xi+\th)}=\f{\pa}{\pa\xi}\left[\f{\pa |y_i|}{\pa \r}g_0\right]_{(rw,\xi+\th)}
\eeq
and then integrated by parts.
Note that it is not clear a priori whether these integrands are integrable at $(w,\xi) = (1,0)$, so we need to write them as principal value integrals in the $w$-variable in the first place.
Yet, it will be clear in the following that all these integrands are absolutely integrable. 
For this reason, we omitted the notations for the principal value integral.

We start with bounding $J_{\th,1}^{(1)}-J_{\th,1}^{(2)}$.
\beq
\begin{split}
&\;J_{\th,1}^{(1)}-J_{\th,1}^{(2)}\\
=&\;\f{r}{4\pi}\int_{\BT}d\xi \int_0^{1+4\d} dw\,\left(\f{\pa K}{\pa s}(\tilde{b}_1, \xi)-\f{\pa K}{\pa s}(\tilde{b}_2,\xi)\right)\left(\f{\pa \tilde{b}_1}{\pa \th}-\f{\pa b_1}{\pa \th}\right)\left[\f{\pa |y_1|}{\pa \r}g_0\right]_{(rw,\xi+\th)}\\
&\;+\f{r}{4\pi}\int_{\BT}d\xi \int_0^{1+4\d} dw\,\f{\pa K}{\pa s}(\tilde{b}_2,\xi)\left(\f{\pa \tilde{b}_1}{\pa \th}-\f{\pa b_1}{\pa \th}-\f{\pa \tilde{b}_2}{\pa \th}+\f{\pa b_2}{\pa \th}\right)\left[\f{\pa |y_1|}{\pa \r}g_0\right]_{(rw,\xi+\th)}\\
&\;+\f{r}{4\pi}\int_{\BT}d\xi \int_0^{1+4\d} dw\, \left(\f{\pa K}{\pa s}(\tilde{b}_1, \xi)-\f{\pa K}{\pa s}(b_1,\xi) -\f{\pa K}{\pa s}(\tilde{b}_2, \xi)+\f{\pa K}{\pa s}(b_2,\xi)\right)\f{\pa b_1}{\pa \th}\left[\f{\pa |y_1|}{\pa \r}g_0\right]_{(rw,\xi+\th)}\\
&\;+\f{r}{4\pi}\int_{\BT}d\xi \int_0^{1+4\d} dw\,\left(\f{\pa K}{\pa s}(\tilde{b}_2, \xi)-\f{\pa K}{\pa s}(b_2,\xi)\right)\left(\f{\pa b_1}{\pa \th}-\f{\pa b_2}{\pa \th}\right)\left[\f{\pa |y_1|}{\pa \r}g_0\right]_{(rw,\xi+\th)}\\
&\;+\f{r}{4\pi}\int_{\BT}d\xi \int_0^{1+4\d} dw\, \f{\pa K}{\pa s}(\tilde{b}_2,\xi)\left(\f{\pa \tilde{b}_2}{\pa \th}-\f{\pa b_2}{\pa \th}\right)\left[\left(\f{\pa |y_1|}{\pa \r}-\f{\pa |y_2|}{\pa \r}\right)g_0\right]_{(rw,\xi+\th)}\\
&\;+\f{r}{4\pi}\int_{\BT}d\xi \int_0^{1+4\d} dw\, \left(\f{\pa K}{\pa s}(\tilde{b}_2, \xi)-\f{\pa K}{\pa s}(b_2,\xi)\right)\f{\pa b_2}{\pa \th}\left[\left(\f{\pa |y_1|}{\pa \r}-\f{\pa |y_2|}{\pa \r}\right)g_0\right]_{(rw,\xi+\th)}.
\end{split}
\label{eqn: decomposition of difference of J theta 1}
\eeq
By Lemma \ref{lem: difference of kernels}, Lemma \ref{lem: estimates involving b}, Lemma \ref{lem: properties of Poisson kernel} and \eqref{eqn: ingredient 1},
\beq
\begin{split}
&\;|J_{\th,1}^{(1)}-J_{\th,1}^{(2)}|\\
\leq &\;Cr\|g_0\|_{L^\infty}\int_{\BT}d\xi \int_{0}^{1+4\d} dw\,\f{|\tilde{b}_1-\tilde{b}_2|}{(|1-w|+|\xi|)^3}\cdot |\eta_\d| \|h_1'\|_{\dot{C}^\b}|\xi|^\b\\
&\;+Cr\|g_0\|_{L^\infty}\int_{\BT}d\xi \int_{0}^{1+4\d} dw\,\f{|\eta_\d||\xi|^\b(\|h_1'-h_2'\|_{\dot{C}^\b}+\|h_2'\|_{\dot{C}^\b}
\|h_1-h_2\|_{L^\infty})}{(|1-w|+|\xi|)^2} \\
&\;+Cr\|g_0\|_{L^\infty}\int_{\BT}d\xi \int_0^{1+4\d} dw\, \bigg[\f{|\tilde{b}_1-b_1-\tilde{b}_2+b_2|}{(|1-w|+|\xi|)^3}\\
&\;\qquad \qquad \qquad +\f{(|\tilde{b}_1-b_1|+|\tilde{b}_2-b_2|)(|\tilde{b}_1-\tilde{b}_2|+|b_1-b_2|)} {(|1-w|+|\xi|)^4}\bigg] \|h_1'\|_{L^\infty}|1-\eta_\d|\\
&\;+Cr\|g_0\|_{L^\infty}\int_{\BT}d\xi \int_0^{1+4\d} dw\,\f{|\tilde{b}_2-b_2|}{(|1-w|+|\xi|)^3}\cdot|1-\eta_\d|\D m_0\\
&\;+Cr\|g_0\|_{L^\infty}\int_{\BT}d\xi \int_0^{1+4\d} dw\, \f{1}{(|1-w|+|\xi|)^2}\cdot |\eta_\d|\|h_2'\|_{\dot{C}^\beta}|\xi|^\beta\cdot |\eta_\d+w\eta_\d'| \|h_1-h_2\|_{L^\infty}\\
&\;+Cr\|g_0\|_{L^\infty}\int_{\BT}d\xi \int_0^{1+4\d} dw\,\f{|\tilde{b}_2-b_2|}{(|1-w|+|\xi|)^3}\cdot \|h_2'\|_{L^\infty} |1-\eta_\d|\cdot |\eta_\d+w\eta_\d'| \|h_1-h_2\|_{L^\infty}\\
\leq &\;Cr\|g_0\|_{L^\infty}\left[ \d^\b(\|h_1'\|_{\dot{C}^\b}+\|h_2'\|_{\dot{C}^\b})\D m_0 +\d^\b\|h_1'-h_2'\|_{\dot{C}^\b}+(\|h_1'\|_{L^\infty}+\|h_2'\|_{L^\infty})\D m_0\right].
\end{split}
\label{eqn: estimate for difference of J theta 1 from two configurations}
\eeq
In the last inequality, when calculating the integrals, we used the facts that $\eta_\d$ is supported on $[1-2\d,1+2\d]$ and that $\eta_\d(1-\eta_\d)$ is supported on $[1-2\d,1-\d]\cup [1+\d,1+2\d]$.

For $J_{\th,2}^{(i)}$ and $J_{\th,3}^{(i)}$, by \eqref{eqn: decomposition of derivative of tangent derivative},
\beq
\begin{split}
&\;(J_{\th,2}^{(1)} +J_{\th,3}^{(1)} )- (J_{\th,3}^{(2)}+ J_{\th,2}^{(2)})\\
= &\;\f{r}{4\pi}\int_{\BT}d\xi \int_0^{1+4\d} dw\, \left[\f{\pa K}{\pa s}(b_1,\xi)\f{\pa b_1}{\pa \th} - \f{\pa K}{\pa s}(\tilde{b}_1,\xi)\f{\pa \tilde{b}_1}{\pa \xi}-\f{\pa K}{\pa \xi}(\tilde{b}_1,\xi)\right]\\
&\;\qquad \qquad \cdot \left(\left[\left(\f{\pa |y_1|}{\pa \r}-\f{\pa |y_2|}{\pa \r}\right)g_0\right]_{(rw,\xi+\th)}-\left[\left(\f{\pa |y_1|}{\pa \r}-\f{\pa |y_2|}{\pa \r}\right)g_0\right]_{(rw,\th)}\right)\\
&\;+\f{r}{4\pi}\int_{\BT}d\xi \int_0^{1+4\d} dw\, \left(\f{\pa K}{\pa s}(b_1,\xi)-\f{\pa K}{\pa s}(b_2,\xi)\right)\f{\pa b_1}{\pa \th}\left(\left[\f{\pa |y_2|}{\pa \r}g_0\right]_{(rw,\xi+\th)}-\left[\f{\pa |y_2|}{\pa \r}g_0\right]_{(rw,\th)}\right)\\
&\;-\f{r}{4\pi}\int_{\BT}d\xi \int_0^{1+4\d} dw\, \left[\left(\f{\pa K}{\pa s}(\tilde{b}_1,\xi)-\f{\pa K}{\pa s}(\tilde{b}_2,\xi)\right)\f{\pa \tilde{b}_1}{\pa \xi}+ \left(\f{\pa K}{\pa \xi}(\tilde{b}_1,\xi)-\f{\pa K}{\pa \xi}(\tilde{b}_2,\xi)\right)\right]\\
&\;\qquad \qquad \cdot \left(\left[\f{\pa |y_2|}{\pa \r}g_0\right]_{(rw,\xi+\th)}-\left[\f{\pa |y_2|}{\pa \r}g_0\right]_{(rw,\th)}\right)\\
&\;+\f{r}{4\pi}\int_{\BT}d\xi \int_0^{1+4\d} dw\, \left[\f{\pa K}{\pa s}(b_2,\xi)\f{\pa (b_1-b_2)}{\pa \th}-\f{\pa K}{\pa s}(\tilde{b}_2,\xi)\f{\pa (\tilde{b}_1-\tilde{b}_2)}{\pa \xi}\right]\\
&\;\qquad \qquad \cdot \left(\left[\f{\pa |y_2|}{\pa \r}g_0\right]_{(rw,\xi+\th)}-\left[\f{\pa |y_2|}{\pa \r}g_0\right]_{(rw,\th)}\right).
\end{split}
\label{eqn: splitting J theta 2 3 difference}
\eeq
We derive in a similar manner.
\beq
\begin{split}
&\;\left|(J_{\th,2}^{(1)}+J_{\th,3}^{(1)}) - (J_{\th,2}^{(2)}+J_{\th,3}^{(2)})\right|\\
\leq  &\;Cr\int_{\BT}d\xi \int_{0}^{1+4\d}  \f{dw}{(|1-w|+|\xi|)^2}\cdot |\eta_\d+w\eta_\d'|\\
&\;\qquad \qquad \cdot (\|h_1-h_2\|_{L^\infty}|g_0(rw,\xi+\th)-g_0(rw,\th)| +\|h_1-h_2\|_{\dot{C}^{\b}}|\xi|^{\b} \|g_0\|_{L^\infty})\\
&\;+Cr\int_{\BT}d\xi \int_{0}^{1+4\d} dw\, \f{|b_1-b_2|}{(|1-w|+|\xi|)^3}\cdot |1-\eta_\d|\|h_1'\|_{L^\infty}\\
&\;\qquad \qquad \cdot (|g_0(rw,\xi+\th)-g_0(rw,\th)|+|\eta_\d+w\eta_\d'||\xi|^{\b}\|h_2\|_{\dot{C}^{\b}}\|g_0\|_{L^\infty})\\
&\;+Cr\int_{\BT}d\xi \int_0^{1+4\d} dw\, \f{|\tilde{b}_1-\tilde{b}_2|}{(|1-w|+|\xi|)^3}\\
&\;\qquad \qquad \cdot (|g_0(rw,\xi+\th)-g_0(rw,\th)|+|\eta_\d+w\eta_\d'||\xi|^{\b}\|h_2\|_{\dot{C}^{\b}}\|g_0\|_{L^\infty})\\
&\;+Cr \int_{\BT}d\xi \int_0^{1+4\d} \f{dw}{(|1-w|+|\xi|)^2}\cdot \D m_0\\
&\;\qquad \qquad \cdot (|g_0(rw,\xi+\th)-g_0(rw,\th)|+|\eta_\d+w\eta_\d'||\xi|^{\b}\|h_2\|_{\dot{C}^{\b}}\|g_0\|_{L^\infty})\\
\leq  &\;Cr\|g_0\|_{L^\infty}\d^{\b-1} (\|h_1-h_2\|_{\dot{C}^{\b}} +\D m_0 \|h_2\|_{\dot{C}^{\b}})\\
&\;+Cr\D m_0  \int_{\BT}d\xi \int_0^{1+4\d} dw\,\f{|g_0(rw,\xi+\th)-g_0(rw,\th)|}{(|1-w|+|\xi|)^2}.
\end{split}
\label{eqn: bounding J theta 2 3 difference}
\eeq
By Minkowski inequality and H\"{o}lder's inequality, with arbitrary $s\in (\f12,\f12+\f1p)$ (for definiteness, take $s= \f12+\f{1}{2p}$),
\beq
\begin{split}
&\;\left\|(J_{\th,2}^{(1)}+J_{\th,3}^{(1)}) - (J_{\th,2}^{(2)}+J_{\th,3}^{(2)})\right\|_{L^p(\BT)}\\
\leq  &\;Cr\|g_0\|_{L^\infty}\D m_0\\
&\;+Cr\D m_0 \int_0^{1+4\d} dw\left[\int_{\BT}d\xi \,\f{\|g_0(rw,\xi+\cdot)-g_0(rw,\cdot)\|_{L^p_\th(\BT)}^2}{|\xi|^{1+2s}}\right]^{1/2}\left[\int_\BT \f{|\xi|^{1+2s}\,d\xi}{(|1-w|+|\xi|)^4}\right]^{1/2}\\
\leq  &\;Cr\|g_0\|_{L^\infty}\D m_0+Cr\D m_0 \int_0^{1+4\d} dw \,\f{\|g_0(rw,\cdot)\|_{\dot{B}^{s}_{p,2}(\BT)}}{|1-w|^{1-s}}\\
\leq  &\;Cr\|g_0\|_{L^\infty}\D m_0+Cr\D m_0 \int_0^{1+4\d} dw \,\f{\|\pa_\th g_0(rw,\cdot)\|_{L^2(\BT)}}{|1-w|^{1-s}}\\
\leq  &\;Cr\D m_0 (\|g_0\|_{L^\infty}+\|e_\th\cdot \na g_0\|_{L^2(B_{(1+4\d)r})} ).
\end{split}
\label{eqn: L p bound for difference of J theta 2 and J theta 3 from two configurations}
\eeq
See e.g.\;\cite[\S 2.5.12 and \S 2.7.1]{triebel2010theory} for the definition of $B_{p,2}^s(\BT)$-space and the embedding of $H^1(\BT)$ into it.
Combining this with \eqref{eqn: decomposition of derivative of tangent derivative} and \eqref{eqn: estimate for difference of J theta 1 from two configurations}, we conclude with \eqref{eqn: W 1p estimate for difference of tangent derivative}.
\end{proof}
\end{lem}

\begin{lem}\label{lem: W 1p estimate of normal derivative difference}
Under the assumptions of Lemma \ref{lem: W 1p estimate of tangent derivative},
\beq
\begin{split}
&\;\|(e_r\cdot \na(\G*g_1))_{\g_1(\th)}
-(e_r\cdot \na(\G*g_2))_{\g_2(\th)}\|_{\dot{W}^{1,p}(\BT)}\\
\leq
&\;Cr\|g_0\|_{L^\infty(B_{(1+4\d)r})}\left[(1+ \d^\b(\|h_1'\|_{\dot{C}^\b}+\|h_2'\|_{\dot{C}^\b}))\D m_0 +\d^\b\|h_1'-h_2'\|_{\dot{C}^\b}\right]\\
&\; +Cr\D m_0\|e_\th\cdot \na g_0\|_{L^2(B_{(1+4\d)r})}.
\label{eqn: W 1p estimate for normal derivative}
\end{split}
\eeq
where $C = C(p,\b)$.
\begin{proof}
We proceed as the proof of Lemma \ref{lem: W 1p estimate of tangent derivative}.
By \eqref{eqn: formula for normal derivative} and integration by parts,
\beq
\begin{split}
&\;\f{d}{d\th}(e_r\cdot \na(\G*g_i))_{\g_i(\th)}\\
=&\;\f{r}{4\pi}\int_{\BT}d\xi\int_{0}^{1+4\d}
\left[\f{\pa J}{\pa s}(\tilde{b}_i,\xi)\f{\pa \tilde{b}_i}{\pa \th}-\f{\pa J}{\pa s}(b_i,\xi)\f{\pa b_i}{\pa \th}\right]\left[\f{\pa|y_i|}{\pa \r}g_0\right]_{(rw, \th+\xi)}\,dw\\
&\;+\f{r}{4\pi}\int_{\BT}d\xi\int_{0}^{1+4\d} \f{\pa J}{\pa s}(b_i,\xi)\f{\pa b_i}{\pa \th}\left(\left[\f{\pa|y_i|}{\pa \r}g_0\right]_{(rw, \th+\xi)}-\left[\f{\pa|y_i|}{\pa \r}g_0\right]_{(rw, \th)}\right) \,dw\\
&\;-\f{r}{4\pi}\int_{\BT}d\xi\int_{0}^{1+4\d} \left[ \f{\pa J}{\pa s}(\tilde{b}_i,\xi)\f{\pa \tilde{b}_i}{\pa \xi}+ \f{\pa J}{\pa\xi}(\tilde{b}_i,\xi)\right]\left(\left[\f{\pa|y_i|}{\pa \r}g_0\right]_{(rw, \th+\xi)}-\left[\f{\pa|y_i|}{\pa \r}g_0\right]_{(rw, \th)}\right)\,dw\\
&\;+\f{r}{4\pi}\int_{\BT}d\xi\int_{0}^{1+4\d}\f{\pa J}{\pa s}(b_i,\xi)\f{\pa b_i}{\pa \th} \left[\f{\pa|y_i|}{\pa \r}g_0\right]_{(rw, \th)}\,dw\\
=:&\;J_{r,1}^{(i)}+J_{r,2}^{(i)}+J_{r,3}^{(i)}+J_{r,4}^{(i)}.
\end{split}
\label{eqn: decomposition of derivative of normal derivative}
\eeq
Estimates concerning $J_{r,1}^{(i)}+J_{r,2}^{(i)}+J_{r,3}^{(i)}$ can be derived exactly as in Lemma \ref{lem: W 1p estimate of tangent derivative}.
It remains to bound $J_{r,4}^{(1)}-J_{r,4}^{(2)}$.
%
By Lemma \ref{lem: properties of Poisson kernel},
\beq
\int_\BT \f{\pa J}{\pa s}(s,\xi)\,d\xi =
\begin{cases}
-4\pi, & \mbox{if } s\in[0,1), \\
0, & \mbox{if } s>1.
\end{cases}
\label{eqn: integral of kernel J}
\eeq
Hence, thanks to Lemma \ref{lem: estimates involving b} and \eqref{eqn: ingredient 1}, 
\beq
\begin{split}
&\;|J_{r,4}^{(1)}-J_{r,4}^{(2)}|\\
= &\;r \left|\int_0^1 \f{\pa b_1}{\pa \th} \left[\f{\pa|y_1|}{\pa \r}g_0\right]_{(rw, \th)}-\f{\pa b_2}{\pa \th} \left[\f{\pa|y_2|}{\pa \r}g_0\right]_{(rw, \th)}\,dw\right|\\
\leq &\;Cr\|g_0\|_{L^\infty} \int_0^1 \left|\f{\pa b_1}{\pa \th}-\f{\pa b_2}{\pa \th}\right| \left|\f{\pa|y_1|}{\pa \r}\right|+\left|\f{\pa b_2}{\pa \th}\right| \left|\f{\pa|y_1|}{\pa \r}-\f{\pa|y_2|}{\pa \r}\right|\,dw\\
\leq &\;Cr\|g_0\|_{L^\infty} \D m_0.
\end{split}
\eeq
This completes the proof.
\end{proof}
\end{lem}

\begin{lem}\label{lem: W 1p estimate of derivative of growth potential inner}
Assume $h\in C^{1,\beta}(\mathbb{T})$ for some $\beta\in (0,1)$, such that $m_{0} \ll 1$.
Define $g(x) = g_0(X(x))$.
Then for all $p\in [2,\infty)$,
\beq
\begin{split}
&\;\|(e_\th\cdot \na(\G*g))_{\g(\th)}\|_{\dot{W}^{1,p}(\BT)}+\|(e_r\cdot \na(\G*g))_{\g(\th)}\|_{\dot{W}^{1,p}(\BT)}\\
\leq
&\;Cr(\|g_0\|_{L^\infty(B_{(1+4\d)r})}m_\b+ \|e_\th\cdot \na g_0\|_{L^2(B_{(1+4\d)r})}).
\label{eqn: W 1p estimate of derivative of growth potential inner}
\end{split}
\eeq
where $C = C(p,\b)$.
Here $m_\b$ is defined as in \eqref{eqn: def of m_alpha}.
\begin{proof}
As in Lemma \ref{lem: L^inf estimate of derivative of growth potential inner}, we first study the case with $h = 0$.
By \eqref{eqn: decomposition of derivative of tangent derivative},
\beq
\f{d}{d\th}(e_\th\cdot \na(\G*g_0))_{\pa B_r}
=
-\f{r}{4\pi}\int_{\BT}d\xi \int_0^{1+4\d} dw\,
\f{\pa K}{\pa \xi}(w,\xi)(g_0(rw,\xi+\th)-g_0(rw,\th)).
\eeq
Hence, arguing as in \eqref{eqn: L p bound for difference of J theta 2 and J theta 3 from two configurations},
\beq
\begin{split}
&\;
\|(e_\th\cdot \na(\G*g_0))_{\pa B_r}\|_{\dot{W}^{1,p}}\\
\leq &\;Cr\int_{\BT}d\xi \int_0^{1+4\d} dw\,
\f{\|g_0(rw,\xi+\cdot)-g_0(rw,\cdot)\|_{L^p_\th(\BT)}}{(|1-w|+|\xi|)^2}\\
\leq &\;Cr\|e_\th\cdot \na g_0\|_{L^2(B_{(1+4\d)r})}.
\end{split}
\eeq

Now taking $h_1 = h$ and $h_2 = 0$ in Lemma \ref{lem: W 1p estimate of tangent derivative}, we find that
\beq
\begin{split}
&\;\|(e_\th\cdot \na(\G*g))_{\g(\th)}\|_{\dot{W}^{1,p}(\BT)}\\
\leq &\;\|(e_\th\cdot \na(\G*g))_{\g(\th)}-(e_\th\cdot \na(\G*g_0))_{\pa B_r}\|_{\dot{W}^{1,p}(\BT)}
+\|(e_\th\cdot \na(\G*g_0))_{\pa B_r}\|_{\dot{W}^{1,p}(\BT)}\\
\leq &\;Cr\|g_0\|_{L^\infty(B_{(1+4\d)r})}( m_0 +\d^\b\|h'\|_{\dot{C}^\b})+Cr \|e_\th\cdot \na g_0\|_{L^2(B_{(1+4\d)r})}.
\label{eqn: W 1p estimate of tangent derivative inner}
\end{split}
\eeq
The estimate for $(e_r\cdot \na(\G*g))_{\g(\th)}$ can be derived in exactly the same way.
\end{proof}
\end{lem}

\subsection{Estimates along $\tilde{\g}$}
Next, we derive estimates for $e_r\cdot \na (\G*g)$ and $e_\th\cdot \na (\G*g)$ along $\tilde{\g}$, with $g(x) = g_0(X(x))$.
We calculate as in \eqref{eqn: formula for tangent derivative crude form} that
\beq
\begin{split}
(e_\th\cdot \na(\G*g))_{\tilde{\g}(\th)} =&\;\f{1}{2\pi}\int_{\BT}d\xi \int_0^{r(1+4\d)} \f{|y|\sin\xi \cdot g_0(\rho, \th+\xi)}{F(\th)^2+|y|^2-2|y|F(\th)\cos\xi}\cdot\f{\pa |y|}{\pa \r}|y|\,d\r\\
=&\;\f{r}{4\pi}\int_{\BT}d\xi \int_0^{1+4\d} K(\tilde{B},\xi)\cdot\f{\pa |y|}{\pa \r}g_0(rw, \th+\xi)\,dw,
\label{eqn: tangent derivative of growth potential simplified at outer interface}
\end{split}
\eeq
and
\beq
(e_r\cdot \na(\G*g))_{\tilde{\g}(\th)} 
= 
\f{r}{4\pi}\int_{\BT}d\xi\int_{0}^{1+4\d} J(\tilde{B},\xi)\cdot  \f{\pa|y|}{\pa \r}g_0(rw, \th+\xi)\,dw.
\label{eqn: normal derivative of growth potential along outer interface}
\eeq

Arguing as in Lemma \ref{lem: L inf estimate of difference same source}, we can show
\begin{lem}\label{lem: L inf estimate of difference same source at outer interface}
Under the assumptions of Lemma \ref{lem: L inf estimate of difference same source},
\beq
\|(e_\th\cdot \na(\G*g_1) )_{\tilde{\g}_1(\th)}-(e_\th\cdot \na(\G*g_2) )_{\tilde{\g}_2(\th)}\|_{L^\infty(\BT)}
\leq  \f{Cr^2}{R}\d |\ln \d| (\D m_0+\D M_0)\|g_0\|_{L^\infty},
\eeq
where $C$ is universal.
Moreover, $\|(e_r\cdot \na(\G*g_1) )_{\tilde{\g}_1(\th)}-(e_r\cdot \na(\G*g_2) )_{\tilde{\g}_2(\th)}\|_{L^\infty(\BT)}$ satisfies the same estimate.

\end{lem}

We omit its proof here, but only note that $|\tilde{B}_i|\leq \f{Cr}{R}$ and $|\ln (1-\f{(1+4\d)r}{R})|\leq \f{Cr}{R}|\ln \d|$.

Then we prove as in Lemma \ref{lem: L^inf estimate of derivative of growth potential inner} that
\begin{lem}\label{lem: L^inf estimate of derivative of growth potential outer}
Let $h,H\in W^{1,\infty}(\BT)$ such that $m_{0},M_{0}\ll 1$, which define the map $x$ in \eqref{eqn: change of variables to the reference coordinate} and $g = g_0(X(x))$.
Then
\beq
\begin{split}
&\;\|(e_\th\cdot \na(\G*g))_{\tilde{\g}(\th)} \|_{L^\infty(\BT)}+
\|(e_r\cdot \na(\G*g))_{\tilde{\g}(\th)}-\tilde{c}_{g_0} \|_{L^\infty(\BT)}\\
\leq &\;\f{Cr^2}{R}( (m_0+M_0) \d |\ln \d| \|g_0\|_{L^\infty(B_{(1+4\d)r})}+ \|e_\th\cdot \na g_0 \|_{L^2(B_{r(1+4\d)})}),
\end{split}
\label{eqn: L inf bound of derivative of potential outer}
\eeq
where $C$ is universal and
\beq
\tilde{c}_{g_0} := -\f{1}{2\pi R}\int_{B_{r(1+4\d)}}g_0(X)\,dX.
\label{eqn: def of c characteristic normal derivative outer interface}
\eeq
\begin{proof}

Let $\bar{g}_0$ be as in Lemma \ref{lem: L^inf estimate of derivative of growth potential inner}.
We proceed as in \eqref{eqn: bound tangent derivative of potential for a perfect circle inner} by noticing that $K(\f{rw}{R},\cdot)$ is an odd kernel.
\beq
\begin{split}
&\;|(e_\th\cdot \na(\G*g_0) )_{\pa B_R}|\\
= &\; \f{r}{4\pi}\left|\int_{\BT}d\xi \int_0^{1+4\d} K\left(\f{rw}{R},\xi\right)(g_0(rw, \th+\xi)-\bar{g}_0(rw))\,dw\right|\\
\leq &\; Cr \int_0^{1+4\d}\left\|\f{\f{r}{R}}{|1-\f{rw}{R}|+|\xi|}\right\|_{L^1_\xi(\BT)}\|g_0(rw, \cdot)-\bar{g}_0(rw)\|_{L^\infty_\xi(\BT)}\,dw\\
\leq &\; \f{Cr^2}{R} \|e_\th\cdot \na g_0 \|_{L^2(B_{r(1+4\d)})}.
\end{split}
\label{eqn: bound tangent derivative of potential for a perfect circle outer}
\eeq
Combining this and Lemma \ref{lem: L inf estimate of difference same source at outer interface} with $h_1= h$, $H_1 = H$ and $h_2 = H_2 = 0$, we argue as in \eqref{eqn: bounding L inf of tangential derivative of growth potential at inner interface} to find that $\|(e_\th\cdot \na(\G*g))_{\tilde{\g}(\th)} \|_{L^\infty(\BT)}$ satisfies the desired bound.

Similarly,
\beq
\begin{split}
&\;(e_r\cdot \na(\G*g_0) )_{\pa B_R}\\
=&\;\f{r}{4\pi}\int_{\BT}d\xi \int_0^{1+4\d} J\left(\f{rw}{R},\xi\right)(g_0(rw, \th+\xi)-\bar{g}_0(rw))\,dw \\
&\;+ \f{r}{4\pi} \int_0^{1+4\d} \int_{\BT}d\xi \,J\left(\f{rw}{R},\xi\right)\bar{g}_0(rw)\,dw.
\end{split}
\eeq
The first term can be bounded exactly as in \eqref{eqn: bound tangent derivative of potential for a perfect circle outer}.
For the second term, we notice that $\f{r(1+4\d)}{R}\leq 1$.
By \eqref{eqn: def of function J},
\beq
\begin{split}
&\;\f{r}{4\pi} \int_0^{1+4\d} \int_{\BT}d\xi \,J\left(\f{rw}{R},\xi\right)\bar{g}_0(rw)\,dw \\
= &\;-\f{r^2}{R} \int_0^{1+4\d} w\bar{g}_0(rw)\,dw = -\f{1}{2\pi R}\int_{B_{r(1+4\d)}} g_0(X)\,dX.
\end{split}
\eeq
Then the desired estimate follows.
\end{proof}
\end{lem}

We shall follow Lemma \ref{lem: W 1p estimate of tangent derivative} and Lemma \ref{lem: W 1p estimate of normal derivative difference} to prove $W^{1,p}$-estimates concerning $(e_\th\cdot \na(\G*g))_{\tilde{\g}(\th)}$ and $(e_r\cdot \na(\G*g))_{\tilde{\g}(\th)}$.

\begin{lem}\label{lem: W 1p estimate of difference of tangent derivative outer interface}
Assume $h_i,H_i\in W^{1,\infty}(\BT)$ $(i = 1,2)$ such that $m_{0,i}+M_{0,i} \ll 1$.
Let $\D m_0$ and $\D M_0$ be defined in \eqref{eqn: def of normalized W 1 inf difference of two h} and \eqref{eqn: def of normalized W 1 inf difference of two H large}, respectively.
Define $g_i(x) = g_0(X_i(x))$ as before.
Then for all $p\in [2,\infty)$, 
\beq
\begin{split}
&\;\|(e_\th\cdot \na(\G*g_1))_{\tilde{\g}_1(\th)}-(e_\th\cdot \na(\G*g_2))_{\tilde{\g}_2(\th)}\|_{\dot{W}^{1,p}(\BT)}\\
\leq  &\;\f{Cr^2}{R}\|g_0\|_{L^\infty(B_{(1+4
\d)r})}(\D m_0 +(m_{0,1}+m_{0,2})\D M_0)\\
&\;+\f{Cr^2}{R} (\D m_0+\D M_0)\|e_\th\cdot \na g_0\|_{L^2(B_{(1+4\d)r})},
\end{split}
\label{eqn: W 1p estimate for difference of tangent derivative outer interface}
\eeq
where $C = C(p)$.
\begin{proof}

Following \eqref{eqn: decomposition of derivative of tangent derivative} and \eqref{eqn: tangent derivative of growth potential simplified at outer interface},
\beq
\begin{split}
&\;
\f{d}{d\th}(e_\th\cdot \na(\G*g_i))_{\tilde{\g}_i(\th)}\\
= &\;
\f{r}{4\pi}\int_{\BT}d\xi \int_0^{1+4\d} dw\, \left[\f{\pa K}{\pa s}(\tilde{B}_i, \xi)\f{\pa \tilde{B}_i}{\pa \th}-\f{\pa K}{\pa s}(B_i,\xi)\f{\pa B_i}{\pa \th}\right]\left[\f{\pa |y_i|}{\pa \r}g_0\right]_{(rw,\xi+\th)}\\
&\;+\f{r}{4\pi}\int_{\BT}d\xi \int_0^{1+4\d} dw\, \f{\pa K}{\pa s}(B_i,\xi)\f{\pa B_i}{\pa \th}\left(\left[\f{\pa |y_i|}{\pa \r}g_0\right]_{(rw,\xi+\th)}-\left[\f{\pa |y_i|}{\pa \r}g_0\right]_{(rw,\th)}\right)\\
&\;-\f{r}{4\pi}\int_{\BT}d\xi \int_0^{1+4\d} dw\,
 \left[\f{\pa K}{\pa s}(\tilde{B}_i,\xi)\f{\pa \tilde{B}_i}{\pa \xi}+\f{\pa K}{\pa \xi}(\tilde{B}_i,\xi)\right]\left(\left[\f{\pa |y_i|}{\pa \r}g_0\right]_{(rw,\xi+\th)} -\left[\f{\pa |y_i|}{\pa \r}g_0\right]_{(rw,\th)}\right)\\
=:&\;\tilde{J}_{\th,1}^{(i)}+\tilde{J}_{\th,2}^{(i)}+\tilde{J}_{\th,3}^{(i)}.
\end{split}
\label{eqn: decomposition of derivative of tangent derivative outer interface}
\eeq
Then we derive as in \eqref{eqn: decomposition of difference of J theta 1} and \eqref{eqn: estimate for difference of J theta 1 from two configurations} to find that
\beq
\begin{split}
&\;|\tilde{J}_{\th,1}^{(1)}-\tilde{J}_{\th,1}^{(2)}|\\
\leq &\;\f{Cr^2}{R}\|g_0\|_{L^\infty}(\D m_0+\D M_0) (\|h_1'\|_{L^\infty}+\|h_2'\|_{L^\infty})+\f{Cr^2}{R}\|g_0\|_{L^\infty} \|h_1'-h_2'\|_{L^\infty}.
\end{split}
\label{eqn: estimate for difference of tilde J theta 1 from two configurations}
\eeq
Here we used the fact that $|1-\f{rw}{R}|\geq C\d$ for all $w\in [0,1+4\d]$.
Moreover, as in \eqref{eqn: splitting J theta 2 3 difference} and \eqref{eqn: bounding J theta 2 3 difference},
\beq
\begin{split}
&\;\left|(\tilde{J}_{\th,2}^{(1)}+\tilde{J}_{\th,3}^{(1)}) - (\tilde{J}_{\th,2}^{(2)}+\tilde{J}_{\th,3}^{(2)})\right|\\
\leq  &\;\f{Cr^2}{R}\|g_0\|_{L^\infty}(\D m_0 +\d^{\b-1}\|h_2\|_{\dot{C}^\b}\D M_0)\\
&\;+\f{Cr^2}{R} (\D m_0+\D M_0)\int_{\BT}d\xi \int_0^{1+4\d} dw\,\f{|g_0(rw,\xi+\th)-g_0(rw,\th)|}{(|1-\f{rw}{R}|+|\xi|)^2}.
\end{split}
\eeq
We proceed as in \eqref{eqn: L p bound for difference of J theta 2 and J theta 3 from two configurations} to obtain that
\beq
\begin{split}
&\;\left\|(\tilde{J}_{\th,2}^{(1)}+\tilde{J}_{\th,3}^{(1)}) - (\tilde{J}_{\th,2}^{(2)}+\tilde{J}_{\th,3}^{(2)})\right\|_{L^p(\BT)}\\
\leq  &\;\f{Cr^2}{R}\|g_0\|_{L^\infty}(\D m_0 +\d^{\b-1}\|h_2\|_{\dot{C}^\b}\D M_0)\\
&\;+\f{Cr^2}{R} (\D m_0+\D M_0)\|e_\th\cdot \na g_0\|_{L^2(B_{(1+4\d)r})} .
\end{split}
\eeq
Combining this with \eqref{eqn: decomposition of derivative of tangent derivative outer interface} and \eqref{eqn: estimate for difference of tilde J theta 1 from two configurations}, we prove \eqref{eqn: W 1p estimate for difference of tangent derivative outer interface}.
\end{proof}
\end{lem}

\begin{lem}\label{lem: W 1p estimate of difference of normal derivative outer interface}
Under the assumptions of Lemma \ref{lem: W 1p estimate of difference of tangent derivative outer interface},
\beq
\begin{split}
&\;\|(e_r\cdot \na(\G*g_1))_{\tilde{\g}_1(\th)}-(e_r\cdot \na(\G*g_2))_{\tilde{\g}_2(\th)}\|_{\dot{W}^{1,p}(\BT)}\\
\leq  &\;\f{Cr^2}{R}(\D m_0 +\D M_0)(\|g_0\|_{L^\infty(B_{(1+4
\d)r})}+\|e_\th\cdot \na g_0\|_{L^2(B_{(1+4\d)r})}),
\end{split}
\label{eqn: W 1p estimate for difference of normal derivative outer interface}
\eeq
where $C = C(p)$.
\begin{proof}
Following the proofs of Lemma \ref{lem: W 1p estimate of normal derivative difference} and Lemma \ref{lem: W 1p estimate of difference of tangent derivative outer interface}, we know that it remains to bound $\tilde{J}_{r,4}^{(1)}-\tilde{J}_{r,4}^{(2)}$, where
\beq
\tilde{J}_{r,4}^{(i)} := \f{r}{4\pi}\int_{\BT}d\xi\int_{0}^{1+4\d}\f{\pa J}{\pa s}(B_i,\xi)\f{\pa B_i}{\pa \th} \left[\f{\pa|y_i|}{\pa \r}g_0\right]_{(rw, \th)}\,dw.
\eeq
Since for all $w\in [0,1+4\d]$ and $\xi\in \BT$, $B_i\leq 1$.
By Lemma \ref{lem: estimates involving b}, \eqref{eqn: ingredient 1} and \eqref{eqn: integral of kernel J}, 
\beq
\begin{split}
&\;|\tilde{J}_{r,4}^{(1)}-\tilde{J}_{r,4}^{(2)}| \\
\leq &\;Cr\|g_0\|_{L^\infty}\int_{0}^{1+4\d}\left|\f{\pa B_1}{\pa \th} -\f{\pa B_2}{\pa \th} \right| \left|\f{\pa|y_1|}{\pa \r}\right|
+\left|\f{\pa B_2}{\pa \th} \right| \left|\f{\pa|y_1|}{\pa \r}-\f{\pa|y_2|}{\pa \r}\right|\,dw\\
\leq &\;\f{Cr^2}{R}\|g_0\|_{L^\infty}(\D m_0+\D M_0).
\end{split}
\eeq
Then by Lemma \ref{lem: W 1p estimate of difference of tangent derivative outer interface}, \eqref{eqn: W 1p estimate for difference of normal derivative outer interface} follows.
\end{proof}
\end{lem}

\begin{lem}\label{lem: W 1p estimate of derivative of growth potential outer}
Assume $h, H\in W^{1,\infty}(\mathbb{T})$, such that $m_{0}+M_0 \ll 1$.
Define $g(x) = g_0(X(x))$.
Then for all $p\in [2,\infty)$,
\beq
\begin{split}
&\;\|(e_\th\cdot \na(\G*g))_{\tilde{\g}(\th)}\|_{\dot{W}^{1,p}(\BT)} +\|(e_r\cdot \na(\G*g))_{\tilde{\g}(\th)}\|_{\dot{W}^{1,p}(\BT)}\\
\leq  &\;\f{Cr^2}{R}((m_0 + M_0)\|g_0\|_{L^\infty(B_{(1+4
\d)r})}+\|e_\th\cdot \na g_0\|_{L^2(B_{(1+4\d)r})}),
\end{split}
\label{eqn: W 1p estimate for normal derivative outer interface}
\eeq
where $C = C(p)$.
\begin{proof}
We first study the case with $h = H = 0$.
By \eqref{eqn: decomposition of derivative of tangent derivative outer interface},
\beq
\begin{split}
&\;
\f{d}{d\th}(e_\th\cdot \na(\G*g_0))_{\pa B_R}\\
=&\;-\f{r}{4\pi}\int_{\BT}d\xi \int_0^{1+4\d} dw\,
\f{\pa K}{\pa \xi}\left(\f{rw}{R},\xi\right)(g_0(rw,\xi+\th)-g_0(rw,\th)).
\end{split}
\eeq
Hence, arguing as in \eqref{eqn: L p bound for difference of J theta 2 and J theta 3 from two configurations},
\beq
\begin{split}
&\;
\|(e_\th\cdot \na(\G*g_0))_{\pa B_R}\|_{\dot{W}^{1,p}}\\
\leq &\;Cr\int_{\BT}d\xi \int_0^{1+4\d} dw\,\f{r}{R}\cdot
\f{\|g_0(rw,\xi+\cdot)-g_0(rw,\cdot)\|_{L^p_\th(\BT)}}{(|1-\f{rw}{R}|+|\xi|)^2}\\
\leq &\;\f{Cr^2}{R}\|e_\th\cdot \na g_0\|_{L^2(B_{(1+4\d)r})}.
\end{split}
\eeq
The rest of the proof is the same as that of Lemma \ref{lem: W 1p estimate of derivative of growth potential inner}.
\end{proof}
\end{lem}

\section{Estimates for Singular Integral Operators $\CK_{\g}$ and $\CK_{\tilde{\g}}$}\label{sec: estimates for singular integral operators}
In this section, we shall derive estimates for singular integrals of type $\g'(\th)^\perp\cdot\mathcal{K}_\g\p$ and $\g'(\th)\cdot\mathcal{K}_\g\p$ (see the definition in \eqref{eqn: def of operator K}.)
Singular integrals involving $\CK_{\tilde{\g}}$ then follow similar estimates.

For convenience, for $\xi\in \BT\backslash\{ 0\}$, denote
\beq
\D f(\th) := \f{f(\th+\xi)-f(\th)}{2\sin\f{\xi}{2}},
\eeq
and
\beq
l(\th,\th+\xi) := \f{(\D f)^2}{f(\th)f(\th+\xi)} = \f{(\D h)^2}{(1+h(\th))(1+h(\th+\xi))}.
\label{eqn: def of l}
\eeq

We first derive a H\"{o}lder estimate for $\g'^\perp\cdot\mathcal{K}_\g\p$ for future use.

\begin{lem}\label{lem: Holder estimate of singular integral normal component}
Fix $\b\in (0,1)$.
Assume $h\in C^{1,\b}(\mathbb{T})$, such that $m_0 \ll 1$.
Then
\beq
\|\g'(\th)^\perp\cdot\mathcal{K}_\g\p\|_{\dot{C}^\b}\\
\leq  C\|h'\|_{\dot{C}^\beta} (\|\p\|_{C^\b}+\|\p\|_{L^\infty}\|h'\|_{\dot{C}^\beta}\|h'\|_{L^\infty}),
\label{eqn: Holder estimate of singular integral normal component}
\eeq
where $C = C(\b)$.
\begin{proof}
Using $\g(\th) = f(\th)(\cos\th,\sin\th)$, 
\beq
2\pi\g'(\th)^\perp\cdot\mathcal{K}_\g\p
=\mathrm{p.v.}\int_\BT \f{-f(\th)^2+f(\th)f(\th+\xi)\cos\xi-f'(\th)f(\th+\xi)\sin\xi} {f(\th)^2+f(\th+\xi)^2-2f(\th)f(\th+\xi)\cos\xi}\p(\th+\xi)\,d\xi.
\eeq
With $f(\th) = r(1+h(\th))$, it can be rewritten as
\beq
\begin{split}
&\;2\pi\g'(\th)^\perp\cdot\mathcal{K}_\g\p\\
= &\;-\f{1}{2}\int_{\BT}\p\,d\xi-\f{1}{2}\int_\BT \f{(f(\th+\xi)-f(\th))^2} {(f(\th)-f(\th+\xi))^2+f(\th)f(\th+\xi)\cdot 4\sin^2\f{\xi}{2}}\p(\th+\xi)\,d\xi\\
&\;+\mathrm{p.v.}\int_\BT \f{(f(\th+\xi)-f(\th))f(\th+\xi)-f'(\th)f(\th+\xi)\sin\xi} {(f(\th)-f(\th+\xi))^2+f(\th)f(\th+\xi)\cdot 4\sin^2\f{\xi}{2}}\p(\th+\xi)\,d\xi\\
= &\;-\f{1}{2}\int_{\BT}\p\,d\xi-\f{1}{2}\int_\BT \f{l(\th,\th+\xi)} {1+l(\th,\th+\xi)}\p(\th+\xi)\,d\xi\\
&\;+\f{1}{1+h(\th)}\mathrm{p.v.}\int_\BT \f{\D h}{2\sin\f{\xi}{2}}\cdot \f{\p(\th+\xi)} {1+l(\th,\th+\xi)}\,d\xi\\
&\;+\f{1}{1+h(\th)}\mathrm{p.v.}\int_\BT -\f{h'(\th)}{2\tan\f{\xi}{2}}\cdot \f{\p(\th+\xi)} {1+l(\th,\th+\xi)}\,d\xi\\
%
%
=:&\;L_0+L_1(\th)+L_2(\th)+L_3(\th).
\end{split}
\label{eqn: crude form of singular integral dot with normal}
\eeq

Since $\|fg\|_{\dot{C}^\b}\leq \|f\|_{\dot{C}^\b}\|g\|_{L^\infty}+\|f\|_{L^\infty}\|g\|_{\dot{C}^\b}$,
\beq
\begin{split}
\|L_1\|_{\dot{C}^\b}\leq &\;C\sup_{\xi\in \BT}\left\|\f{l} {1+l}\p(\th+\xi)\right\|_{\dot{C}^\b_\th}\\
\leq &\;C\sup_{\xi\in \BT}\left\|\f{l} {1+l}\right\|_{\dot{C}^\b_\th}\|\p\|_{L^\infty}
+C\sup_{\xi\in \BT}\left\|\f{l} {1+l}\right\|_{L^\infty_\th}\|\p\|_{\dot{C}^\b}.
\end{split}
\label{eqn: Holder bound for L_1 crude form}
\eeq
By the Lipschitz continuity of $\f{x}{1+x}$ on $[0,+\infty)$ and the smallness of $h$,
\beq
\begin{split}
\|L_1\|_{\dot{C}^\b}
\leq &\;C\sup_{\xi\in \BT}\left\|\f{(\D h)^2}{(1+h(\th))(1+h(\th+\xi))}\right\|_{\dot{C}^\b_\th}\|\p\|_{L^\infty}
+C\|h'\|_{L^\infty}^2\|\p\|_{\dot{C}^\b}\\
\leq &\;C(\|h'\|_{\dot{C}^\beta}\|h'\|_{L^\infty}\|\p\|_{L^\infty}+\|h'\|_{L^\infty}^2\|\p\|_{\dot{C}^\b}).
\end{split}
\label{eqn: estimate finite difference of L1}
\eeq
%
%
Here we used
\beq
\|\D h\|_{\dot{C}^\b_\th}= \f{\|h(\th+\xi)-h(\th)\|_{\dot{C}^\b_\th}}{\left|2\sin\f{\xi}{2}\right|}\leq \left|\f{1}{2\sin\f{\xi}{2}}\int_0^\xi \|h'(\th+\eta)\|_{\dot{C}^\b_\th}\,d\eta\right|
\leq C\|h'\|_{\dot{C}^\beta}.
\label{eqn: bounding difference between two finite differences}
\eeq

Take $\e\in \BT$ and $\e\geq 0$ without loss of generality.
Write
\beq
\begin{split}
&\;(L_2+L_3)(\th+\e)-(L_2+L_3)(\th)\\
=&\;\left(\f{1}{1+h(\th+\e)}-\f{1}{1+h(\th)}\right)\int_\BT \f{\D h(\th+\e)-\cos\f{\xi}{2}h'(\th+\e)}{2\sin\f{\xi}{2}}\cdot \f{\p(\th+\e+\xi)} {1+l(\th+\e,\th+\e+\xi)}\,d\xi\\
&\;+\f{1}{1+h(\th)}\int_\BT\f{\D h(\th+\e)-\cos\f{\xi}{2}\cdot h'(\th+\e)}{2\sin\f{\xi}{2}} \\
&\;\qquad \cdot \left(\f{\p(\th+\e+\xi)} {1+l(\th+\e,\th+\e+\xi)}-\f{\p(\th+\xi)} {1+l(\th,\th+\xi)}\right)\,d\xi\\
&\;+\f{1}{1+h(\th)} \int_\BT\f{\D h(\th+\e)-\D h(\th)-\cos\f{\xi}{2}(h'(\th+\e)- h'(\th))}{2\sin\f{\xi}{2}} \f{\p(\th)} {1+\f{ h'(\th)^2}{(1+h(\th))^2}}\,d\xi\\
&\;+\f{1}{1+h(\th)}\int_\BT\f{\D h(\th+\e)-\D h(\th)-\cos\f{\xi}{2}(h'(\th+\e)- h'(\th))}{2\sin\f{\xi}{2}} \\
&\;\qquad \cdot\left(\f{\p(\th+\xi)} {1+l(\th,\th+\xi)}-\f{\p(\th)} {1+\f{ h'(\th)^2}{(1+h(\th))^2}}\right)\,d\xi.
\end{split}
\label{eqn: sum of L_2 and L_3}
\eeq
We derive that
\beq
\begin{split}
&\;\left|\D h(\th+\e)-\cos\f{\xi}{2}\cdot h'(\th+\e)\right|\\
\leq &\;\left|\f{ \int_0^\xi h'(\th+\e+\eta)- h'(\th+\e)\,d\eta}{2\sin\f{\xi}{2}}\right|+\left|\f{\xi-\sin\xi}{2\sin\f{\xi}{2}}h'(\th+\e)\right|\\
\leq &\;C|\xi|^\b\|h'\|_{\dot{C}^\b},
\end{split}
\label{eqn: difference of difference quotient and the derivative}
\eeq
and
\beq
\begin{split}
&\;\left|\D h(\th+\e)-\D h(\th)-\cos\f{\xi}{2}(h'(\th+\e)-h'(\th))\right|\\
\leq &\;
\left|\f{1}{2\sin\f{\xi}{2}}\int_0^\xi h'(\th+\e+\eta)-h'(\th+\eta)\,d\eta\right|+| h'(\th+\e)- h'(\th)|\\
\leq&\; C\e^\b\|h'\|_{\dot{C}^\b}.
\end{split}
\label{eqn: absolute difference of Taylor expansion type error}
\eeq
Thanks to \eqref{eqn: estimate finite difference of L1} and \eqref{eqn: bounding difference between two finite differences},
\beq
\begin{split}
&\;\left|\f{\p(\th+\e+\xi)} {1+l(\th+\e,\th+\e+\xi)}-\f{\p(\th+\xi)} {1+l(\th,\th+\xi)}\right|\\
\leq &\;C\e^\b\|\p\|_{\dot{C}^\b}\\
&\;+C\|\p\|_{L^\infty}\left|\f{(\D h(\th+\e))^2}{(1+h(\th+\e))(1+h(\th+\e+\xi))}-\f{(\D h)^2}{(1+h(\th))(1+h(\th+\xi))}\right|\\
\leq &\;C\e^\b(\|\p\|_{\dot{C}^\b}+\|\p\|_{L^\infty}\|h'\|_{\dot{C}^\beta}\|h'\|_{L^\infty}),
\end{split}
\eeq
and similarly,
\beq
\left|\f{\p(\th+\xi)} {1+l(\th,\th+\xi)}-\f{\p(\th)} {1+\f{h'(\th)^2}{(1+h(\th))^2}}\right|
\leq C|\xi|^\b(\|\p\|_{\dot{C}^\b}+\|\p\|_{L^\infty}\|h'\|_{\dot{C}^\b}\|h'\|_{L^\infty}).
\eeq
Lastly, 
\beq
\begin{split}
&\;\left|\int_\BT\f{\D h(\th+\e)-\D h(\th)-\cos\f{\xi}{2}(h'(\th+\e)- h'(\th))}{2\sin\f{\xi}{2}} \,d\xi\right|\\
=&\;\left|\mathrm{p.v.}\int_\BT\f{h(\th+\e+\xi)-h(\th+\e)-h(\th+\xi)+h(\th)}{4\sin^2\f{\xi}{2}} \,d\xi\right|\\
= &\;C|\CH h'(\th+\e)-\CH h'(\th)|\\
\leq &\;C\e^\b\|h'\|_{\dot{C}^\b}.
\end{split}
\label{eqn: Holder estimate using Hilbert transform}
\eeq
Note that Hilbert transform is bounded in $C^\b(\BT)$.

Combining these estimates with \eqref{eqn: sum of L_2 and L_3}, we obtain that
\beq
\begin{split}
&\;|(L_2+L_3)(\th+\e)-(L_2+L_3)(\th)|\\
\leq &\;C\e^\b\|h'\|_{\dot{C}^\b}(\|\p\|_{C^\b}+\|\p\|_{L^\infty}\|h'\|_{\dot{C}^\b}\|h'\|_{L^\infty}).
\end{split}
\label{eqn: estimate finite difference of L2+L3}
\eeq
Then \eqref{eqn: Holder estimate of singular integral normal component} follows from \eqref{eqn: crude form of singular integral dot with normal}, \eqref{eqn: estimate finite difference of L1} and \eqref{eqn: estimate finite difference of L2+L3}.
\end{proof}
\end{lem}

Now we turn to a $\dot{W}^{1,p}$-estimate of $\g'^\perp\cdot\mathcal{K}_\g\p$.

\begin{lem}\label{lem: W 1p estimate of normal component of singular integral}
Fix $p\in [2,\infty)$.
Assume $h\in C^{1,\beta}(\mathbb{T})$ for some $\beta\in (0,1)$, such that $m_0 \ll 1$ with the needed smallness depending on $p$.
Then
\beq
\begin{split}
&\;\|\g'(\th)^\perp\cdot\mathcal{K}_\g\p\|_{\dot{W}^{1,p}}\\
\leq &\;C\| h''\|_{L^p}\|\p\|_{L^\infty}(1+\|h'\|_{\dot{C}^\b})+C( \|h''\|_{L^p}\|\p\|_{\dot{C}^\b} +\|h'\|_{L^\infty}\|\p'\|_{L^p}),
\end{split}
\eeq
where $C = C(p,\b)$.
\begin{proof}
Let $C_*$ and $C_\dag$ be the constants introduced in Lemma \ref{lem: Lp bound for multi-term commutator} and Lemma \ref{lem: W 1p estimate for multi-term commutator}, respectively, both of which only depend on $p$.
Without loss of generality, we may assume $C_\dag\geq C_*\geq 1$.
We also recall that $l$ is defined in \eqref{eqn: def of l}.

Using the notation in \eqref{eqn: crude form of singular integral dot with normal}, we take $\th$-derivative of $L_1$ to derive that 
\beq
\begin{split}
\|L_1\|_{\dot{W}^{1,p}}\leq &\;C\left\|\int_{\BT}\|h'\|_{L^\infty}^2|\p'(\th+\xi)|\,d\xi\right\|_{L^p}+C\left\|\int_{\BT}(\|h'\|_{L^\infty}|\D h'|+\|h'\|_{L^\infty}^3)\|\p\|_{L^\infty}\,d\xi\right\|_{L^p}\\
\leq &\;C\|h'\|_{L^\infty}^2\|\p'\|_{L^p}+C\|h'\|_{L^\infty}\| h''\|_{L^p}\|\p\|_{L^\infty}.
\end{split}
\label{eqn: W1p estimate for L1}
\eeq

Thanks to the smallness of $h$, we may assume $|l| <1$.
Hence, by Taylor expanding $(1+l)^{-1}$, 
we may rewrite $L_2$ in \eqref{eqn: crude form of singular integral dot with normal} as
\beq
\begin{split}
L_2 = &\;\sum_{j = 0}^\infty(-1)^j(1+h(\th))^{-(j+1)}
\mathrm{p.v.}\int_\BT(\D h)^{2j+1}(1+h(\th+\xi))^{-j}\cdot \f{\p(\th+\xi)}{2\sin\f{\xi}{2}}\,d\xi=:\sum_{j=0}^\infty L_{2,j}.
\end{split}
\label{eqn: decomposition of L2}
\eeq
%
%
By virtue of Lemma \ref{lem: Lp bound for multi-term commutator}, 
\beq
\begin{split}
&\;\left\|\mathrm{p.v.}\int_\BT(\D h)^{2j+1}(1+h(\th+\xi))^{-j}\cdot \f{\p(\th+\xi)}{2\sin\f{\xi}{2}}\,d\xi\right\|_{L^p}\\
\leq &\;C_*^{2j+3}\|h'\|_{L^\infty}^{2j+1}\|(1+h)^{-j}\p\|_{L^p}\\
\leq &\;C (C_*^{2}C_2\|h'\|_{L^\infty}^{2})^j\|h'\|_{L^\infty}\|\p\|_{L^p}.
\end{split}
\label{eqn: Lp estimate of L2j}
\eeq
Here $C_2$ is a universal constant such that $\|(1+h)^{-1}\|_{L^\infty}\leq C_2$.
Similarly, by Lemma \ref{lem: W 1p estimate for multi-term commutator},
\beq
\begin{split}
&\;\left\|\mathrm{p.v.}\int_\BT(\D h)^{2j+1}(1+h(\th+\xi))^{-j}\cdot \f{\p(\th+\xi)}{2\sin\f{\xi}{2}}\,d\xi\right\|_{\dot{W}^{1,p}}\\
\leq &\;(2j+2)C_\dag^{2j+2}\|h'\|_{L^\infty}^{2j}(\|((1+h)^{-j}\p)'\|_{L^p}\|h'\|_{L^\infty}+\|(1+h)^{-j}\p\|_{L^\infty}\|h''\|_{L^p})\\
\leq &\;C(j+1) (C_\dag^{2}C_2\|h'\|_{L^\infty}^{2})^j(j\|h'\|_{L^\infty}^2\|\p\|_{L^p} +\|h'\|_{L^\infty}\|\p'\|_{L^p}+\|\p\|_{L^\infty}\|h''\|_{L^p}).
\end{split}
\label{eqn: W1p estimate of L2j}
\eeq
Hence, with the assumption $C_\dag\geq C_*$,
\beq
\begin{split}
&\;\|L_{2,j}\|_{\dot{W}^{1,p}}\\
\leq &\;\|(1+h)^{-(j+1)}\|_{\dot{W}^{1,\infty}}\left\|\mathrm{p.v.}
\int_\BT(\D h)^{2j+1}(1+h(\th+\xi))^{-j}\cdot \f{\p(\th+\xi)}{2\sin\f{\xi}{2}}\,d\xi
\right\|_{L^p}\\
&\;+\|(1+h)^{-(j+1)}\|_{L^\infty}\left\|\mathrm{p.v.}
\int_\BT(\D h)^{2j+1}(1+h(\th+\xi))^{-j}\cdot \f{\p(\th+\xi)}{2\sin\f{\xi}{2}}\,d\xi
\right\|_{\dot{W}^{1,p}}\\
\leq &\;C(j+1)(C_\dag C_2\|h'\|_{L^\infty})^{2j} ((j+1)\|h'\|_{L^\infty}^2\|\p\|_{L^p} +\|h'\|_{L^\infty}\|\p'\|_{L^p}+\|\p\|_{L^\infty}\|h''\|_{L^p}).
\end{split}
\label{eqn: W 1 p estimate for L 2 j}
\eeq
To this end, by assuming $\|h'\|_{L^\infty}\ll 1$, where the smallness depends on $p$, we derive from \eqref{eqn: decomposition of L2} that
\beq
\|L_2\|_{\dot{W}^{1,p}}
\leq 
C (\|h'\|_{L^\infty}\|\p'\|_{L^p}+\|\p\|_{L^\infty}\|h''\|_{L^p}).
\label{eqn: W1p estimate for L2}
\eeq

Similarly, we write 
\beq
L_3=\sum_{j = 0}^\infty h'(\th)(-1-h(\th))^{-(j+1)}\mathrm{p.v.}\int_\BT (\D h)^{2j}(1+h(\th+\xi))^{-j}\cdot \f{\p(\th+\xi)}{2\tan\f{\xi}{2}}\,d\xi=:\sum_{j=0}^\infty L_{3,j}.
\label{eqn: decomposition of L3}
\eeq
In order to bound $\dot{W}^{1,p}$-semi-norm of $L_{3,j}$, we need an $L^\infty$-bound of the integral above.
This is possible thanks to the H\"{o}lder regularity of $h'$ and $\p$.
Indeed, by the mean value theorem,
\beq
\begin{split}
&\;\left|\mathrm{p.v.}\int_\BT (\D h)^{2j}(1+h(\th+\xi))^{-j} \f{\p(\th+\xi)}{2\tan\f{\xi}{2}}\,d\xi\right|\\
=&\;\left|\int_\BT [(\D h)^{2j}(1+h(\th+\xi))^{-j}\p(\th+\xi)
-h'(\th)^{2j}(1+h(\th))^{-j} \p(\th)]\f{1}{2\tan\f{\xi}{2}}\,d\xi\right|\\
\leq &\;C\int_\BT 2j(C_1\|h'\|_{L^\infty})^{2j-1}|\D h-h'(\th)|\cdot C_2^{j}\|\p\|_{L^\infty}|\xi|^{-1}\,d\xi\\
&\;+C\int_\BT \|h'\|_{L^\infty}^{2j}\cdot jC_2^{j+1}|h(\th+\xi)-h(\th)|\cdot \|\p\|_{L^\infty}|\xi|^{-1}\,d\xi\\
&\;+C\int_\BT \|h'\|_{L^\infty}^{2j}\cdot C_2^{j}|\p(\th+\xi)-\p(\th)| |\xi|^{-1}\,d\xi\\
\leq &\;C(2jC_1^{2j}C_2^{j}\|h'\|_{L^\infty}^{2j-1}\|h'\|_{\dot{C}^\b} \|\p\|_{L^\infty}
+C_2^{j}\|h'\|_{L^\infty}^{2j}\|\p\|_{\dot{C}^\b}).
\end{split}
\label{eqn: L inf bound of a term in L_3}
\eeq
Here $C_1 = \f{\pi}{2}$ introduced in the proof of Lemma \ref{lem: Lp bound for multi-term commutator}; note that $|\D h|\leq C_1 \|h'\|_{L^\infty}$.
Arguing as in \eqref{eqn: Lp estimate of L2j}-\eqref{eqn: W 1 p estimate for L 2 j},
\beq
\left\|\mathrm{p.v.}\int_\BT (\D h)^{2j}(1+h(\th+\xi))^{-j}\cdot \f{\p(\th+\xi)}{2\tan\f{\xi}{2}}\,d\xi\right\|_{L^p}
\leq 
C(C_*^{2}C_2\|h'\|_{L^\infty}^{2})^j\|\p\|_{L^p},
\eeq
\beq
\begin{split}
&\;\left\|\mathrm{p.v.}\int_\BT (\D h)^{2j}(1+h(\th+\xi))^{-j}\cdot \f{\p(\th+\xi)}{2\tan\f{\xi}{2}}\,d\xi\right\|_{\dot{W}^{1,p}}\\
\leq &\; C (2j+1) (C_\dag^{2}C_2\|h'\|_{L^\infty}^{2})^j(j\|h'\|_{L^\infty}\|\p\|_{L^p}+\|\p'\|_{L^p}+\mathds{1}_{\{j>0\}}\|h'\|_{L^\infty}^{-1}\|h''\|_{L^p}\|\p\|_{L^\infty}).
\end{split}
\eeq
and hence, 
\beq
\begin{split}
&\;\|L_{3,j}\|_{\dot{W}^{1,p}}\\
\leq &\;\|h''\|_{L^p}\|(1+h(\th))^{-(j+1)}\|_{L^\infty}\left\|\mathrm{p.v.}\int_\BT (\D h)^{2j}(1+h(\th+\xi))^{-j}\cdot \f{\p(\th+\xi)}{2\tan\f{\xi}{2}}\,d\xi\right\|_{L^\infty}\\
&\;+\|h'\|_{L^\infty}\|(1+h(\th))^{-(j+1)}\|_{\dot{W}^{1,\infty}}\left\|\mathrm{p.v.}\int_\BT (\D h)^{2j}(1+h(\th+\xi))^{-j}\cdot \f{\p(\th+\xi)}{2\tan\f{\xi}{2}}\,d\xi\right\|_{L^p}\\
&\;+\|h'\|_{L^\infty}\|(1+h(\th))^{-(j+1)}\|_{L^\infty}\left\|\mathrm{p.v.}\int_\BT (\D h)^{2j}(1+h(\th+\xi))^{-j}\cdot \f{\p(\th+\xi)}{2\tan\f{\xi}{2}}\,d\xi\right\|_{\dot{W}^{1,p}}\\
\leq &\;C\cdot (C_2\|h'\|_{L^\infty})^{2j-1}\cdot
(j C_1^{2j}\|h'\|_{\dot{C}^\b} \|\p\|_{L^\infty}+\|h'\|_{L^\infty}\|\p\|_{\dot{C}^\b})\|h''\|_{L^p}\\
&\;+C\cdot (j+1)(C_\dag C_2\|h'\|_{L^\infty})^{2j}\\
&\;\qquad\cdot ((j+1)\|h'\|_{L^\infty}^2\|\p\|_{L^p}+\|h'\|_{L^\infty}\|\p'\|_{L^p}+\mathds{1}_{\{j>0\}}\|h''\|_{L^p}\|\p\|_{L^\infty}).
%
\end{split}
\eeq
By \eqref{eqn: decomposition of L3}, provided that $\|h'\|_{L^\infty}\ll 1$, 
\beq
\|L_3\|_{\dot{W}^{1,p}}
%
\leq 
C(\|h''\|_{L^p}\|h'\|_{\dot{C}^\b}\|\p\|_{L^\infty}+ \|h''\|_{L^p}\|\p\|_{\dot{C}^\b} +\|h'\|_{L^\infty}\|\p'\|_{L^p}).
\label{eqn: W1p estimate for L3}
\eeq

Combining \eqref{eqn: W1p estimate for L1}, \eqref{eqn: W1p estimate for L2} and \eqref{eqn: W1p estimate for L3}, we prove the desired estimate.
\end{proof}
\end{lem}

We also prove a $\dot{W}^{1,p}$-estimate for $\g'\cdot\mathcal{K}_\g\p-\f{1}{2}\CH\p$.

\begin{lem}\label{lem: W 1p difference from Hilbert transform}
Under the assumptions of Lemma \ref{lem: W 1p estimate of normal component of singular integral},
\beq
\begin{split}
&\;\left\|\g'(\th)\cdot \CK_\g \p - \f{1}{2}\CH \p\right\|_{\dot{W}^{1,p}}\\
\leq &\;C\| h''\|_{L^p}\|\p\|_{L^\infty}(1+\|h'\|_{\dot{C}^\b})+C(\|h'\|_{L^\infty}\|h''\|_{L^p}\|\p\|_{\dot{C}^\b}+\|h'\|_{L^\infty}^2\|\p'\|_{L^p}),
\end{split}
\label{eqn: W1p difference estimate of the singular integral tangent component with Hilbert transform}
\eeq
where $C = C(p,\b)$.
\begin{proof}

Using $\g(\th) = f(\th)(\cos\th,\sin\th)$, by definition,
\beq
2\pi\g'(\th)\cdot\mathcal{K}_\g\p
=\mathrm{p.v.}\int_\BT \f{f'(\th)f(\th)-f'(\th)f(\th+\xi)\cos\xi-f(\th)f(\th+\xi)\sin\xi} {f(\th)^2+f(\th+\xi)^2-2f(\th)f(\th+\xi)\cos\xi}\p(\th+\xi)\,d\xi.
\eeq
With $f(\th) = r(1+h(\th))$ and $l(\th,\th+\xi)$ defined in \eqref{eqn: def of l}, it can be rewritten as
\beq
\begin{split}
&\;2\pi\g'(\th)\cdot\mathcal{K}_\g\p\\
=&\;f'(\th)\int_\BT \f{f(\th+\xi)\cdot 2\sin^2\f{\xi}{2}} {(f(\th+\xi)-f(\th))^2+f(\th)f(\th+\xi)\cdot 4\sin^2\f{\xi}{2}}\p(\th+\xi)\,d\xi\\
&\;-f'(\th)\mathrm{p.v.}\int_\BT \f{f(\th+\xi)-f(\th)} {(f(\th+\xi)-f(\th))^2+f(\th)f(\th+\xi)\cdot 4\sin^2\f{\xi}{2}}\p(\th+\xi)\,d\xi\\
&\;-\mathrm{p.v.}\int_\BT \f{f(\th)f(\th+\xi)\sin\xi} {(f(\th+\xi)-f(\th))^2+f(\th)f(\th+\xi)\cdot 4\sin^2\f{\xi}{2}}\p(\th+\xi)\,d\xi\\
%
%
=&\;\f{h'(\th)}{2(1+h(\th))}
\left(\int_\BT \p\,d\xi-\int_\BT \f{l(\th,\th+\xi)} {1+l(\th,\th+\xi)}\p(\th+\xi)\,d\xi\right)\\
&\;-\f{h'(\th)}{1+h(\th)}\mathrm{p.v.}\int_\BT \f{\f{\D h}{2\sin\f{\xi}{2}}} {1+l(\th,\th+\xi)}\f{\p(\th+\xi)}{1+h(\th+\xi)}\,d\xi\\
&\; +\mathrm{p.v.}\int_\BT \f{l(\th,\th+\xi)} {1+l(\th,\th+\xi)} \f{\p(\th+\xi)}{2\tan\f{\xi}{2}}\,d\xi+\pi \CH \p\\
=:&\;\tilde{L}_1(\th)+\tilde{L}_2(\th)+\tilde{L}_3(\th)+\pi\CH \p.
\end{split}
\label{eqn: decomposition of singular integral tangent}
\eeq

Since
\beq
\tilde{L}_1  = \f{h'(\th)}{1+h(\th)}\left(\f12 \int_\BT \p\,d\xi + L_1\right),
\label{eqn: relation between L_1 and tilde L_1}
\eeq
we derive by \eqref{eqn: W1p estimate for L1} that
\beq
\begin{split}
\|\tilde{L}_1\|_{\dot{W}^{1,p}}\leq &\; C\left\|\f{h'}{1+h}\right\|_{\dot{W}^{1,p}}\|\p\|_{L^\infty}+C\|h'\|_{L^\infty}\|L_1\|_{\dot{W}^{1,p}}\\
\leq &\;C(\|h''\|_{L^p}\|\p\|_{L^\infty}+\|h'\|_{L^\infty}^3\|\p'\|_{L^p}).
\end{split}
\label{eqn: W1p estimate for L1 tilde}
\eeq
For $\tilde{L}_2$,
\beq
\tilde{L}_2
= \sum_{j = 0}^\infty h'(\th)(-1-h(\th))^{-(j+1)}\mathrm{p.v.}\int_\BT (\D h)^{2j+1}(1+h(\th+\xi))^{-(j+1)}\f{\p(\th+\xi)}{2\sin\f{\xi}{2}}.
\eeq
Arguing as in \eqref{eqn: L inf bound of a term in L_3},
\beq
\begin{split}
&\;\left\|\mathrm{p.v.}\int_\BT (\D h)^{2j+1}(1+h(\th+\xi))^{-(j+1)}\f{\p(\th+\xi)}{2\sin\f{\xi}{2}}\,d\xi\right\|_{L^\infty}\\
\leq &\;C(C_2\|h'\|_{L^\infty}^2)^j((2j+1)C_1^{2j}\|h'\|_{\dot{C}^\beta}\|\p\|_{L^\infty}+\|h'\|_{L^\infty}\|\p\|_{\dot{C}^\b}).
\end{split}
\eeq
Moreover, by Lemma \ref{lem: Lp bound for multi-term commutator} and Lemma \ref{lem: W 1p estimate for multi-term commutator},
\beq
\begin{split}
&\;\left\|\mathrm{p.v.}\int_\BT (\D h)^{2j+1}(1+h(\th+\xi))^{-(j+1)}\f{\p(\th+\xi)}{2\sin\f{\xi}{2}}\,d\xi\right\|_{L^p}\\
\leq &\;C(C_*^2 C_2 \|h'\|_{L^\infty}^2)^j\|h'\|_{L^\infty}\|\p\|_{L^\infty},
\end{split}
\eeq
and
\beq
\begin{split}
&\;\left\|\mathrm{p.v.}\int_\BT (\D h)^{2j+1}(1+h(\th+\xi))^{-(j+1)}\f{\p(\th+\xi)}{2\sin\f{\xi}{2}}\,d\xi\right\|_{\dot{W}^{1,p}}\\
\leq &\; C(2j+2)(C_\dag^2 C_2 \|h'\|_{L^\infty}^2)^j((j+1)\|h'\|_{L^\infty}^2\|\p\|_{L^\infty}+\|h''\|_{L^p}\|\p\|_{L^\infty}+\|h'\|_{L^\infty}\|\p'\|_{L^p}).
\end{split}
\eeq
Hence,
\beq
\|\tilde{L}_2\|_{\dot{W}^{1,p}}\leq C(\|h''\|_{L^p}\|h'\|_{\dot{C}^\b}\|\p\|_{L^\infty}+\|h'\|_{L^\infty}\|h''\|_{L^p}\|\p\|_{\dot{C}^\b}+\|h'\|_{L^\infty}^2\|\p'\|_{L^p}).
\label{eqn: W1p estimate for L2 tilde}
\eeq

For $\tilde{L}_3$,
\beq
\tilde{L}_3
= \sum_{j = 0}^\infty (-1)^j(1+h(\th))^{-(j+1)}\mathrm{p.v.}\int_\BT (\D h)^{2j+2}(1+h(\th+\xi))^{-(j+1)}\f{\p(\th+\xi)}{2\tan\f{\xi}{2}}\,d\xi.
\eeq
Since
\beq
\begin{split}
&\;\left\|\mathrm{p.v.}\int_\BT (\D h)^{2j+2}(1+h(\th+\xi))^{-(j+1)}\f{\p(\th+\xi)}{2\tan\f{\xi}{2}}\,d\xi
\right\|_{L^p}\\
\leq &\;C(C_*^2 C_2 \|h'\|_{L^\infty}^2)^j\|h'\|_{L^\infty}^2\|\p\|_{L^\infty},
\end{split}
\eeq
and
%
\beq
\begin{split}
&\;\left\|\mathrm{p.v.}\int_\BT (\D h)^{2j+2}(1+h(\th+\xi))^{-(j+1)}\f{\p(\th+\xi)}{2\tan\f{\xi}{2}}\,d\xi
\right\|_{\dot{W}^{1,p}}\\
\leq &\; C(2j+3)(C_\dag^2 C_2 \|h'\|_{L^\infty}^2)^j\|h'\|_{L^\infty}\\
&\;\quad\cdot((j+1)\|h'\|_{L^\infty}^2\|\p\|_{L^\infty}+\|h'\|_{L^\infty}\|\p'\|_{L^p}+\|h''\|_{L^p}\|\p\|_{L^\infty}),
\end{split}
\eeq
we find that
\beq
\|\tilde{L}_3\|_{\dot{W}^{1,p}}
\leq C(\|h'\|_{L^\infty}\|h''\|_{L^p}\|\p\|_{L^\infty}+\|h'\|_{L^\infty}^2 \|\p'\|_{L^p}).
\label{eqn: W1p estimate for L3 tilde}
\eeq

Combining \eqref{eqn: decomposition of singular integral tangent}, \eqref{eqn: W1p estimate for L1 tilde}, \eqref{eqn: W1p estimate for L2 tilde} and \eqref{eqn: W1p estimate for L3 tilde}, we obtain \eqref{eqn: W1p difference estimate of the singular integral tangent component with Hilbert transform}.
\end{proof}
\end{lem}

In order to show uniqueness of the solution in Section \ref{sec: uniqueness}, we need the following three lemmas, which are generalizations of Lemmas \ref{lem: Holder estimate of singular integral normal component}-\ref{lem: W 1p difference from Hilbert transform}, respectively.

\begin{lem}
\label{lem: Holder estimate of singular integral normal component difference}
Fix $\b\in (0,1)$.
Assume $h_1,h_2\in C^{1,\b}(\mathbb{T})$, such that $m_{0,1},m_{0,2} \ll 1$.
Here $m_{0,i}$ are defined for $i=1,2$ as in \eqref{eqn: bound assumption on the Lipschitz norm of h}.
Then
\beq
\begin{split}
&\;\|\g'_1(\th)^\perp\cdot\mathcal{K}_{\g_1}\p-\g'_2(\th)^\perp\cdot\mathcal{K}_{\g_2}\p\|_{\dot{C}^\b}\\
\leq &\; C\|h_1-h_2\|_{C^{1,\b}} (1+\|h_1'\|_{\dot{C}^\b}+\|h_2'\|_{\dot{C}^\b})^2\|\p\|_{C^\b},
\end{split}
\eeq
where $C = C(\b)$.
\end{lem}

\begin{lem}\label{lem: W 1p estimate of normal component of singular integral difference}
Fix $p\in [2,\infty)$ and $\beta\in (0,1)$.
Assume $h_i\in C^{1,\beta}\cap W^{2,p}(\mathbb{T})$ $(i = 1,2)$, such that $m_{0,i} \ll 1$ with the needed smallness depending only on $p$.
Then
\beq
\begin{split}
&\;\|\g_1'(\th)^\perp\cdot\mathcal{K}_{\g_1}\p-\g_2'(\th)^\perp\cdot\mathcal{K}_{\g_2}\p\|_{\dot{W}^{1,p}}\\
\leq &\; C \|h_1''-h_2''\|_{L^p} (1+\|h_1'\|_{\dot{C}^\b} +\|h_2'\|_{\dot{C}^\b}) \|\p\|_{C^\b}\\
&\;+C(\|h_1''\|_{L^p}+\|h_2''\|_{L^p} ) \|h_1-h_2\|_{C^{1,\b}}(1+\|h_1'\|_{\dot{C}^\b}+\|h_2'\|_{\dot{C}^\b}) \|\p\|_{C^\b}\\
&\;+C\|h_1-h_2\|_{W^{1,\infty}}\|\p'\|_{L^p}.
\end{split}
\eeq
where $C = C(p,\b)$.
\end{lem}

\begin{lem}\label{lem: W 1p difference from Hilbert transform difference}
Under the assumptions of Lemma \ref{lem: W 1p estimate of normal component of singular integral difference},
\beq
\begin{split}
&\;\|\g_1'(\th)\cdot\mathcal{K}_{\g_1}\p-\g_2'(\th)\cdot\mathcal{K}_{\g_2}\p\|_{\dot{W}^{1,p}}\\
\leq &\;C\|h_1''-h_2''\|_{L^p} \left|\int_\BT \p\,d\xi\right|\\
&\;+C\|h_1''-h_2''\|_{L^p}\|\p\|_{C^\b} (\|h_1\|_{C^{1,\b}}+\|h_2\|_{C^{1,\b}})(1+\|h_1\|_{C^{1,\b}}+\|h_2\|_{C^{1,\b}})^2\\
&\;+C(\|h_1''\|_{L^p}+\|h_2''\|_{L^p})\|\p\|_{C^\b} \|h_1-h_2\|_{C^{1,\b}}(1+\|h_1\|_{C^{1,\b} } +\|h_2\|_{C^{1,\b}})^3\\
&\;+C\|h_1-h_2\|_{W^{1,\infty}}
\|\p'\|_{L^p} (\|h_1\|_{W^{1,\infty}}+\|h_2\|_{W^{1,\infty}}),
\end{split}
\label{eqn: W1p difference estimate of the singular integral tangent component}
\eeq
where $C = C(p,\b)$.
\end{lem}

These estimates can be justified by following similar arguments as those in Lemmas \ref{lem: Holder estimate of singular integral normal component}-\ref{lem: W 1p difference from Hilbert transform}.
However, since their proofs turn out to be extremely lengthy and somewhat tedious,
we shall leave them to Appendix \ref{sec: proofs of estimates for difference of singular integral operators}.

\section{Estimates for Integral Operators $\CK_{\g,\tilde{\g}}$ and $\CK_{\tilde{\g},\g}$}
\label{sec: estimates for interaction operators}
%
%
Recall that the integral operators $\CK_{\g,\tilde{\g}}$ and $\CK_{\tilde{\g},\g}$ are defined in \eqref{eqn: def of integral operators for a pair of interfaces}, while the Poisson kernel $P$ on the 2-D unit disc and its conjugate $Q$ are defined in \eqref{eqn: def of P} and \eqref{eqn: def of Q}.
For convenience, we denote
\beq
P_{\f{r}{R}} := P\left(\f{r}{R},\cdot \right)\quad \mbox{ and }\quad Q_{\f{r}{R}} := Q\left(\f{r}{R},\cdot \right).
\eeq

\begin{lem}\label{lem: L inf norm of derivative of interaction kernel from outer to inner}
Assume $h,H\in W^{1,\infty}(\BT)$, such that $\d^{-1}(\|h\|_{L^\infty}+\|H\|_{L^\infty})\ll 1$.
Denote $\bar{\p} = (2\pi)^{-1}\int_{\BT}\p(\th)\,d\th$.
Then
\beq
\begin{split}
&\;\left\|fe_r(\th)\cdot \mathcal{K}_{\g,\tilde{\g}}\p+\f{1}{4\pi}P_{\f{r}{R}}*(\p-\bar{\p})\right\|_{L^\infty(\BT)}\\
&\;+\left\|fe_\th(\th)\cdot \mathcal{K}_{\g,\tilde{\g}}\p-\f{1}{4\pi}Q_{\f{r}{R}}*(\p-\bar{\p})\right\|_{L^\infty(\BT)}\\
\leq &\;
\f{Cr}{R}\d^{-1}(\|h\|_{L^\infty}+\|H\|_{L^\infty})\|\p\|_{L^\infty},
\end{split}
\label{eqn: L inf norm of derivative of interaction kernel from outer to inner}
\eeq
where $C$ is a universal constant.

\begin{proof}
With $\th' = \th+\xi$ and $D(\th,\th+\xi) := f(\th)/F(\th+\xi)$,
we calculate that
\beq
\begin{split}
&\;2\pi e_r(\th)\cdot \mathcal{K}_{\g,\tilde{\g}}\p\\
=&\;\int_\BT \f{e_r(\th)\cdot (\g(\th)-\tilde{\g}(\th'))}{|\g(\th)-\tilde{\g}(\th')|^2}\p(\th')\,d\th'\\
=&\;f(\th)^{-1}\int_\BT \left[\f12-\f12\cdot\f{1 - D(\th,\th+\xi)^{2}}{1+D(\th,\th+\xi)^{2}-2 D(\th,\th+\xi)\cos\xi}\right]\p(\th+\xi)\,d\xi\\
=:&\; f(\th)^{-1}(I_{r,1}+I_{r,2}),
\end{split}
\label{eqn: decomposition of I r}
\eeq
where
\begin{align}
I_{r,1}=&\;-\f{1}{2}\int_\BT P\left(\f{r}{R},\xi\right)(\p(\th+\xi)-\bar{\p})\,d\xi = -\f{1}{2}P_{\f{r}{R}}*(\p-\bar{\p}),\label{eqn: I r 1}\\
I_{r,2}=&\;\f{1}{2}\int_\BT \left[P\left(\f{r}{R},\xi\right)
-P(D,\xi)\right]\p(\th+\xi)\,d\xi.
\end{align}
Here we used the fact that $P_{\f{r}{R}}$ is an even function and has integral $2\pi$ on $\BT$. 
%
%
%
$I_{r,1}$ is already in the desired shape.
For $I_{r,2}$, since 
\beq
\left|\f{r}{R}-D(\th,\th+\xi)\right| = \f{r}{R}\left|1-\f{1+h(\th)}{1+H(\th+\xi)}\right|\leq \f{Cr}{R}(\|h\|_{L^\infty}+\|H\|_{L^\infty}),
\label{eqn: deviation of B inv from a constant}
\eeq
we may assume that $D\in [0,1-C\d]$ for some universal $C>0$.
Hence, by the mean value theorem and Lemma \ref{lem: properties of Poisson kernel},
\beq
\begin{split}
\|I_{r,2}\|_{L^\infty}
\leq &\; \f{Cr}{R}(\|h\|_{L^\infty}+\|H\|_{L^\infty})\|\p\|_{L^\infty}\int_\BT (\d^2+\xi^2)^{-1}\,d\xi\\
\leq &\;\f{Cr}{R}\d^{-1}(\|h\|_{L^\infty}+\|H\|_{L^\infty})\|\p\|_{L^\infty}.
\end{split}
\label{eqn: L inf estimate of I r 2}
\eeq
The estimate for $fe_r \cdot \mathcal{K}_{\g,\tilde{\g}}\p$ in \eqref{eqn: L inf norm of derivative of interaction kernel from outer to inner} follows.

Similarly, since $Q_{\f{r}{R}}$ is an odd kernel,
\beq
\begin{split}
&\;2\pi e_\th(\th)\cdot \mathcal{K}_{\g,\tilde{\g}}\p\\
=&\;-f(\th)^{-1}\int_\BT \f{D(\th,\th+\xi)\cdot \sin\xi} {1+D(\th,\th+\xi)^2-2D(\th,\th+\xi)\cos\xi}\p(\th+\xi)\,d\xi. \\
=:&\;f(\th)^{-1}(I_{\th,1}+I_{\th,2}),
\end{split}
\label{eqn: decomposition of I theta}
\eeq
where
%
\begin{align}
I_{\th,1}=&\;-\f12\int_\BT Q\left(\f{r}{R},\xi\right)(\p(\th+\xi)-\bar{\p})\,d\xi = \f12 Q_{\f{r}{R}}*(\p-\bar{\p}),\label{eqn: I theta 1}\\
I_{\th,2}=&\;\f12\int_\BT \left[Q\left(\f{r}{R},\xi\right)
-Q(D,\xi)\right]\p(\th+\xi)\,d\xi.
\end{align}
Then the estimate for $fe_\th\cdot \mathcal{K}_{\g,\tilde{\g}}\p$ in \eqref{eqn: L inf norm of  derivative of interaction kernel from outer to inner} can be derived as before.
\end{proof}
\end{lem}

\begin{lem}\label{lem: Holder norm of derivative of interaction kernel from outer to inner}

Assume $h,H\in C^{1,\alpha}(\BT)$ for some $\alpha\in (0,1)$, such that $m_0+M_0\ll 1$.
Then for $\b\in (0,\f{\alpha}{1+\alpha})$,
\beq
\begin{split}
&\;\left\|fe_r(\th)\cdot \mathcal{K}_{\g,\tilde{\g}}\p+\f{1}{4\pi}P_{\f{r}{R}}*(\p-\bar{\p})\right\|_{\dot{C}^\b(\BT)}\\
&\;+\left\|fe_\th(\th)\cdot \mathcal{K}_{\g,\tilde{\g}}\p-\f{1}{4\pi}Q_{\f{r}{R}}*(\p-\bar{\p})\right\|_{\dot{C}^\b(\BT)}\\
\leq &\;
\f{Cr}{R}(m_0+M_0)\|\p\|_{\dot{C}^\beta}\\
&\;+\f{Cr}{R}\|\p\|_{L^\infty}(\d^{-1}(\|h\|_{L^\infty}+\|H\|_{L^\infty})+\|h'\|_{\dot{C}^{\alpha}}+\|H'\|_{\dot{C}^{\alpha}}),
\end{split}
\label{eqn: Holder norm of derivative of interaction kernel from outer to inner}
\eeq
where $C = C(\alpha,\b)$.
\begin{proof}

Let $I_{r,1}$, $I_{r,2}$, $I_{\th,1}$ and $I_{\th,2}$ be defined as in the proof of Lemma \ref{lem: L inf norm of derivative of interaction kernel from outer to inner}.


Consider $I_{r,2}$.
For $\th_1,\th_2\in \BT$,
\beq
\begin{split}
&\;I_{r,2}(\th_1)-I_{r,2}(\th_2)\\
=&\;\f{1}{2}\int_\BT \left[P(D(\th_2,\th_2+\xi),\xi)
-P(D(\th_1,\th_1+\xi),\xi)\right](\p(\th_1+\xi)-\p(\th_1))\,d\xi\\
&\;+\f{1}{2}\p(\th_1)\int_\BT P(D(\th_2,\th_2+\xi),\xi)
-P(D(\th_1,\th_1+\xi),\xi)\,d\xi\\
&\;-\f{1}{2}\int_\BT \left[P\left(\f{r}{R},\xi\right)
-P(D(\th_2,\th_2+\xi),\xi)\right](\p(\th_2+\xi)-\p(\th_1+\xi))\,d\xi\\
=:&\;I_{r,2,1}+I_{r,2,2}+I_{r,2,3}.
\label{eqn: splitting I r 2}
\end{split}
\eeq
Following the argument of \eqref{eqn: deviation of B inv from a constant} and \eqref{eqn: L inf estimate of I r 2},
\beq
\begin{split}
|I_{r,2,1}|\leq &\;C\int_{\BT} \f{1}{\d^2+|\xi|^2}\cdot \left|\f{f(\th_1)}{F(\th_1+\xi)}-\f{f(\th_2)}{F(\th_2+\xi)}\right|\cdot |\xi|^\beta\|\p\|_{\dot{C}^\beta}\,d\xi\\
\leq &\;C\|\p\|_{\dot{C}^\beta}\int_{\BT} \f{|\xi|^\beta}{\d^2+|\xi|^2}\cdot \f{r}{R}|\th_1-\th_2|^\beta( \|h\|_{\dot{C}^\beta}+\|H\|_{\dot{C}^\beta})\,d\xi\\
\leq &\;C|\th_1-\th_2|^\beta\cdot \f{r}{R}(m_0+M_0)\|\p\|_{\dot{C}^\beta},
\end{split}
\label{eqn: estimate for I r 2 1}
\eeq
and similarly,
\beq
|I_{r,2,3}|
\leq C |\th_1-\th_2|^\beta\cdot \f{r}{R}(m_0+M_0)\|\p\|_{\dot{C}^\beta}.
\label{eqn: estimate for I r 2 3}
\eeq

To handle $I_{r,2,2}$, we first note that
\beq
\begin{split}
&\;\int_\BT P(D(\th_2,\th_2+\xi),\xi)
-P(D(\th_1,\th_1+\xi),\xi)\,d\xi\\
=&\; \int_\BT P(D(\th_2,\th_2+\xi),\xi)-P(D(\th_2,\th_2),\xi)-P(D(\th_1,\th_1+\xi),\xi)+P(D(\th_1,\th_1),\xi)\,d\xi.
\end{split}
\label{eqn: two representation of the difference}
\eeq
We may bound the integrands in \eqref{eqn: two representation of the difference} as follows.
By the mean value theorem and Lemma \ref{lem: properties of Poisson kernel},
\beq
\begin{split}
&\;|P(D(\th_2,\th_2+\xi),\xi)-P(D(\th_2,\th_2),\xi)-P(D(\th_1,\th_1+\xi),\xi)+P(D(\th_1,\th_1),\xi)|\\
\leq &\; \f{C}{\d^2+\xi^2}(|D(\th_2,\th_2+\xi)-D(\th_2,\th_2)|+|D(\th_1,\th_1+\xi)-D(\th_1,\th_1)|)\\
\leq &\; \f{C|\xi|^{\beta'}}{\d^2+\xi^2}\cdot \f{r}{R}\|H\|_{\dot{C}^{\beta'}},
\label{eqn: bound the difference approach 1}
\end{split}
\eeq
where $\beta'\in(0,1)$ is to be determined.
Here we used the bound $|\pa_s P|\leq C(\d^2+\xi^2)^{-1}$ since $D\leq 1-C\d$ (see the proof of Lemma \ref{lem: L inf norm of derivative of interaction kernel from outer to inner}).
Alternatively,
\beq
\begin{split}
&\;|P(D(\th_2,\th_2+\xi),\xi)-P(D(\th_1,\th_1+\xi),\xi)-P(D(\th_2,\th_2),\xi)+P(D(\th_1,\th_1),\xi)|\\
\leq &\; \f{C}{\d^2+\xi^2}(|D(\th_2,\th_2+\xi)-D(\th_1,\th_1+\xi)|+|D(\th_2,\th_2)-D(\th_1,\th_1)|)\\
\leq &\; \f{C}{\d^2+\xi^2}\cdot \f{r}{R}|\th_1-\th_2|(\|h'\|_{L^\infty}+\|H'\|_{L^\infty}).
\end{split}
\label{eqn: bound the difference approach 2}
\eeq
If $|\th_1 -\th_2 |\geq \d$, by \eqref{eqn: two representation of the difference} and \eqref{eqn: bound the difference approach 1},
\beq
\begin{split}
&\;\left|\int_\BT P(D(\th_2,\th_2+\xi),\xi)
-P(D(\th_1,\th_1+\xi),\xi)\,d\xi\right|\\
\leq &\;\f{Cr}{R}\|H\|_{\dot{C}^{\beta'}}\d^{\beta'-\beta-1}|\th_1-\th_2|^\beta.
\end{split}
\eeq
Otherwise, if $|\th_1 -\th_2 |\leq \d$, we deduce by \eqref{eqn: two representation of the difference} and \eqref{eqn: bound the difference approach 2} that
\beq
\begin{split}
&\;\left|\int_\BT P(D(\th_2,\th_2+\xi),\xi)
-P(D(\th_1,\th_1+\xi),\xi)\,d\xi\right|\\
\leq &\;\f{Cr}{R}|\th_1-\th_2|^\beta \d^{-\beta}(\|h'\|_{L^\infty}+\|H'\|_{L^\infty}).
\end{split}
\eeq
Recall that $\beta<\f{\alpha}{1+\alpha}$, so we take $\beta' = \f{\beta(1+\alpha)}{\alpha}$.
Combining these estimates with the definition of $I_{r,2,2}$ in \eqref{eqn: splitting I r 2}, by interpolation inequality,
\beq
|I_{r,2,2}|\leq \f{Cr}{R}|\th_1-\th_2|^\beta \|\p\|_{L^\infty}(\d^{-1}(\|h\|_{L^\infty}+\|H\|_{L^\infty})+\|h'\|_{\dot{C}^{\alpha}}+\|H'\|_{\dot{C}^{\alpha}}).
\eeq
Combining this with \eqref{eqn: splitting I r 2}-\eqref{eqn: estimate for I r 2 3}, we obtain that
\beq
\|I_{r,2}\|_{\dot{C}^\beta}\leq \f{Cr}{R}(m_0+M_0)\|\p\|_{\dot{C}^\beta}+\f{Cr}{R}\|\p\|_{L^\infty}(\d^{-1}(\|h\|_{L^\infty}+\|H\|_{L^\infty})+\|h'\|_{\dot{C}^{\alpha}}+\|H'\|_{\dot{C}^{\alpha}}).
\eeq

The estimate for $I_{\th,2}$ can be derived in the same manner.
\end{proof}
\end{lem}

\begin{lem}\label{lem: W 1p norm of derivative of interaction kernel from outer to inner}

Assume $h\in W^{1,\infty}(\BT)$ and $H\in W^{2,p}(\BT)$ for some $p\in (1,\infty)$,
satisfying that $m_0+M_0\ll 1$.
Then
\beq
\begin{split}
&\;\left\|fe_r(\th)\cdot \mathcal{K}_{\g,\tilde{\g}}\p+\f{1}{4\pi}P_{\f{r}{R}}*(\p-\bar{\p})\right\|_{\dot{W}^{1,p}(\BT)}\\
&\;+\left\|fe_\th(\th)\cdot \mathcal{K}_{\g,\tilde{\g}}\p-\f{1}{4\pi}Q_{\f{r}{R}}*(\p-\bar{\p})\right\|_{\dot{W}^{1,p}(\BT)}\\
\leq &\;
\f{Cr}{R}(\|H''\|_{L^p}\|\p\|_{L^\infty}+(m_0+M_0)\|\p'\|_{L^p}),
\end{split}
\label{eqn: W 1 p norm of derivative of interaction kernel from outer to inner}
\eeq
where $C = C(p)$.
\begin{proof}
Let $I_{r,1}$, $I_{r,2}$, $I_{\th,1}$ and $I_{\th,2}$ be defined as in the proof of Lemma \ref{lem: L inf norm of derivative of interaction kernel from outer to inner}.

We calculate that
\beq
I_{r,2}'(\th)=\f{1}{2}\int_\BT \left[P\left(\f{r}{R},\xi\right)
-P(D,\xi)\right]\p'(\th+\xi)\,d\xi-\f{1}{2}\int_\BT \pa_s P(D,\xi)\f{\pa D}{\pa \th}\p(\th+\xi)\,d\xi.
\eeq
Arguing as in \eqref{eqn: deviation of B inv from a constant} and \eqref{eqn: L inf estimate of I r 2}, 
%
\beq
\left\|\int_\BT \left[P\left(\f{r}{R},\xi\right)
-P(D,\xi)\right]\p'(\th+\xi)\,d\xi\right\|_{L^p}
\leq \f{Cr}{R}\d^{-1}(\|h\|_{L^\infty}+\|H\|_{L^\infty})\|\p'\|_{L^p}.
\label{eqn: easy term in I 2 r derivative}
\eeq
For the second term in $I_{r,2}'$, we derive by Lemma \ref{lem: properties of Poisson kernel} that
\beq
\begin{split}
&\;\int_\BT \pa_s P(D,\xi)\f{\pa D}{\pa \th}\p(\th+\xi)\,d\xi\\
=&\;\f{f'(\th)}{f(\th)}\int_\BT D\pa_s P(D,\xi)\p(\th+\xi)\,d\xi
+\int_\BT \pa_s P(D,\xi)\f{\pa D}{\pa \xi}\p(\th+\xi)\,d\xi\\
=&\;\f{f'(\th)}{f(\th)}\int_\BT \f{\pa Q(D,\xi)}{\pa \xi}\p(\th+\xi)\,d\xi+\int_\BT \pa_s P(D,\xi)\f{\pa D}{\pa \xi}\p(\th+\xi)\,d\xi\\
&\;-\f{f'(\th)}{f(\th)}\int_\BT \pa_s Q(D,\xi)\f{\pa D}{\pa \xi}\p(\th+\xi)\,d\xi\\
=:&\;I_{r,2,a}+I_{r,2,b}+I_{r,2,c}.
\end{split}
\eeq
Here $\f{\pa Q(D,\xi)}{\pa \xi}$ denotes total derivative of $Q(D(\th,\th+\xi),\xi)$ with respect to $\xi$. 

We integrate by parts in $I_{2,r,a}$.
Arguing as in \eqref{eqn: easy term in I 2 r derivative},
\beq
\begin{split}
\|I_{r,2,a}\|_{L^p}
\leq &\;C\|h'\|_{L^\infty}\left(\left\|\int_\BT \left[Q(D,\xi)-Q\left(\f{r}{R},\xi\right)\right]\p'(\th+\xi)\,d\xi\right\|_{L^p}+\| Q_{\f{r}{R}}*\p'\|_{L^p}\right)\\
\leq &\;\f{Cr}{R}\|h'\|_{L^\infty}\d^{-1}(\|h\|_{L^\infty}+\|H\|_{L^\infty})\|\p'\|_{L^p}+C\|h'\|_{L^\infty}\| P_{\f{r}{R}}*\CH\p'\|_{L^p}.
\end{split}
\label{eqn: estimate for I r 2 a crude form}
\eeq
Using the fact that $\CH\p'$ has mean zero on $\BT$, we derive that
\beq
P_{\f{r}{R}}*\CH\p' = \int_{\BT}\left(P_{\f{r}{R}}(\xi)-P_{\f{r}{R}}(\pi)\right)\CH\p'(\th-\xi)\,d\xi.
\eeq
By Young's inequality,
\beq
\|P_{\f{r}{R}}*\CH\p' \|_{L^p}\leq \int_{\BT}\left|P_{\f{r}{R}}(\xi)-P_{\f{r}{R}}(\pi)\right|\,d\xi \cdot \|\CH\p'\|_{L^p}\leq \f{Cr}{R}\|\p'\|_{L^p}.
\eeq
Therefore,
\beq
\|I_{r,2,a}\|_{L^p}\leq \f{Cr}{R}\|h'\|_{L^\infty}\|\p'\|_{L^p}.
\label{eqn: estimate for I r 2 a}
\eeq

Next we deal with $I_{r,2,b}$. 
Since
\beq
\f{\pa D}{\pa \xi} = -D\f{F'(\th+\xi)}{F(\th+\xi)},
\eeq
we find by Lemma \ref{lem: properties of Poisson kernel} that
\beq
\begin{split}
I_{r,2,b} =&\; -\int_{\BT}D\pa_s P(D,\xi)\f{F'(\th+\xi)}{F(\th+\xi)}\p(\xi+\th)\,d\xi\\
=&\; -\int_{\BT}\left[\f{\pa Q(D,\xi)}{\pa \xi}-\pa_s Q(D,\xi)\f{\pa D}{\pa \xi}\right]\cdot \f{F'\p}{F}(\xi+\th)\,d\xi\\
=&\; -\int_{\BT}\f{\pa Q(D,\xi)}{\pa \xi}\cdot \f{F'\p}{F}(\xi+\th)\,d\xi
-\int_{\BT}D\pa_s Q(D,\xi)\cdot \f{F'^2\p}{F^2}(\xi+\th)\,d\xi\\
=&\; -\int_{\BT}\f{\pa Q(D,\xi)}{\pa \xi}\cdot \f{F'\p}{F}(\xi+\th)\,d\xi
+\int_{\BT}\f{\pa P(D,\xi)}{\pa \xi}\cdot \f{F'^2\p}{F^2}(\xi+\th)\,d\xi\\
&\;-\int_{\BT}\pa_s P(D,\xi)\f{\pa D}{\pa \xi}\cdot \f{F'^2\p}{F^2}(\xi+\th)\,d\xi.
\end{split}
\label{eqn: rewriting I r 2 b}
\eeq
Arguing as in \eqref{eqn: estimate for I r 2 a crude form}-\eqref{eqn: estimate for I r 2 a},
\beq
\begin{split}
&\;\left\|\int_{\BT}\f{\pa Q(D,\xi)}{\pa \xi}\cdot \f{F'\p}{F}(\xi+\th)\,d\xi\right\|_{L^p}
+\left\|\int_{\BT}\f{\pa P(D,\xi)}{\pa \xi}\cdot \f{F'^2\p}{F^2}(\xi+\th)\,d\xi\right\|_{L^p}\\
\leq &\;\f{Cr}{R}\left\|\f{F'\p}{F}\right\|_{\dot{W}^{1,p}}
+\f{Cr}{R}\left\|\f{F'^2\p}{F^2}\right\|_{\dot{W}^{1,p}}\\
\leq &\;\f{Cr}{R}(\|H''\|_{L^p}\|\p\|_{L^\infty}+\|H'\|_{L^\infty}\|\p'\|_{L^p}).
\end{split}
\label{eqn: bound nice terms in I r 2 b}
\eeq
We notice that the last term in \eqref{eqn: rewriting I r 2 b}, which has not been bounded, is in a similar form as the original $I_{r,2,b}$.
Following \eqref{eqn: rewriting I r 2 b} and \eqref{eqn: bound nice terms in I r 2 b}, it is not difficult to argue by induction that for all $k\in \mathbb{N}$,
\beq
\begin{split}
&\;\|I_{r,2,b}\|_{L^p}\\
\leq &\; \f{Cr}{R}(\|H''\|_{L^p}\|\p\|_{L^\infty}+\|H'\|_{L^\infty}\|\p'\|_{L^p})+\left\|\int_{\BT}\pa_s P(D,\xi)\f{\pa D}{\pa \xi} \f{F'^{2k}\p}{F^{2k}}(\xi+\th)\,d\xi\right\|_{L^p}\\
\leq &\; \f{Cr}{R}(\|H''\|_{L^p}\|\p\|_{L^\infty}+\|H'\|_{L^\infty}\|\p'\|_{L^p})+\f{Cr}{R}\int_{\BT}\f{d\xi}{\d^2+\xi^2}\left(\f{\|H'\|_{L^\infty}}{1-\|H\|_{L^\infty}}\right)^{2k+1}\|\p\|_{L^p}.
\end{split}
\eeq
Here the constants $C$ are uniformly bounded in $k$ provided the smallness of $H$.
Since $M_0\ll 1$, we take $k\to \infty$ and obtain
\beq
\|I_{r,2,b}\|_{L^p}
\leq \f{Cr}{R}(\|H''\|_{L^p}\|\p\|_{L^\infty}+\|H'\|_{L^\infty}\|\p'\|_{L^p}).
\eeq

$\|I_{r,2,c}\|_{L^p}$ can be estimated in a similar manner, so is $\|I'_{\th,2}\|_{L^p}$.
\end{proof}
\end{lem}

Estimates for the operator $\CK_{\tilde{\g},\g}$ can be derived in a similar manner.

\begin{lem}\label{lem: interaction kernel from inner to outer}
\begin{enumerate}
\item Under the assumptions of Lemma \ref{lem: L inf norm of derivative of interaction kernel from outer to inner}, 
\beq
\begin{split}
&\;\left\|Fe_r(\th)\cdot \mathcal{K}_{\tilde{\g},\g}\p-\bar{\p}-\f{1}{4\pi}P_{\f{r}{R}}*(\p-\bar{\p})\right\|_{L^\infty(\BT)}\\
&\;+\left\|Fe_\th(\th)\cdot \mathcal{K}_{\tilde{\g},\g}\p-\f{1}{4\pi}Q_{\f{r}{R}}*(\p-\bar{\p})\right\|_{L^\infty(\BT)}\\
\leq &\;
\f{Cr}{R}\d^{-1}(\|h\|_{L^\infty}+\|H\|_{L^\infty})\|\p\|_{L^\infty},
\end{split}
\label{eqn: L inf norm of derivative of interaction kernel from inner to outer}
\eeq
where $C$ is a universal constant.
\item Under the assumptions of Lemma \ref{lem: Holder norm of derivative of interaction kernel from outer to inner},
\beq
\begin{split}
&\;\left\|Fe_r(\th)\cdot \mathcal{K}_{\tilde{\g},\g}\p-\f{1}{4\pi}P_{\f{r}{R}}*(\p-\bar{\p})\right\|_{\dot{C}^\b(\BT)}\\
&\;+\left\|Fe_\th(\th)\cdot \mathcal{K}_{\tilde{\g},\g}\p-\f{1}{4\pi}Q_{\f{r}{R}}*(\p-\bar{\p})\right\|_{\dot{C}^\b(\BT)}\\
\leq &\;
\f{Cr}{R}(m_0+M_0)\|\p\|_{\dot{C}^\beta}\\
&\;+\f{Cr}{R}\|\p\|_{L^\infty}(\d^{-1}(\|h\|_{L^\infty}+\|H\|_{L^\infty})+\|h'\|_{\dot{C}^{\alpha}}+\|H'\|_{\dot{C}^{\alpha}}),
\end{split}
\label{eqn: Holder norm of derivative of interaction kernel from inner to outer}
\eeq
where $C = C(\alpha,\b)$.
\item
Assume $h\in W^{2,p}(\BT)$ for some $p\in (1,\infty)$ and $H\in W^{1,\infty}(\BT)$,
satisfying that $m_0+M_0\ll 1$.
Then
%
%
\beq
\begin{split}
&\;\left\|Fe_r(\th)\cdot \mathcal{K}_{\tilde{\g},\g}\p-\f{1}{4\pi}P_{\f{r}{R}}*(\p-\bar{\p})\right\|_{\dot{W}^{1,p}(\BT)}\\
&\;+\left\|Fe_\th(\th)\cdot \mathcal{K}_{\tilde{\g},\g}\p-\f{1}{4\pi}Q_{\f{r}{R}}*(\p-\bar{\p})\right\|_{\dot{W}^{1,p}(\BT)}\\
\leq &\;\f{Cr}{R}(\|h''\|_{L^p}\|\p\|_{L^\infty}+(m_0+M_0)\|\p'\|_{L^p}),
\end{split}
\label{eqn: W 1 p norm of derivative of interaction kernel from inner to outer}
\eeq
where $C = C(p)$.
\end{enumerate}

\begin{proof}
We derive as in Lemma \ref{lem: L inf norm of derivative of interaction kernel from outer to inner}.
\beq
\begin{split}
&\;2\pi F(\th)e_r(\th)\cdot \CK_{\tilde{\g},\g}\p-2\pi \bar{\p}\\
=&\;\f12\int_{\BT} P\left(\f{r}{R},\xi\right)(\p(\th+\xi)-\bar{\p})\,d\xi+\f12\int_{\BT} \left[P\left(\f{f(\th+\xi)}{F(\th)},\xi\right)-P\left(\f{r}{R},\xi\right)\right]\p(\th+\xi)\,d\xi,
\end{split}
\eeq
and
\beq
\begin{split}
&\;2\pi F(\th)e_\th (\th)\cdot \CK_{\tilde{\g},\g}\p\\
=&\;-\f12\int_\BT Q\left(\f{r}{R},\xi\right)(\p(\th+\xi)-\bar{\p})\,d\xi+\f12\int_\BT \left[Q\left(\f{r}{R},\xi\right)
-Q\left(\f{f(\th+\xi)}{F(\th)},\xi\right)\right]\p(\th+\xi)\,d\xi.
\end{split}
\eeq
Then the desired estimate can be proved by arguing as in Lemmas \ref{lem: L inf norm of derivative of interaction kernel from outer to inner}-\ref{lem: W 1p norm of derivative of interaction kernel from outer to inner}.
\end{proof}
\end{lem}

Lastly, for those convolution terms on the left hand sides of the estimates in Lemmas \ref{lem: L inf norm of derivative of interaction kernel from outer to inner}-\ref{lem: interaction kernel from inner to outer}, we have that
\begin{lem}\label{lem: estimates for the potential kernel part in interaction kernels}
For $\beta\in(0,1)$, we have
\begin{align}
\|P_{\f{r}{R}}*(\p-\bar{\p})\|_{L^\infty}\leq &\;\f{4\pi r}{R+r}\|\p-\bar{\p}\|_{L^\infty},\\
\|Q_{\f{r}{R}}*(\p-\bar{\p})\|_{L^\infty}\leq &\;C\|\p\|_{\dot{C}^\beta},
\end{align}
and
\begin{align}
\|P_{\f{r}{R}}*(\p-\bar{\p})\|_{\dot{C}^{\beta}}\leq &\; \f{4\pi r}{R+r}\|\p\|_{\dot{C}^\beta},\\
\|Q_{\f{r}{R}}*(\p-\bar{\p})\|_{\dot{C}^{\beta}}\leq &\; C\|\p\|_{\dot{C}^\beta}.
\end{align}
where these two constants $C$ depend on $\b$.
Moreover, for $p\in(1,\infty)$,
\begin{align}
\|P_{\f{r}{R}}*(\p-\bar{\p})\|_{\dot{W}^{1,p}}\leq &\; \f{4\pi r}{R+r} \|\p'\|_{L^p},\\
\|Q_{\f{r}{R}}*(\p-\bar{\p})\|_{\dot{W}^{1,p}}\leq &\; C\|\p'\|_{L^p},
\end{align}
where $C$ depends on $p$.

\begin{proof}
Since
\beq
P_{\f{r}{R}}*(\p-\bar{\p}) = \int_{\BT} (P_{\f{r}{R}}(\xi)-P_{\f{r}{R}}(\pi))(\p(\th-\xi)-\bar{\p})\,d\xi,
\eeq
and $P_{\f{r}{R}}(\xi)\geq P_{\f{r}{R}}(\pi)$, we have that
\beq
\|P_{\f{r}{R}}*(\p-\bar{\p})\|_{L^\infty}\leq \|\p-\bar{\p}\|_{L^\infty}\int_\BT P_{\f{r}{R}}(\xi)-P_{\f{r}{R}}(\pi)\,d\xi
=\f{4\pi r}{R+r}\|\p-\bar{\p}\|_{L^\infty}.
\eeq
Since $Q_{\f{r}{R}}$ has integral zero over $\BT$, by Lemma \ref{lem: properties of Poisson kernel},
\beq
|Q_{\f{r}{R}}*(\p-\bar{\p})| = \left|\int_{\BT} Q_{\f{r}{R}}(\xi)(\p(\th-\xi)- \p(\th))\,d\xi\right|\leq C\|\p\|_{\dot{C}^\beta}\int_{\BT}\f{|\xi|^\beta}{\d+|\xi|}\,d\xi\leq C\|\p\|_{\dot{C}^\beta}.
\eeq

It is straightforward to derive that for $\th_1,\th_2\in \BT$,
\beq
\begin{split}
&\;|P_{\f{r}{R}}*(\p-\bar{\p})(\th_1)-P_{\f{r}{R}}*(\p-\bar{\p})(\th_2)|\\
= &\;\left|\int_\BT (P_{\f{r}{R}}(\xi)-P_{\f{r}{R}}(\pi))(\p(\th_1-\xi)-\p(\th_2-\xi))\,d\xi\right|\\
\leq &\;\f{4\pi r}{R+r}\|\p\|_{\dot{C}^\beta}|\th_1-\th_2|^\beta.
\end{split}
\eeq
%
Moreover, by Young's inequality,
\beq
\|P_{\f{r}{R}}*(\p-\bar{\p})\|_{\dot{W}^{1,p}}= \|(P_{\f{r}{R}}-P_{\f{r}{R}}(\pi))*\p'\|_{L^p}\leq \f{4\pi r}{R+r}\|\p'\|_{L^p}.
\eeq
The estimates involving $Q_{\f{r}{R}}$ follows from the fact $Q_{\f{r}{R}} = \CH P_{\f{r}{R}}$.
Note that the boundedness of Hilbert transform on $C^\b(\BT)$ can be justified by that of its counterpart on $C^\b(\mathbb{R})$ with some adaptation.
\end{proof}
\end{lem}

\section{Existence, Uniqueness and Estimates for $[\va]$ and $\phi$}
\label{sec: bounds for boundary potentials}
This section aims at establishing well-definedness, regularity and estimates for $[\va]_\g$ and $\phi$.
The main approach is to apply a fixed-point argument to static equations \eqref{eqn: equation for derivative of jump of potential final form} and \eqref{eqn: equation for derivative of potential on outer interface}, by using many estimates in Sections \ref{sec: estimate for pressure in circular geometry}-\ref{sec: estimates for interaction operators}.

With the domain determined by $r$, $R$, $h$ and $H$, let $\tilde{p}$ be defined by \eqref{eqn: pressure in the reference coordinate} and \eqref{eqn: equation for p tilde}, and let the radially symmetric solution $p_*$ be defined as in \eqref{eqn: equation for p star}.
Recall that $c_*$ and $\tilde{c}_*$ are defined in \eqref{eqn: def of c characteristic normal derivative}.
In fact, $c_* = -\mu|\nabla p_*(r^-)|$ and $\tilde{c}_* = -\nu |\na p_*(R)|$, so their estimates can be found in Lemma \ref{lem: estimate for p_*}.
Also recall that $\CS\p:=\f{1}{2\pi}P_{\f{r}{R}}*\p$ defined in \eqref{eqn: def of operator S}.
Then $\CH\CS\p = \f{1}{2\pi}Q_{\f{r}{R}}*\p$ thanks to Lemma \ref{lem: properties of Poisson kernel}.

In the spirit of the linearized equations \eqref{eqn: linearized equation for derivative of jump of potential} and \eqref{eqn: linearized equation for derivative of potential on outer interface}, we rewrite \eqref{eqn: equation for derivative of jump of potential final form} and \eqref{eqn: equation for derivative of potential on outer interface} as
\begin{align}
[\va]'-2Ac_* f'-A\CS\phi' = &\;\CR_{[\va]'},\label{eqn: introducing remainder in va equation}\\
\phi' +2\tilde{c}_* F'-\CS[\va]' = &\; \CR_{\phi'},\label{eqn: introducing remainder in phi equation}
\end{align}
where
\beq
\begin{split}
\CR_{[\va]'}:= &\;2A f'(\th) (e_r\cdot \na(\G*g)|_{\g}-c_*)
+ 2Af(\th) e_\th\cdot \na(\G*g)|_{\g}\\
&\;+ 2A\g'(\th)^\perp\cdot \CK_\g [\va]'
+ 2A\left(\g'(\th)^\perp\cdot \CK_{\g,\tilde{\g}} \phi'-\frac{1}{2}\CS\phi'\right),
\end{split}
\label{eqn: def of remainder in va equation}
\eeq
\beq
\begin{split}
\CR_{\phi'}:= &\;-2F' (e_r\cdot \na(\G*g)|_{\tilde{\g}}-\tilde{c}_*)-2 F e_\th\cdot \na(\G*g)|_{\tilde{\g}}\\
&\;-2\tilde{\g}'(\th)^\perp\cdot \CK_{\tilde{\g}} \phi'
-2\left(\tilde{\g}'(\th)^\perp\cdot \CK_{\tilde{\g},\g} [\va]'+\f{1}{2}\CS[\va]'\right).
\end{split}
\label{eqn: def of remainder in phi equation}
\eeq

In what follows, we will need to apply the lemmas in Section \ref{sec: gradient estimate of growth potential} with $g_0 = G(\tilde{p}(X))\chi_{B_r}(X)$.
For that purpose, according to \eqref{eqn: def of c characteristic normal derivative inner interface} and \eqref{eqn: def of c characteristic normal derivative outer interface}, we define
\beq
c = -\f{1}{2\pi r}\int_{B_r} G(\tilde{p}(X))\,dX,\quad \tilde{c} 
= \f{r}{R}c.
\label{eqn: def of c characteristic normal derivative for actual solution}
\eeq
We can show the following relation between $c$ and $c_*$.

\begin{lem}\label{lem: bound error between c and c_*}
Let $c_*$ and $c$ be defined in \eqref{eqn: def of c characteristic normal derivative} and \eqref{eqn: def of c characteristic normal derivative for actual solution}, respectively.
Then under the assumption $m_0+M_0\ll 1$,
\beq
|c-c_*|\leq Cr (m_0+M_0)(\d R^2)^{1/2},
\eeq
where $C =C(\mu,\nu, G)$.
\begin{proof}
Thanks to the $C^1$-smoothness of $G$,
\beq
|c-c_*|\leq Cr^{-1}\int_{B_r}|\tilde{p}-p_*|\,dX.
\eeq

If $r\geq R/2$, by Lemma \ref{lem: estimates for difference between tilde p and p_*}, H\"{o}lder's inequality and Poincar\'{e} inequality,
\beq
|c-c_*|\leq C\|\tilde{p}-p_*\|_{L^2(B_R)}\leq CR\|\na (\tilde{p}-p_*)\|_{L^2(B_R)}\leq CR(m_0+M_0)(\d R^2)^{1/2}.
\eeq
Since $r$ and $R$ are comparable, the desired estimate follows.

Otherwise, the estimate follows from \eqref{eqn: L^inf bound of tilde p - p* when R is way larger than r}.
\end{proof}
\end{lem}

Then we turn to prove that the static equations \eqref{eqn: equation for derivative of jump of potential final form} and \eqref{eqn: equation for derivative of potential on outer interface} have solutions $[\va]'$ and $\phi'$.
\begin{prop}\label{prop: Holder estimate for jump of potential and potential along outer interface}
Let $\b'\in (0,1)$ and $\b\in(0,\f{\b'}{1+\b'})$.
Suppose $h,H\in C^{1,\beta'}(\BT)$, such that
\beq
m_0+M_0+\|h'\|_{\dot{C}^{\beta'}}+\|H'\|_{\dot{C}^{\beta'}}\ll 1,
\label{eqn: smallness Holderassumption for the main estimates}
\eeq
where the smallness depends on $\mu$, $\nu$, $\b$ and $\b'$.
Then there exist unique $[\va]',\phi'\in C^{\b}(\BT)$ solving \eqref{eqn: equation for derivative of jump of potential final form} and \eqref{eqn: equation for derivative of potential on outer interface}, or equivalently \eqref{eqn: introducing remainder in va equation}-\eqref{eqn: def of remainder in phi equation}.
They satisfy that 
\beq
\begin{split}
&\;\|[\va]'\|_{\dot{C}^\beta}+\|\phi'\|_{\dot{C}^\beta}\\
\leq &\; C|c_*|r(\|h'\|_{\dot{C}^\beta}+\|H'\|_{\dot{C}^\beta})+Cr^2 (\d^{\b}\|h'\|_{\dot{C}^\b}+(m_0+M_0) (1+\d R^2)^{1/2})\\
=:&\; N_{1,\b},
\end{split}
\label{eqn: Holder estimate for derivative of potentials}
\eeq
where $C = C(\mu,\nu, G, \b, \b')$.
\begin{proof}
We will first derive a priori estimates for $[\va]'$ and $\phi'$, and then briefly discuss the proof of their existence and uniqueness at the end.

By Lemmas \ref{lem: estimates for difference between tilde p and p_*}, \ref{lem: L^inf estimate of derivative of growth potential inner} and
\ref{lem: W 1p estimate of derivative of growth potential inner} (with $p=(1-\beta)^{-1}$), the $C^1$-smoothness of $G$ and the smallness of $h$,
\beq
\begin{split}
&\;\|f'(e_r\cdot \na (\G* g)|_{\g}-c)\|_{\dot{C}^{\beta}}+\|fe_\th\cdot \na (\G* g)|_{\g}\|_{\dot{C}^{\beta}}\\
\leq &\; \|f'\|_{\dot{C}^{\beta}} \|e_r\cdot \na(\G*g)|_{\g}-c\|_{L^\infty}
+\|f'\|_{L^\infty} \|e_r\cdot \na(\G*g)|_{\g}\|_{\dot{W}^{1,p}}\\
&\;+\|f\|_{\dot{C}^\beta} \|e_\th\cdot \na(\G*g)|_{\g}\|_{L^\infty}
+\|f\|_{L^\infty}\|e_\th\cdot \na(\G*g)|_{\g}\|_{\dot{W}^{1,p}}\\
\leq &\;Cr^2\|h'\|_{\dot{C}^\beta} ( m_0 \d |\ln \d| + \|\na (\tilde{p}-p_*) \|_{L^2(B_r)})\\
&\;+Cr^2 ( m_\b+ \|\na (\tilde{p}-p_*)\|_{L^2(B_r)})\\
\leq &\;Cr^2 (m_\b+(m_0+M_0) (\d R^2)^{1/2}).
\end{split}
\label{eqn: Holder estimate for source terms inner interface}
\eeq
On the other hand, for $\beta\in(0, \f{\beta'}{1+\beta'})$, by Lemmas \ref{lem: Holder estimate of singular integral normal component}, \ref{lem: L inf norm of derivative of interaction kernel from outer to inner},
\ref{lem: Holder norm of derivative of interaction kernel from outer to inner} and \ref{lem: estimates for the potential kernel part in interaction kernels},
\beq
\begin{split}
&\;\|\g'(\th)^\perp\cdot \CK_\g [\va]'\|_{\dot{C}^{\beta}}
+\left\|\g'(\th)^\perp\cdot \CK_{\g,\tilde{\g}} \phi'-\f12 \CS\phi'\right\|_{\dot{C}^{\beta}}\\
\leq &\;C\|h'\|_{\dot{C}^\beta} (\|[\va]'\|_{C^\b}+\|[\va]'\|_{L^\infty}\|h'\|_{\dot{C}^\beta}\|h'\|_{L^\infty})\\
&\;+\|f'/f\|_{\dot{C}^{\beta}}\|fe_\th\cdot \CK_{\g,\tilde{\g}} \phi'\|_{L^\infty}+\|f'/f\|_{L^\infty}\|fe_\th\cdot \CK_{\g,\tilde{\g}} \phi'\|_{\dot{C}^{\beta}}\\
&\;+\left\|fe_r\cdot \CK_{\g,\tilde{\g}} \phi'+\f12\CS\phi'\right\|_{\dot{C}^{\beta}}\\
%
\leq &\;C\|h'\|_{\dot{C}^{\beta}}\|[\va]'\|_{\dot{C}^\b}
+C(m_0+M_0+\|h'\|_{\dot{C}^{\beta'}}+\|H'\|_{\dot{C}^{\beta'}})\|\phi'\|_{\dot{C}^\beta},
\end{split}
\label{eqn: Holder estimate for interaction potentials inner interface}
\eeq
where $C = C(\b,\b')$.
Hence, by \eqref{eqn: def of remainder in va equation}, Lemma \ref{lem: bound error between c and c_*} and the fact that $|A|\leq 1$,
\beq
\begin{split}
\|\CR_{[\va]'}\|_{\dot{C}^\beta}
\leq &\; |A|C(\b,\b')(m_0+M_0+\|h'\|_{\dot{C}^{\beta'}}+\|H'\|_{\dot{C}^{\beta'}})\|\phi'\|_{\dot{C}^\beta}\\
&\;+C(\b,\b')\|h'\|_{\dot{C}^{\beta}}\|[\va]'\|_{\dot{C}^\beta}\\
&\;+Cr^2 (m_\b+(m_0+M_0) (\d R^2)^{1/2}),
\end{split}
\label{eqn: Holder estimate for remainder of derivative of jump of potential at inner interface}
\eeq
and thus by \eqref{eqn: introducing remainder in va equation},
\beq
\begin{split}
\|[\va]'\|_{\dot{C}^\beta}
\leq &\; |A|\left(\f{2r}{R+r}+C(\b,\b')(m_0+M_0+\|h'\|_{\dot{C}^{\beta'}}+\|H'\|_{\dot{C}^{\beta'}})\right)\|\phi'\|_{\dot{C}^\beta}\\
&\;+C(\b,\b')\|h'\|_{\dot{C}^{\beta}}\|[\va]'\|_{\dot{C}^\beta}\\
&\;+C|c_*|r\|h'\|_{\dot{C}^\beta}+Cr^2 (m_\b+(m_0+M_0) (\d R^2)^{1/2}),
\end{split}
\label{eqn: Holder estimate for derivative of jump of potential at inner interface}
\eeq
where $C = C(\mu,\nu, G, \b, \b')$ unless otherwise stated.

Similarly, by Lemmas \ref{lem: estimates for difference between tilde p and p_*}, \ref{lem: L^inf estimate of derivative of growth potential outer} and \ref{lem: W 1p estimate of derivative of growth potential outer},
\beq
\begin{split}
&\;\|F' (e_r\cdot \na(\G*g)-\tilde{c})|_{\tilde{\g}}\|_{\dot{C}^{\beta}} +\|F e_\th\cdot \na(\G*g)|_{\tilde{\g}}\|_{\dot{C}^{\beta}}\\
\leq &\; \|F'\|_{\dot{C}^{\beta}} \|e_r\cdot \na(\G*g)|_{\tilde{\g}}-\tilde{c}\|_{L^\infty}
+\|F'\|_{L^\infty} \|e_r\cdot \na(\G*g)|_{\tilde{\g}}\|_{\dot{W}^{1,p}}\\
&\;+\|F\|_{\dot{C}^\beta} \|e_\th\cdot \na(\G*g)|_{\tilde{\g}}\|_{L^\infty}
+\|F\|_{L^\infty}\|e_\th\cdot \na(\G*g)|_{\tilde{\g}}\|_{\dot{W}^{1,p}}\\
\leq &\;Cr^2 (m_0 + M_0)(1+\d R^2)^{1/2}.
\end{split}
\label{eqn: Holder estimate for source terms outer interface}
\eeq
By Lemmas \ref{lem: Holder estimate of singular integral normal component}, \ref{lem: interaction kernel from inner to outer} and \ref{lem: estimates for the potential kernel part in interaction kernels},
\beq
\begin{split}
&\;\|\tilde{\g}'(\th)^\perp\cdot \CK_{\tilde{\g}} \phi'\|_{\dot{C}^{\beta}}
+\left\|\tilde{\g}'(\th)^\perp\cdot \CK_{\tilde{\g},\g} [\va]'+\f12 \CS [\va]\right\|_{\dot{C}^{\beta}}\\
\leq &\;C\|H'\|_{\dot{C}^\beta} (\|\phi'\|_{C^\b}+\|\phi'\|_{L^\infty}\|H'\|_{\dot{C}^\beta}\|H'\|_{L^\infty})\\
&\;+\|F'/F\|_{\dot{C}^{\beta}}\|Fe_\th\cdot \CK_{\tilde{\g},\g} [\va]'\|_{L^\infty}+\|F'/F\|_{L^\infty}\|Fe_\th\cdot \CK_{\tilde{\g},\g} [\va]'\|_{\dot{C}^{\beta}}\\
&\;+\left\|Fe_r\cdot \CK_{\tilde{\g},\g} [\va]'-\f12\CS[\va]'\right\|_{\dot{C}^{\beta}}\\
\leq &\;C\|H'\|_{\dot{C}^{\beta}}\|\phi'\|_{\dot{C}^\b}
+C(m_0+M_0+\|h'\|_{\dot{C}^{\beta'}}+\|H'\|_{\dot{C}^{\beta'}})\|[\va]'\|_{\dot{C}^\beta},
\end{split}
\label{eqn: Holder estimate for interaction potentials outer interface}
\eeq
where $C = C(\b,\b')$.
Combining them with \eqref{eqn: introducing remainder in phi equation}, \eqref{eqn: def of remainder in phi equation} and Lemma \ref{lem: bound error between c and c_*}, we obtain that
\beq
\begin{split}
\|\CR_{\phi'}\|_{\dot{C}^\beta}
\leq &\;
C(\b,\b')(m_0+M_0+\|h'\|_{\dot{C}^{\beta'}}+\|H'\|_{\dot{C}^{\beta'}})\|[\va]'\|_{\dot{C}^\beta}\\
&\;+C(\b,\b')\|H'\|_{\dot{C}^{\beta}}\|\phi'\|_{\dot{C}^\b}+Cr^2 (m_0 + M_0)(1+\d R^2)^{1/2},
\end{split}
\label{eqn: Holder estimate for remainder of derivative of outer potential}
\eeq
and
\beq
\begin{split}
\|\phi'\|_{\dot{C}^\beta}
\leq &\;
\left(\f{2r}{R+r}+C(\b,\b')(m_0+M_0+\|h'\|_{\dot{C}^{\beta'}}+\|H'\|_{\dot{C}^{\beta'}})\right)\|[\va]'\|_{\dot{C}^\beta}\\
&\;+C(\b,\b')\|H'\|_{\dot{C}^{\beta}}\|\phi'\|_{\dot{C}^\b}\\
&\;+C|\tilde{c}_*|R\|H'\|_{\dot{C}^\beta}+Cr^2 (m_0 + M_0)(1+\d R^2)^{1/2},
\end{split}
\label{eqn: Holder estimate for derivative of outer potential}
\eeq
where $C = C(\mu,\nu, G,\b,\b')$.

Since $|A|<1$ and $\tilde{c}_* = \f{r}{R}c_*$, by the smallness assumption \eqref{eqn: smallness Holderassumption for the main estimates},
we combine \eqref{eqn: Holder estimate for derivative of jump of potential at inner interface} and \eqref{eqn: Holder estimate for derivative of outer potential} to obtain \eqref{eqn: Holder estimate for derivative of potentials}. 

Let us briefly explain the proof of existence and uniqueness of $[\va]'$ and $\phi'$.
Let $V$ denote the space of $C^\beta(\BT)$-functions with mean zero.
Take $h$ and $H$ satisfying the assumptions.
According to \eqref{eqn: introducing remainder in va equation} and \eqref{eqn: introducing remainder in phi equation}, define a map from $V\times V$ to itself by
\beq
([\va]', \phi')
\mapsto \left(2Ac_* f'+A\CS\phi' +\CR_{[\va]'},-2\tilde{c}_* F'+\CS[\va]' +\CR_{\phi'}\right).
\label{eqn: mapping T}
\eeq
Thanks to the estimates above, one can easily show that the map is well-defined and it is a contraction mapping provided the smallness of $h$ and $H$.
Then the existence and uniqueness of $([\va]',\phi')$ follow.
\end{proof}
\end{prop}

\begin{prop}\label{prop: W 1p estimate for jump of potential and potential along outer interface}
Let $\beta'\in (0,1)$, $\beta\in(0,\f{\beta'}{1+\beta'})$ and $p\in [2,\infty)$.
Suppose $h,H\in C^{1,\beta'}\cap W^{2,p}(\BT)$, such that
\beq
m_0+M_0+\|h'\|_{\dot{C}^{\beta'}}+\|H'\|_{\dot{C}^{\beta'}}\ll 1,
\eeq
where the smallness depends on $\mu$, $\nu$, $p$, $\b$ and $\b'$.
Then $[\va]'$ and $\phi'$ obtained in Proposition \ref{prop: Holder estimate for jump of potential and potential along outer interface} also belong to $W^{1,p}(\BT)$.
They satisfy 
\beq
\begin{split}
&\;\|[\va]''\|_{L^p}+\|\phi''\|_{L^p}\\
\leq
&\;C|c_*|r(\|h''\|_{L^p}+\|H''\|_{L^p})\\
&\;+Cr^2(1+\|h''\|_{L^p}+\|H''\|_{L^p})( \d^\b\|h'\|_{\dot{C}^\b}+(m_0+M_0) (1+\d R^2)^{1/2})\\
=:&\;N_{2,p},
\end{split}
\label{eqn: W 2 p estimate of boundary potentials final form}
\eeq
where $C = C(\mu,\nu,p,G,\b,\b')$.
\begin{proof}
The proof is similar to that of Proposition \ref{prop: Holder estimate for jump of potential and potential along outer interface}.

Let $c$ and $\tilde{c}$ be defined as in \eqref{eqn: def of c characteristic normal derivative for actual solution}. 
We proceed as before.
\beq
\begin{split}
&\;\|f'(e_r\cdot \na (\G* g)|_{\g}-c)\|_{\dot{W}^{1,p}}+\|fe_\th\cdot \na (\G* g)|_{\g}\|_{\dot{W}^{1,p}}\\
\leq &\; \|f'\|_{\dot{W}^{1,p}} \|e_r\cdot \na(\G*g)|_{\g}-c\|_{L^\infty}
+\|f'\|_{L^\infty} \|e_r\cdot \na(\G*g)|_{\g}\|_{\dot{W}^{1,p}}\\
&\;+\|f'\|_{L^p} \|e_\th\cdot \na(\G*g)|_{\g}\|_{L^\infty}
+\|f\|_{L^\infty}\|e_\th\cdot \na(\G*g)|_{\g}\|_{\dot{W}^{1,p}}\\
\leq &\;Cr^2\|h''\|_{L^p}( m_0 \d |\ln \d| + \|\na (\tilde{p}-p_*)\|_{L^2(B_r)})\\
&\;+Cr^2 (m_\b+ \|\na (\tilde{p}-p_*)\|_{L^2(B_r)})\\
\leq &\;Cr^2 \d^\b\|h'\|_{\dot{C}^\b}+ Cr^2(1+\|h''\|_{L^p})(m_0+M_0) (1+\d R^2)^{1/2},
\end{split}
\label{eqn: W 1p estimate for source term on inner interface}
\eeq
%
where $C = C(\mu,\nu,p, G, \b)$, and by Lemma \ref{lem: W 1p estimate of normal component of singular integral} and Lemma \ref{lem: W 1p norm of derivative of interaction kernel from outer to inner},
\beq
\begin{split}
&\;\|\g'(\th)^\perp\cdot \CK_\g [\va]'\|_{\dot{W}^{1,p}}+\left\|\g'(\th)^\perp\cdot \CK_{\g,\tilde{\g}} \phi'-\f{1}{2}\CS\phi'\right\|_{\dot{W}^{1,p}}\\
\leq &\;C\| h''\|_{L^p}\|[\va]'\|_{L^\infty}(1+\|h'\|_{\dot{C}^\b})+C( \|h''\|_{L^p}\|[\va]'\|_{\dot{C}^\b} +\|h'\|_{L^\infty}\|[\va]''\|_{L^p})\\
&\;+\|f'/f\|_{\dot{W}^{1,p}}\|fe_\th\cdot \CK_{\g,\tilde{\g}} \phi'\|_{L^\infty}
+\|f'/f\|_{L^\infty}\|fe_\th\cdot \CK_{\g,\tilde{\g}} \phi'\|_{\dot{W}^{1,p}}\\
&\;+\left\|fe_r\cdot \CK_{\g,\tilde{\g}} \phi'
+\f{1}{2}\CS \phi'\right\|_{\dot{W}^{1,p}}\\
\leq &\;C( \|h''\|_{L^p}\|[\va]'\|_{\dot{C}^\b} +\|h'\|_{L^\infty}\|[\va]''\|_{L^p})\\
&\;+C\|h''\|_{L^p}\|\phi'\|_{\dot{C}^\beta}+C\|h'\|_{L^\infty}\|\phi''\|_{L^p}\\
&\;+C(\|H''\|_{L^p}\|\phi'\|_{L^\infty}+(m_0+M_0)\|\phi''\|_{L^p})
\\
\leq &\;C(m_0+M_0)\|\phi''\|_{L^p}+C\|h'\|_{L^\infty}\|[\va]''\|_{L^p} \\
&\;+C(\|h''\|_{L^p}+\|H''\|_{L^p})(\|[\va]'\|_{\dot{C}^\b} +\|\phi'\|_{\dot{C}^\beta}),
\label{eqn: W 1p bound for interaction potentials inner interface}
\end{split}
\eeq
where $C = C(p,\b)$.
Combining them with \eqref{eqn: introducing remainder in va equation} and \eqref{eqn: def of remainder in va equation}, by Lemma \ref{lem: estimates for the potential kernel part in interaction kernels} and Lemma \ref{lem: bound error between c and c_*}
\beq
\begin{split}
\|\CR_{[\va]'}'\|_{L^p}
\leq
&\;C(p,\b)(m_0+M_0)\|\phi''\|_{L^p}+C(p,\b)\|h'\|_{L^\infty}\|[\va]''\|_{L^p} \\
&\;+C(\|h''\|_{L^p}+\|H''\|_{L^p})(\|[\va]'\|_{\dot{C}^\b} +\|\phi'\|_{\dot{C}^\beta})\\
&\;+Cr^2 \d^\b\|h'\|_{\dot{C}^\b}+ Cr^2(1+\|h''\|_{L^p})(m_0+M_0) (1+\d R^2)^{1/2},
\end{split}
\label{eqn: W 2 p estimate of remainder of jump of potential crude form}
\eeq
and thus
\beq
\begin{split}
&\;\|[\va]''\|_{L^p}\\
\leq
&\;\left(\f{2|A|r}{R+r}+C(p,\b)(m_0+M_0)\right)\|\phi''\|_{L^p}+C(p,\b)\|h'\|_{L^\infty}\|[\va]''\|_{L^p} \\
&\;+C(\|h''\|_{L^p}+\|H''\|_{L^p})(\|[\va]'\|_{\dot{C}^\b} +\|\phi'\|_{\dot{C}^\beta})\\
&\;+C|c_*|r\|h''\|_{L^p}+Cr^2 \d^\b\|h'\|_{\dot{C}^\b}+ Cr^2(1+\|h''\|_{L^p})(m_0+M_0) (1+\d R^2)^{1/2},
\end{split}
\label{eqn: W 2 p estimate of jump of potential crude form}
\eeq
where $C = C(\mu,\nu,p,G,\b)$ unless otherwise stated.

Moreover,
\beq
\begin{split}
&\;\|F' (e_r\cdot \na(\G*g)-\tilde{c})|_{\tilde{\g}}\|_{\dot{W}^{1,p}} +\|F e_\th\cdot \na(\G*g)|_{\tilde{\g}}\|_{\dot{W}^{1,p}}\\
\leq &\; \|F'\|_{\dot{W}^{1,p}} \|e_r\cdot \na(\G*g)|_{\tilde{\g}}-\tilde{c}\|_{L^\infty}
+\|F'\|_{L^\infty} \|e_r\cdot \na(\G*g)|_{\tilde{\g}}\|_{\dot{W}^{1,p}}\\
&\;+\|F\|_{\dot{W}^{1,p}} \|e_\th\cdot \na(\G*g)|_{\tilde{\g}}\|_{L^\infty}
+\|F\|_{L^\infty}\|e_\th\cdot \na(\G*g)|_{\tilde{\g}}\|_{\dot{W}^{1,p}}\\
\leq &\;Cr^2(1+\|H''\|_{L^p}) (m_0+M_0) (1+\d R^2)^{1/2},
\end{split}
\label{eqn: W 1p estimate for source term on outer interface}
\eeq
where $C = C(p,G,\b)$,
and 
\beq
\begin{split}
&\;\|\tilde{\g}'(\th)^\perp\cdot \CK_{\tilde{\g}} \phi'\|_{\dot{W}^{1,p}}
+\left\|\tilde{\g}'(\th)^\perp\cdot \CK_{\tilde{\g},\g} [\va]'+\f{1}{2}\CS[\va]'\right\|_{\dot{W}^{1,p}}\\
\leq &\;C\| H''\|_{L^p}\|\phi'\|_{L^\infty}(1+\|H'\|_{\dot{C}^\b})+C( \|H''\|_{L^p}\|\phi'\|_{\dot{C}^\b} +\|H'\|_{L^\infty}\|\phi''\|_{L^p})\\
&\;+\|F'/F\|_{\dot{W}^{1,p}}\|Fe_\th\cdot \CK_{\tilde{\g},\g} [\va]'\|_{L^\infty}+\|F'/F\|_{L^\infty}\|Fe_\th\cdot \CK_{\tilde{\g},\g} [\va]'\|_{\dot{W}^{1,p}}\\
&\;+\left\|Fe_r\cdot \CK_{\tilde{\g},\g} [\va]'-\f{1}{2}\CS[\va]'\right\|_{\dot{W}^{1,p}}\\
\leq &\;C(m_0+M_0)\|[\va]''\|_{L^p}+C\|H'\|_{L^\infty}\|\phi''\|_{L^p}\\
&\;+C(\|h''\|_{L^p}+\|H''\|_{L^p})(\|\phi'\|_{\dot{C}^\b}+\|[\va]'\|_{\dot{C}^\b}),
\end{split}
\label{eqn: W 1p bound for interaction potentials outer interface}
\eeq
where $C = C(p,\b)$.
Hence, by \eqref{eqn: introducing remainder in phi equation}, \eqref{eqn: def of remainder in phi equation}, Lemma \ref{lem: estimates for the potential kernel part in interaction kernels} and Lemma \ref{lem: bound error between c and c_*}, with $C = C(p,G,\b)$,
\beq
\begin{split}
\|\CR_{\phi'}'\|_{L^p}
\leq  &\;C(p,\b)(m_0+M_0)\|[\va]''\|_{L^p}+C(p,\b)\|H'\|_{L^\infty}\|\phi''\|_{L^p}\\
&\;+C(\|h''\|_{L^p}+\|H''\|_{L^p})(\|\phi'\|_{\dot{C}^\b}+\|[\va]'\|_{\dot{C}^\b} )\\
&\;+Cr^2(1+\|H''\|_{L^p}) (m_0+M_0) (1+\d R^2)^{1/2},
\end{split}
\label{eqn: W 2 p estimate of remainder of potential on outer interface crude form}
\eeq
and
\beq
\begin{split}
\|\phi''\|_{L^p}
\leq  &\;\left(\f{2r}{R+r}+C(p,\b)(m_0+M_0)\right)\|[\va]''\|_{L^p}+C(p,\b)\|H'\|_{L^\infty}\|\phi''\|_{L^p}\\
&\;+C(\|h''\|_{L^p}+\|H''\|_{L^p})(\|\phi'\|_{\dot{C}^\b}+\|[\va]'\|_{\dot{C}^\b} )\\
&\;+C|\tilde{c}_*|R\|H''\|_{L^p}+Cr^2(1+\|H''\|_{L^p}) (m_0+M_0) (1+\d R^2)^{1/2}.
\end{split}
\label{eqn: W 2 p estimate of potential on outer interface crude form}
\eeq
Since $|A|<1$ and $m_0+M_0\ll 1$, we combine \eqref{eqn: W 2 p estimate of jump of potential crude form}
and \eqref{eqn: W 2 p estimate of potential on outer interface crude form} to obtain that
\beq
\begin{split}
&\;\|[\va]''\|_{L^p}+\|\phi''\|_{L^p}\\
\leq
&\;C(\|h''\|_{L^p}+\|H''\|_{L^p})(\|[\va]'\|_{\dot{C}^\b} +\|\phi'\|_{\dot{C}^\beta})+C|c_*|r(\|h''\|_{L^p}+\|H''\|_{L^p})\\
&\;+Cr^2 \d^\b\|h'\|_{\dot{C}^\b}+ Cr^2(1+\|h''\|_{L^p}+\|H''\|_{L^p})(m_0+M_0) (1+\d R^2)^{1/2},
\end{split}
\eeq
where $C = C(\mu,\nu,p,G,\b)$.
Applying Proposition \ref{prop: Holder estimate for jump of potential and potential along outer interface} yields the desired estimate.

To prove $[\va]',\phi'\in W^{1,p}(\BT)$, we simply define $\tilde{V}$ to be the space of mean-zero $C^{\beta}\cap W^{1,p}(\BT)$-functions.
One can show that the map in \eqref{eqn: mapping T} is well-defined from $\tilde{V}\times \tilde{V}$ to itself and it is a contraction mapping, provided smallness of $h$ and $H$.
\end{proof}
\end{prop}

\begin{lem}\label{lem: estimate for principle part in R_f and R_F}
Under the assumptions of Proposition \ref{prop: W 1p estimate for jump of potential and potential along outer interface},
\beq
\begin{split}
&\;\|\CR_{[\va]'}\|_{\dot{C}^\beta}+\|\CR_{\phi'}\|_{\dot{C}^\beta}\\
\leq &\; C|c_*|r(\|h'\|_{\dot{C}^{\beta'}}+\|H'\|_{\dot{C}^{\beta'}})^2\\
&\;+Cr^2 (\d^{\b}\|h'\|_{\dot{C}^\b}+(m_0+M_0) (1+\d R^2)^{1/2}),
\end{split}
\eeq
and
\beq
\begin{split}
&\;\|\CR_{[\va]'}\|_{\dot{W}^{1,p}}+\|\CR_{\phi'}\|_{\dot{W}^{1,p}}\\
\leq &\;C|c_*|r(\|h''\|_{L^p}+\|H''\|_{L^p})(\|h'\|_{\dot{C}^\beta}+\|H'\|_{\dot{C}^\beta})\\
&\;+Cr^2(1+\|h''\|_{L^p}+\|H''\|_{L^p}) (\d^{\b}\|h'\|_{\dot{C}^\b}+(m_0+M_0) (1+\d R^2)^{1/2})\\
%
=:&\;\tilde{N}_{2,p},
\end{split}
\eeq
where $C = C(\mu,\nu,p,G,\b,\b')$.
\begin{proof}
The estimates immediately follow by combining \eqref{eqn: Holder estimate for remainder of derivative of jump of potential at inner interface}, \eqref{eqn: Holder estimate for remainder of derivative of outer potential},
\eqref{eqn: W 2 p estimate of remainder of jump of potential crude form} and \eqref{eqn: W 2 p estimate of remainder of potential on outer interface crude form} with Proposition \ref{prop: Holder estimate for jump of potential and potential along outer interface}, Proposition \ref{prop: W 1p estimate for jump of potential and potential along outer interface} and Lemma \ref{lem: estimate for p_*}.
\end{proof}
\end{lem}

\section{Local Existence}\label{sec: local well-posedness}
In this section, we prove existence of local solutions of \eqref{eqn: equation for inner interface new notation}-\eqref{eqn: initial condition}.

\subsection{Preliminaries}
Inspired by \eqref{eqn: simple linearized f equation} and \eqref{eqn: simple linearized F big equation}, 
we may rewrite \eqref{eqn: equation for inner interface new notation} and \eqref{eqn: equation for outer interface new notation} as
\begin{align}
\pa_t h+\f{c_*}{r}=&\;-\f{Ac_*}{r}(-\D)^{1/2}h -\f{1+A}{2r^2}\CH\CS\phi'+\f{1}{r}\CR_h, \label{eqn: backbone equation for h}\\
\pa_t H+\f{\tilde{c}_*}{R}=&\;\f{\tilde{c}_*}{R}(-\D)^{1/2}H -\f{1}{R^2}\CH\CS[\va]' +\f{1}{R}\CR_H,\label{eqn: backbone equation for H big}
\end{align}
where
\beq
\begin{split}
\CR_h := &\;-\f{1}{f}\g'(\th)\cdot \CK_{\g}\CR_{[\va]'}\\
&\;-\left(\f{1}{f}\g'(\th)\cdot \CK_{\g}(2Ac_*f'+A\CS\phi')-\f{1}{2r}\CH(2Ac_*f'+A\CS\phi')\right)\\
&\;+\left(\f{1}{f}\na (\G* g)|_{\g}\cdot \g'(\th)^\perp+c_*\right)-\left(\f{1}{f}\g'(\th)\cdot \CK_{\g,\tilde{\g}}\phi'-\f{1}{2r}\CH\CS\phi'\right),
\end{split}
\label{eqn: def of R_F inner}
\eeq
and
\beq
\begin{split}
\CR_H := &\;-\f{1}{F}\tilde{\g}'(\th)\cdot \CK_{\tilde{\g}}\CR_{\phi'}\\
&\;-\left(\f{1}{F}\tilde{\g}'(\th)\cdot \CK_{\tilde{\g}}(-2\tilde{c}_*F'+\CS[\va]')-\f{1}{2R}\CH(-2\tilde{c}_*F'+\CS[\va]')\right)\\
&\;+\left(\f{1}{F}\na (\G* g)|_{\tilde{\g}}\cdot \tilde{\g}'(\th)^\perp+\tilde{c}_*\right)-\left(\f{1}{F}\tilde{\g}'(\th)\cdot \CK_{\tilde{\g},\g}[\va]'-\f{1}{2R}\CH\CS[\va]'\right).
\end{split}
\label{eqn: def of R_F outer}
\eeq
For future use, we also denote
\beq
\tilde{\CR}_h := -\f{1+A}{2r}\CH\CS\phi'+\CR_h,\quad \tilde{\CR}_H := -\f{1}{R}\CH\CS[\va]'+\CR_H.
\label{eqn: def of remainder tilde terms in h and H equations}
\eeq

We need estimates for $\CR_h$ and $\CR_H$.

\begin{lem}\label{lem: W 1p estimate for remainder terms}
Under the assumptions of Proposition \ref{prop: W 1p estimate for jump of potential and potential along outer interface}, 
\beq
r\|\CR_h\|_{\dot{W}^{1,p}}+R\|\CR_H\|_{\dot{W}^{1,p}}\leq C\tilde{N}_{2,p},
\label{eqn: W 1p estimate of residuals} 
\eeq
where $C = C(\mu,\nu,p,G,\b,\b')$.
\begin{proof}
By \eqref{eqn: introducing remainder in va equation}, $\CR_{[\va]'}$ has zero integral on $\BT$.
By Lemma \ref{lem: W 1p difference from Hilbert transform} and Lemma \ref{lem: estimate for principle part in R_f and R_F}, 
\beq
\begin{split}
\|\g'(\th)\cdot \CK_{\g}\CR_{[\va]'}\|_{\dot{W}^{1,p}}
\leq &\;C\|\CR_{[\va]'}\|_{\dot{W}^{1,p}}+C\|h''\|_{L^p}\|\CR_{[\va]'}\|_{\dot{C}^\b}\leq C\tilde{N}_{2,p}.
\label{eqn: W 1 p estimate for transform of R varphi'}
\end{split}
\eeq
When $\g'$ and $\p$ are H\"{o}lder continuous on $\BT$ and $h$ satisfies the smallness assumption, one can rigorously show that
\beq
\g'\cdot \CK_{\g}\p = \f{d}{d\th}\left[\f{1}{2\pi}\int_{\BT}\ln|\g(\th)-\g(\th')|\p(\th')\,d\th'\right],
\label{eqn: singular integral is derivative of something}
\eeq
and thus it has mean zero on $\BT$.
Hence, by Poincar\'{e} inequality and \eqref{eqn: W 1 p estimate for transform of R varphi'},
\beq
\|f^{-1}\g'(\th)\cdot \CK_{\g}\CR_{[\va]'}\|_{\dot{W}^{1,p}}\leq Cr^{-1}\tilde{N}_{2,p}.
\eeq
Similarly,
\beq
\begin{split}
&\;\left\|\f{1}{f}\g'(\th)\cdot \CK_{\g}(2Ac_*f'+A\CS\phi')-\f{1}{2r}\CH(2Ac_*f'+A\CS\phi')\right\|_{\dot{W}^{1,p}}\\
\leq &\;\left\|\f{1}{f}\left(\g'(\th)\cdot \CK_{\g}(2Ac_*f'+A\CS\phi')-\f{1}{2}\CH(2Ac_*f'+A\CS\phi')\right)\right\|_{\dot{W}^{1,p}}\\
&\;+\left\|\left(\f{1}{2f}-\f{1}{2r}\right)\CH(2Ac_*f'+A\CS\phi')\right\|_{\dot{W}^{1,p}}\\
\leq &\;Cr^{-1}\| h''\|_{L^p}\|2Ac_*f'+A\CS\phi'\|_{\dot{C}^\b}+Cr^{-1}m_0\|2Ac_*f'+A\CS\phi'\|_{\dot{W}^{1,p}}\\
\leq &\;Cr^{-1}\tilde{N}_{2,p}.
\end{split}
\eeq
By Lemmas \ref{lem: estimates for difference between tilde p and p_*}, \ref{lem: L^inf estimate of derivative of growth potential inner} and \ref{lem: W 1p estimate of derivative of growth potential inner},
\beq
\begin{split}
&\;\|f^{-1}\na (\G* g)|_{\g}\cdot \g'(\th)^\perp+c_*\|_{\dot{W}^{1,p}}\\
\leq &\;C\|f'/f\|_{\dot{W}^{1,p}}\|e_\th\cdot \na (\G* g)|_{\g}\|_{L^\infty}+C\|f'/f\|_{L^\infty}\|e_\th\cdot \na (\G* g)|_{\g}\|_{\dot{W}^{1,p}}\\
&\;+\|e_r\cdot \na (\G* g)|_{\g}\|_{\dot{W}^{1,p}}\\
%
\leq &\;Cr\|h''\|_{L^p}( m_0 \d |\ln \d| + \|\na (\tilde{p}-p_*) \|_{L^2(B_{r})})\\
&\;+Cr(m_\b+ \|\na (\tilde{p}-p_*)\|_{L^2(B_r)})\\
\leq &\;Cr^{-1}\tilde{N}_{2,p}.
\end{split}
\eeq
Finally, by Lemmas \ref{lem: L inf norm of derivative of interaction kernel from outer to inner},
\ref{lem: W 1p norm of derivative of interaction kernel from outer to inner} and \ref{lem: estimates for the potential kernel part in interaction kernels},
\beq
\begin{split}
&\;\left\|\f{1}{f}\g'(\th)\cdot \CK_{\g,\tilde{\g}}\phi'-\f{1}{2r}\CH\CS\phi'\right\|_{\dot{W}^{1,p}}\\
\leq &\;Cr^{-1}\left\|fe_\th\cdot \CK_{\g,\tilde{\g}}\phi'-\f12\CH\CS\phi'\right\|_{\dot{W}^{1,p}}
+Cr^{-1}\|f'e_r\cdot \CK_{\g,\tilde{\g}}\phi'\|_{\dot{W}^{1,p}}\\
&\;+Cr^{-1}\|h\|_{W^{1,\infty}}\|\CH\CS\phi'\|_{\dot{W}^{1,p}}\\
\leq &\;Cr^{-1}(\|H''\|_{L^p}\|\phi'\|_{L^\infty}+(m_0+M_0)\|\phi''\|_{L^p})+Cr^{-1}\|h''\|_{L^p}\|\phi'\|_{L^\infty}\\
\leq &\;Cr^{-1}\tilde{N}_{2,p}.
\end{split}
\eeq
Combining these estimates with \eqref{eqn: def of R_F inner}, we obtain the estimate for $\CR_h$ in \eqref{eqn: W 1p estimate of residuals}.

The estimate for $\CR_H$ can be derived in a similar manner.
\end{proof}
\end{lem}

We shall also need bounds for integrals of $\CR_h$ and $\CR_H$ on $\BT$.
\begin{lem}\label{lem: estimate for the integral of R_h and R_H}
Under the assumptions of Proposition \ref{prop: W 1p estimate for jump of potential and potential along outer interface},
\beq
\begin{split}
&\;r\left|\int_\BT \CR_h\,d\th\right|+R\left|\int_\BT \CR_H\,d\th\right|\\
\leq &\;C(\|h\|_{L^\infty}+\|H\|_{L^\infty})N_{2,p}+Cr^{2} (m_0+M_0) (1+\d R^2)^{1/2},
\end{split}
\eeq
where $C = C(\mu,\nu, p, G, \b,\b')$.
\begin{proof}
We shall again use the fact that, provided $\g'$, $\tilde{\g}'$ and $\p$ to be H\"{o}lder continuous on $\BT$, 
\beq
 (\g'\cdot \CK_{\g}\p),\;(\tilde{\g}'\cdot \CK_{\tilde{\g}}\p), \;(\g'\cdot \CK_{\g,\tilde{\g}}\p),\; (\tilde{\g}'\cdot \CK_{\tilde{\g},\g}\p)\mbox{ have integrals }0\mbox{ on }\BT.
 \label{eqn: singular integrals have mean zero}
\eeq
This is because they all can be represented as $\th$-derivatives of certain quantities as in \eqref{eqn: singular integral is derivative of something}.

Applying this fact to \eqref{eqn: def of R_F inner},
\beq
\begin{split}
\int_\BT\CR_h\,d\th = &\;\int_\BT\left(\f{1}{r}-\f{1}{f}\right)(\g'(\th)\cdot \CK_{\g}(\CR_{[\va]'}+2Ac_*f'+A\CS\phi'))\,d\th\\
&\;+\int_\BT (-e_r\cdot \na (\G* g)|_{\g}+c_*)\,d\th+\int_\BT \f{f'}{f}e_\th\cdot \na (\G* g)|_{\g}\,d\th\\
&\;+\int_\BT\left(\f{1}{r}-\f{1}{f}\right)\g'(\th)\cdot \CK_{\g,\tilde{\g}}\phi'\,d\th.
\end{split}
\eeq
By \eqref{eqn: introducing remainder in va equation}, Poincar\'{e} inequality, Lemmas \ref{lem: estimates for difference between tilde p and p_*}, \ref{lem: L^inf estimate of derivative of growth potential inner}, \ref{lem: W 1p difference from Hilbert transform}, \ref{lem: L inf norm of derivative of interaction kernel from outer to inner}, \ref{lem: estimates for the potential kernel part in interaction kernels} and \ref{lem: bound error between c and c_*}, as well as Propositions \ref{prop: Holder estimate for jump of potential and potential along outer interface} and \ref{prop: W 1p estimate for jump of potential and potential along outer interface}, we derive that
\beq
\begin{split}
\left|\int_\BT\CR_h\,d\th \right|\leq  &\; Cr^{-1}\|h\|_{L^\infty}\|\g'(\th)\cdot \CK_{\g}[\va]'\|_{\dot{W}^{1,p}}\\
&\;+C\|e_r\cdot \na (\G* g)|_{\g}-c_*\|_{L^\infty}+C\|h'\|_{L^\infty}\|e_\th\cdot \na (\G* g)|_{\g}\|_{L^\infty}\\
&\;+Cr^{-1}\|h\|_{L^\infty}(\|h'\|_{L^\infty}\|fe_r\cdot \CK_{\g,\tilde{\g}}\phi'\|_{L^\infty}+\|fe_\th\cdot \CK_{\g,\tilde{\g}}\phi'\|_{L^\infty})\\
\leq  &\; Cr^{-1}\|h\|_{L^\infty}(\| h''\|_{L^p}\|[\va]'\|_{\dot{C}^\b}+\|[\va]''\|_{L^p})\\
&\;+Cr( m_0 \d |\ln \d|+ (m_0+M_0) (\d R^2)^{1/2})\\
&\;+Cr^{-1}\|h\|_{L^\infty}(\d^{-1}(\|h\|_{L^\infty}+\|H\|_{L^\infty})\|\phi'\|_{L^\infty}+\|\phi'\|_{\dot{C}^\b})\\
\leq  &\; Cr^{-1}\|h\|_{L^\infty}N_{2,p}+Cr (m_0+M_0) (1+\d R^2)^{1/2},
\end{split}
\eeq
where $C = C(\mu,\nu, p, G, \b,\b')$.

The estimate for the $\int_\BT \CR_H$ can be derived similarly.
%
\end{proof}
\end{lem}

\subsection{Proof of existence of local solutions}
Now we are ready to show existence of local solutions.

\begin{proof}[Proof of Theorem \ref{thm: local well-posedness}]
The proof is an application of the Schauder fixed-point theorem.

\setcounter{step}{0}
\begin{step}[Setup]
Let $\d$ be chosen according to \eqref{eqn: condition on delta}.
Also recall that $\alpha = 1-\f2p$, and $\varepsilon>0$ and $M$ are given in \eqref{eqn: smallness of fractional Sobolev norm of initial data}.
We assume $M\leq 1$. The exact smallness of $M$ will be specified later.

With $0<T\leq  \min\{1,\d M\}$ to be determined, we define
\beq
\begin{split}
X_{M,T}:=&\;\Big\{v\in L^p_{[0,T]}W^{2,p}\cap C_{[0,T]}C^{1,\alpha}(\BT):\;v_t \in L^p_{[0,T]}W^{1,p}(\BT),\\
&\;\qquad \qquad  v|_{t = 0} = 0,\;\|v\|_{C_{[0,T]}L^{\infty}(\BT)}\leq \d M,\\
&\;\qquad \qquad \|v\|_{L^p_{[0,T]}\dot{W}^{2,p}(\BT)}+\|v\|_{C_{[0,T]}\dot{C}^{1,\alpha}(\BT)}+\|v_t\|_{L^p_{[0,T]}\dot{W}^{1,p}(\BT)}\leq \d^{-\alpha+\e}M\Big\}.
\end{split}
\label{eqn: def of X_MT}
\eeq
$X_{M,T}$ is a non-empty, convex, closed subset of $\{v\in C_{[0,T]}C^{1,\alpha}(\BT):\;v_t \in L^p_{[0,T]}W^{1,p}(\BT)\}$.
Take $\alpha'\in (0,\alpha)$ to be determined.
Denote
\beq
Z: = L^\infty_{[0,T]}C^{1,\alpha'}(\BT).
\label{eqn: def of space Z}
\eeq
By Aubin-Lions Lemma \cite{temam1984navier}, the embedding
\beq
\{v\in C_{[0,T]}C^{1,\alpha}(\BT):\;v_t \in L^p_{[0,T]}W^{1,p}(\BT)\}\hookrightarrow Z
\eeq
is compact, so $X_{M,T}$ is compact in $Z$.
In what follows, we shall apply Schauder fixed-point theorem on
\beq
Y_{M,T}:=\left(e^{-\f{Ac_*}{r}t(-\D)^{1/2}}h_0-\f{c_*t}{r}+X_{M,T}\right)\times \left(e^{\f{\tilde{c}_*}{R}t(-\D)^{1/2}}H_0-\f{\tilde{c}_*t}{R}+X_{M,T}\right),
\eeq
which is a non-empty, convex, compact subset of $Z\times Z$.
\end{step}

\begin{step}[Estimates for elements in $Y_{M,T}$]\label{step: estimate of elements in the set}
%

Take $(h,H)\in Y_{M,T}$.
By the definition of $X_{M,T}$ and Lemma \ref{lem: estimate for p_*},
\beq
\|h\|_{C_{[0,T]}L^\infty(\BT)}
\leq
\left\|e^{-\f{Ac_*}{r}t(-\D)^{1/2}}h_0\right\|_{L^\infty(\BT)}+\f{|c_*| T}{r}+\d M\leq C(G)\d M.
\label{eqn: L inf bound for the proposed solution h}
\eeq
By the definition of the $\dot{W}^{2-\f1p,p}(\BT)$-seminorm in \eqref{eqn: def of W s p semi norm},
\beq
\|h\|_{L^p_{[0,T]}\dot{W}^{2,p}(\BT)}\leq \left(\f{r}{|Ac_*|}\right)^\f{1}{p}\|h_0\|_{\dot{W}^{2-\f1p,p}(\BT)}+\d^{-\alpha+\e}M\leq C(p,\mu, \nu, r/|c_*|)\d^{-\alpha+\e}M.
\label{eqn: W 2p bound for the proposed solution h}
\eeq
Moreover, $W^{2-\f1p,p}(\BT)\hookrightarrow h^{1,\alpha}(\BT)$ \cite[\S\,2.7]{triebel2010theory}, where $h^{1,\alpha}(\BT)$ is the closure of $C^\infty(\BT)$ in the $C^{1,\alpha}$-topology.
So
$e^{-\f{Ac_*}{r}t(-\D)^{1/2}}h_0$
is continuous in $t$ valued in $C^{1,\alpha}(\BT)$ and hence
\beq
\|h\|_{C_{[0,T]}\dot{C}^{1,\alpha}(\BT)}\leq \|h_0\|_{\dot{C}^{1,\alpha}(\BT)}+\d^{-\alpha+\e}M\leq C(p)\d^{-\alpha+\e}M.
\label{eqn: big C 1 alpha bound for the proposed solution h}
\eeq
Applying interpolation to \eqref{eqn: L inf bound for the proposed solution h} and \eqref{eqn: big C 1 alpha bound for the proposed solution h} yields
\beq
\|h\|_{C_{[0,T]}\dot{C}^{1,\b'}(\BT)}\leq C(G,p)\d^{1-\f{1+\alpha-\e}{1+\alpha}(1+\b')}M.
\eeq
Hence, taking
\beq
\b' = \f{\e}{1+\alpha-\e},
\label{eqn: def of beta'}
\eeq
we find that
\beq
\|h\|_{C_{[0,T]}\dot{C}^{1,\b'}(\BT)}\leq C(G,p)M.
\label{eqn: C 1 beta norm of the proposed solution h}
\eeq

Similarly,
\begin{align}
\|H\|_{C_{[0,T]}L^\infty(\BT)}
\leq&\; C(G)\d M,\label{eqn: L inf bound for the proposed solution H big}\\
\|H\|_{L^p_{[0,T]}\dot{W}^{2,p}(\BT)}\leq &\; C(p, R/|\tilde{c}_*|)\d^{-\alpha+\e}M,
\label{eqn: W 2p bound for the proposed solution H big}\\
\|H\|_{C_{[0,T]}\dot{C}^{1,\alpha}(\BT)}\leq &\; C(p)\d^{-\alpha+\e}M,
\label{eqn: big C 1 alpha bound for the proposed solution H big}
\end{align}
and, with the same $\b'$ as above, 
\beq
\|H\|_{C_{[0,T]}\dot{C}^{1,\b'}(\BT)}\leq C(G,p)M.
\label{eqn: C 1 beta norm of the proposed solution H big}
\eeq
In what follows, we shall assume $M$ to be suitably small, which depends on $p$ and $G$, so that \eqref{eqn: L inf bound for the proposed solution h}, \eqref{eqn: big C 1 alpha bound for the proposed solution h},
\eqref{eqn: C 1 beta norm of the proposed solution h}, \eqref{eqn: L inf bound for the proposed solution H big}, \eqref{eqn: big C 1 alpha bound for the proposed solution H big}
and \eqref{eqn: C 1 beta norm of the proposed solution H big} implies that for $(h,H)\in Y_{M,T}$,
\beq
\sup_{t\in [0,T]}(m_{1,\alpha}+M_{1,\alpha}+\|h'\|_{\dot{C}^{\b'}}+\|H'\|_{\dot{C}^{\b'}})\leq C(G,p)M\ll 1.
\label{eqn: L inf in time bound of Holder norm for proposed solution}
\eeq

\end{step}

\begin{step}[Construction of a map on $Y_{M,T}$]
Inspired by \eqref{eqn: backbone equation for h} and \eqref{eqn: backbone equation for H big}, for given $(h,H)\in Y_{M,T}$, we let $(h_\dag, H_\dag)$ solve
\begin{align}
\pa_t h_\dag=&\;-\f{Ac_*}{r}(-\D)^{1/2}h_\dag+ \f{1}{r}\tilde{\CR}_h, \quad h_\dag|_{t = 0} = 0,\label{eqn: equation for h dag}\\
\pa_t H_\dag=&\;\f{\tilde{c}_*}{R}(-\D)^{1/2}H_\dag +\f{1}{R}\tilde{\CR}_H,\quad H_\dag|_{t = 0} = 0.\label{eqn: equation for H big dag}
\end{align}
Recall that $\tilde{\CR}_h$ and $\tilde{\CR}_H$ are defined in \eqref{eqn: def of remainder tilde terms in h and H equations}, which are uniquely determined by $(h,H)$ via \eqref{eqn: equation for derivative of jump of potential final form} (c.f., Proposition \ref{prop: W 1p estimate for jump of potential and potential along outer interface}), \eqref{eqn: equation for derivative of potential on outer interface}, \eqref{eqn: def of R_F inner} and \eqref{eqn: def of R_F outer}.
Then let
%
\beq
(\tilde{h}, \tilde{H}) = \left(e^{-\f{Ac_*}{r}t(-\D)^{1/2}}h_0-\f{c_*t}{r}+h_\dag, \; e^{\f{\tilde{c}_*}{R}t(-\D)^{1/2}}H_0-\f{\tilde{c}_*t}{R}+H_\dag\right).
\label{eqn: decomposition of solution}
\eeq
A fixed-point of the map $\CT:\,(h,H)\mapsto (\tilde{h}, \tilde{H})$ is then a solution of \eqref{eqn: backbone equation for h} and \eqref{eqn: backbone equation for H big}.

We shall show that $\CT$ is continuous from $Y_{M,T}$ to itself in the topology of $Z\times Z$ and then apply Schauder fixed-point theorem.
It suffices to prove that:
\begin{itemize}
  \item the map $\CT':\,(h,H)\mapsto (h_\dag, H_\dag)$ is well-defined as a continuous function on $Y_{M,T}$ in the topology of $Z\times Z$, and
  \item $(h_\dag,H_\dag)\in X_{M,T}\times X_{M,T}$ for properly chosen $M$ and $T$.
\end{itemize}

\end{step}

\begin{step}[Continuity of $\CT'$]
\label{step: continuity of the map}
We choose $\alpha'<\alpha''<\min\{\f14,\alpha\}$.
By \eqref{eqn: backbone equation for h} and \eqref{eqn: backbone equation for H big},
\beq
(\tilde{\CR}_h, \tilde{\CR}_H) = (r\pa_t h+c_*+Ac_*(-\D)^{1/2}h, R\pa_t H+\tilde{c}_*-\tilde{c}_*(-\D)^{1/2}H).
\eeq
%
By \eqref{eqn: L inf in time bound of Holder norm for proposed solution} and Lemma \ref{lem: stability of the interface velocities}, provided that $M \ll 1$ which depends on $p$, $G$ and $\alpha''$, for any pair $(h_1,H_1),(h_2,H_2)\in Y_{M,T}$,
\beq
\begin{split}
&\;\|\tilde{\CR}_{h_1}-\tilde{\CR}_{h_2}\|_{L^\infty_{[0,T]}C^{\alpha''}(\BT)} +\|\tilde{\CR}_{H_1}-\tilde{\CR}_{H_2}\|_{L^\infty_{[0,T]}C^{\alpha''}(\BT)}\\
\leq &\;C(\alpha'',\mu,\nu,r,R,G)\cdot d_{\alpha''}((h_1,H_1),(h_2,H_2)), 
\end{split}
\label{eqn: Lipschitz continuity of remainder term}
\eeq
where
\beq
d_{\alpha''}((h_1,H_1),(h_2,H_2)) : =\|h_1-h_2\|_{L^\infty_{[0,T]}C^{1,\alpha''}(\BT)}+\|H_1-H_2\|_{L^\infty_{[0,T]}C^{1,\alpha''}(\BT)}.
\label{eqn: def of metric of Z times Z}
\eeq
We abbreviate it as $d_{\alpha''}$ if it incurs no confusion.
By taking $h_2 = H_2 = 0$ in \eqref{eqn: Lipschitz continuity of remainder term} which corresponds to $\tilde{\CR}_{h_2} = \tilde{\CR}_{H_2} =0$, we show that $\tilde{\CR}_{h_1}, \tilde{\CR}_{H_1} \in L^\infty_{[0,T]}C^{\alpha''}(\BT)$; so are $\tilde{\CR}_{h_2}$ and $\tilde{\CR}_{H_2}$.
Following a similar argument, we may apply Lemma \ref{lem: stability of the interface velocities} to different time slices of $h_i$ and $H_i$, and use the time continuity $h_i, H_i \in C_{[0,T]}C^{1,\alpha''}(\BT)$ to prove  $\tilde{\CR}_{h_i}, \tilde{\CR}_{H_i} \in C_{[0,T]}C^{\alpha''}(\BT)$.

Let $(h_{i,\dag},H_{i,\dag})$ $(i = 1,2)$ be the unique solution of \eqref{eqn: equation for h dag} and \eqref{eqn: equation for H big dag} in $Z\times Z$ corresponding to $(h_i,H_i)\in Y_{M,T}$.
By Lemma \ref{lem: Holder bound for the fractional heat equation} and \eqref{eqn: Lipschitz continuity of remainder term},
\beq
\|h_{1,\dag}-h_{2,\dag}\|_{C_{[0,T]}\dot{C}^{1,\alpha'}(\BT)}
\leq  C(\alpha',\alpha'', \mu,\nu,r, R, G)\cdot d_{\alpha''}.
%
\label{eqn: bound C 1 alpha semi norm of difference of two new solutions}
\eeq
%
On the other hand, let 
$\bar{h}_{i,\dag} = \f{1}{2\pi}\int_\BT h_{i,\dag}\,d\th$.
By \eqref{eqn: equation for h dag} and \eqref{eqn: Lipschitz continuity of remainder term},
\beq
\| \bar{h}_{1,\dag}-\bar{h}_{2,\dag}\|_{C_{[0,T]}L^\infty(\BT)}\leq Cr^{-1}\|\tilde{\CR}_{h_1}-\tilde{\CR}_{h_2}\|_{C_{[0,T]}C^{\alpha''}(\BT)}\leq C(\alpha'',\mu,\nu,r,R,G)\cdot d_{\alpha''}.
\eeq
Combining this with \eqref{eqn: bound C 1 alpha semi norm of difference of two new solutions}, we use interpolation as well as \eqref{eqn: L inf bound for the proposed solution h}, \eqref{eqn: big C 1 alpha bound for the proposed solution h}, \eqref{eqn: L inf bound for the proposed solution H big} and \eqref{eqn: big C 1 alpha bound for the proposed solution H big} to derive that
\beq
\begin{split}
\|h_{1,\dag}-h_{2,\dag}\|_{C_{[0,T]}C^{1,\alpha'}(\BT)}
\leq &\; C(\alpha',\alpha'',\mu,\nu,r,R,G)\cdot d_{\alpha'}^{\th}d_{\alpha}^{1-\th}\\
\leq &\; C(\alpha',\alpha'',p,\mu,\nu,r,R,G)\cdot d_{\alpha'}^{\th},
\end{split}
\eeq
where $\th = \f{\alpha-\alpha''}{\alpha-\alpha'}$.
Similarly, $\|H_{1,\dag}-H_{2,\dag}\|_{C_{[0,T]}C^{1,\alpha'}(\BT)}$ enjoys the same bound.
This proves (H\"{o}lder) continuity of $\CT'$ in $Y_{M,T}$ in the topology of $Z\times Z$.

In fact, if one improves Lemma \ref{lem: Holder bound for the fractional heat equation}, it can be shown that $\CT'$ is log-Lipschitz continuous in $Y_{M,T}$ in the topology of $Z\times Z$.
We omit the details although it may be of independent interest.

\end{step}

\begin{step}[Justification of $(h_\dag,H_\dag)\in X_{M,T}\times X_{M,T}$]
\label{step: justification image is still in the set}
Let $\b'$ be taken as before, and let $\b=\f{\b'}{4}<\f{\b'}{1+\b'}$.
It is not difficult to show that
\beq
\|\CH\CS\p'\|_{\dot{W}^{1,p}}\leq C\|\CS\p'\|_{\dot{W}^{1,p}}\leq C\d^{\b-1+\f1p}\|\p'\|_{\dot{C}^\b}.
\eeq
Combining with Lemma \ref{lem: W 1p estimate for remainder terms},
\beq
\|\tilde{\CR}_h\|_{\dot{W}^{1,p}}+\|\tilde{\CR}_H\|_{\dot{W}^{1,p}}\leq Cr^{-1}(\tilde{N}_{2,p}+\d^{\b-1+\f1p}N_{1,\b}),
\label{eqn: spatial estimate for the source terms in the fixed point argument}
\eeq
Then we derive by Lemma \ref{lem: estimate for p_*}, Proposition \ref{prop: Holder estimate for jump of potential and potential along outer interface}, Lemma \ref{lem: estimate for principle part in R_f and R_F}, \eqref{eqn: W 2p bound for the proposed solution h}, \eqref{eqn: W 2p bound for the proposed solution H big} and \eqref{eqn: L inf in time bound of Holder norm for proposed solution} that
\beq
\begin{split}
&\;\|r^{-1}\tilde{\CR}_{h}\|_{L^p_{[0,T]}\dot{W}^{1,p}}+\|R^{-1}\tilde{\CR}_{H}\|_{L^p_{[0,T]}\dot{W}^{1,p}}\\
\leq
&\;C|c_*|r^{-1}(\|h''\|_{L^p_{[0,T]}L^p}+\|H''\|_{L^p_{[0,T]}L^p})\sup_{t\in [0,T]}(\|h'\|_{\dot{C}^\beta}+\|H'\|_{\dot{C}^\beta})\\
&\;+C(T^{1/p}+\|h''\|_{L^p_{[0,T]}L^p}+\|H''\|_{L^p_{[0,T]}L^p}) \sup_{t\in [0,T]}(\d^{\b}\|h'\|_{\dot{C}^\b}+(m_0+M_0) (1+\d R^2)^{1/2})\\
&\;+C\d^{\b-1+\f1p}T^{1/p}|c_*|r^{-1}\sup_{t\in [0,T]}(\|h'\|_{\dot{C}^\beta}+\|H'\|_{\dot{C}^\beta})\\
&\;+C\d^{\b-1+\f1p}T^{1/p} \sup_{t\in [0,T]}(\d^{\b}\|h'\|_{\dot{C}^\b}+(m_0+M_0) (1+\d R^2)^{1/2})\\
\leq
&\;C \d^{-\alpha+\e}M^2(1+\d R^2)^{1/2}+C\d^{\b-1+\f1p}T^{1/p} M(1+\d R^2)^{1/2},
\end{split}
\label{eqn: derivation of space-time estimate for the source terms in the fixed point argument}
\eeq
where $C = C(p,\e,\mu,\nu,R/|\tilde{c}_*|, G)$.
Here we rewrote the $\b$- and $\b'$-dependence into dependence on $p$ and $\e$, and used the fact that $r/|c_*|\leq R/|\tilde{c}_*|$.
In particular, $C$ does not deteriorate as $\d$ becomes smaller.
Hence, 
\beq
\|r^{-1}\tilde{\CR}_{h}\|_{L^p_{[0,T]}\dot{W}^{1,p}}+\|R^{-1}\tilde{\CR}_{H}\|_{L^p_{[0,T]}\dot{W}^{1,p}}
\leq
C (\d^{\b-1+\f1p}T^{1/p}+\d^{-\alpha+\e}M) M,
\label{eqn: space-time estimate for the source terms in the fixed point argument}
\eeq
where $C = C(p,\e,\mu,\nu,R/|\tilde{c}_*|,G, \d R^2)$. 

To this end, applying Lemma \ref{lem: L p regularity theory of fractional heat equation} and Lemma \ref{lem: spatial Holder estimate for solution with Sobolev source term} to \eqref{eqn: equation for h dag} and \eqref{eqn: equation for H big dag}, we obtain that
\beq
\begin{split}
&\;\|h_\dag\|_{L^p_{[0,T]}\dot{W}^{2,p}}+\|\pa_t h_\dag\|_{L^p_{[0,T]}\dot{W}^{1,p}}+\|h_\dag\|_{C_{[0,T]}\dot{C}^{1,\alpha}}\\
&\;+\|H_\dag\|_{L^p_{[0,T]}\dot{W}^{2,p}}+\|\pa_t H_\dag\|_{L^p_{[0,T]}\dot{W}^{1,p}}+\|H_\dag\|_{C_{[0,T]}\dot{C}^{1,\alpha}}\\
\leq &\;C (\d^{\b-1+\f1p}T^{1/p}+\d^{-\alpha+\e}M) M,
\end{split}
\label{eqn: sobolev bound for new solution}
\eeq
%
Here the universal constant $C$ has the same dependence as above.
Now we take
\begin{align}
M\leq &\; M_*(p,\e,\mu,\nu,R/|\tilde{c}_*|,G, \d R^2 
)\ll 1,\label{eqn: def of M_*}\\
T\leq &\;T_*(\d, p,\e,\mu,\nu,R/|\tilde{c}_*|,G, \d R^2 
)\ll 1,\label{eqn: upper bound for T}
\end{align}
so that \eqref{eqn: sobolev bound for new solution} becomes
\beq
\begin{split}
&\;\|h_\dag\|_{L^p_{[0,T]}\dot{W}^{2,p}}+\|\pa_t h_\dag\|_{L^p_{[0,T]}\dot{W}^{1,p}}+\|h_\dag\|_{C_{[0,T]}\dot{C}^{1,\alpha}}\\
&\; +\|H_\dag\|_{L^p_{[0,T]}\dot{W}^{2,p}}+\|\pa_t H_\dag\|_{L^p_{[0,T]}\dot{W}^{1,p}}+\|H_\dag\|_{C_{[0,T]}\dot{C}^{1,\alpha}}\\
\leq &\;\d^{-\alpha+\e}M.
\end{split}
\label{eqn: bound for the new solutions final}
\eeq
Note that the smallness needed for $M$ will not be more stringent as $\d$ becomes smaller.

Finally, we show $(h_\dag,H_\dag)$ satisfies the $C_{[0,T]}L^\infty(\BT)$-bound in the definition \eqref{eqn: def of X_MT} of $X_{M,T}$.
By Lemma \ref{lem: estimate for the integral of R_h and R_H}, Sobolev inequality and \eqref{eqn: spatial estimate for the source terms in the fixed point argument},
\beq
\begin{split}
&\;\|r^{-1}\tilde{\CR}_{h}\|_{L^\infty}+\|R^{-1}\tilde{\CR}_{H}\|_{L^\infty}\\
\leq
&\;Cr^{-2}(\|h\|_{L^\infty}+\|H\|_{L^\infty})N_{2,p}+C (m_0+M_0) (1+\d R^2)^{1/2}\\
&\;+Cr^{-2}(\tilde{N}_{2,p}+\d^{\b-1+\f1p}N_{1,\b})\\
\leq
&\;Cr^{-2}(\tilde{N}_{2,p}+\d^{\b-1+\f1p}N_{1,\b}).
\end{split}
\eeq
Following \eqref{eqn: derivation of space-time estimate for the source terms in the fixed point argument} and \eqref{eqn: space-time estimate for the source terms in the fixed point argument},
\beq
\|r^{-1}\tilde{\CR}_{h}\|_{L^1_{[0,T]}L^\infty} +\|R^{-1}\tilde{\CR}_{H}\|_{L^1_{[0,T]}L^\infty}
\leq
C T^{1-\f1p}(\d^{\b-1+\f1p}T^{1/p}+\d^{-\alpha+\e}M) M.
\eeq
Combining this with \eqref{eqn: equation for h dag} and \eqref{eqn: equation for H big dag}, we use the fact $\|e^{-t(-\D)^{1/2}}\|_{L^\infty\to L^\infty}\leq 1$ to obtain that
%
\beq
\|h_\dag\|_{C_{[0,T]}L^\infty(\BT)}+\|H_\dag\|_{C_{[0,T]}L^\infty(\BT)}\leq C T^{1-\f1p}(\d^{\b-1+\f1p}T^{1/p}+\d^{-\alpha+\e}M) M.
\label{eqn: L inf bound for part of local solution}
\eeq
where $C = C(p,\e,\mu,\nu,R/|\tilde{c}_*|,G, \d R^2 
)$.
Take $T_*$ in \eqref{eqn: upper bound for T} even smaller if necessary, 
so that the required $C_{[0,T]}L^\infty(\BT)$-bound for $(h_\dag,H_\dag)$ in \eqref{eqn: def of X_MT} is achieved.

This shows that $\CT'$ has its image $(h_\dag,H_\dag)$ in $X_{M,T}\times X_{M,T}$.

\end{step}

\begin{step}[Existence and estimates]
By Schauder fixed-point theorem, the map $\CT$ has a fixed-point $(h,H)\in Y_{M,T}$,
which is a mild solution of \eqref{eqn: backbone equation for h} and \eqref{eqn: backbone equation for H big}.
Moreover, the pointwise well-definedness of $\pa_t h$ and $\pa_t H$ has been readily shown in Step \ref{step: continuity of the map}, as they are at least in $C_{[0,T]}C^{\alpha''}(\BT)$, where $\alpha''<\min\{\f14,\alpha\}$ is arbitrary.
Therefore, $(h,H)$ is a strong solution of \eqref{eqn: backbone equation for h} and \eqref{eqn: backbone equation for H big}.

Estimates for $h$ and $H$ follow from \eqref{eqn: L inf bound for the proposed solution h}-\eqref{eqn: big C 1 alpha bound for the proposed solution h} and \eqref{eqn: L inf bound for the proposed solution H big}-\eqref{eqn: big C 1 alpha bound for the proposed solution H big}.
For $\pa_t h$ and $\pa_t H$, we derive by \eqref{eqn: decomposition of solution}, \eqref{eqn: bound for the new solutions final} and the definition of $W^{2-\f1p,p}(\BT)$-space \eqref{eqn: def of W s p semi norm},
\beq
\begin{split}
\|\pa_t h\|_{L^p_{[0,T]}\dot{W}^{1,p}}
\leq &\;\|\pa_t e^{-\f{Ac_*}{r}t(-\D)^{1/2}}h_0\|_{L^p_{[0,T]}\dot{W}^{1,p}}+\|\pa_t h_\dag\|_{L^p_{[0,T]}\dot{W}^{1,p}}\\
\leq &\; C(\mu, \nu,p, G)\|h_0\|_{\dot{W}^{2-\f1p,p}}+\d^{-\alpha+\e}M,
\end{split}
\eeq
and similarly,
\beq
\|\pa_t H\|_{L^p_{[0,T]}\dot{W}^{1,p}}
\leq  C(\mu, \nu,p, G)\|H_0\|_{\dot{W}^{2-\f1p,p}}+\d^{-\alpha+\e}M.
\eeq
\end{step}
\end{proof}

\subsection{Continuation of the local solutions}
A local solution can be extended to longer time intervals as long as $f(T)$ and $F(T)$ still satisfy the smallness assumption \eqref{eqn: smallness of fractional Sobolev norm of initial data} on the initial data.
We start with the following lemma that links estimates for $f(T)$ and $F(T)$ when they are treated as new initial datum, with the estimates for $f_0$ and $F_0$.

\begin{lem}\label{lem: growth of local solution after one stage}
Under the assumptions of Theorem \ref{thm: local well-posedness} with $M_*$ suitably small, let
$f$ and $F$ be a local solution over $[0,T]$.
Define $f_1(\th) = f(\th,T)$ and $F_1(\th) = F(\th,T)$.
Let
\beq
r_{1} := \f{1}{2\pi}\int_\BT f_1(\th)\,d\th,\quad R_{1} := \f{1}{2\pi}\int_\BT F_1(\th)\,d\th,
\eeq
and according to \eqref{eqn: def of h and H},
\beq
h_{1}(\th) := \f{f_1}{r_1}-1,\quad H_{1}(\th) := \f{F_1}{R_1}-1.
\eeq
Let
\beq
\d_1 = \f{1-\f{r_1}{R_1}}{1-\f{r}{R}}\cdot \d.
\eeq
Then $r_1$, $R_1$ and $\d_1$ satisfy \eqref{eqn: condition on delta}.
Moreover, with some universal constant $\tilde{C} = \tilde{C}(p,\e,G)$,
\beq
\d_1^{-1}(\|h_1\|_{L^\infty(\BT)}+\|H_1\|_{L^\infty(\BT)}) +\d_1^{\alpha-\e}\left(\|h_1\|_{\dot{W}^{2-\f1p,p}(\BT)}+\|H_1\|_{\dot{W}^{2-\f1p,p}(\BT)}\right)
\leq \tilde{C}(p,\e,G)M,
\label{eqn: estimate for new initial data}
\eeq
where $M$ is defined in \eqref{eqn: smallness of fractional Sobolev norm of initial data} with $h_0$, $H_0$ and $\d$.

\begin{proof}
That $r_1$, $R_1$ and $\d_1$ satisfy \eqref{eqn: condition on delta} is obvious since $r$, $R$ and $\d$ satisfy \eqref{eqn: condition on delta}.

To show \eqref{eqn: estimate for new initial data}, we first study $h(T)$ and $H(T)$.
Note that \eqref{eqn: bounds for the local solution} readily provides 
\beq
\d^{-1}(\|h(T)\|_{L^\infty}+\|H(T)\|_{L^\infty})\leq
C(p,G) M.
\label{eqn: L^inf bound at the terminal time}
\eeq
A bound for $W^{2-\f1p,p}$-seminorm of $h(T)$ and $H(T)$ may be derived as follows.
Denote $h_* = h-e^{-\f{Ac_*}{r}t(-\D)^{1/2}}h_0$ and $H_* = H-e^{\f{\tilde{c}_*}{R}t(-\D)^{1/2}}H_0$.
By \eqref{eqn: decomposition of solution} and \eqref{eqn: bound for the new solutions final}, they satisfy
\beq
\pa_{x} h_*, \pa_{x} H_* \in  W^{1,p}([0,T]\times \BT)
\eeq
and
\beq
h_*|_{t = 0} = H_*|_{t = 0} = \pa_x h_*|_{t = 0} = \pa_x H_*|_{t = 0} = 0.
\eeq
We make zero extension of $h_*$ and $H_*$ to the region $t<0$ while still denote the extension to be $h_*$ and $H_*$.
Then the above properties imply that $\pa_x h_*,\pa_x H_*\in W^{1,p}((-\infty,T]\times \BT)$.
By trace theorem (see e.g., \cite[\S 2.7.2]{triebel2010theory}) and \eqref{eqn: bound for the new solutions final},
\beq
\begin{split}
&\;\|\pa_x h_*(T)\|_{\dot{W}^{1-\f1p,p}(\BT)}+\|\pa_x H_*(T)\|_{\dot{W}^{1-\f1p,p}(\BT)}\\
\leq &\;C\left(\|\pa_x h_*\|_{\dot{W}^{1,p}((-\infty,T]\times \BT)}
+\|\pa_x H_*\|_{\dot{W}^{1,p}((-\infty,T]\times \BT)}\right)\\
\leq &\;C\left(\|h_*\|_{L^p_{[0,T]}\dot{W}^{2,p}(\BT)}
+\|H_*\|_{L^p_{[0,T]}\dot{W}^{2,p}(\BT)}\right)\\
&\;+C\left(\|\pa_t h_*\|_{L^p_{[0,T]}\dot{W}^{1,p}(\BT)}
+\|\pa_t H_*\|_{L^p_{[0,T]}\dot{W}^{1,p}(\BT)}\right)\\
\leq &\;C\d^{-\alpha+\e}M.
\end{split}
\label{eqn: trace bound for inhomogeneous part of the local solution}
\eeq
It is noteworthy that the constants $C$ may only depend on $p$ but not on $T$.
On the other hand, by the definition \eqref{eqn: def of W s p semi norm} of the $W^{2-\f1p,p}(\BT)$-seminorm,
\beq
\begin{split}
&\;\|e^{-\f{Ac_*}{r}t(-\D)^{1/2}}h_0(T)\|_{\dot{W}^{2-\f1p,p}(\BT)}+\|e^{-\f{\tilde{c}_*}{R}t(-\D)^{1/2}}H_0(T)\|_{\dot{W}^{2-\f1p,p}(\BT)}\\
\leq &\;\|h_0\|_{\dot{W}^{2-\f1p,p}(\BT)}+\|H_0\|_{\dot{W}^{2-\f1p,p}(\BT)}.
\end{split}
\eeq
Combining this with \eqref{eqn: smallness of fractional Sobolev norm of initial data} and \eqref{eqn: trace bound for inhomogeneous part of the local solution}, we conclude that
\beq
\|h(T)\|_{\dot{W}^{2-\f1p,p}(\BT)}+\|H(T)\|_{\dot{W}^{2-\f1p,p}(\BT)}\leq C(p)\d^{-\alpha+\e}M.
\label{eqn: fractional Sobolev bound at the terminal time}
\eeq

Thanks to \eqref{eqn: L^inf bound at the terminal time} and the way $r_1$ and $R_1$ are defined
\beq
\left|\f{r_1}{r}-1\right|+\left|\f{R_1}{R}-1\right|\leq C(p,G)\d M.
\eeq
Assume $M_*$ is already small enough, depending on $p$ and $G$, to guarantee that the right hand side of the above inequality is sufficiently small and that 
\beq
c_1\d \leq \d_1 \leq c_2\d
\eeq
for some universal $0<c_1<1<c_2$.
Hence,
\beq
\d_1^{-1}(\|h_1\|_{L^\infty(\BT)}+\|H_1\|_{L^\infty(\BT)})\leq C\d^{-1}(\|h(T)\|_{L^\infty(\BT)}+\|H(T)\|_{L^\infty(\BT)})+C(G)M,
\eeq
and
\beq
\begin{split}
&\;\d_1^{\alpha-\e}\left(\|h_1\|_{\dot{W}^{2-\f1p,p}(\BT)}+\|H_1\|_{\dot{W}^{2-\f1p,p}(\BT)}\right)\\
\leq &\; C(p,\e)\d^{\alpha-\e}\left(\|h(T)\|_{\dot{W}^{2-\f1p,p}(\BT)}+\|H(T)\|_{\dot{W}^{2-\f1p,p}(\BT)}\right).
\end{split}
\eeq
They combined with \eqref{eqn: L^inf bound at the terminal time} and \eqref{eqn: fractional Sobolev bound at the terminal time} imply \eqref{eqn: estimate for new initial data}. 
%
%
%
\end{proof}
\end{lem}

\begin{proof}[Proof of Corollary \ref{cor: long time small solution}]
We would like to construct a local solution over $[0,\tilde{T}]$ by making successive continuations.

\setcounter{step}{0}
\begin{step}[Setup]
We can always start with $f_0$ and $F_0$ satisfying the smallness condition of Theorem \ref{thm: local well-posedness}.
To make the notations more systematic, we rewrite $r$, $R$ and $\d$ in Theorem \ref{thm: local well-posedness} as $r_0$, $R_0$ and $\d_0$, respectively.
Let $h_0$ and $H_0$ be defined as in \eqref{eqn: def of h and H}.
Since
\beq
M^0:= \d_0^{-1}(\|h_0\|_{L^\infty(\BT)}+\|H_0\|_{L^\infty(\BT)})
+
\d_0^{\alpha - \e}\left(\|h_0\|_{\dot{W}^{2-\f1p,p}(\BT)}+\|H_0\|_{\dot{W}^{2-\f1p,p}(\BT)}\right)\leq M_{*,0},
\label{eqn: smallness condition for initial data}
\eeq
where according to \eqref{eqn: def of M_*},
\beq
M_{*,0} := M_*(p,\e,\mu,\nu,R_0/|\tilde{c}_*(r_0,R_0)|,G, \d_0 R_0^2 ),
\eeq
by Theorem \ref{thm: local well-posedness}, there exists a solution $(f^0, F^0)$ on $[0,t_0]$, where by \eqref{eqn: upper bound for T},
\beq
t_0 \leq T_*(\d_0, p,\e,\mu,\nu,R_0/|\tilde{c}_*(r_0,R_0)|,G, \d_0 R_0^2).
\eeq
Define $T_0 = t_0$.

Suppose we have obtained a solution on $[0,T_{k-1}]$ for some $k\in \mathbb{Z}_+$.
We define
\beq
f_{k} = f(T_{k-1}),\quad F_{k}(t= 0) = F(T_{k-1}),
\eeq
\beq
r_k = \f{1}{2\pi} \int_{\BT} f_{k}(\th)\,d\th,\quad
R_k = \f{1}{2\pi} \int_{\BT} F_{k}(\th)\,d\th.
\eeq
Also let
\beq
\d_{k} = \f{1-\f{r_k}{R_k}}{1-\f{r_{k-1}}{R_{k-1}}}\cdot \d_{k-1}.
\eeq
With this choice, $r_k$, $R_k$ and $\d_k$ satisfy \eqref{eqn: condition on delta}.
Let $h_k$ and $H_k$ be defined by $f_k$, $F_k$, $r_k$ and $R_k$ as in \eqref{eqn: def of h and H}.
Then if
\beq
M^k:=\d_k^{-1}(\|h_k\|_{L^\infty(\BT)}+\|H_k\|_{L^\infty(\BT)}) +\d_k^{\alpha-\e}\left(\|h_k\|_{\dot{W}^{2-\f1p,p}(\BT)}+\|H_k\|_{\dot{W}^{2-\f1p,p}(\BT)}\right)\leq M_{*,k},
\label{eqn: smallness for the k-th stage}
\eeq
where
\beq
M_{*,k} := M_*(p,\e,\mu,\nu,R_k/|\tilde{c}_*(r_k,R_k)|,G, \d_k R_k^2),
\label{eqn: smallness threshold for the k-th stage}
\eeq
Theorem \ref{thm: local well-posedness} claims that
there exists a solution $(f^k, F^k)$ on $[0,t_k]$, where by \eqref{eqn: upper bound for T},
\beq
t_k\leq T_*(\d_k, p,\e,\mu,\nu,R_k/|\tilde{c}_*(r_k, R_k)|,G, \d_k R_k^2).
\eeq
To this end, we let $T_k = T_{k-1}+t_k$, and define $f(t) = f^k(t-T_{k-1})$ and $F(t) = F^k(t-T_{k-1})$ for $t\in [T_{k-1},T_k]$.
Then it is easy to verify that $(f,F)$ is a local strong solution on $[0,T_k]$.

Starting from the initial data, if we are able to make such continuation until $T_K \geq \tilde{T}$ for some finite $K$, then we prove the existence of a strong solution on $[0,\tilde{T}]$.
Otherwise,
\begin{enumerate}
  \item \label{case: bad case blow up before T} either \eqref{eqn: smallness for the k-th stage} is first violated for some finite $K_*$ (depending on the initial data) with $T_{K_*}<\tilde{T}$;
  \item \label{case: bad case infinite continuations} or we are able to make continuation for infinitely many times but still can not reach $\tilde{T}$.
      This implies that for all $k\in \mathbb{N}$, $T_k<\tilde{T}$ and \eqref{eqn: smallness for the k-th stage} holds, while
      \beq
      \lim_{k\to \infty} T_*(\d_k, p,\e,\mu,\nu,R_k/|\tilde{c}_*(r_k, R_k)|,G, \d_k R_k^2)= 0.
      \label{eqn: infinite continuation}
      \eeq
\end{enumerate}
We are going to show that both of them would not occur if
we take initial datum $h_0$ and $H_0$ to be sufficiently small.

\end{step}

\begin{step}[A priori estimates for configurations staying almost circular]
Consider an arbitrary $k$ such that $T_k< \tilde{T}$ and \eqref{eqn: smallness for the k-th stage} holds for all numbers from $0$ to $k$.
We shall first derive upper and lower bounds for $r_k$ and $R_k$.

Since \eqref{eqn: smallness for the k-th stage} holds, in which $M_*$ is sufficiently small, the inner and outer interfaces at times $T_{-1},\cdots, T_{k-1}$ are all sufficiently close to circles (we use the convention $T_{-1} = 0$).
In this case, we must have $r_k<R_k$ as the interfaces can not cross by the proof of Theorem \ref{thm: local well-posedness}.
Moreover, with some universal constants $c$ and $C$,
\beq
c|\Om_{T_{k-1}}|^{1/2}\leq r_k < R_k \leq C|\tilde{\Om}_{T_{k-1}}|^{1/2}.
\label{eqn: bounding r_k and R_k}
\eeq
The increment of $|\tilde{\Omega}|$ is due to the growth of the tumor, which provides a naive bound for $|\tilde{\Om}|$
\beq
\frac{d}{dt}|\tilde{\Omega}|\leq G(0)|\tilde{\Omega}|.
\eeq
Therefore, for all such $k$, $r_k$ and $R_k$ admit an upper bound that only depends on $G$, $|\tilde{\Om}_0|$ and $\tilde{T}$.
Since the initial data is assumed to satisfy the smallness condition  \eqref{eqn: smallness condition for initial data}, $|\tilde{\Om}_0|$ is comparable with $R_0^2$ up to universal constants.
Hence, the $|\tilde{\Om}_0|$-dependence can be rewritten as $R_0$-dependence.
We note that Lemma \ref{lem: estimate for p_*} may provide a better upper bound that depends linearly on $T$, but the naive bound here is enough for this qualitative discussion.
On the other hand, because of the growth of the tumor, $|\Om_{T_{k-1}}|\geq |\Om_0|$.
This gives a positive lower bound for $r_k$ and $R_k$ that only depends on $|\Om_0|$, and thus only on $r_0$ by the same reasoning as above.

To this end, we note that $R/|\tilde{c}_*(r,R)|$ is a continuous function in $r,R\in \BR_+$.
The continuity can be justified using Lemma \ref{lem: stability of the interface velocities} with $h_1= H_1 = 0$ and $h_2$ and $H_2$ being small constants.
Indeed, $|\tilde{c}_*(r,R)|$ is the speed of the outer interface when the interfaces are concentric circles with radii $r$ and $R$, respectively.
Therefore, for all such $k$, $R_k/|\tilde{c}_*(r_k,R_k)|$ admits positive lower and upper bounds depending only on $\mu$, $\nu$, $G$, $r_0$, $R_0$ and $\tilde{T}$.

By Remark \ref{rmk: delta R^2 is bounded}, $\d_k R_k^2$ has lower and upper bounds that only depend on $|\tilde{\Om}_0\backslash \Om_0|$.
This together with the bound for $R_k$ implies that $\d_k$ has positive lower and upper bounds only depending on $G$, $r_0$, $R_0$ and $\tilde{T}$.

By the proof of Theorem \ref{thm: local well-posedness} (c.f., \eqref{eqn: upper bound for T} and \eqref{eqn: L inf bound for part of local solution}), $T_*$ has continuous dependence on $\d$, $R/|\tilde{c}_*|$ and $\d R^2$.
Combining all the facts above, there is a universal $T_{**}=T_{**}(\mu, \nu, G, r_0, R_0,\tilde{T})>0$, such that for all such $k$,
\beq
T_*(\d_k, p,\e,\mu,\nu,R_k/|\tilde{c}_*(r_k, R_k)|,G, \d_k R_k^2)\geq T_{**}.
\label{eqn: lower bound for life span}
\eeq
This contradicts with \eqref{eqn: infinite continuation}, so case \eqref{case: bad case infinite continuations} above is ruled out.

Similarly, there exists a universal $M_{**} = M_{**}(\mu, \nu, G, r_0, R_0,\tilde{T})>0$ such that for all such $k$,
\beq
M_*(p,\e,\mu,\nu,R_k/|\tilde{c}_*(r_k, R_k)|,G, \d_k R_k^2)\geq M_{**}.
\label{eqn: lower bound for the smallness}
\eeq
\end{step}

\begin{step}[Estimates for total number of continuations]
It suffices to consider the case \eqref{case: bad case blow up before T} above.

Thanks to \eqref{eqn: lower bound for life span}, if \eqref{eqn: smallness for the k-th stage} always holds, we only need to make continuation for finitely many times to cover the time interval $[0,\tilde{T}]$.
To be more precise, by choosing the longest possible lifespan of the local solution in each stage of continuation, we can have $T_N\geq \tilde{T}$ for some $N$ that admits an upper bound
\beq
N \leq N_{**}(\mu, \nu, G, r_0, R_0,\tilde{T}),
\eeq
provided that \eqref{eqn: smallness for the k-th stage} is not violated along the way.
In order to make \eqref{eqn: smallness for the k-th stage} hold for $N_{**}$ times, we take $M$ sufficiently small (recall that $M$ is defined by $h_0$, $H_0$ and $\d_0$ in \eqref{eqn: smallness of fractional Sobolev norm of initial data}), such that
\beq
\tilde{C}(p,\e,G)^{N_{**}}\cdot M\leq M_{**},
\label{eqn: smallness for multiple stages}
\eeq
where $\tilde{C}$ is given in Lemma \ref{lem: growth of local solution after one stage} and $M_{**}$ is introduced in \eqref{eqn: lower bound for the smallness}.
Note that the required smallness for $M$ only depends on $\mu$, $\nu$, $G$, $r_0$, $R_0$ and $\tilde{T}$.
With \eqref{eqn: smallness for multiple stages}, it is easy to justify by Lemma \ref{lem: growth of local solution after one stage} that \eqref{eqn: smallness for the k-th stage} will always be satisfied before the solution is extended beyond $\tilde{T}$.
\end{step}

This completes the proof.
\end{proof}

\section{Uniqueness}
\label{sec: uniqueness}
In this section, we prove uniqueness of the local solution under the additional assumption $G\in C^{1,1}$.

\subsection{Basic setup}
We start with basic setups that will be used throughout this section.
Let $p\in (2,\infty)$ and $\e>0$ as in Theorem \ref{thm: local well-posedness}, and $\alpha = 1-\f2p$.
Let $\b'$ be defined in \eqref{eqn: def of beta'} and $\b = \b'/4$ as in the proof of local existence (see step \ref{step: justification image is still in the set}).
In particular, $\b<\f{\b'}{1+\b'}$ and $\b<\f14$.

Suppose there are two solutions $f_i$ and $F_i$ $(i = 1,2)$ of \eqref{eqn: equation for inner interface new notation}-\eqref{eqn: initial condition} with regularity and estimates given in Theorem \ref{thm: local well-posedness}.
We define $h_i$ and $H_i$ $(i = 1,2)$ as in \eqref{eqn: def of h and H}.
Let $m_{0,i}$, $M_{0,i}$, $m_{\alpha,i}$ and $M_{\alpha,i}$ be defined as in \eqref{eqn: bound assumption on the Lipschitz norm of h}, \eqref{eqn: bound assumption on the Lipschitz norm of H}, \eqref{eqn: def of m_alpha} and \eqref{eqn: def of M_alpha}, respectively, and let $\D m_0$, $\D M_0$, $\D m_\alpha$ and $\D M_\alpha$ be defined in \eqref{eqn: def of normalized W 1 inf difference of two h}-\eqref{eqn: def of normalized C 1 alpha difference of two H large}.
By virtue of \eqref{eqn: bounds for the local solution}, by imposing sufficient smallness in \eqref{eqn: smallness of fractional Sobolev norm of initial data} that depends on $G$, $p$ and $\e$, we may assume that for all $t\in [0,T]$, $\g_i(t)\subset B_{r(1+\d)}$ and $\tilde{\g}_i(t)\subset B_{r(1+5\d)}^c$, and
\beq
m_{0,i}+M_{0,i}+\|h_i\|_{\dot{C}^{\b'}}+\|H_i\|_{\dot{C}^{\b'}}\ll 1,
\label{eqn: smallness of two solutions}
\eeq
Later we shall see the smallness needs to depend on $p$ and $\e$.

Let $p_i$ solve \eqref{eqn: equation for p simplified case} and \eqref{eqn: p boundary condition} in the (time-varying) physical domain that is determined by $f_i$ and $F_i$.
Let $x_i(X)$ be the diffeomorphism between the physical and the (time-invariant) reference domains, determined by $h_i$ and $H_i$ via \eqref{eqn: change of variables to the reference coordinate}, and let $X_i(x)$ be its inverse.
Define $\tilde{p}_i(X) := p_i(x_i(X))$ as the pull-back of $p_i$ to the reference domain.
Let $\va_i$ be the potential defined in \eqref{eqn: def of potential varphi} corresponding to $p_i$.
Let $c_i$ and $\tilde{c}_i$ be defined as in \eqref{eqn: def of c characteristic normal derivative for actual solution}.

The idea of proving uniqueness is to first derive bounds for $\tilde{\CR}_{h_1}-\tilde{\CR}_{h_2}$ and $\tilde{\CR}_{H_1}-\tilde{\CR}_{H_2}$ (see \eqref{eqn: def of remainder tilde terms in h and H equations}) in terms of $h_1-h_2$ and $H_1-H_2$ by following the arguments in previous sections, and then use regularity theory of \eqref{eqn: backbone equation for h} and \eqref{eqn: backbone equation for H big} to conclude that $h_1-h_2$ and $H_1-H_2$ can only be zero if they initially are.
Such a process would be extremely involved if carried out naively, requiring more estimates than we currently have.
To slightly reduce the complexity, we shall segregate inner and outer interfaces by a cut-off function in space, which decouples their dynamics in some sense.

With abuse of notation, let $\eta(x)$ be a time-independent, radially symmetric, smooth cut-off function on the physical domain, such that $\eta\in [0,1]$ in $\BR^2$, $\eta \equiv 1$ on $B_{r(1+3\d)}$, and $\eta \equiv 0$ outside $B_{r(1+4\d)}$.
Moreover, we need $|\na \eta|\leq C(r\d)^{-1}$ and $|\na^2 \eta|\leq C(r\d)^{-2}$ for some universal $C$.
For $i = 1,2$, define
\beqo
\p_i = \eta\va_i,\quad \P_i = (1-\eta)\va_i.
\eeqo
The equation satisfied by $\p_i$ can be derived from \eqref{eqn: equation for p simplified case}, \eqref{eqn: p boundary condition} and \eqref{eqn: def of potential varphi}.
Proceeding as in Section \ref{sec: equations of the problem},
\beq
\begin{split}
\p_i = &\;-\CD_{\g_i}[\va_i]+ \G*(G(p_i)\chi_{\Om_i}-2\na\va_i\na \eta-\va_i\D\eta)\\
= &\;-\CD_{\g_i}[\va_i]+\G*(g_{\p,i}(X_i))\quad \mbox{in }\BR^2\backslash \g_i,
\end{split}
\label{eqn: representation of psi inner}
\eeq
where we define in the reference coordinate
\beq
g_{\p,i}(X) = G(\tilde{p}_i(X))\chi_{B_r}(X)-2\nu \na \tilde{p}_i(X) \na \eta(X)-\nu \tilde{p}_i(X)\D \eta(X).
\eeq
Note that the last two terms above are only supported on $\overline{B}_{r(1+4\d)}\backslash B_{r(1+3\d)}$, where the diffeomorphism is
identity.
Comparing \eqref{eqn: representation of phi} and \eqref{eqn: representation of psi inner}, we find
\beq
-\CD_{\tilde{\g}_i}\phi_i+\G*(2\na\va_i\na \eta +\va_i\D \eta) = 0\quad\mbox{in }B_{r(1+3\d)}.
\label{eqn: cut-off potential to replace the double layer potential on outer boundary}
\eeq
Hence, we claim that
\beq
\P_i = -\CD_{\tilde{\g}_i}\phi_i + \G*(2\na \va_i\na\eta+\va_i\D \eta)\quad \mbox{in }\tilde{\Om}_i.
\label{eqn: representation of psi outer}
\eeq
Indeed, we may first assume $\P_i = -\CD_{\tilde{\g}_i}\Phi_i + \G*(2\na \va_i\na\eta+\va_i\D \eta)$ for some boundary potential $\Phi_i$ to be determined along $\tilde{\g}_i$.
Then we observe $\CD_{\tilde{\g}_i}\Phi_i$ and $\CD_{\tilde{\g}_i}\phi_i$ have to coincide in $B_{r(1+3\d)}$ because of \eqref{eqn: cut-off potential to replace the double layer potential on outer boundary} and the fact $\P_i = 0$ there.
Since $\CD_{\tilde{\g}_i}\Phi_i$ and $\CD_{\tilde{\g}_i}\phi_i$ are harmonic inside $\tilde{\Om}_i$, this proves $\Phi_i = \phi_i$.
For convenience, we also introduce
\beq
g_{\P,i}(X) = 2\nu \na \tilde{p}_i(X) \na \eta(X)+\nu \tilde{p}_i(X)\D \eta(X).
\label{eqn: outer potential segregated}
\eeq
Then \eqref{eqn: representation of psi outer} becomes
\beq
\P_i = -\CD_{\tilde{\g}_i}\phi_i + \G*g_{\P,i}(X_i(x))\quad \mbox{in }\tilde{\Om}_i.
\eeq
This also implies
\beq
\G*g_{\P,i}(X_i(x)) = \G*(G(p_i)\chi_{\Om_i})-\CD_{\g_i}[\va_i]\quad\mbox{in }B_{r(1+4\d)}^c.
\label{eqn: cut-off potential to replace the source and double layer potential on inner boundary}
\eeq

Recall that $[\va_i]$ and $\phi_i$ satisfy 
%
%
%
%
\eqref{eqn: introducing remainder in va equation}-\eqref{eqn: def of remainder in phi equation}.
They can be rewritten as (see \eqref{eqn: equation for derivative of jump of potential final form} and \eqref{eqn: equation for derivative of potential on outer interface})
%
\begin{align}
[\va_i]'-2Ac_* f_i' = &\; \tilde{\CR}_{[\va_i]'},\label{eqn: new equation for phi_i' segregated}\\
\phi_i' +2\tilde{c}_* F_i' = &\; \tilde{\CR}_{\phi_i'},
\label{eqn: new equation for Phi big_i' segregated}
\end{align}
where
\beq
\begin{split}
\tilde{\CR}_{[\va_i]'}:= &\;2A f_i'(\th) (e_r\cdot \na(\G*g_{\p,i}(X_i))|_{\g_i}-c_*)\\
&\;+ 2Af_i(\th) e_\th\cdot \na(\G*g_{\p,i}(X_i))|_{\g_i}+ 2A\g_i'^\perp\cdot \CK_{\g_i} [\va_i]',
\end{split}
\label{eqn: def of tilde R varphi'_i}
\eeq
and
\beq
\begin{split}
\tilde{\CR}_{\phi_i'}:= &\;-2F_i'(\th) (e_r\cdot \na(\G*g_{\psi,i}(X_i))|_{\tilde{\g}_i}-\tilde{c}_*)\\
&\;-2 F_i(\th) e_\th\cdot \na(\G*g_{\psi,i}(X_i))|_{\tilde{\g}_i}-2\tilde{\g}_i'^\perp\cdot \CK_{\tilde{\g}_i} \phi_i'.
\end{split}
\label{eqn: def of tilde R Phi' big_i}
\eeq
On the other hand, following the derivation of \eqref{eqn: equation for inner interface new notation} and \eqref{eqn: equation for outer interface new notation}, \eqref{eqn: backbone equation for h}-\eqref{eqn: def of remainder tilde terms in h and H equations} admit the following new representations,
\begin{align}
\pa_t h_i+\f{c_*}{r}=&\;-\f{Ac_*}{r}(-\D)^{1/2}h_i +\f{1}{r}\tilde{\CR}_{h_i}, \label{eqn: reformulation of equation for h_i}\\
\pa_t H_i+\f{\tilde{c}_*}{R}=&\;\f{\tilde{c}_*}{R}(-\D)^{1/2}H_i +\f{1}{R}\tilde{\CR}_{H_i}, \label{eqn: reformulation of equation for H_i big}
\end{align}
where 
\beq
\begin{split}
\tilde{\CR}_{h_i} = &\;-\f{1}{f_i}\g_i'\cdot \CK_{\g_i}\tilde{\CR}_{[\va_i]'}-2Ac_*\left(\f{1}{f_i}\g_i'\cdot \CK_{\g_i}f_i'-\f{1}{2r}\CH f_i'\right)\\
&\;+\left(\f{f_i'}{f_i}e_\th\cdot \na (\G* g_{\p,i}(X_i))|_{\g_i} - e_r \cdot \na (\G* g_{\p,i}(X_i))|_{\g_i}+c_*\right),
\end{split}
\label{eqn: new representation of tilde R h_i}
\eeq
and
\beq
\begin{split}
\tilde{\CR}_{H_i} = &\;-\f{1}{F_i}\tilde{\g}_i'\cdot \CK_{\tilde{\g}_i}\tilde{\CR}_{\phi_i'}+2\tilde{c}_*\left(\f{1}{F_i}\tilde{\g}_i'\cdot \CK_{\tilde{\g}_i}F_i'-\f{1}{2R}\CH F_i'\right)\\
&\;+\left(\f{F_i'}{F_i}e_\th\cdot \na (\G* g_{\P,i}(X_i))|_{\tilde{\g}_i} - e_r \cdot \na (\G* g_{\P,i}(X_i))|_{\tilde{\g}_i} +\tilde{c}_*\right).
\end{split}
\label{eqn: new representation of tilde R H big_i}
\eeq
\eqref{eqn: reformulation of equation for h_i} and \eqref{eqn: reformulation of equation for H_i big} are coupled with initial data $h_i(t= 0) = h_0$ and $H_i(t = 0) = H_0$.

\subsection{Estimates for differences of two solutions}
Next we shall bound $\tilde{R}_{h_1}-\tilde{R}_{h_2}$ and $\tilde{R}_{H_1}-\tilde{R}_{H_2}$.

\begin{lem}\label{lem: bounds for generalized source term}
$g_{\p,i}$ and $g_{\P,i}$ are supported in $B_{r(1+4\d)}$, satisfying that
\begin{align}
\|g_{\p,i}\|_{L^\infty}+\|g_{\P,i}\|_{L^\infty}\leq &\;C(\nu,r,R,G),\\
\|g_{\p,1}-g_{\p,2}\|_{L^\infty}+\|g_{\P,1}-g_{\P,2}\|_{L^\infty}\leq &\;C(\b,\mu,\nu, r, R, G)(\D m_0+\D M_0),
\end{align}
and
\begin{align}
\|e_\th\cdot \na g_{\p,i}\|_{L^2}+\|e_\th\cdot \na g_{\P,i}\|_{L^2}\leq &\;C(\mu,\nu,r,R,G)(m_{0,i}+M_{0,i}),\\
\|e_\th\cdot \na (g_{\p,1}-g_{\p,2})\|_{L^2}+\|e_\th\cdot \na (g_{\P,1}-g_{\P,2})\|_{L^2}\leq &\;C(\b,\mu,\nu, r, R, G)(\D m_\b+\D M_\b).
\end{align}
\begin{proof}
Note that $\tilde{p}_i$ and $p_*$ are harmonic in a neighborhood (whose size depends on $r$ and $R$) of the support of $\na \eta$, so gradient estimates apply.
Then the desired estimates follow from Lemma \ref{lem: estimates for difference between tilde p and p_*} and Lemma \ref{lem: C 1 alpha bounds for two different pressures in reference coordinate}.
The assumption $G\in C^{1,1}$ is used when proving the last inequality.
\end{proof}
\end{lem}

\begin{proposition}\label{prop: estimates for difference of varphi_i segregated}
Assume \eqref{eqn: smallness of two solutions} with the smallness depending on $p$ and $\b$ (and thus on $p$ and $\e$.)
\beq
\|[\va_1]'-[\va_2]'\|_{\dot{C}^\b}
\leq C r^2 (\|h_1'-h_2'\|_{\dot{C}^\b}+ \D m_0 + \D M_\b),
\label{eqn: Holder norm of difference of varphi' jump}
\eeq
and
\beq
\begin{split}
&\;\|[\va_1]''-[\va_2]''\|_{L^p}\\
\leq &\;C(p,\e,\mu,\nu,G)r^2\|h_1''-h_2''\|_{L^p}(1+\d R^2)^{1/2}\\
&\;+Cr^2(\|h_1'-h_2'\|_{\dot{C}^\b}+ \D m_0 + \D M_\b)
(1+\|h_1''\|_{L^p}+\|h_2''\|_{L^p}+\|H_1''\|_{L^p}).
\end{split}
\label{eqn: W 2 p estimate for difference of varphi_i jump}
\eeq
where $C = C(p,\e,\mu,\nu,r,R, G)$ unless otherwise stated.
\begin{proof}
We proceed as in Proposition \ref{prop: Holder estimate for jump of potential and potential along outer interface} and Proposition \ref{prop: W 1p estimate for jump of potential and potential along outer interface}.
By Lemmas \ref{lem: C 1 alpha bounds for two different pressures in reference coordinate}, \ref{lem: L inf estimate of difference same source}-
\ref{lem: W 1p estimate of derivative of growth potential inner}, \ref{lem: bound error between c and c_*} and \ref{lem: bounds for generalized source term},
\beq
\begin{split}
&\;\|f_1'(\th) (e_r\cdot \na(\G*g_{\p,1}(X_1))|_{\g_1}-c_*)
-f_2'(\th) (e_r\cdot \na(\G*g_{\p,2}(X_2))|_{\g_2}-c_*)\|_{\dot{C}^\b}\\
\leq &\;\|(f_1'-f_2') (e_r\cdot \na(\G*g_{\p,1}(X_1))|_{\g_1}-c_*)\|_{\dot{C}^\b}\\
&\;+
\|f_2'(\th)( e_r\cdot \na(\G*g_{\p,1}(X_1))|_{\g_1}-e_r\cdot \na(\G*g_{\p,1}(X_2))|_{\g_2})\|_{\dot{C}^\b}\\
&\;+
\|f_2'(\th) (e_r\cdot \na(\G*(g_{\p,1}-g_{\p,2})(X_2))|_{\g_2}\|_{\dot{C}^\b}\\
\leq &\;Cr^2 \|h_1'-h_2'\|_{\dot{C}^\b} ( m_{0,1} \d |\ln \d| \|g_{\p,1}\|_{L^\infty(B_{(1+4\d)r})}\\
&\;\qquad + \|e_\th\cdot \na g_{\p,1} \|_{L^2(B_{r(1+4\d)})}+(m_{0,1}+M_{0,1})(\d R^2)^{1/2})\\
&\;+Cr^2\|h_1'-h_2'\|_{L^\infty} (\|g_{\p,1}\|_{L^\infty(B_{(1+4\d)r})} m_{\b,1}+ \|e_\th\cdot \na g_{\p,1}\|_{L^2(B_{(1+4\d)r})})\\
&\;+Cr^2
\|h_2'\|_{\dot{C}^\b} \cdot \d |\ln \d| \D m_0\|g_{\p,1}\|_{L^\infty(B_{r(1+4\d)})}\\
&\;+Cr^2\|h_2'\|_{L^\infty} (\|g_{\p,1}\|_{L^\infty(B_{(1+4\d)r})}\D m_\b +\D m_0\|e_\th\cdot \na g_{\p,1}\|_{L^2(B_{(1+4\d)r})})\\
&\;+Cr^2
\|h_2'\|_{\dot{C}^\b} ( m_{0,2} \d |\ln \d| \|g_{\p,1}-g_{\p,2}\|_{L^\infty(B_{(1+4\d)r})}\\
&\;\qquad + \|e_\th\cdot \na (g_{\p,1}-g_{\p,2}) \|_{L^2(B_{r(1+4\d)})}+|c_1-c_2|)\\
&\;+Cr^2\|h_2'\|_{L^\infty} (\|g_{\p,1}-g_{\p,2}\|_{L^\infty(B_{(1+4\d)r})} m_{\b,2}+ \|e_\th\cdot \na (g_{\p,1}-g_{\p,2})\|_{L^2(B_{(1+4\d)r})})\\
\leq &\;Cr^2 \|h_1'-h_2'\|_{\dot{C}^\b} (m_{\b,1}+M_{0,1})
+Cr^2
\|h_2'\|_{\dot{C}^\b} (\D m_\b+\D M_\b),
\end{split}
\eeq
where $C = C(\b,\mu,\nu,r,R,G)$.
Here we used the estimate by \eqref{eqn: def of c characteristic normal derivative for actual solution} and Lemma \ref{lem: C 1 alpha bounds for two different pressures in reference coordinate} that
\beq
|c_1-c_2|\leq \f{C}{r}\int_{B_r}|G(\tilde{p}_1)-G(\tilde{p}_2)|\,dX \leq Cr \|\tilde{p}_1-\tilde{p}_2\|_{L^\infty(B_r)}\leq Cr(\D m_0+\D M_0),
\label{eqn: stability of normal velocity of inner interface}
\eeq
where $C = C(\b,\mu,\nu, r, R, G)$.
Similarly,
\beq
\begin{split}
&\;\|f_1(\th) e_\th\cdot \na(\G*g_{\p,1}(X_1))|_{\g_1}- f_2(\th) e_\th\cdot \na(\G*g_{\p,2}(X_2))|_{\g_2}\|_{\dot{C}^\b}\\
\leq &\;Cr^2(\D m_\b + \D M_\b).
\end{split}
\eeq
On the other hand, by \eqref{eqn: smallness of two solutions}, Lemma \ref{lem: Holder estimate of singular integral normal component} and Lemma \ref{lem: Holder estimate of singular integral normal component difference}, 
\beq
\begin{split}
&\;\|\g_1'^\perp\cdot \CK_{\g_1} [\va_1]' - \g_2'^\perp\cdot \CK_{\g_2} [\va_2]'\|_{\dot{C}^{\b}}\\
\leq &\; C(\b)(\|h_1-h_2\|_{C^{1,\b}}\|[\va_1]'\|_{\dot{C}^\b} + \|h_2'\|_{\dot{C}^\b}\|[\va_1]'-[\va_2]'\|_{\dot{C}^\b}).
\end{split}
\eeq
Combining these estimates with \eqref{eqn: smallness of two solutions}, \eqref{eqn: new equation for phi_i' segregated}, \eqref{eqn: def of tilde R varphi'_i} and Proposition \ref{prop: Holder estimate for jump of potential and potential along outer interface} yields
\beq
\begin{split}
&\;\|\tilde{\CR}_{[\va_1]'}-\tilde{\CR}_{[\va_2]'}\|_{\dot{C}^\b}\\
\leq &\; Cr^2\|h_1'-h_2'\|_{\dot{C}^\b} (m_{0,1}+M_{0,1}+\|h_1'\|_{\dot{C}^\beta}+\|H_1'\|_{\dot{C}^\beta}) +Cr^2(\D m_\b + \D M_\b)\\
&\; + C(\b)\|h_2'\|_{\dot{C}^\b}\|[\va_1]'-[\va_2]'\|_{\dot{C}^\b},
\end{split}
\label{eqn: Holder estimate for difference of R varphi'}
\eeq
and
\beq
\|[\va_1]'-[\va_2]'\|_{\dot{C}^\b}
\leq C r^2 \|h_1'-h_2'\|_{\dot{C}^\b}+ Cr^2(\D m_\b + \D M_\b),
\eeq
where $C = C(\b, \b',\mu,\nu,r,R, G)$.
Note that $\b$ and $\b'$ essentially depend on $p$ and $\e$.

To show \eqref{eqn: W 2 p estimate for difference of varphi_i jump}, we derive as in \eqref{eqn: W 1p estimate for source term on inner interface} that
\beq
\begin{split}
&\;\|f_1'(\th) (e_r\cdot \na(\G*g_{\p,1}(X_1))|_{\g_1}-c_*)
-f_2'(\th) (e_r\cdot \na(\G*g_{\p,2}(X_2))|_{\g_2}-c_*)\|_{\dot{W}^{1,p}}\\
\leq &\;\|f_1'-f_2'\|_{\dot{W}^{1,p}} \|e_r\cdot \na(\G*g_{\p,1}(X_1))|_{\g_1}-c_*\|_{L^\infty}\\
&\;+\|f_1'-f_2'\|_{L^\infty} \|e_r\cdot \na(\G*g_{\p,1}(X_1))|_{\g_1}\|_{\dot{W}^{1,p}}\\
&\;+
\|f_2'(\th)( e_r\cdot \na(\G*g_{\p,1}(X_1))|_{\g_1}-e_r\cdot \na(\G*g_{\p,1}(X_2))|_{\g_2})\|_{\dot{W}^{1,p}}\\
&\;+
\|f_2'(\th) e_r\cdot \na(\G*(g_{\p,1}-g_{\p,2})(X_2))|_{\g_2}\|_{\dot{W}^{1,p}}\\
%
\leq &\;\|f_1'-f_2'\|_{\dot{W}^{1,p}} \|e_r\cdot \na(\G*g_{\p,1}(X_1))|_{\g_1}-c_*\|_{L^\infty}\\
&\;+Cr^2\|h_1'-h_2'\|_{L^\infty} ( m_{\b,1}+ M_{0,1})+Cr^2
\|h_2''\|_{L^p}(\D m_\b+\D M_\b).
%
\end{split}
\label{eqn: W 1p difference of a component of growth potential crude form}
\eeq
We shall need an estimate for $\|e_r\cdot \na(\G*g_{\p,1}(X_1))|_{\g_1}-c_*\|_{L^\infty}$ with explicit $r$- and $R$-dependence.
By \eqref{eqn: cut-off potential to replace the double layer potential on outer boundary}, and then \eqref{eqn: gradient of double layer away from the boundary}, Lemmas \ref{lem: estimates for difference between tilde p and p_*}, \ref{lem: L^inf estimate of derivative of growth potential inner}, \ref{lem: L inf norm of derivative of interaction kernel from outer to inner}, \ref{lem: estimates for the potential kernel part in interaction kernels}, \ref{lem: bound error between c and c_*} and Proposition \ref{prop: Holder estimate for jump of potential and potential along outer interface},
\beq
\begin{split}
&\;\|e_r\cdot \na(\G*g_{\p,1}(X_1))|_{\g_1}-c_*\|_{L^\infty}\\
\leq &\; \|e_r\cdot \na(\G*(G(p_1)\chi_{\Om_1}))|_{\g_1}-c_*\|_{L^\infty}+\|e_r\cdot \na (\CD_{\tilde{\g}_1}\phi_1)|_{\g_1}\|_{L^\infty}\\
\leq &\; \|e_r\cdot \na(\G*(G(p_1)\chi_{\Om_1}))|_{\g_1}-c_1\|_{L^\infty}+|c_1-c_*|+\|e_\th\cdot \CK_{\g_1,\tilde{\g}_1}\phi_1'\|_{L^\infty}\\
\leq &\;C(\mu,\nu,G, \b,\b')r(\|h_1'\|_{\dot{C}^\b}+\|H_1'\|_{\dot{C}^\b}+(m_{0,1}+M_{0,1})(1+\d R^2)^{1/2}).
\end{split}
\label{eqn: special handling of L^inf norm of gradient of the growth potential}
\eeq
Hence,
\beq
\begin{split}
&\;\|f_1'(\th) (e_r\cdot \na(\G*g_{\p,1}(X_1))|_{\g_1}-c_*)
-f_2'(\th) (e_r\cdot \na(\G*g_{\p,2}(X_2))|_{\g_2}-c_*)\|_{\dot{W}^{1,p}}\\
\leq &\;C(\mu,\nu,G, \b,\b')r^2\|h_1''-h_2''\|_{L^p}(\|h_1'\|_{\dot{C}^\b}+\|H_1'\|_{\dot{C}^\b}+(m_{0,1}+M_{0,1})(1+\d R^2)^{1/2})\\
&\;+Cr^2\|h_1'-h_2'\|_{L^\infty} ( m_{\b,1}+ M_{0,1})+Cr^2
\|h_2''\|_{L^p}(\D m_\b+\D M_\b).
\end{split}
\label{eqn: W 1p difference of a component of growth potential}
\eeq
On the other hand,
\beq
\begin{split}
&\;\|f_1(\th) e_\th\cdot \na(\G*g_{\p,1}(X_1))|_{\g_1}- f_2(\th) e_\th\cdot \na(\G*g_{\p,2}(X_2))|_{\g_2}\|_{\dot{W}^{1,p}}\\
\leq &\; Cr^2(\D m_\b + \D M_\b),
\end{split}
\eeq
and by Lemma \ref{lem: W 1p estimate of normal component of singular integral} and Lemma \ref{lem: W 1p estimate of normal component of singular integral difference},
\beq
\begin{split}
&\;\|\g_1'^\perp\cdot\mathcal{K}_{\g_1}[\va_1]'-\g_2'^\perp\cdot\mathcal{K}_{\g_2}[\va_2]'\|_{\dot{W}^{1,p}}\\
\leq &\; C(\D m_0+\|h_1'-h_2'\|_{\dot{C}^\b})(\|h_1''\|_{L^p}+\|h_2''\|_{L^p})\|[\va_1]'\|_{\dot{C}^\b}
\\
&\;+C\|h_1''-h_2''\|_{L^p} \| [\va_1]'\|_{\dot{C}^\b}+C\D m_0 \|[\va_1]''\|_{L^p}\\
&\;+C( \|h_2''\|_{L^p}\|[\va_1]'-[\va_2]'\|_{\dot{C}^\b} +\|h_2'\|_{L^\infty}\|[\va_1]''-[\va_2]''\|_{L^p}),
\end{split}
\eeq
where $C = C(p,\b)$.
Combining these estimates with \eqref{eqn: smallness of two solutions}, \eqref{eqn: new equation for phi_i' segregated}, \eqref{eqn: def of tilde R varphi'_i}, \eqref{eqn: Holder norm of difference of varphi' jump} as well as Propositions \ref{prop: Holder estimate for jump of potential and potential along outer interface} and \ref{prop: W 1p estimate for jump of potential and potential along outer interface}, we can show
\beq
\begin{split}
&\;\|\tilde{\CR}_{[\va_1]'}-\tilde{\CR}_{[\va_2]'}\|_{\dot{W}^{1,p}}\\
\leq &\;C(p,\e,\mu,\nu,G)r^2\|h_1''-h_2''\|_{L^p}(\|h_1'\|_{\dot{C}^\b}+\|H_1'\|_{\dot{C}^\b}+(m_{0,1}+M_{0,1})(1+\d R^2)^{1/2})\\
&\;+Cr^2(\|h_1'-h_2'\|_{\dot{C}^\b}+ \D m_0 + \D M_\b)
(1+\|h_1''\|_{L^p}+\|h_2''\|_{L^p}+\|H_1''\|_{L^p})\\
&\;+C(p,\b)\|h_2'\|_{L^\infty}\|[\va_1]''-[\va_2]''\|_{L^p},
\end{split}
\label{eqn: W 1 p estimate for difference of R varphi'}
\eeq
and thus \eqref{eqn: W 2 p estimate for difference of varphi_i jump}.
\end{proof}
\end{proposition}

\begin{proposition}\label{prop: estimates for difference of Phi_big i segregated}
Under the assumption of Proposition \ref{prop: estimates for difference of varphi_i segregated}, 
\beq
\|\phi_1'-\phi_2'\|_{\dot{C}^{\b}}\leq  Cr^2 (\|H_1'-H_2'\|_{\dot{C}^\b} +\D m_\b+\D M_0),
\label{eqn: Holder estimate for difference of Phi_i'}\\
\eeq
%
and
\beq
\begin{split}
\|\phi_1''-\phi_2''\|_{L^p}\leq&\;C(p,\e,\mu,\nu,G)r^2 \|H_1''-H_2''\|_{L^p}(1+\d R^2)^{1/2}\\
&\; +Cr^2(\D m_\b+\D M_0+\|H_1'-H_2'\|_{\dot{C}^\b}) (1+\|h_1''\|_{L^p}+\|H_1''\|_{L^p}+\|H_2''\|_{L^p}),
\end{split}
\label{eqn: W 2 p estimate for difference of Phi_i'}
\eeq
%
%
%
%
where $C = C(p,\e,\mu,\nu,r,R,G)$ unless otherwise stated.
\begin{proof}
We justify as before.
By Lemmas \ref{lem: C 1 alpha bounds for two different pressures in reference coordinate}, \ref{lem: L inf estimate of difference same source at outer interface}-\ref{lem: W 1p estimate of derivative of growth potential outer},
\ref{lem: bound error between c and c_*} and \ref{lem: bounds for generalized source term},
\beq
\begin{split}
&\;\|F_1'(\th) (e_r\cdot \na(\G*g_{\P,1}(X_1))|_{\tilde{\g}_1}-\tilde{c}_*)
-F_2'(\th) (e_r\cdot \na(\G*g_{\P,2}(X_2))|_{\tilde{\g}_2}-\tilde{c}_*)\|_{\dot{C}^\b}\\
\leq &\;Cr^2 \|H_1'-H_2'\|_{\dot{C}^\b}  (m_{0,1}+M_{0,1})
+Cr^2\|H_2'\|_{\dot{C}^\b} (\D m_\b+\D M_\b),
\end{split}
\label{eqn: Holder estimate for difference of ingredient 1 in R Phi'}
\eeq
where $C = C(\b,\mu,\nu,r,R,G)$.
Here we used the fact that, by \eqref{eqn: def of c characteristic normal derivative outer interface}, \eqref{eqn: def of c characteristic normal derivative for actual solution} and \eqref{eqn: outer potential segregated}, with $\sigma$ being the unit outer normal vector of $\pa B_{r(1+3\d)}$,
\beqo
\tilde{c}_{g_{\P,i}} = -\f{1}{2\pi R} \int_{B_{r(1+4\d)}\backslash B_{r(1+3\d)}} \nu \D (\eta \tilde{p}_i)\,dX = \f{\nu}{2\pi R}\int_{\pa B_{r(1+3\d)}}\f{\pa \tilde{p}_i}{\pa \sigma} \,dy = \tilde{c}_i,
\eeqo
which yields by \eqref{eqn: stability of normal velocity of inner interface} that
\beq
|\tilde{c}_1-\tilde{c}_2| \leq \f{Cr^2}{R}(\D m_0 + \D M_0).
\label{eqn: difference of tilde c outer interface segregated}
\eeq
Similarly,
\beq
\begin{split}
&\;\|F_1(\th) e_\th\cdot \na(\G*g_{\P,1}(X_1))|_{\tilde{\g}_1}
-F_2(\th) e_\th\cdot \na(\G*g_{\P,2}(X_2))|_{\tilde{\g}_2}\|_{\dot{C}^\b}\\
\leq &\;Cr^2 (\D m_\b+\D M_\b).
\end{split}
\label{eqn: Holder estimate for difference of ingredient 2 in R Phi'}
\eeq
Again by \eqref{eqn: smallness of two solutions}, Lemma \ref{lem: Holder estimate of singular integral normal component} and Lemma \ref{lem: Holder estimate of singular integral normal component difference}, 
\beq
\begin{split}
&\;\|\tilde{\g}_1'^\perp\cdot \CK_{\tilde{\g}_1} \phi_1' - \tilde{\g}_2'^\perp\cdot \CK_{\tilde{\g}_2} \phi_2'\|_{\dot{C}^{\b}}\\
\leq &\; C(\b)(\|H_1-H_2\|_{C^{1,\b}}\|\phi_1'\|_{\dot{C}^\b} + \|H_2'\|_{\dot{C}^\b}\|\phi_1'-\phi_2'\|_{\dot{C}^\b}).
\end{split}
\label{eqn: Holder estimate for difference of ingredient 3 in R Phi'}
\eeq
By \eqref{eqn: def of tilde R Phi' big_i}, \eqref{eqn: Holder estimate for difference of ingredient 1 in R Phi'}, \eqref{eqn: Holder estimate for difference of ingredient 2 in R Phi'}, \eqref{eqn: Holder estimate for difference of ingredient 3 in R Phi'} and Proposition \ref{prop: Holder estimate for jump of potential and potential along outer interface},
\beq
\begin{split}
&\;\|\tilde{\CR}_{\phi_1'}-\tilde{\CR}_{\phi_2'}\|_{\dot{C}^\b} \\
\leq&\; Cr^2 \|H_1'-H_2'\|_{\dot{C}^\b}  (\|h_1'\|_{\dot{C}^\beta}+\|H_1'\|_{\dot{C}^\beta}+m_{0,1}+M_{0,1})
+Cr^2 (\D m_\b+\D M_\b)\\
&\;+C(\b) \|H_2'\|_{\dot{C}^\b}\|\phi_1'-\phi_2'\|_{\dot{C}^\b},
\end{split}
\label{eqn: Holder estimate for difference of R Phi'}
\eeq
where $C = C(p,\b,\mu,\nu,r,R,G)$.
Combining this with \eqref{eqn: new equation for Phi big_i' segregated} yields \eqref{eqn: Holder estimate for difference of Phi_i'}.

In addition, thanks to \eqref{eqn: cut-off potential to replace the source and double layer potential on inner boundary},
\beq
\begin{split}
&\;\|F_1'(\th) (e_r\cdot \na(\G*g_{\P,1}(X_1))|_{\tilde{\g}_1}-\tilde{c}_*)
-F_2'(\th) (e_r\cdot \na(\G*g_{\P,2}(X_2))|_{\tilde{\g}_2}-\tilde{c}_*)\|_{\dot{W}^{1,p}}\\
\leq &\;R \|H_1''-H_2''\|_{L^p}\|e_r\cdot \na(\G*g_{\P,1}(X_1))|_{\tilde{\g}_1}-\tilde{c}_*\|_{L^\infty}\\
&\;+Cr^2\|H_1'-H_2'\|_{L^\infty} (m_{0,1} + M_{0,1})+Cr^2
\|H_2''\|_{L^p}(\D m_\b+\D M_\b)\\
\leq &\;C(\mu,\nu,G,\b,\b')r^2 \|H_1''-H_2''\|_{L^p}(\|h_1'\|_{\dot{C}^\b}+\|H_1'\|_{\dot{C}^\b}+(m_{0,1}+M_{0,1})(1+\d R^2)^{1/2})\\
&\;+Cr^2\|H_1'-H_2'\|_{L^\infty} (m_{0,1} + M_{0,1})+Cr^2
\|H_2''\|_{L^p}(\D m_\b+\D M_\b),
\end{split}
\label{eqn: W 1 p estimate for difference of ingredient 1 in R Phi'}
\eeq
and
\beq
\begin{split}
&\;\|F_1(\th) e_\th\cdot \na(\G*g_{\P,1}(X_1))|_{\tilde{\g}_1}
-F_2(\th) e_\th\cdot \na(\G*g_{\P,2}(X_2))|_{\tilde{\g}_2}\|_{\dot{W}^{1,p}}\\
\leq &\;Cr^2 (\D m_\b+\D M_\b).
\end{split}
\label{eqn: W 1 p estimate for difference of ingredient 2 in R Phi'}
\eeq
By Lemma \ref{lem: W 1p estimate of normal component of singular integral} and Lemma \ref{lem: W 1p estimate of normal component of singular integral difference},
\beq
\begin{split}
&\;\|\tilde{\g}_1'^\perp\cdot \CK_{\tilde{\g}_1} \phi_1' - \tilde{\g}_2'^\perp\cdot \CK_{\tilde{\g}_2} \phi_2'\|_{\dot{W}^{1,p}}\\
\leq &\; C(\D M_0+\|H_1'-H_2'\|_{\dot{C}^\b}) (\|H_1''\|_{L^p}+\|H_2''\|_{L^p})\|\phi_1'\|_{\dot{C}^\b}\\
&\;+C(p,\b)\|H_1''-H_2''\|_{L^p} \| \phi_1'\|_{\dot{C}^\b}+C\D M_0 \|\phi_1''\|_{L^p}\\
&\;+C(p,\b)(\| H_2''\|_{L^p}\|\phi_1'-\phi_2'\|_{\dot{C}^\b} +\|H_2'\|_{L^\infty}\|\phi_1''-\phi_2''\|_{L^p}).
\end{split}
\label{eqn: W 1 p estimate for difference of ingredient 3 in R Phi'}
\eeq
Combining \eqref{eqn: W 1 p estimate for difference of ingredient 1 in R Phi'}-\eqref{eqn: W 1 p estimate for difference of ingredient 3 in R Phi'} with \eqref{eqn: Holder estimate for difference of Phi_i'} and Propositions \ref{prop: Holder estimate for jump of potential and potential along outer interface} and \ref{prop: W 1p estimate for jump of potential and potential along outer interface}, we find
\beq
\begin{split}
&\;\|\tilde{\CR}_{\phi_1'}-\tilde{\CR}_{\phi_2'}\|_{\dot{W}^{1,p}} \\
\leq &\;C(p,\e,\mu,\nu,G)r^2 \|H_1''-H_2''\|_{L^p}(\|h_1'\|_{\dot{C}^\b}+\|H_1'\|_{\dot{C}^\b}+(m_{0,1}+M_{0,1})(1+\d R^2)^{1/2})\\
&\;+Cr^2(\|H_1'-H_2'\|_{\dot{C}^\b}+\D m_\b+\D M_0)(1+\|h_1''\|_{L^p}+
\|H_1''\|_{L^p}+\|H_2''\|_{L^p})\\
&\;+C(p,\b)\|H_2'\|_{L^\infty}\|\phi_1''-\phi_2''\|_{L^p}.
\end{split}
\label{eqn: W 1 p estimate for difference of R Phi'}
\eeq
Then \eqref{eqn: W 2 p estimate for difference of Phi_i'} follows from \eqref{eqn: new equation for Phi big_i' segregated} and \eqref{eqn: W 1 p estimate for difference of R Phi'}.
%
%
\end{proof}
\end{proposition}

\begin{lem}\label{lem: difference of residue in phi and Phi}
Under the assumption of Proposition \ref{prop: estimates for difference of varphi_i segregated},
\beq
\begin{split}
&\;\|\tilde{\CR}_{[\va_1]'}-\tilde{\CR}_{[\va_2]'}\|_{\dot{C}^\b}\\
\leq &\; Cr^2(\D m_\b + \D M_\b)\\
&\;+Cr^2\|h_1-h_2\|_{\dot{C}^{1,\b}}(\|h_1'\|_{\dot{C}^\b}+ \|h_2'\|_{\dot{C}^\b}+ \|H_1'\|_{\dot{C}^\b}+ m_{0,1} + M_{0,1}),
\end{split}
\eeq
\beq
\begin{split}
&\;\|\tilde{\CR}_{[\va_1]'}-\tilde{\CR}_{[\va_2]'}\|_{\dot{W}^{1,p}}\\
\leq &\;C(p,\e,\mu,\nu,G)r^2\|h_1''-h_2''\|_{L^p}(\|h_1'\|_{\dot{C}^\b}+\|H_1'\|_{\dot{C}^\b}+(m_{0,1}+m_{0,2}+M_{0,1})(1+\d R^2)^{1/2})\\
&\;+Cr^2(\|h_1'-h_2'\|_{\dot{C}^\b}+ \D m_0 + \D M_\b)
(1+\|h_1''\|_{L^p}+\|h_2''\|_{L^p}+\|H_1''\|_{L^p}),
\end{split}
\eeq
\beq
\begin{split}
&\;\|\tilde{\CR}_{\phi_1'}-\tilde{\CR}_{\phi_2'}\|_{\dot{C}^\b} \\
\leq&\; Cr^2 (\D m_\b+\D M_\b)\\
&\;+Cr^2\|H_1-H_2\|_{\dot{C}^{1,\b}}(\|h_1'\|_{\dot{C}^\b}+\|H_1'\|_{\dot{C}^\b}+\|H_2'\|_{\dot{C}^\b} +m_{0,1}+ M_{0,1}),
\end{split}
\eeq
and
\beq
\begin{split}
&\;\|\tilde{\CR}_{\phi_1'}-\tilde{\CR}_{\phi_2'}\|_{\dot{W}^{1,p}} \\
\leq &\;C(p,\e,\mu,\nu,G)r^2 \|H_1''-H_2''\|_{L^p}(\|h_1'\|_{\dot{C}^\b}+\|H_1'\|_{\dot{C}^\b}+(m_{0,1}+M_{0,1}+M_{0,2})(1+\d R^2)^{1/2})\\
&\;+Cr^2(\|H_1'-H_2'\|_{\dot{C}^\b}+\D m_\b+\D M_0)(1+\|h_1''\|_{L^p}+
\|H_1''\|_{L^p}+\|H_2''\|_{L^p}),
\end{split}
\eeq
where $C = C(p,\e,\mu,\nu,r,R, G)$ unless otherwise stated.
\begin{proof}
It suffices to apply Proposition \ref{prop: estimates for difference of varphi_i segregated} and Proposition \ref{prop: estimates for difference of Phi_big i segregated} to \eqref{eqn: Holder estimate for difference of R varphi'},
\eqref{eqn: W 1 p estimate for difference of R varphi'}, \eqref{eqn: Holder estimate for difference of R Phi'}
and \eqref{eqn: W 1 p estimate for difference of R Phi'}.
\end{proof}
\end{lem}

\begin{lem}\label{lem: difference of remainder terms in h and H equations}
Under the assumption of Proposition \ref{prop: estimates for difference of varphi_i segregated},
\beq
\begin{split}
&\;\|\tilde{\CR}_{h_1}-\tilde{\CR}_{h_2}\|_{W^{1,p}}\\
\leq &\;C(p,\e,\mu,\nu,G)r\|h_1''-h_2''\|_{L^p}\\
&\;\quad \cdot (\|h_1'\|_{\dot{C}^\b}+\|h_2'\|_{\dot{C}^\b}+\|H_1'\|_{\dot{C}^\b}+(m_{0,1}+m_{0,2}+M_{0,1})(1+\d R^2)^{1/2})\\
&\;+Cr(\|h_1'-h_2'\|_{\dot{C}^\b}+ \D m_0 + \D M_\b)
(1+\|h_1''\|_{L^p}+\|h_2''\|_{L^p}+\|H_1''\|_{L^p}).
\end{split}
\label{eqn: difference of R h_i}
\eeq
and
\beq
\begin{split}
&\;\|\tilde{\CR}_{H_1}-\tilde{\CR}_{H_2}\|_{W^{1,p}}\\
\leq &\;C(p,\e,\mu,\nu,G)R^{-1}r^2  \|H_1''-H_2''\|_{L^p}\\
&\;\quad \cdot (\|h_1'\|_{\dot{C}^\b}+\|H_1'\|_{\dot{C}^\b}+\|H_2'\|_{\dot{C}^\b}+(m_{0,1}+M_{0,1}+M_{0,2})(1+\d R^2)^{1/2})\\
&\;+CR^{-1}r^2(\|H_1'-H_2'\|_{\dot{C}^\b}+\D m_\b+\D M_0) (1+\|h_1''\|_{L^p}+
\|H_1''\|_{L^p}+\|H_2''\|_{L^p}),
\end{split}
\label{eqn: difference of R H big_i}
\eeq
where $C = C(p,\e,\mu,\nu,r,R, G)$ unless otherwise stated.

\begin{proof}
We argue as in Lemma \ref{lem: W 1p estimate for remainder terms}.
Note that $\tilde{\CR}_{[\va_i]'}$ has mean zero on $\BT$.
By Poincar\'{e} inequality and Lemmas \ref{lem: W 1p difference from Hilbert transform} and \ref{lem: W 1p difference from Hilbert transform difference},
\beq
\begin{split}
&\;\left\|\f{1}{f_1}\g_1'\cdot \CK_{\g_1}\tilde{\CR}_{[\va_1]'}-\f{1}{f_2}\g_2'\cdot \CK_{\g_2}\tilde{\CR}_{[\va_2]'}\right\|_{W^{1,p}}\\
\leq &\;C\left\|\f{1}{f_1}-\f{1}{f_2}\right\|_{W^{1,\infty}}\|\g_1'\cdot \CK_{\g_1}\tilde{\CR}_{[\va_1]'}\|_{\dot{W}^{1,p}}\\
&\;+C\|f_2^{-1}\|_{W^{1,\infty}}\|\g_1'\cdot \CK_{\g_1}\tilde{\CR}_{[\va_1]'}-\g_2'\cdot \CK_{\g_2}\tilde{\CR}_{[\va_1]'}\|_{\dot{W}^{1,p}}\\
&\;+C\|f_2^{-1}\|_{W^{1,\infty}}\|\g_2'\cdot \CK_{\g_2}(\tilde{\CR}_{[\va_1]'}-\tilde{\CR}_{[\va_2]'})\|_{\dot{W}^{1,p}}\\
\leq &\;Cr^{-1}\D m_0 \|\tilde{\CR}_{[\va_1]'}\|_{\dot{W}^{1,p}}\\
&\;+C(p,\b)r^{-1}\|\tilde{\CR}_{[\va_1]'}\|_{\dot{C}^\b}(\|h_1''-h_2''\|_{L^p}+(\|h_1''\|_{L^p}+\|h_2''\|_{L^p}) ( \|h_1'-h_2'\|_{\dot{C}^\b}+\D m_0))\\
&\;+C(p,\b)r^{-1}(\|\tilde{\CR}_{[\va_1]'}-\tilde{\CR}_{[\va_2]'}\|_{\dot{W}^{1,p}}+\|\tilde{\CR}_{[\va_1]'}-\tilde{\CR}_{[\va_2]'}\|_{\dot{C}^{\b}}\|h_2''\|_{L^p}).
\end{split}
\eeq
By \eqref{eqn: new equation for phi_i' segregated} and \eqref{eqn: new equation for Phi big_i' segregated},
\begin{align}
\|\tilde{\CR}_{[\va_i]'}\|_{\dot{C}^\b}+\|\tilde{\CR}_{\phi_i'}\|_{\dot{C}^\b}\leq &\; \|[\va_i]'\|_{\dot{C}^\b}+\|\phi_i'\|_{\dot{C}^\b}+C(\mu,\nu,G)r^2(\|h_i'\|_{\dot{C}^\b}+\|H_i'\|_{\dot{C}^\b}),\\
\|\tilde{\CR}_{[\va_i]'}\|_{\dot{W}^{1,p}}+\|\tilde{\CR}_{\phi_i'}\|_{\dot{W}^{1,p}}\leq &\; \|[\va_i]'\|_{\dot{W}^{1,p}}+\|\phi_i'\|_{\dot{W}^{1,p}}+C(\mu,\nu,G)r^2(\|h_i'\|_{\dot{W}^{1,p}}+\|H_i'\|_{\dot{W}^{1,p}}).
\end{align}
So by 
Proposition \ref{prop: Holder estimate for jump of potential and potential along outer interface}, Proposition \ref{prop: W 1p estimate for jump of potential and potential along outer interface} 
and Lemma \ref{lem: difference of residue in phi and Phi},
\beq
\begin{split}
&\;\left\|\f{1}{f_1}\g_1'\cdot \CK_{\g_1}\tilde{\CR}_{[\va_1]'}-\f{1}{f_2}\g_2'\cdot \CK_{\g_2}\tilde{\CR}_{[\va_2]'}\right\|_{W^{1,p}}\\
\leq &\;Cr^{-1}\D m_0 (\|[\va_1]''\|_{L^p}+\|\phi_1''\|_{L^p}+r^2(\|h_1''\|_{L^p}+\|H_1''\|_{L^p}))\\
&\;+C(p,\e,\mu,\nu,G)r^{-1}(\|[\va_1]'\|_{\dot{C}^\b}+\|\phi_1'\|_{\dot{C}^\b}+r^2(\|h_1'\|_{\dot{C}^\b}+\|H_1'\|_{\dot{C}^\b}))\\
&\;\qquad \cdot (\|h_1''-h_2''\|_{L^p}+(\|h_1''\|_{L^p}+\|h_2''\|_{L^p}) ( \|h_1'-h_2'\|_{\dot{C}^\b}+\D m_0))\\
&\;+C(p,\e,\mu,\nu,G)r\|h_1''-h_2''\|_{L^p}(\|h_1'\|_{\dot{C}^\b}+\|H_1'\|_{\dot{C}^\b}+(m_{0,1}+m_{0,2}+M_{0,1})(1+\d R^2)^{1/2})\\
&\;+Cr(\|h_1'-h_2'\|_{\dot{C}^\b}+ \D m_0 + \D M_\b)
(1+\|h_1''\|_{L^p}+\|h_2''\|_{L^p}+\|H_1''\|_{L^p})\\
&\;+Cr(\D m_\b + \D M_\b)\|h_2''\|_{L^p}\\
&\;+Cr\|h_1-h_2\|_{\dot{C}^{1,\b}}(\|h_1'\|_{\dot{C}^\b}+ \|h_2'\|_{\dot{C}^\b}+ \|H_1'\|_{\dot{C}^\b}+ m_{0,1} + M_{0,1})\|h_2''\|_{L^p},\\
%
\leq &\;C(p,\e,\mu,\nu,G)r\|h_1''-h_2''\|_{L^p}(\|h_1'\|_{\dot{C}^\b}+\|H_1'\|_{\dot{C}^\b}+(m_{0,1}+m_{0,2}+M_{0,1})(1+\d R^2)^{1/2})\\
&\;+Cr(\|h_1'-h_2'\|_{\dot{C}^\b}+ \D m_0 + \D M_\b)
(1+\|h_1''\|_{L^p}+\|h_2''\|_{L^p}+\|H_1''\|_{L^p}),
%
\end{split}
\eeq
where $C = C(p,\e,\mu,\nu,r,R, G)$ unless otherwise stated.
%
%
Similarly,
\beq
\begin{split}
&\;\left\|\left(\f{1}{f_1}\g_1'\cdot \CK_{\g_1}f_1'-\f{1}{2r}\CH f_1'\right)-\left(\f{1}{f_2}\g_2'\cdot \CK_{\g_2}f_2'-\f{1}{2r}\CH f_2'\right)\right\|_{W^{1,p}}\\
\leq &\;\left\|\f{1}{f_1}-\f{1}{f_2}\right\|_{W^{1,\infty}}\left\|\g_1'\cdot \CK_{\g_1}f_1'\right\|_{\dot{W}^{1,p}}+\left\|\f{1}{f_2}\right\|_{W^{1,\infty}}\left\|\g_1'\cdot \CK_{\g_1}f_1'-\g_2'\cdot \CK_{\g_2}f_1'\right\|_{\dot{W}^{1,p}}\\
&\;+\left\|\f{1}{f_2}-\f1r\right\|_{W^{1,\infty}}\left\|\g_2'\cdot \CK_{\g_2}(f_1-f_2)'\right\|_{\dot{W}^{1,p}}+\f1r\left\|\g_2'\cdot \CK_{\g_2}(f_1-f_2)'-\f{1}{2}\CH (f_1-f_2)'\right\|_{\dot{W}^{1,p}}\\
\leq &\;C(p,\b)\|h_1''-h_2''\|_{L^p} (m_{0,1}+m_{0,2}+\|h_1'\|_{\dot{C}^\b}+\|h_2'\|_{\dot{C}^\b})\\
&\;+C(\|h_1''\|_{L^p}+\|h_2''\|_{L^p})( \|h_1'-h_2'\|_{\dot{C}^\b}+\D m_0).
\end{split}
\eeq
%
%
By \eqref{eqn: stability of normal velocity of inner interface} and Lemmas \ref{lem: L inf estimate of difference same source}-\ref{lem: W 1p estimate of derivative of growth potential inner} and \ref{lem: bounds for generalized source term},
\beq
\begin{split}
&\;\|e_r \cdot \na (\G* g_{\p,1}(X_1))|_{\g_1} -e_r \cdot \na (\G* g_{\p,2}(X_2))|_{\g_2}\|_{L^\infty}\\
\leq &\;\|e_r \cdot \na (\G* g_{\p,1}(X_1))|_{\g_1} -e_r \cdot \na (\G* g_{\p,1}(X_2))|_{\g_2}\|_{L^\infty}\\
&\;+\|e_r \cdot \na (\G* (g_{\p,1}-g_{\p,2})(X_2))|_{\g_2}\|_{L^\infty}\\
\leq &\;Cr(\D m_\b+\D M_\b)+|c_1-c_2|\\
\leq &\;Cr(\D m_\b+\D M_\b).
\end{split}
\eeq
and
\beq
\begin{split}
&\;\|e_r \cdot \na (\G* g_{\p,1}(X_1))|_{\g_1} -e_r \cdot \na (\G* g_{\p,2}(X_2))|_{\g_2}\|_{\dot{W}^{1,p}}\\
\leq &\;Cr(\D m_\b+\D M_\b).
\end{split}
\eeq
Note that this term is not of mean zero on $\BT$, so we have to bound its $L^\infty$-norm and $\dot{W}^{1,p}$-seminorm in order to prove \eqref{eqn: difference of R h_i}.
Finally,
\beq
\begin{split}
&\;\left\|\f{f_1'}{f_1}e_\th\cdot \na (\G* g_{\p,1}(X_1))|_{\g_1}-\f{f_2'}{f_2}e_\th\cdot \na (\G* g_{\p,2}(X_2))|_{\g_2}\right\|_{L^\infty}\\
\leq &\;C\left\|\f{f_1'}{f_1}-\f{f_2'}{f_2}\right\|_{L^\infty}\|e_\th\cdot \na (\G* g_{\p,1}(X_1))|_{\g_1}\|_{L^\infty}\\
&\;+C\left\|\f{f_2'}{f_2}\right\|_{L^\infty}\|(e_\th\cdot \na (\G* g_{\p,1}(X_1))|_{\g_1}-e_\th\cdot \na (\G* g_{\p,2}(X_2))|_{\g_2})\|_{L^\infty}\\
%
%
\leq &\;Cr(m_{0,1}+m_{0,2}+ M_{0,1})(\D m_\b+\D M_\b),
\end{split}
\eeq
and by proceeding as in \eqref{eqn: W 1p difference of a component of growth potential crude form}-\eqref{eqn: W 1p difference of a component of growth potential},
%
%
%
\beq
\begin{split}
&\;\left\|\f{f_1'}{f_1}e_\th\cdot \na (\G* g_{\p,1}(X_1))|_{\g_1}-\f{f_2'}{f_2}e_\th\cdot \na (\G* g_{\p,2}(X_2))|_{\g_2}\right\|_{\dot{W}^{1,p}}\\
\leq &\;C(p,\e,\mu,\nu,G) r\|h_1''-h_2''\|_{L^p}(\|h_1'\|_{\dot{C}^\b}+\|H_1'\|_{\dot{C}^\b}+(m_{0,1}+M_{0,1})(1+\d R^2)^{1/2})\\
&\;+Cr(m_{\b,1}+M_{0,1}+\|h_2''\|_{L^p})(\D m_\b+\D M_\b).
\end{split}
\eeq
Combining these estimates with \eqref{eqn: new representation of tilde R h_i}, we use the fact $|c_*|\leq Cr$ by Lemma \ref{lem: estimate for p_*} to prove \eqref{eqn: difference of R h_i}.

To show \eqref{eqn: difference of R H big_i}, we derive as before.
%
%
%
\beq
\begin{split}
&\;\left\|\f{1}{F_1}\tilde{\g}_1'\cdot \CK_{\tilde{\g}_1}\tilde{\CR}_{\phi_1'}-\f{1}{F_2}\tilde{\g}_2'\cdot \CK_{\tilde{\g}_2}\tilde{\CR}_{\phi_2'}\right\|_{W^{1,p}}\\
\leq &\;C(p,\e,\mu,\nu,G)R^{-1}r^2  \|H_1''-H_2''\|_{L^p}\\
&\;\quad \cdot (\|h_1'\|_{\dot{C}^\b}+\|H_1'\|_{\dot{C}^\b}+(m_{0,1}+M_{0,1}+M_{0,2})(1+\d R^2)^{1/2})\\
&\;+CR^{-1}r^2(\|H_1'-H_2'\|_{\dot{C}^\b}+\D m_\b+\D M_0) (1+\|h_1''\|_{L^p}+
\|H_1''\|_{L^p}+\|H_2''\|_{L^p})
\end{split}
\eeq
and
\beq
\begin{split}
&\;\left\|\left(\f{1}{F_1}\tilde{\g}_1'\cdot \CK_{\tilde{\g}_1}F_1'-\f{1}{2R}\CH F_1'\right)-\left(\f{1}{F_2}\tilde{\g}_2'\cdot \CK_{\tilde{\g}_2}F_2'-\f{1}{2R}\CH F_2'\right)\right\|_{W^{1,p}}\\
\leq &\;C(p,\b)\|H_1''-H_2''\|_{L^p} (M_{0,1}+M_{0,2}+\|H_1'\|_{\dot{C}^\b}+\|H_2'\|_{\dot{C}^\b})\\
&\;+C (\|H_1''\|_{L^p}+\|H_2''\|_{L^p})(\|H_1'-H_2'\|_{\dot{C}^\b}+ \D M_0).
\end{split}
\eeq
By \eqref{eqn: difference of tilde c outer interface segregated} and Lemmas \ref{lem: L inf estimate of difference same source at outer interface}-\ref{lem: W 1p estimate of derivative of growth potential outer} and \ref{lem: bounds for generalized source term},
\beq
\begin{split}
&\;\|e_r \cdot \na (\G* g_{\P,1}(X_1))|_{\tilde{\g}_1} -e_r \cdot \na (\G* g_{\P,2}(X_2))|_{\tilde{\g}_2}\|_{L^\infty}\\
&\;+\|e_r \cdot \na (\G* g_{\P,1}(X_1))|_{\tilde{\g}_1} -e_r \cdot \na (\G* g_{\P,2}(X_2))|_{\tilde{\g}_2}\|_{\dot{W}^{1,p}}\\
\leq &\;CR^{-1}r^2(\D m_\b+\D M_\b),
\end{split}
\eeq
\beq
\begin{split}
&\;\left\|\f{F_1'}{F_1}e_\th\cdot \na (\G* g_{\P,1}(X_1))|_{\tilde{\g}_1}-\f{F_2'}{F_2}e_\th\cdot \na (\G* g_{\P,2}(X_2))|_{\tilde{\g}_2}\right\|_{L^\infty}\\
\leq
&\;C R^{-1}r^2(m_{0,1}+M_{0,1}+M_{0,2})(\D m_\b+\D M_\b),
\end{split}
\eeq
and
\beq
\begin{split}
&\;\left\|\f{F_1'}{F_1}e_\th\cdot \na (\G* g_{\P,1}(X_1))|_{\tilde{\g}_1}-\f{F_2'}{F_2}e_\th\cdot \na (\G* g_{\P,2}(X_2))|_{\tilde{\g}_2}\right\|_{\dot{W}^{1,p}}\\
\leq &\;C(p,\e,\mu,\nu,G)R^{-1}r^2\|H_1''-H_2''\|_{L^p}\\
&\;\quad \cdot (\|h_1'\|_{\dot{C}^\b}+\|H_1'\|_{\dot{C}^\b}+(m_{0,1}+M_{0,1})(1+\d R^2)^{1/2})\\
&\;+CR^{-1}r^2(m_{0,1}+M_{0,1}+\|H_2''\|_{L^p})(\D m_\b+\D M_\b).
\end{split}
\eeq
Combining these estimates and the fact $\tilde{c}_*\leq C(\mu,\nu,G)R^{-1}r^2$ with \eqref{eqn: new representation of tilde R H big_i} yields \eqref{eqn: difference of R H big_i}.
\end{proof}
\end{lem}

\subsection{Proof of the uniqueness}
Now we are ready to prove uniqueness.
\begin{proof}[Proof of Theorem \ref{thm: uniqueness}]
In this proof, we always assume that the constant $C$ has the dependence $C = C(p,\e,\mu,\nu,r,R, G)$ unless otherwise stated.

As stated at the beginning of this section, suppose there are two solutions $f_i$ and $F_i$ $(i = 1,2)$ of \eqref{eqn: equation for inner interface new notation}-\eqref{eqn: initial condition} with regularity and estimates given in Theorem \ref{thm: local well-posedness}.
By \eqref{eqn: reformulation of equation for h_i} and \eqref{eqn: reformulation of equation for H_i big}, $(h_1-h_2)$ and $(H_1-H_2)$ solve
\begin{align}
\pa_t (h_1-h_2)=&\;-\f{Ac_*}{r}(-\D)^{1/2}(h_1-h_2) +\f{1}{r}(\tilde{\CR}_{h_1}-\tilde{\CR}_{h_2}), \label{eqn: equation for difference of h}\\
\pa_t (H_1-H_2)=&\;\f{\tilde{c}_*}{R}(-\D)^{1/2}(H_1-H_2) +\f{1}{R}(\tilde{\CR}_{H_1}-\tilde{\CR}_{H_2}),\label{eqn: equation for difference of H big}
\end{align}
with initial condition $(h_1 - h_2)|_{t = 0} = (H_1 - H_2)|_{t = 0} = 0$.

Let $T_0\in (0,T)$, $T_0<1$ to be chosen.
By virtue of Lemma 
\ref{lem: L p regularity theory of fractional heat equation} and Lemma \ref{lem: spatial Holder estimate for solution with Sobolev source term}, with $\alpha = 1-\f2p$,
\beq
\begin{split}
&\;\|h_1''-h_2''\|_{L^p_{[0,T_0]}L^p(\BT)}+\|H_1''-H_2''\|_{L^p_{[0,T_0]}L^p(\BT)}\\
&\;+\|h_1'-h_2'\|_{C_{[0,T_0]}\dot{C}^\alpha(\BT)}+\|H_1'-H_2'\|_{C_{[0,T_0]}\dot{C}^\alpha(\BT)}\\
\leq &\;C(p, \mu, \nu, G)\left(\f{r}{|Ac_*|}\cdot \f1r\|\tilde{\CR}_{h_1}-\tilde{\CR}_{h_2}\|_{L^p_{[0,T_0]}\dot{W}^{1,p}(\BT)}+\f{R}{|\tilde{c}_*|}\cdot \f1R\|\tilde{\CR}_{H_1}-\tilde{\CR}_{H_2}\|_{L^p_{[0,T_0]}\dot{W}^{1,p}(\BT)}\right).
\end{split}
\eeq
Here we first applied change of time variables to normalize the coefficients of fractional Laplacians in \eqref{eqn: equation for difference of h} and \eqref{eqn: equation for difference of H big}, and then applied Lemma 
\ref{lem: L p regularity theory of fractional heat equation} and Lemma \ref{lem: spatial Holder estimate for solution with Sobolev source term} to obtain these estimates.
To fulfill the condition of Lemma \ref{lem: spatial Holder estimate for solution with Sobolev source term}, we need
\beq
T_0\leq \min\left\{\f{r}{Ac_*},\f{R}{|\tilde{c}_*|}\right\}.
\eeq
Note that by Lemma \ref{lem: estimate for p_*}, the right hand side is bounded from below by some constant depending only on $\mu$, $\nu$ and $G$.

On the other hand, by Sobolev embedding (in space) and H\"{o}lder's inequality (in time)
\beq
\begin{split}
&\;\|h_1-h_2\|_{C_{[0,T_0]}L^\infty(\BT)}+\|H_1-H_2\|_{C_{[0,T_0]}L^\infty(\BT)}\\
\leq &\;C(p)T_0^{1-\f1p}\left( \f1r\|\tilde{\CR}_{h_1}-\tilde{\CR}_{h_2}\|_{L^p_{[0,T_0]}W^{1,p}(\BT)}
+\f1R\|\tilde{\CR}_{H_1}-\tilde{\CR}_{H_2}\|_{L^p_{[0,T_0]}W^{1,p}(\BT)}\right).
\end{split}
\eeq

Denote
\beq
\begin{split}
\mathcal{N}(T_0):= &\;\|h_1''-h_2''\|_{L^p_{[0,T_0]}L^p(\BT)}+\|H_1''-H_2''\|_{L^p_{[0,T_0]}L^p(\BT)}\\
&\;+\|h_1'-h_2'\|_{C_{[0,T_0]}\dot{C}^\b(\BT)}+\|H_1'-H_2'\|_{C_{[0,T_0]}\dot{C}^\b(\BT)}\\
&\;+\d^{-1}\|h_1-h_2\|_{C_{[0,T_0]}L^\infty(\BT)}+\d^{-1}\|H_1-H_2\|_{C_{[0,T_0]}L^\infty(\BT)}.
\end{split}
\eeq
By interpolation and Lemma \ref{lem: difference of remainder terms in h and H equations}, with $\th = (1-\f1p)\cdot \f{\alpha-\b}{1+\alpha}$,
\beq
\begin{split}
&\;\mathcal{N}(T_0)\\
\leq &\;C(p,\mu,\nu,G)\left(\f{r}{|Ac_*|}\cdot \f1r\|\tilde{\CR}_{h_1}-\tilde{\CR}_{h_2}\|_{L^p_{[0,T_0]}\dot{W}^{1,p}(\BT)}+\f{R}{|\tilde{c}_*|}\cdot \f1R\|\tilde{\CR}_{H_1}-\tilde{\CR}_{H_2}\|_{L^p_{[0,T_0]}\dot{W}^{1,p}(\BT)}\right)\\
&\;+CT_0^{\th}\left( \f1r\|\tilde{\CR}_{h_1}-\tilde{\CR}_{h_2}\|_{L^p_{[0,T_0]}\dot{W}^{1,p}(\BT)}+ \f1R\|\tilde{\CR}_{H_1}-\tilde{\CR}_{H_2}\|_{L^p_{[0,T_0]}\dot{W}^{1,p}(\BT)}\right)\\
\leq &\;\left[C(p,\e,\mu,\nu,G)\cdot \f{r}{|c_*|}+C_1T_0^{\th}\right]\mathcal{N}(T_0)\\
&\;\cdot \sup_{t\in [0,T_0]} (\|h_1'\|_{\dot{C}^\b}+\|h_2'\|_{\dot{C}^\b}+\|H_1'\|_{\dot{C}^\b}+\|H_2'\|_{\dot{C}^\b}\\
&\;\qquad \quad +(m_{0,1}+m_{0,2}+M_{0,1}+M_{0,2})(1+\d R^2)^{1/2})\\
&\;+C_2\mathcal{N}(T_0) (T_0^{1/p}+\|h_1''\|_{L^p_{[0,T_0]}L^p}+\|h_2''\|_{L^p_{[0,T_0]}L^p}+\|H_1''\|_{L^p_{[0,T_0]}L^p}+\|H_2''\|_{L^p_{[0,T_0]}L^p}).
\end{split}
\label{eqn: inequality for N(T_0)}
\eeq
Here the constants $C_1$ and $C_2$ have the same dependence as $C$ introduced above.

Now we take $T_0$ such that $C_1T_0^\th\leq \f12$ and
\beq
C_2(T_0^{1/p}+\|h_1''\|_{L^p_{[0,T_0]}L^p}+\|h_2''\|_{L^p_{[0,T_0]}L^p}+\|H_1''\|_{L^p_{[0,T_0]}L^p}+\|H_2''\|_{L^p_{[0,T_0]}L^p})\leq \f12.
\eeq
Such $T_0$ relies on $p$, $\e$, $\mu$, $\nu$, $r$, $R$, $G$ as well as the fixed solutions $h_i$ and $H_i$.
Then \eqref{eqn: inequality for N(T_0)} becomes
\beq
\begin{split}
\mathcal{N}(T_0)\leq &\;\left[C(p,\e,\mu,\nu,G)\cdot \f{r}{|c_*|}+1\right]\mathcal{N}(T_0)\\
&\; \cdot \sup_{t\in [0,T_0]} (\|h_1'\|_{\dot{C}^\b}+\|h_2'\|_{\dot{C}^\b}+\|H_1'\|_{\dot{C}^\b}+\|H_2'\|_{\dot{C}^\b}\\
&\;\qquad \quad +(m_{0,1}+m_{0,2}+M_{0,1}+M_{0,2})(1+\d R^2)^{1/2})\\
\leq &\;C(p,\e,\mu,\nu,G,r/|c_*|,\d R^2 
)M\cdot \mathcal{N}(T_0).
\end{split}
\eeq
In the last inequality, we used the estimate \eqref{eqn: bounds for the local solution}.
If we assume $M$ to be suitably small, depending only on $p$, $\e$, $\mu$, $\nu$, $G$, $r/|c_*|$ and $\d R^2$
, we obtain that $\mathcal{N}(T_0) = 0$.
Note that here the smallness of $M$ has no additional dependence on other parameters compared to that in the proof of existence of local solutions.

We can continue this process starting from $t = T_0$ and find a second time interval $[T_0, T_0+T_1]$ on which uniqueness holds.
By repeating this argument for finitely many times (see \eqref{eqn: bounds for the local solution L^p W^2,p} and the way we chose $T_0$ above), we can prove the uniqueness of local solution on $[0,T]$.
\end{proof}

\newpage
\appendix
\section{Some Auxiliary Estimates}

\subsection{Estimates for the Poisson kernel and its conjugate}
\begin{lem}\label{lem: properties of Poisson kernel}
Let Poisson kernel $P$ on the 2-D unit disc and its conjugate $Q$ be defined as in \eqref{eqn: def of P} and \eqref{eqn: def of Q}, respectively.
\begin{enumerate}
\item Let $\CH_\xi$ denote the Hilbert transform on $\BT$ with respect to $\xi$.
Then for $s\not = 1$,
\beq
Q(s,\xi)= \mathrm{sgn}(1-s) \CH_{\xi}P(s,\xi).
\label{eqn: Q is the hilbert transform of P}
\eeq
  \item For all $\xi\in \BT$ and all $s \in [0,2]$,
\beq
|P(s,\xi)|+|Q(s,\xi)|\leq C(|1-s|^2+\xi^2)^{-1/2}.\label{eqn: bound for P and Q}
\eeq
\item For derivatives of $P$ and $Q$, we have
\begin{align}
\left|\f{\pa P}{\pa s}(s,\xi)\right|+\left|\f{\pa Q}{\pa \xi}(s,\xi)\right| \leq &\;C((1-s)^2+\xi^2)^{-1},
\label{eqn: bound for derivative of Poisson kernel}\\
\left|\f{\pa P}{\pa \xi}(s,\xi)\right|+\left|\f{\pa Q}{\pa s}(s,\xi)\right| \leq &\; C|\sin\xi|((1-s)^2+\xi^2)^{-3/2},
\end{align}
and
\beq
\left|\f{\pa^2 P}{\pa s^2}(s,\xi) \right|+\left|\f{\pa^2 P}{\pa \xi \pa s}(s,\xi) \right|+\left|\f{\pa^2 Q}{\pa s^2}(s,\xi)\right| \leq  C((1-s)^2+\xi^2)^{-3/2}.
\label{eqn: bound for mix 2nd derivative of P}
\eeq
Moreover,
\beq
\f{\pa P}{\pa \xi}(s,\xi) =-s\f{\pa Q}{\pa s},\quad
\f{\pa Q}{\pa \xi}(s,\xi) =s\f{\pa P}{\pa s}.
\label{eqn: conjugacy between P and Q}
\eeq
\end{enumerate}
\begin{proof}
\eqref{eqn: Q is the hilbert transform of P} can be proved by calculating Fourier transforms of $P(s,\cdot)$ and $Q(s,\cdot)$.

For any $s\geq 0$, 
\beq
1+s^2-2s\cos\xi = (1-s\cos\xi)^2+(s\sin\xi)^2 = (s-\cos\xi)^2+\sin^2\xi\geq 0.
\label{eqn: rewriting denominator of Poisson kernel}
\eeq
If $\cos\xi\geq \f{1}{2}$,
\beq
\begin{split}
1+s^2-2s\cos\xi
= &\; (1+s^2)(1-\cos\xi)+\cos\xi(1-s)^2\\
\geq &\;C(|\xi|^2+|1-s|^2).
\end{split}
\label{eqn: lower bound for denominator of Poisson kernel case 1}
\eeq
Otherwise,
\beq
1+s^2-2s\cos\xi\geq C(1+s^2)\geq C(|\xi|^2+|1-s|^2).
\label{eqn: lower bound for denominator of Poisson kernel case 2}
\eeq
Then \eqref{eqn: bound for P and Q} follows easily.

Finally, we calculate that
\begin{align}
\f{\pa P}{\pa s}(s,\xi) = &\;\f{2(1+s^2)\cos\xi -4s}{(1+s^2-2s\cos\xi)^2}= \f{2\cos\xi}{1+s^2-2s\cos\xi}-\f{4s\sin^2\xi}{(1+s^2-2s\cos\xi)^2},
\label{eqn: spatial derivative of poisson kernel}\\
\f{\pa Q}{\pa s}(s,\xi) =&\;\f{2(1-s^2)\sin\xi}{(1+s^2 -2s\cos\xi)^{2}},
\label{eqn: spatial derivative of conjugate poisson kernel}\\
\f{\pa^2 P}{\pa \xi \pa s}(s,\xi) = &\;-\f{2(1+s^2)\sin \xi}{(1+s^2-2s\cos \xi)^2} - \f{\pa P}{\pa s}\cdot \f{4s\sin \xi}{1+s^2-2s\cos\xi},\\
\f{\pa^2 P}{\pa s^2}(s,\xi) = &\;\f{4s\cos \xi-4}{(1+s^2-2s\cos \xi)^2} - \f{8(s-\cos\xi)((1+s^2)\cos\xi-2s)}{(1+s^2-2s\cos\xi)^3},
\end{align}
and
\begin{align}
\f{\pa P}{\pa \xi}(s,\xi) =&\;-s\f{\pa Q}{\pa s},\\
\f{\pa Q}{\pa \xi}(s,\xi) =&\;s\f{\pa P}{\pa s},
\label{eqn: angular derivative of conjugate poisson kernel}\\
\f{\pa^2 Q}{\pa s^2}(s,\xi) =&\;-\f{1}{s}\left(\f{\pa^2 P}{\pa \xi\pa s}+\f{\pa Q}{\pa s}\right).
\label{eqn: spatial 2nd derivative of conjugate poisson kernel}
\end{align}
Then \eqref{eqn: bound for derivative of Poisson kernel}-
\eqref{eqn: conjugacy between P and Q} follow.
\end{proof}
\end{lem}

\subsection{Some Calder\'{o}n-commutator-type estimates}
In this part we shall establish some Calder\'{o}n-commutator-type estimates used in Section \ref{sec: estimates for singular integral operators}.
Recall that
\beq
\D f (\th): = \f{f(\th+\xi)-f(\xi)}{2\sin\f{\xi}{2}}.
\eeq

\begin{lem}\label{lem: Lp bound for multi-term commutator}
Let $\mathbf{k} = (k_1,\cdots,k_n)$ be a multi-index of length $n\in \mathbb{Z}_+$.
Assume $h_1,\cdots, h_n\in W^{1,\infty}(\BT)$ and $\p\in L^p(\BT)$ for some $p\in [2,\infty)$.
Define
\begin{align}
M_{\mathbf{k},\p}(\th) = &\;\mathrm{p.v.}\int_\BT \prod_{i = 1}^n(\D h_i)^{k_i}\cdot\f{\p(\th+\xi)}{2\tan\f{\xi}{2}}\,d\xi,
\label{eqn: def of multi-term Calderon commutator tan}\\
N_{\mathbf{k},\p}(\th) = &\;\mathrm{p.v.}\int_\BT \prod_{i = 1}^n(\D h_i)^{k_i}\cdot\f{\p(\th+\xi)}{2\sin\f{\xi}{2}}\,d\xi.
\label{eqn: def of multi-term Calderon commutator sin}
\end{align}
Then 
\beq
\|M_{\mathbf{k},\p}\|_{L^p}+\|N_{\mathbf{k},\p}\|_{L^p}\leq C_*^{|\mathbf{k}|+2}\|\psi\|_{L^p}\prod_{i=1}^n\|h_i'\|_{L^\infty}^{k_i},
\label{eqn: L^p bound for multi-term Calderon commutator}
\eeq
where $C_*$ is a universal constant depending only on $p$.
Here $|\mathbf{k}| := \sum_{i = 1}^n k_i$.
\begin{proof}
The proof essentially follows the classic argument of $L^p$-boundedness of the Calder\'{o}n commutator 
\cite[$\S\,9.3$]{duoandikoetxea2001fourier}.
For completeness, we elaborate it as follows.

First we notice that $\sin (\xi/2)$ is not continuous on $\BT$ at $\pm \pi$.
For this technical reason, with abuse of notations, we introduce an even cut-off function $\eta\in C_0^\infty([-2,2])$, such that $\eta \equiv 1$ on $[-1,1]$, $\eta\in [0,1]$ on $[-2,2]$, and $|\eta'| \leq C$.
Write \eqref{eqn: def of multi-term Calderon commutator tan} as
\beq
M_{\mathbf{k},\p} =  \mathrm{p.v.}\int_\BT \prod_{i = 1}^n (\D h_i)^{k_i}\cdot\f{\p(\th+\xi)}{2\tan\f{\xi}{2}}[\eta(\xi)+(1-\eta(\xi))]\,d\xi =: M_{\mathbf{k},\p}^{(1)}+M_{\mathbf{k},\p}^{(2)}.
\label{eqn: decomposition of multi-term commutator tan}
\eeq
It is straightforward to bound $M_{\mathbf{k},\p}^{(2)}$ as it involves no singularity,
\beq
\|M_{\mathbf{k},\p}^{(2)}\|_{L^p}\leq CC^{|\mathbf{k}|}_1\|\p\|_{L^p}\prod_{i = 1}^n\|h_i'\|_{L^\infty}^{k_i}.
\label{eqn: Lp bound form Kkl2}
\eeq
Here $C_1 = \pi/2$ comes from the fact that
\beq
\left|2\sin\f{\xi}{2}\right|^{-1}\leq C_1|\xi|^{-1}\quad \mbox{ on }\BT.
\eeq
To derive an $L^p$-bound for $M_{\mathbf{k},\p}^{(1)}$, we first show that $M_{\mathbf{k},1}^{(1)}\in BMO$ by mathematical induction. 

\setcounter{step}{0}
\begin{step}
For $\mathbf{k} = \mathbf{0}$, $M_{\mathbf{0},1}^{(1)} = -\pi \mathcal{H}\eta(0) = 0$ since $\eta$ is even.

\end{step}
\begin{step}
Suppose for some $N\geq 1$ and any multi-index $\mathbf{k}$ such that $|\mathbf{k}|\leq N-1$, we have shown that $M_{\mathbf{k},1}^{(1)}\in BMO$ and, with some constant $C_*$ that will be specified later,
\beq
\|M_{\mathbf{k},1}^{(1)}\|_{BMO}\leq C_*^{|\mathbf{k}|+1}\prod_{i = 1}^n\|h_i'\|_{L^\infty}^{k_i}.
\label{eqn: induction hypo}
\eeq
It is known that the map $\p\mapsto M_{\mathbf{k},\p}^{(1)}$ is associated with the kernel
\beq
\prod_{ i =1}^n\left(\f{h_i(x)-h_i(y)}{2\sin\f{x-y}{2}}\right)^{k_i} \cdot \f{\eta(x-y)}{2\tan\f{x-y}{2}},
\label{eqn: kernel associated with M kl1}
\eeq
which is a standard anti-symmetric kernel, vanishing whenever $|x-y|>2$.
It can be naturally understand as a kernel on
$\mathbb{R}$ with a bound similar to \eqref{eqn: induction hypo}.
Hence, by the $T1$ Theorem, it is $(2,2)$-bounded.
Its operator norm depends linearly \cite[$\S\,9.3$]{duoandikoetxea2001fourier} on the constant in \eqref{eqn: induction hypo} and the kernel constant of \eqref{eqn: kernel associated with M kl1}, which is bounded by
\beq
CC_1^{|\mathbf{k}|+1}(|\mathbf{k}|+1)\prod_{ i =1}^n \|h_i'\|_{L^\infty}^{k_i}.
\label{eqn: constant of the standard kernel}
\eeq
This further implies that \cite[Theorem 6.6]{duoandikoetxea2001fourier} for all $\mathbf{k}$ satisfying $|\mathbf{k}|\leq N-1$, and $\p\in L^\infty$,
\beq
\|M_{\mathbf{k},\p}^{(1)}\|_{BMO}\leq  C(C_1^{|\mathbf{k}|+1}(|\mathbf{k}|+1)+C_*^{|\mathbf{k}|+1})\|\p\|_{L^\infty}\prod_{i = 1}^n\|h_i'\|_{L^\infty}^{k_i}.
\label{eqn: L inf to BMO boundedness}
\eeq

Now consider the case when $|\mathbf{k}|=N$.
Observe that
\beq
\left(\f{1}{2\sin\f{\xi}{2}}\right)^{|\mathbf{k}|}\f{1}{2\tan\f{\xi}{2}} = - \f{1}{|\mathbf{k}|}\cdot\f{d}{d\xi}\left(\f{1}{2\sin\f{\xi}{2}}\right)^{|\mathbf{k}|}.
\eeq
We integrate by parts in $M_{\mathbf{k},1}^{(1)}$. 
For almost all $\th\in \BT$,
\beq
\begin{split}
M_{\mathbf{k},1}^{(1)}(\th) = &\;\f{1}{|\mathbf{k}|}\mathrm{p.v.}\int_{[-2,2]} \left(\f{1}{2\sin\f{\xi}{2}}\right)^{|\mathbf{k}|}\,d\left[\prod(h_i(\th+\xi)-h_i(\th))^{k_i} \eta(\xi)\right]\\
= &\;\f{1}{|\mathbf{k}|}\int_{[-2,2]} \prod_{i = 1}^n(\D h_i)^{k_i} \cdot \eta'(\xi)\,d\xi\\
&\;+\sum_{i = 1}^n\f{k_i}{|\mathbf{k}|}\mathrm{p.v.}\int_{[-2,2]} (\D h_1)^{k_1}\cdots (\D h_i)^{k_i-1}\cdots (\D h_n)^{k_n} \cdot \f{\eta(\xi)h_i'(\th+\xi)}{2\sin\f{\xi}{2}}\,d\xi\\
=:&\;M_{\mathbf{k},1}^{(1,0)}+\sum_{ i = 1}^n M_{\mathbf{k},1}^{(1,i)}.
\end{split}
\label{eqn: intergration by parts in proving commutator estimate}
\eeq
Indeed, this can be rigorously justified by the fact that $h_i$ are differentiable almost everywhere.
It is straightforward to derive that
\beq
\|M_{\mathbf{k},1}^{(1,0)}\|_{L^\infty}\leq C C_1^{|\mathbf{k}|}\prod_{i = 1}^n\|h_i'\|_{L^\infty}^{k_i}.
\eeq
On the other hand, by \eqref{eqn: L inf to BMO boundedness},
\beq
\begin{split}
&\;\|M_{\mathbf{k},1}^{(1,i)}\|_{BMO}\\
\leq &\;\frac{k_i}{|\mathbf{k}|}\|M_{(k_1,\cdots,k_i-1,\cdots, k_n), h_i'}^{(1)}\|_{BMO}+ \frac{k_i}{|\mathbf{k}|}C_1^{|\mathbf{k}|-1}\prod_{j = 1}^n\|h_j'\|_{L^\infty}^{k_j}\left\|\frac{1}{2\sin\f{\xi}{2}}-\frac{1}{2\tan\f{\xi}{2}}\right\|_{L^\infty([-2,2])}\\
\leq &\;\f{Ck_i}{|\mathbf{k}|}(C_1^{|\mathbf{k}|}|\mathbf{k}|+C_*^{|\mathbf{k}|})\prod_{j = 1}^n\|h_j'\|_{L^\infty}^{k_j}.
\end{split}
\eeq
%
%
Hence, with some universal constant $C$,
\beq
\|M_{\mathbf{k},1}^{(1)}\|_{BMO}\leq C(C_1^{|\mathbf{k}|}|\mathbf{k}|+C_*^{|\mathbf{k}|})\prod_{i = 1}^n\|h_i'\|_{L^\infty}^{k_i}.
\eeq
Now assuming $C_*$ sufficiently large but still universal, 
such that
\beq
C\left[\left(\f{C_1}{C_*}\right)^{|\mathbf{k}|}|\mathbf{k}|+1\right]\leq C_*,
\eeq
we conclude 
with \eqref{eqn: induction hypo} for $|\mathbf{k}| = N$.
By induction, \eqref{eqn: induction hypo} holds for all multi-indices $\mathbf{k}$.
\end{step}

To this end, we argue as in \eqref{eqn: kernel associated with M kl1}-\eqref{eqn: L inf to BMO boundedness} to find that 
$\p \mapsto M_{\mathbf{k},\p}^{(1)}$ is bounded from $L^2$ to $L^2$, and also from
$L^\infty$ to $BMO$.
By interpolation, it is $(p,p)$-bounded as well.
In particular,
\beq
\|M_{\mathbf{k},\p}^{(1)}\|_{L^p}\leq C_p(C_1^{|\mathbf{k}|+1}(|\mathbf{k}|+1)+C_*^{|\mathbf{k}|+1})\|\psi\|_{L^p}\prod_{ i =1}^n\|h_i'\|_{L^\infty}^{k_i}.
\label{eqn: Lp bound form Kkl1}
\eeq
Combining \eqref{eqn: Lp bound form Kkl2} and \eqref{eqn: Lp bound form Kkl1} yields a bound for $\|M_{\mathbf{k},\p}\|_{L^p}$ that has the same form as in \eqref{eqn: Lp bound form Kkl1}. 
A bound for $\|N_{\mathbf{k},\p}\|_{L^p}$ can be derived easily since $(M_{\mathbf{k},\p}-N_{\mathbf{k},\p})$ is an integral with no singularity.

Assuming $C_*$ to be even larger if needed, we obtain the desired estimate from \eqref{eqn: Lp bound form Kkl1}.
\end{proof}
\end{lem}

\begin{lem}\label{lem: Lp bound for extra-term commutator}
Let $\mathbf{k} = (k_1,\cdots,k_n)$ be a multi-index of length $n\in \mathbb{Z}_+$.
With $p\in [2,\infty)$, assume that $h_1,\cdots, h_n\in W^{2,p}(\BT)$, and $h_{n+1},\psi\in W^{1,p}(\BT)$.
Define
\begin{align}
\tilde{M}_{\mathbf{k},\p}(\th) = &\;\mathrm{p.v.}\int_\BT \prod_{i =1}^n(\D h_i)^{k_i}\cdot \D h_{n+1} \cdot\f{\p(\th+\xi)}{2\tan\f{\xi}{2}}\,d\xi,
\label{eqn: def of extra-term Calderon commutator tan}\\
\tilde{N}_{\mathbf{k},\p}(\th) = &\;\mathrm{p.v.}\int_\BT \prod_{i = 1}^n(\D h_i)^{k_i}\cdot \D h_{n+1} \cdot\f{\p(\th+\xi)}{2\sin\f{\xi}{2}}\,d\xi.
\label{eqn: def of extra-term Calderon commutator sin}
\end{align}
Then 
\beq
\begin{split}
&\;\|\tilde{M}_{\mathbf{k},\p}\|_{L^p}+\|\tilde{N}_{\mathbf{k},\p}\|_{L^p}\\
\leq &\;C_{**}^{|\mathbf{k}|+1}(\|h_{n+1}'\|_{L^p}\|\p\|_{L^\infty}+\|h_{n+1}\|_{L^\infty}\|\p'\|_{L^p})\prod_{ i =1}^n\|h_i'\|_{L^\infty}^{k_i}\\
&\;+C_{**}^{|\mathbf{k}|+1}\|h_{n+1}\|_{L^\infty}\|\p\|_{L^\infty}\sum_{i = 1}^n \|h_1'\|_{L^\infty}^{k_1}\cdots \|h_i'\|_{L^\infty}^{k_i-1}\cdots \|h_n'\|_{L^\infty}^{k_n}\cdot \mathds{1}_{\{k_i>0\}}\|h_i''\|_{L^p},
\label{eqn: Lp bound for extra-term Calderon commutator}
\end{split}
\eeq
where $C_{**}$ is a universal constant depending only on $p$.
\begin{proof}
We shall prove \eqref{eqn: Lp bound for extra-term Calderon commutator} by induction.
It suffices to prove it for $h_{n+1}$ and $\p$ being smooth.

\setcounter{step}{0}
\begin{step}
Consider $\mathbf{k} = \mathbf{0}$.
Note that even in this simple case, the estimate \eqref{eqn: Lp bound for extra-term Calderon commutator} does not trivially follow from Lemma \ref{lem: Lp bound for multi-term commutator}.

By integration by parts as in \eqref{eqn: intergration by parts in proving commutator estimate}, 
\beq
\begin{split}
&\;\tilde{M}_{\mathbf{0},\p}(\th)\\
= &\;\mathrm{p.v.}\int_\BT  \f{1}{2\sin\f{\xi}{2}}\,d[(h_{n+1}(\th+\xi)- h_{n+1}(\th))\p(\th+\xi)]\\
&\;-[(h_{n+1}(\th+\pi)- h_{n+1}(\th))\p(\th+\pi)]\\
=&\;\int_\BT  \f{1-\cos\f{\xi}{2}}{2\sin\f{\xi}{2}}[h_{n+1}'(\th+\xi)\p(\th+\xi)+(h_{n+1}(\th+\xi)- h_{n+1}(\th))\p'(\th+\xi)]\,d\xi\\
&\;+\mathrm{p.v.}\int_\BT  \f{1}{2\tan\f{\xi}{2}}[h_{n+1}'(\th+\xi)\p(\th+\xi)+(h_{n+1}(\th+\xi)- h_{n+1}(\th))\p'(\th+\xi)]\,d\xi\\
&\;-[(h_{n+1}(\th+\pi)- h_{n+1}(\th))\p(\th+\pi)].
\end{split}
\eeq
By Sobolev embedding and $L^p$-boundedness of the Hilbert transform,
\beq
\begin{split}
\|\tilde{M}_{\mathbf{0},\p}\|_{L^p}\leq C(\|h_{n+1}'\|_{L^p}\|\p\|_{L^\infty}+\|h_{n+1}\|_{L^\infty}\|\p'\|_{L^p}).
\end{split}
\label{eqn: bound for tilde M 00}
\eeq
Since
\beq
|\tilde{N}_{\mathbf{0},\p}-\tilde{M}_{\mathbf{0},\p}|\leq C\int_\BT|h_{n+1}(\th+\xi)-h_{n+1}(\th)||\p(\th+\xi)|\,d\xi,
\eeq
it is easy to show that $\tilde{N}_{\mathbf{0},\p}$ satisfies the same estimate as \eqref{eqn: bound for tilde M 00}.
\end{step}

\begin{step}
Suppose \eqref{eqn: Lp bound for extra-term Calderon commutator} holds for all multi-indices $\mathbf{k}$ satisfying $|\mathbf{k}|\leq N-1$, where $C_{**}>0$ is some constant to be chosen later.
Then consider the case with $|\mathbf{k}| = N$.
By integration by parts as in \eqref{eqn: intergration by parts in proving commutator estimate}, for almost all $\th\in \BT$,
\beq
\begin{split}
&\;\tilde{M}_{\mathbf{k},\p}(\th)\\
= &\;\sum_{i = 1}^n\f{k_i}{|\mathbf{k}|+1}\mathrm{p.v.}\int_\BT  (\D h_1)^{k_1}\cdots (\D h_i)^{k_i-1}\cdots (\D h_n)^{k_n} \cdot \D h_{n+1}\f{h_i'(\th+\xi)\p(\th+\xi)}{2\sin\f{\xi}{2}}\,d\xi\\
&\;+\f{1}{|\mathbf{k}|+1}\mathrm{p.v.}\int_\BT  \prod_{i = 1}^n(\D h_i)^{k_i}\f{h_{n+1}'(\th+\xi)\p(\th+\xi)}{2\sin\f{\xi}{2}}\,d\xi\\
&\;+\f{1}{|\mathbf{k}|+1}\mathrm{p.v.}\int_\BT  \prod_{i = 1}^n(\D h_i)^{k_i}(h_{n+1}(\th+\xi)-h_{n+1}(\th))\cdot\f{ \p'(\th+\xi)}{2\sin\f{\xi}{2}}\,d\xi\\
&\; - \f{1}{|\mathbf{k}|+1}\cdot \f{1-(-1)^{|\mathbf{k}|+1}}{2^{|\mathbf{k}|+1}}\prod_{i = 1}^n(h_i(\th+\pi)-h_i(\th))^{k_i}(h_{n+1}(\th+\pi)-h_{n+1}(\th))\p(\th+\pi)\\
=&\;\sum_{i = 1}^n \f{k_i}{|\mathbf{k}|+1}\tilde{N}_{(k_1,\cdots, k_i-1,\cdots, k_n),h_i'\p}
+\f{1}{|\mathbf{k}|+1} (N_{\mathbf{k},(h_{n+1}\p)'} - h_{n+1}(\th)N_{\mathbf{k},\p'})\\
&\; - \f{1}{|\mathbf{k}|+1}\cdot \f{1-(-1)^{|\mathbf{k}|+1}}{2^{|\mathbf{k}|+1}}\prod_{i = 1}^n(h_i(\th+\pi)-h_i(\th))^{k_i}(h_{n+1}(\th+\pi)-h_{n+1}(\th))\p(\th+\pi).
\end{split}
\label{eqn: rewrite extra-term commutator}
\eeq
By the induction hypothesis \eqref{eqn: Lp bound for extra-term Calderon commutator},
\beq
\begin{split}
&\;\|k_i\tilde{N}_{(k_1,\cdots, k_i-1,\cdots, k_n),h_i'\p}\|_{L^p}\\
\leq &\;k_i
C_{**}^{|\mathbf{k}|} (\|h_{n+1}'\|_{L^p}\|h_i'\p\|_{L^\infty}+\|h_{n+1}\|_{L^\infty}\|(h_i'\p)'\|_{L^p})\cdot \|h_1'\|_{L^\infty}^{k_1}\cdots \|h_i'\|_{L^\infty}^{k_i-1}\cdots \|h_n'\|_{L^\infty}^{k_n}\\
&\;+ k_i C_{**}^{|\mathbf{k}|}\|h_{n+1}\|_{L^\infty}\|h_i'\p\|_{L^\infty} \\
&\;\qquad \cdot \sum_{j = 1,\cdots n\atop j\not = i} \|h_1'\|_{L^\infty}^{k_1}\cdots \|h_i'\|_{L^\infty}^{k_i-1}\cdots \|h_j'\|_{L^\infty}^{k_j-1}\cdots \|h_n'\|_{L^\infty}^{k_n}\cdot \mathds{1}_{\{k_j> 0\}}\|h_j''\|_{L^p}\\
&\;+ k_i C_{**}^{|\mathbf{k}|} \|h_{n+1}\|_{L^\infty}\|h_i'\p\|_{L^\infty} \cdot \|h_1'\|_{L^\infty}^{k_1}\cdots \|h_i'\|_{L^\infty}^{k_i-2}\cdots \|h_n'\|_{L^\infty}^{k_n}\cdot \mathds{1}_{\{k_i>1\}}\|h_i''\|_{L^p}\\
\leq &\;k_i
C_{**}^{|\mathbf{k}|} (\|h_{n+1}'\|_{L^p}\|\p\|_{L^\infty}+ \|h_{n+1}\|_{L^\infty}\|\p'\|_{L^p})\cdot \prod_{j = 1}^n\|h_j'\|_{L^\infty}^{k_j}\\
&\;+Ck_i C_{**}^{|\mathbf{k}|} \|h_{n+1}\|_{L^\infty}\|\p\|_{L^\infty} \sum_{j = 1}^n \|h_1'\|_{L^\infty}^{k_1}\cdots \|h_j'\|_{L^\infty}^{k_j-1}\cdots \|h_n'\|_{L^\infty}^{k_n}\cdot \mathds{1}_{\{k_j> 0\}}\|h_j''\|_{L^p}.
\end{split}
\eeq
By Lemma \ref{lem: Lp bound for multi-term commutator},
\beq
\begin{split}
&\;\|(N_{\mathbf{k},(h_{n+1}\p)'} - h_{n+1}(\th)N_{\mathbf{k},\p'})\|_{L^p}\\
\leq &\;C_*^{|\mathbf{k}|+2}(\|h_{n+1}'\|_{L^p}\|\p\|_{L^\infty}+\|h_{n+1}\|_{L^\infty}\|\p'\|_{L^p})\prod_{i = 1}^n \|h'_i\|_{L^\infty}^{k_i}.
\end{split}
\eeq
Combining these estimates with \eqref{eqn: rewrite extra-term commutator}, we obtain by Sobolev embedding that
\beq
\begin{split}
&\;\|\tilde{M}_{\mathbf{k},\p}\|_{L^p}\\
\leq &\;
C_{**}^{|\mathbf{k}|} (\|h_{n+1}'\|_{L^p}\|\p\|_{L^\infty}+ \|h_{n+1}\|_{L^\infty}\|\p'\|_{L^p}) \prod_{i = 1}^n\|h_i'\|_{L^\infty}^{k_i}\\
&\;+C C_{**}^{|\mathbf{k}|} \|h_{n+1}\|_{L^\infty}\|\p\|_{L^\infty} \sum_{i = 1}^n \|h_1'\|_{L^\infty}^{k_1}\cdots \|h_i'\|_{L^\infty}^{k_i-1}\cdots \|h_n'\|_{L^\infty}^{k_n}\cdot \mathds{1}_{\{k_i> 0\}}\|h_i''\|_{L^p}\\
&\;+C_*^{|\mathbf{k}|+2}(\|h_{n+1}'\|_{L^p}\|\p\|_{L^\infty}+\|h_{n+1}\|_{L^\infty}\|\p'\|_{L^p})\prod_{i = 1}^n \|h'_i\|_{L^\infty}^{k_i}\\
&\;+CC_1^{|\mathbf{k}|}\prod_{i = 1}^n\|h_i'\|_{L^\infty}^{k_i}\cdot \|h_{n+1}'\|_{L^p}\|\p\|_{L^\infty}\\
\leq &\;
(C_{**}^{|\mathbf{k}|} +C_*^{|\mathbf{k}|+2}+CC_1^{|\mathbf{k}|}) (\|h_{n+1}'\|_{L^p}\|\p\|_{L^\infty}+ \|h_{n+1}\|_{L^\infty}\|\p'\|_{L^p})\prod_{i = 1}^n\|h_i'\|_{L^\infty}^{k_i}\\
&\;+C C_{**}^{|\mathbf{k}|} \|h_{n+1}\|_{L^\infty}\|\p\|_{L^\infty} \sum_{i = 1}^n \|h_1'\|_{L^\infty}^{k_1}\cdots \|h_i'\|_{L^\infty}^{k_i-1}\cdots \|h_n'\|_{L^\infty}^{k_n}\cdot \mathds{1}_{\{k_i> 0\}}\|h_i''\|_{L^p}\\
\end{split}
\eeq
The estimate for $\tilde{N}_{\mathbf{k},\p}$ can be derived easily, since
\beq
|\tilde{N}_{\mathbf{k},\p}-\tilde{M}_{\mathbf{k},\p}| \leq CC_1^{|\mathbf{k}|}\prod_{i = 1}^n\|h_i'\|_{L^\infty}^{k_i} \int_{\BT} |h_{n+1}(\th+\xi)-h_{n+1}(\th)||\p(\th+\xi)|\,d\xi.
\eeq
Taking $C_{**}>0$ to be suitably large, we prove \eqref{eqn: Lp bound for extra-term Calderon commutator} when $|\mathbf{k}| = N$.
\end{step}
This completes the proof.
\end{proof}
\end{lem}

\begin{lem}\label{lem: W 1p estimate for multi-term commutator}
Under the hypotheses of Lemma \ref{lem: Lp bound for multi-term commutator}, we additionally assume $h_1,\cdots, h_n \in W^{2,p}(\BT)$ and $\p\in W^{1,p}(\BT)$.
Then 
\beq
\begin{split}
&\;\|M_{\mathbf{k},\p}\|_{\dot{W}^{1,p}}+\|N_{\mathbf{k},\p}\|_{\dot{W}^{1,p}}\\
\leq &\;(|\mathbf{k}|+1)C_\dag^{|\mathbf{k}|+1}\|\p'\|_{L^p} \prod_{i = 1}^n \|h_i'\|_{L^\infty}^{k_i}\\
&\; +(|\mathbf{k}|+1)C_\dag^{|\mathbf{k}|+1}\|\p\|_{L^\infty}\sum_{i = 1}^n \|h_1'\|_{L^\infty}^{k_1}\cdots \|h_i'\|_{L^\infty}^{k_i-1} \cdots \|h_n'\|_{L^\infty}^{k_n}\cdot \mathds{1}_{\{k_i>0\}}\|h_i''\|_{L^p}.
\end{split}
\label{eqn: W 1p bound for multi-term Calderon commutator}
\eeq
where $C_\dag$ is a universal constant depending only on $p$.

\begin{proof}
Instead of studying weak derivatives of $M_{\mathbf{k},\p}$ and $N_{\mathbf{k},\p}$ directly, we turn to difference quotients first.
Without loss of generality, let $\e>0$ be arbitrary and sufficiently small.
It suffices to prove uniform-in-$\e$ $L^p$-bounds for $\e^{-1}(M_{\mathbf{k},\p}(\th+\e)-M_{\mathbf{k},\p}(\th))$ and $\e^{-1}(N_{\mathbf{k},\p}(\th+\e)-N_{\mathbf{k},\p}(\th))$.
Write
\beq
\begin{split}
&\;\e^{-1}(M_{\mathbf{k},\p}(\th+\e)-M_{\mathbf{k},\p}(\th))\\
= &\;\sum_{i = 1}^n \int_\BT (\D h_1(\th))^{k_1}\cdots (\D h_{i-1}(\th))^{k_{i-1}}(\D h_{i+1}(\th+\e))^{k_{i+1}}\cdot (\D h_n(\th+\e))^{k_n}\\
&\;\qquad \cdot \sum_{l = 0}^{k_i-1}(\D h_i(\th))^l(\D h_i(\th+\e))^{k_i-1-l} \D\left(\f{h_i(\th+\e)-h_i(\th)}{\e}\right)\f{\p(\th+\e+\xi)}{2\tan\f{\xi}{2}}\,d\xi\\
&\;+\int_\BT\prod_{i = 1}^n(\D h_i(\th))^{k_i} \cdot \f{\e^{-1}(\p(\th+\e+\xi)-\p(\th+\xi))}{2\tan\f{\xi}{2}}\,d\xi.
\end{split}
\eeq
Applying Lemma \ref{lem: Lp bound for multi-term commutator} and Lemma \ref{lem: Lp bound for extra-term commutator},
\beq
\begin{split}
&\;\|\e^{-1}(M_{\mathbf{k},\p}(\th+\e)-M_{\mathbf{k},\p}(\th))\|_{L^p}\\
\leq
&\;\sum_{i = 1}^n k_i C_{**}^{|\mathbf{k}|}(\|\e^{-1}(h_i(\th+\e)-h_i(\th))'\|_{L^p}\|\p\|_{L^\infty}+\|\e^{-1}(h_i(\th+\e)-h_i(\th))\|_{L^\infty}\|\p'\|_{L^p})\\
&\;\qquad \cdot \|h_1'\|_{L^\infty}^{k_1}\cdots \|h_i'\|_{L^\infty}^{k_i-1}\cdots \|h_n'\|_{L^\infty}^{k_n}\\
&\;+\sum_{i = 1}^n C_{**}^{|\mathbf{k}|}\|\e^{-1}(h_i(\th+\e)-h_i(\th))\|_{L^\infty}\|\p\|_{L^\infty}\\
&\;\qquad \cdot k_i\sum_{j = 1,\cdots, n\atop j \not = i} \|h_1'\|_{L^\infty}^{k_1}\cdots \|h_i'\|_{L^\infty}^{k_i-1}\cdots \|h_j'\|_{L^\infty}^{k_j-1} \cdots \|h_n'\|_{L^\infty}^{k_n}\cdot \mathds{1}_{\{k_j>0\}}\|h_j''\|_{L^p}\\
&\;+\sum_{i = 1}^n C_{**}^{|\mathbf{k}|}\|\e^{-1}(h_i(\th+\e)-h_i(\th))\|_{L^\infty}\|\p\|_{L^\infty}\\
&\;\qquad \cdot \|h_1'\|_{L^\infty}^{k_1}\cdots \|h_i'\|_{L^\infty}^{k_i-2}\cdots \|h_n'\|_{L^\infty}^{k_n}\cdot C(k_i-1)\mathds{1}_{\{k_i>1\}} \|h_i''\|_{L^p}\\
&\;+C_*^{|\mathbf{k}|+2}\prod_{i = 1}^n \|h_i'\|_{L^\infty}^{k_i}\cdot \|\e^{-1}(\p(\cdot+\e)-\p(\cdot))\|_{L^p}\\
\leq &\;C (|\mathbf{k}|C_{**}^{|\mathbf{k}|}+C_*^{|\mathbf{k}|+2})\|\p'\|_{L^p} \prod_{i = 1}^n \|h_i'\|_{L^\infty}^{k_i}\\
&\; +C |\mathbf{k}| C_{**}^{|\mathbf{k}|}\|\p\|_{L^\infty}\sum_{i = 1}^n \|h_1'\|_{L^\infty}^{k_1}\cdots \|h_i'\|_{L^\infty}^{k_i-1} \cdots \|h_n'\|_{L^\infty}^{k_n}\cdot \mathds{1}_{\{k_i>0\}}\|h_i''\|_{L^p}.
\end{split}
\eeq
Note that this bound is uniform in $\e$.
Hence, $M_{\mathbf{k},\p}(\th)$ has weak derivative, with an identical $L^p$-bound as above.
The estimate for $N_{\mathbf{k},\p}$ can be derived similarly.
Therefore, \eqref{eqn: W 1p bound for multi-term Calderon commutator} holds if $C_{\dag}$ is taken to be suitably large.
\end{proof}
\end{lem}

\subsection{Regularity theory of fractional heat equations}
We focus on the following Cauchy problem of fractional heat equation on $\BT$ with special exponent $\f12$.
\beq
\pa_t v = -(-\D)^{1/2}v +f(t,\th),\quad v(0,\th) = 0.
\label{eqn: fractional heat equation}
\eeq

For our purpose, we have that
\begin{lem}\label{lem: L p regularity theory of fractional heat equation}
Suppose $f \in L^p_{[0,T]}L^p(\BT)$ for some $p\in [2,\infty)$.
Then there exists $v\in L^p_{[0,T]}W^{1,p}(\BT)$ solving \eqref{eqn: fractional heat equation}, satisfying that
\beq
\|v_t \|_{L^p_{[0,T]}L^p(\BT)}+\|(-\D)^{1/2}v\|_{L^p_{[0,T]}L^p(\BT)} \leq C\|f \|_{L^p_{[0,T]}L^p(\BT)},
\eeq
where $C = C(p)$.
\end{lem}
This immediately follows from \cite[Theorem 1]{lamberton1987equations}; see also \cite[Theorem 4.1]{biccari2018local}.

\begin{lem}\label{lem: spatial Holder estimate for solution with Sobolev source term}
Suppose $T\leq 1$ and $p\in (2,\infty)$.
Under the assumption of Lemma \ref{lem: L p regularity theory of fractional heat equation}, $v\in C_{[0,T]}C^\alpha(\BT)$ with $\alpha = 1-\f2p$, satisfying that
\beq
\|v\|_{C_{[0,T]}\dot{C}^{\alpha}(\BT)}\leq C\|f \|_{L^p_{[0,T]}L^p(\BT)},
\eeq
where $C = C(p)$.
\begin{proof}
Let $\CP(t,\th)$ be the Poisson kernel on $\BT$, with $t$ being the time variable, solving
\beq
\pa_t \CP = -(-\D)^{1/2}\CP,\quad \CP(0,\th) = \d_0
\eeq
in the sense of distribution.
Here $\d_0$ is the delta measure at $0\in \BT$.
Note that $\CP(t,\th)$ is related to $P(s,\xi)$, which is defined in Section \ref{sec: gradient estimate of growth potential}, in the following sense
\beq
\CP(t,\th) = \f{1}{2\pi}P(e^{-t},\th).
\label{eqn: relation between two types of poisson kernels}
\eeq
Then $v$ can be represented by
\beq
v(t,\th) = \int_0^t \int_\BT\CP(t-\tau,\th-\xi)f (\tau,\xi)\,d\xi d\tau.
\label{eqn: Duhamel's formula}
\eeq

Take arbitrary $\th_1,\th_2\in \BT$, such that $d_\th := |\th_1-\th_2|\leq 1$.
Denote $\bar{\th} = (\th_1+\th_2)/2$.
Then
\beq
\begin{split}
&\;|v(t,\th_1)-v(t,\th_2)|\\
\leq &\;\int_{[0,t]\times \BT\cap \{(\tau,\xi):\,|t-\tau|+|\bar{\th}-\xi|\leq d_\th\}} (|\CP(t-\tau,\th_1-\xi)|+|\CP(t-\tau,\th_2-\xi)|)|f (\tau,\xi)|\,d\xi d\tau\\
&\;+\int_{[0,t]\times \BT\cap \{(\tau,\xi):\,|t-\tau|+|\bar{\th}-\xi|\geq d_\th\}} |\CP(t-\tau,\th_1-\xi)-\CP(t-\tau,\th_2-\xi)||f (\tau,\xi)|\,d\xi d\tau.
\end{split}
\eeq
By the mean value theorem, Lemma \ref{lem: properties of Poisson kernel} and H\"{o}lder's inequality,
\beq
\begin{split}
&\;|v(t,\th_1)-v(t,\th_2)|\\
\leq &\;C\int_{[0,t]\times \BT\cap \{(\tau,\xi):\,|t-\tau|+|\th_1-\xi|\leq 2d_\th\}} \f{|f (\tau,\xi)|}{|t-\tau|+|\th_1-\xi|}\,d\xi d\tau\\
&\;+C\int_{[0,t]\times \BT\cap \{(\tau,\xi):\,|t-\tau|+|\th_2-\xi|\leq 2d_\th\}} \f{|f (\tau,\xi)|}{|t-\tau|+|\th_2-\xi|}\,d\xi d\tau\\
&\;+C|\th_1-\th_2|\int_{[0,t]\times \BT\cap \{(\tau,\xi):\,|t-\tau|+|\bar{\th}-\xi|\geq d_\th\}} \f{|f (\tau,\xi)|}{|t-\tau|^2+|\bar{\th}-\xi|^2}\,d\xi d\tau\\
\leq &\;C\|f \|_{L^p([0,T]\times \BT)}\left(\int_{0}^{ 2d_\th}\r^{1-p'}\,d\r\right)^{1/p'}\\
&\;+C|\th_1-\th_2|\|f \|_{L^p([0,T]\times \BT)}\left(\int_{d_\th/\sqrt{2}}^\infty\r^{1-2p'}\,d\r\right)^{1/p'}.
\end{split}
\eeq
Here $p'= (1-\f1p)^{-1}\in (1,2)$.
Calculating the integral above yields
\beq
|v(t,\th_1)-v(t,\th_2)|\leq C|\th_1-\th_2|^\alpha\|f \|_{L^p_{[0,T]}L^p(\BT)}.
\eeq
It is then straightforward to justify the case $|\th_1-\th_2|>1$.

The time-continuity of $v$ in $C^{1,\alpha}$ follows from the absolute continuity of the Lebesgue integral with respect to translation.
\end{proof}
\end{lem}

\begin{lem}\label{lem: Holder bound for the fractional heat equation}
Suppose $T\leq 1$ and $f  \in L^\infty_{[0,T]}C^{\alpha}(\BT)$ for some $\alpha\in(0,1)$.
Then for all $\b\in (0,\alpha)$, there exists a unique $v\in C_{[0,T]}C^{1,\b}(\BT)$ solving \eqref{eqn: fractional heat equation}, satisfying that
\beq
\|v\|_{C_{[0,T]}\dot{C}^{1,\b}(\BT)}\leq C \|f \|_{L^\infty_{[0,T]}\dot{C}^\alpha(\BT)},
\label{eqn: estimate for uniform in time Holder norm of solution of fractional heat equation}
\eeq
where $C = C(\alpha,\b)$.
\begin{proof}
Once again, $v$ can be represented by \eqref{eqn: Duhamel's formula}.
It then suffices to bound its $\dot{C}^{1,\b}$-seminorm, which also implies the uniqueness.

For arbitrary $\th_1,\th_2\in \BT$, 
\beq
\begin{split}
&\;\pa_\th v(t,\th_1)-\pa_\th v(t,\th_2)\\
=&\;\int_0^t \int_\BT\pa_\th\CP(t-\tau,\xi)(f (\tau,\th_1-\xi)-f (\tau,\th_1)-f (\tau,\th_2-\xi)+f (\tau,\th_2))\,d\xi d\tau\\
\end{split}
\label{eqn: decomposition of difference of derivative of solution of fractional heat equation}
\eeq
Since
\beq
\begin{split}
&\;|f (\tau,\th_1-\xi)-f (\tau,\th_1)-f (\tau,\th_2-\xi)+f (\tau,\th_2)|\\
\leq &\;C\|f (\tau,\cdot)\|_{\dot{C}^\alpha}\min\{|\xi|^\alpha,|\th_1-\th_2|^\alpha\},
%
\end{split}
\eeq
by \eqref{eqn: relation between two types of poisson kernels} and Lemma \ref{lem: properties of Poisson kernel}, we have that
\beq
\begin{split}
&\;|\pa_\th v(t,\th_1)-\pa_\th v(t,\th_2)|\\
\leq &\;\int_0^t \int_\BT|\pa_\th\CP(t-\tau,\xi)||\xi|^{\alpha-\b}\,d\xi d\tau\cdot |\th_1-\th_2|^\b \|f \|_{L^\infty_{[0,T]}\dot{C}^\alpha(\BT)}\\
\leq&\; C\int_0^t(1-e^{-(t-\tau)})^{\alpha-\b-1}\,d\tau\cdot |\th_1-\th_2|^\b \|f \|_{L^\infty_{[0,T]}\dot{C}^\alpha(\BT)}\\
\leq &\;C|\th_1-\th_2|^\b \|f \|_{L^\infty_{[0,T]}\dot{C}^\alpha(\BT)}.
\end{split}
\label{eqn: bound for V_2}
\eeq

Finally, the time continuity of $v$ can be justified by interpolating between the facts that $v\in C_{[0,T]} C^\alpha(\BT)$ and $v\cap L^\infty_{[0,T]} C^{1,\b'}(\BT)$ for some $\b'\in (\b,\alpha)$.
%
\end{proof}
\end{lem}

\section{Proofs of Lemma \ref{lem: stability of the interface velocities} and Lemma \ref{lem: C 1 alpha bounds for two different pressures in reference coordinate}}
\label{sec: proof of stability lemma}

We need several preparatory results.

Let $h_i$ and $H_i$ be given as in Section \ref{sec: stability of interface velocity}.
Let $x_i(X)$ $(i = 1,2)$ denote the diffeomorphism \eqref{eqn: change of variables to the reference coordinate} defined by $h_i$ and $H_i$,
\beq
x_i(X) = \z_i(X)X,\quad \z_i(X) := 1+h_i(\om)\eta_\d\left(\f{\r}{r}\right)+H_i(\om)\eta_\d\left(\f{\r}{R}\right).
\label{eqn: def of zeta i}
\eeq
Let $p_i$ denote the pressure on the physical domain that is determined by $\g_i$ and $\tilde{\g}_i$, while $\tilde{p}_i$ denotes its pull back into the reference coordinate as in \eqref{eqn: pressure in the reference coordinate}.
By \eqref{eqn: equation for p tilde}, $(\tilde{p}_{1}-\tilde{p}_{2})$ solves
\beq
\begin{split}
&\;-\na_{X_k}\left(a\f{\pa X_k}{\pa x_{1,i}}\f{\pa X_j}{\pa x_{1,i}}\na_{X_j}( \tilde{p}_{1}-\tilde{p}_{2})\right)\\
=&\;\na_{X_k}\left[a\left(\f{\pa X_k}{\pa x_{1,i}}\f{\pa X_j}{\pa x_{1,i}}-\f{\pa X_k}{\pa x_{2,i}}\f{\pa X_j}{\pa x_{2,i}}\right)\na_{X_j}\tilde{p}_{2}\right]\\
&\; +(G(\tilde{p}_{1})-G(\tilde{p}_{2}))\chi_{B_r}-\na_{X_k}\f{\pa X_k}{\pa x_{1,i}}\cdot a\f{\pa X_j}{\pa x_{1,i}}\na_{X_j}(\tilde{p}_{1}-\tilde{p}_{2})\\
&\;-a\left[\na_{X_k}\f{\pa X_k}{\pa x_{1,i}}\cdot \f{\pa X_j}{\pa x_{1,i}}-\na_{X_k}\f{\pa X_k}{\pa x_{2,i}}\cdot \f{\pa X_j}{\pa x_{2,i}}\right]\na_{X_j}\tilde{p}_{2}
\end{split}
\label{eqn: equation for difference between two p tilde solutions}
\eeq
in $B_R$, with $(\tilde{p}_1-\tilde{p}_2)|_{\pa B_R} = 0$.
Here $a =a(X)$ is given in \eqref{eqn: def of a}, and $x_{1,i}$ and $x_{2,i}$ denote $i$-th components of $x_1$ and $x_2$, respectively.

We first derive estimates for several ingredients in \eqref{eqn: equation for difference between two p tilde solutions}.

\begin{lem}\label{lem: Holder estimate of jacobi}
Assume $h_i,H_i\in W^{1,\infty}(\BT)$ satisfy that $m_{0,i}+M_{0,i}\ll 1$.
Then
\begin{align}
\left\|\f{\pa X}{\pa x_{1}}-\f{\pa X}{\pa x_{2}}\right\|_{L^\infty(B_R)}\leq &\; C(\D m_0+\D M_0),\label{eqn: L inf difference of two jacobi}\\
\left\|\f{\pa X_k}{\pa x_{1,i}}\f{\pa X_j}{\pa x_{1,i}}-\f{\pa X_k}{\pa x_{2,i}}\f{\pa X_j}{\pa x_{2,i}}\right\|_{L^\infty(B_R)}\leq &\; C(\D m_0+\D M_0),
\label{eqn: L inf difference of two coefficient matrices}\\
\left\|\na_{X_k}\f{\pa X_k}{\pa x_{1,i}}-\na_{X_k}\f{\pa X_k}{\pa x_{2,i}}\right\|_{L^\infty(B_R)}\leq &\;C(\d r)^{-1}(\D m_0+\D M_0),
\label{eqn: L inf difference of derivative of jacobi matrices}
\end{align}
where the constants $C$ are all universal.

If in addition, $h_i,H_i\in C^{1,\alpha}(\BT)$ for some $\alpha\in (0,1)$, such that $m_{\alpha,i}+M_{\alpha,i}\ll 1$, then
\beq
\left\|\f{\pa X}{\pa x_{1}}-\f{\pa X}{\pa x_{2}}\right\|_{\dot{C}^\alpha(B_R)}\leq C(\d r)^{-\alpha}(\D m_\alpha +\D M_\alpha),
\label{eqn: Holder difference of two jacobi matrices}
\eeq
and
\beq
\left\|\f{\pa X_k}{\pa x_{1,i}}\f{\pa X_j}{\pa x_{1,i}}-\f{\pa X_k}{\pa x_{2,i}}\f{\pa X_j}{\pa x_{2,i}}\right\|_{\dot{C}^\alpha(B_R)}\leq C(\d r)^{-\alpha}(\D m_\alpha +\D M_\alpha).
\label{eqn: Holder difference of two coefficient matrices}
\eeq
Here $C$ are universal constants only depending on $\alpha$.
All the quantities above are only supported on $\overline{B_{r(1+2\d)}}\backslash B_{ r(1-2\d)}$ and $\overline{B_R}\backslash B_{R(1-2\d)}$.
\begin{proof}
The proof is once again a straightforward calculation.

We derive by \eqref{eqn: inverse of jacobi matrix} that
\beq
\begin{split}
\f{\pa X}{\pa x_{1}}-\f{\pa X}{\pa x_{2}} =&\; (\zeta^2_1 + \zeta_1 \r\pa_\r \zeta_1)^{-1}((\zeta_1-\z_2) \cdot Id+ (\na(\zeta_1-\z_2))^{\perp}\otimes X^\perp)\\
&\; +\f{(\zeta^2_2 + \zeta_2 \r\pa_\r \zeta_2)-(\zeta^2_1 + \zeta_1 \r\pa_\r \zeta_1)}{(\zeta^2_1 + \zeta_1 \r\pa_\r \zeta_1)(\zeta^2_2 + \zeta_2 \r\pa_\r \zeta_2)}(\z_2 \cdot Id + (\na\z_2)^{\perp}\otimes X^\perp).
\end{split}
\label{eqn: formula for difference of jacobi matrices}
\eeq
By \eqref{eqn: calculation of rho partial rho zeta},
\beq
\begin{split}
&\;|(\zeta^2_2 + \zeta_2 \r\pa_\r \zeta_2)-(\zeta^2_1 + \zeta_1 \r\pa_\r \zeta_1)|\\
\leq &\; |\z_1-\z_2||\z_1+\z_2+\r\pa_\r\z_2|+|\z_1||\r\pa_\r (\z_1-\z_2)|\\
\leq &\; C\d^{-1}(\|h_1-h_2\|_{L^\infty}+\|H_1-H_2\|_{L^\infty})\\
\leq &\; C(\D m_0+\D M_0).
\end{split}
\label{eqn: L inf different two denominators}
\eeq
Combining \eqref{eqn: denominator of jacobi is close to 1}, \eqref{eqn: grad zeta} and \eqref{eqn: L inf different two denominators} with \eqref{eqn: formula for difference of jacobi matrices}, we find that
\beq
\begin{split}
&\;\left|\f{\pa X}{\pa x_{1}}-\f{\pa X}{\pa x_{2}} \right|\\
\leq &\; C(|\zeta_1-\z_2| + \r|\na(\zeta_1-\z_2)|)\\
&\; +C|(\zeta^2_2 + \zeta_2 \r\pa_\r \zeta_2)-(\zeta^2_1 + \zeta_1 \r\pa_\r \zeta_1)|(|\z_2|+\r|\na\z_2|)\\
\leq &\; C(\D m_0+\D M_0),
\end{split}
\eeq
which proves \eqref{eqn: L inf difference of two jacobi}.
It is easy to derive \eqref{eqn: L inf difference of two coefficient matrices} from \eqref{eqn: L inf difference of two jacobi} and Lemma \ref{lem: L inf estimate of jacobi}.

To show \eqref{eqn: L inf difference of derivative of jacobi matrices}, we use \eqref{eqn: div of jacobi matrix} to derive that
%
\beq
\begin{split}
&\;\na_{X_k}\f{\pa X_k}{\pa x_{1,i}}-\na_{X_k}\f{\pa X_k}{\pa x_{2,i}}\\
=&\;\left(\f{\pa X_j}{\pa x_{2,i}}-\f{\pa X_j}{\pa x_{1,i}}\right)
(\zeta^2_1 + \zeta_1 \r \pa_\r \zeta_1)^{-1}\na_{X_j}(\zeta^2_1 + \zeta_1 \r\cdot\pa_\r\zeta_1)\\
&\;+\f{\pa X_j}{\pa x_{2,i}}
\f{(\zeta^2_1 + \zeta_1 \r \pa_\r \zeta_1)-(\zeta^2_2 + \zeta_2 \r \pa_\r \zeta_2)}{(\zeta^2_1 + \zeta_1 \r \pa_\r \zeta_1)(\zeta^2_2 + \zeta_2 \r \pa_\r \zeta_2)}\na_{X_j}(\zeta^2_1 + \zeta_1 \r\cdot\pa_\r\zeta_1)\\
&\;+\f{\pa X_j}{\pa x_{2,i}}
(\zeta^2_2 + \zeta_2 \r \pa_\r \zeta_2)^{-1}\na_{X_j}[(\zeta^2_2 + \zeta_2 \r\cdot\pa_\r\zeta_2)-(\zeta^2_1 + \zeta_1 \r\cdot\pa_\r\zeta_1)].
\end{split}
\eeq
Then by \eqref{eqn: calculation of rho partial rho zeta}, \eqref{eqn: denominator of jacobi is close to 1}, \eqref{eqn: grad zeta}, \eqref{eqn: grad of radial derivative of zeta}, \eqref{eqn: L inf different two denominators} and Lemma \ref{lem: L inf estimate of jacobi},
\beq
\begin{split}
&\;\left|\na_{X_k}\f{\pa X_k}{\pa x_{1,i}}-\na_{X_k}\f{\pa X_k}{\pa x_{2,i}}\right|\\
\leq &\;C\left|\f{\pa X}{\pa x_{2}}-\f{\pa X}{\pa x_{1}}\right|
|\na_{X_j}(\zeta^2_1 + \zeta_1 \r \pa_\r\zeta_1)|\\
&\;+
C|(\zeta^2_1 + \zeta_1 \r \pa_\r \zeta_1)-(\zeta^2_2 + \zeta_2 \r \pa_\r \zeta_2)||\na_{X_j}(\zeta^2_1 + \zeta_1 \r \pa_\r\zeta_1)|\\
&\;+C|\na_{X_j}[(\zeta^2_2 + \zeta_2 \r \pa_\r\zeta_2)-(\zeta^2_1 + \zeta_1 \r \pa_\r\zeta_1)]|\\
\leq &\;C(\d r)^{-1}(\D m_0+\D M_0).
\end{split}
\eeq

To prove \eqref{eqn: Holder difference of two jacobi matrices}, we start with a H\"{o}lder estimate of $(\z_1^2+\z_1\r\pa_\r \z_1)-(\z_2^2+\z_2\r\pa_\r \z_2)$.
Using the fact that $\|fg\|_{\dot{C}^\alpha}\leq \|f\|_{\dot{C}^\alpha}\|g\|_{L^\infty}+\|f\|_{L^\infty}\|g\|_{\dot{C}^\alpha}$,
\beq
\begin{split}
&\;\|(\z_1^2+\z_1\r\pa_\r \z_1)-(\z_2^2+\z_2\r\pa_\r \z_2)\|_{\dot{C}^\alpha(B_R)}\\
\leq &\;\|\z_1-\z_2\|_{\dot{C}^\alpha(B_R)}(\|\z_1+\z_2\|_{L^\infty}+\|\r \pa_\r \z_1\|_{L^\infty})\\
&\;+\|\z_1-\z_2\|_{L^\infty}(\|\z_1+\z_2\|_{\dot{C}^\alpha(B_R)}+\|\r \pa_\r \z_1\|_{\dot{C}^\alpha(B_R)})\\
&\;+\|\z_2\|_{\dot{C}^\alpha(B_R)}\|\r \pa_\r \z_1-\r \pa_\r \z_2\|_{L^\infty}
+\|\z_2\|_{L^\infty}\|\r \pa_\r \z_1-\r \pa_\r \z_2\|_{\dot{C}^\alpha(B_R)}. 
%
\end{split}
\label{eqn: Holder estimate of difference of denominators crude form}
\eeq
Note that the H\"{o}lder semi-norms are taken over $B_R$ with respect to the Euclidean distance in $X$-coordinate instead of the $(\r,\om)$-coordinate.
Using
\beq
(\z_1-\z_2)(X) = (h_1-h_2)(\om)\eta_\d\left(\f{\r}{r}\right)+(H_1-H_2)(\om)\eta_\d\left(\f{\r}{R}\right),
\eeq
and the fact that $\eta_\d(\f{\r}{r})$ and $\eta_\d(\f{\r}{R})$ are supported near $\pa B_r$ and $\pa B_R$, respectively, we find that
\beq
\begin{split}
&\;\|\z_1-\z_2\|_{\dot{C}^\alpha(B_R)}\\
\leq &\; Cr^{-\alpha}\|h_1-h_2\|_{\dot{C}^\alpha(\BT)}
+C\|h_1-h_2\|_{L^\infty}\left\|\eta_\d\left(\f{\r}{r}\right)\right\|_{\dot{C}^\alpha(B_R)}\\
&\;+CR^{-\alpha}\|H_1-H_2\|_{\dot{C}^\alpha(\BT)}
+C\|H_1-H_2\|_{L^\infty}\left\|\eta_\d\left(\f{\r}{R}\right)\right\|_{\dot{C}^\alpha(B_R)}\\
\leq &\; Cr^{-\alpha}\|h_1-h_2\|_{\dot{C}^\alpha(\BT)}+C\|h_1-h_2\|_{L^\infty}(\d r)^{-\alpha}\\
&\;+CR^{-\alpha}\|H_1-H_2\|_{\dot{C}^\alpha(\BT)}+C\|H_1-H_2\|_{L^\infty}(\d R)^{-\alpha}\\
\leq &\; C\d^{1-\alpha}r^{-\alpha}\D m_0+C\d^{1-\alpha}R^{-\alpha}\D M_0.
\end{split}
\label{eqn: Holder bound for difference of zetas}
\eeq
In the last line, we applied interpolation inequalities.
Setting $h_1=H_1 = 0$ (or $h_2=H_2= 0$), we obtain estimates for $\|\z_i\|_{\dot{C}^\alpha(B_R)}$.

Similarly, since
\beq
\r\pa_\r (\z_1-\z_2) = (h_1-h_2)(\om)\cdot \f{\r}{r}\eta_\d'\left(\f{\r}{r}\right) +(H_1-H_2)(\om)\cdot \f{\r}{R}\eta_\d'\left(\f{\r}{R}\right),
\eeq
we deduce that
\beq
\|\r \pa_\r \z_1-\r \pa_\r \z_2\|_{\dot{C}^\alpha(B_R)}
\leq 
C(\d r)^{-\alpha}\D m_0+C(\d R)^{-\alpha}\D M_0.
\label{eqn: Holder bound for difference of radial derivatives of zetas}
\eeq
Combining \eqref{eqn: Holder bound for difference of zetas} and \eqref{eqn: Holder bound for difference of radial derivatives of zetas} with \eqref{eqn: Holder estimate of difference of denominators crude form} yields that
\beq
\begin{split}
&\;\|(\z_1^2+\z_1\r\pa_\r \z_1)-(\z_2^2+\z_2\r\pa_\r \z_2)\|_{\dot{C}^\alpha(B_R)}\\
\leq &\;C\|\z_1-\z_2\|_{\dot{C}^\alpha(B_R)}\\
&\;+C\d (\D m_0+ \D M_0)\cdot (\|\z_1\|_{\dot{C}^\alpha(B_R)}+\|\z_2\|_{\dot{C}^\alpha(B_R)}+\|\r \pa_\r \z_1\|_{\dot{C}^\alpha(B_R)})\\
&\;+C\|\z_2\|_{\dot{C}^\alpha}(\D m_0+ \D M_0)
+C\|\r \pa_\r \z_1-\r \pa_\r \z_2\|_{\dot{C}^\alpha(B_R)}\\
\leq 
&\;C(\d r)^{-\alpha}(\D m_0+\D M_0).
\end{split}
\label{eqn: Holder difference of two denominators}
\eeq
Setting $h_2 = H_2 = 0$ gives 
\beq
\|\z_1^2+\z_1\r\pa_\r \z_1\|_{\dot{C}^\alpha(B_R)}\leq  C(\d r)^{-\alpha}( m_{0,1}+ M_{0,1}).
\eeq
Thanks to \eqref{eqn: denominator of jacobi is close to 1}, it is not difficult to derive that $\|(\z_1^2+\z_1\r\pa_\r \z_1)^{-1}\|_{\dot{C}^\alpha(B_R)}$ has the same bound, with a different constant $C$.

In addition, by \eqref{eqn: grad zeta},
\beq
\begin{split}
\na (\z_1-\z_2)
= &\;\left[(h_1-h_2)(\om)\cdot \f{1}{r}\eta_\d'\left(\f{\r}{r}\right) +(H_1-H_2)(\om)\cdot \f{1}{R}\eta_\d'\left(\f{\r}{R}\right) \right] e_r\\
&\;+\left[(h_1'-h_2')(\om)\cdot \eta_\d\left(\f{\r}{r}\right)+(H_1'-H_2')(\om)\cdot \eta_\d\left(\f{\r}{R}\right)\right]\r^{-1} e_\theta.
\end{split}
\eeq
So
\beq
\begin{split}
&\;\|\na (\z_1-\z_2)\|_{\dot{C}^\alpha(B_R)}\\
\leq &\; Cr^{-\alpha}\|(h_1-h_2)e_r\|_{\dot{C}^\alpha(\BT)}(\d r)^{-1}
+C\|h_1-h_2\|_{L^\infty}\left\|\f{1}{r}\eta_\d'\left(\f{\r}{r}\right)\right\|_{\dot{C}^\alpha(B_R)}\\
&\;+CR^{-\alpha}\|(H_1-H_2)e_r\|_{\dot{C}^\alpha(\BT)}(\d R)^{-1}+C\|H_1-H_2\|_{L^\infty}\left\|\f{1}{R}\eta_\d'\left(\f{\r}{R}\right)\right\|_{\dot{C}^\alpha(B_R)}\\
&\;+Cr^{-\alpha}\|(h_1'-h_2')e_\th\|_{\dot{C}^\alpha(\BT)}r^{-1}
+C\|h_1'-h_2'\|_{L^\infty}\left\|\f{1}{\r}\eta_\d\left(\f{\r}{r}\right)\right\|_{\dot{C}^\alpha(B_R)}\\
&\;+CR^{-\alpha}\|(H_1'-H_2')e_\th\|_{\dot{C}^\alpha(\BT)} R^{-1}+C\|H_1'-H_2'\|_{L^\infty}\left\|\f{1}{\r}\eta_\d\left(\f{\r}{R}\right)\right\|_{\dot{C}^\alpha(B_R)}\\
%
%
\leq &\;
C\d^{-\alpha}(r^{-1-\alpha}\D m_\alpha +R^{-1-\alpha}\D M_\alpha).
\end{split}
\label{eqn: Holder bound for difference of two gradients}
\eeq
Here we used the fact that $\D m_0 +\D M_0\leq C(\D m_\alpha+\D M_\alpha)$ by interpolation.

To this end, combining \eqref{eqn: denominator of jacobi is close to 1}, \eqref{eqn: grad zeta}, \eqref{eqn: formula for difference of jacobi matrices}, \eqref{eqn: L inf different two denominators}, \eqref{eqn: Holder bound for difference of zetas}, \eqref{eqn: Holder difference of two denominators} and \eqref{eqn: Holder bound for difference of two gradients},
\beq
\begin{split}
&\;\left\|\f{\pa X}{\pa x_{1}}-\f{\pa X}{\pa x_{2}}\right\|_{\dot{C}^\alpha(B_R)} \\
\leq &\; C\|(\zeta^2_1 + \zeta_1 \r\pa_\r \zeta_1)^{-1}\|_{\dot{C}^\alpha(B_R)}(\|\z_1-\z_2\|_{L^\infty}+\|\r\na(\z_1-\z_2)\|_{L^\infty})\\
&\; +C(\|\zeta_1-\z_2\|_{\dot{C}^\alpha(B_R)}+ \|\na(\zeta_1-\z_2))^{\perp}\otimes X^\perp\|_{\dot{C}^\alpha(B_R)})\\
&\; +C\|(\zeta^2_2 + \zeta_2 \r\pa_\r \zeta_2)-(\zeta^2_1 + \zeta_1 \r\pa_\r \zeta_1)\|_{\dot{C}^\alpha(B_R)}\\
&\; +C\|(\zeta^2_2 + \zeta_2 \r\pa_\r \zeta_2)-(\zeta^2_1 + \zeta_1 \r\pa_\r \zeta_1)\|_{L^\infty}\\
&\;\quad \cdot \left[\|(\zeta^2_1 + \zeta_1 \r\pa_\r \zeta_1)^{-1}\|_{\dot{C}^\alpha(B_R)}+\|(\zeta^2_2 + \zeta_2 \r\pa_\r \zeta_2)^{-1}\|_{\dot{C}^\alpha(B_R)}\right.\\
&\;\qquad \left.+\|\z_2\|_{\dot{C}^\alpha(B_R)} + \|(\na\z_2)^{\perp}\otimes X^\perp\|_{\dot{C}^\alpha(B_R)}\right]\\
\leq &\; C(\d r)^{-\alpha}(\D m_\alpha +\D M_\alpha).
\end{split}
\eeq
In the last inequality, we needed the assumption $m_{\alpha,i}+M_{\alpha,i}\ll 1$.

Finally, \eqref{eqn: Holder difference of two coefficient matrices} follows from \eqref{eqn: L inf difference of two jacobi}, \eqref{eqn: Holder difference of two jacobi matrices} and Lemma \ref{lem: L inf estimate of jacobi}.
\end{proof}
\end{lem}

\begin{lem}\label{lem: C 1 alpha bound for tilde p_2}
Assume $h_2,H_2\in C^{1,\alpha}(\BT)$ with $\alpha<\f14$, satisfying that $m_{\alpha,2}+M_{\alpha,2}\ll 1$.
Then
\beq
\|\tilde{p}_{2}\|_{C^{1,\alpha}(\overline{B_r})}+\| \tilde{p}_{2}\|_{C^{1,\alpha}(\overline{B_R\backslash B_r})}\leq C(\alpha,\mu,\nu, r, R, G).
\eeq

\begin{proof}
By \eqref{eqn: equation for p tilde}, $\tilde{p}_2$ solves
%
\beq
\begin{split}
&\;-\na_{X_k}\left(a\f{\pa X_k}{\pa x_{2,i}}\f{\pa X_j}{\pa x_{2,i}}\na_{X_j} \tilde{p}_{2}\right)= G(\tilde{p}_{2})\chi_{B_r}-\na_{X_k}\f{\pa X_k}{\pa x_{2,i}}\cdot a\f{\pa X_j}{\pa x_{2,i}}\na_{X_j}\tilde{p}_{2}
\end{split}
\label{eqn: equation for p tilde 2}
\eeq
in $B_R$, with $\tilde{p}_2|_{\pa B_R} = 0$.
By putting $h_1= H_1 = 0$ in \eqref{eqn: L inf difference of two coefficient matrices} and \eqref{eqn: Holder difference of two coefficient matrices}, we obtain that
\begin{align}
\left\|\f{\pa X_k}{\pa x_{2,i}}\f{\pa X_j}{\pa x_{2,i}}-Id\right\|_{L^\infty(B_R)}\leq &\;C(m_{0,2} + M_{0,2}),\\
\left\|\f{\pa X_k}{\pa x_{2,i}}\f{\pa X_j}{\pa x_{2,i}}\right\|_{\dot{C}^\alpha(B_R)}\leq &\;C(\d r)^{-\alpha}(m_{\alpha,2} + M_{\alpha,2}).
\end{align}
By assuming $m_{\alpha,2}+M_{\alpha,2}$ to be suitably small (ans thus $m_{0,2}+M_{0,2}$ is small by interpolation), we may have the coefficient matrix satisfy 
\beq
\f12 \min\{\mu,\nu\}Id\leq a\f{\pa X_k}{\pa x_{2,i}}\f{\pa X_j}{\pa x_{2,i}} \leq 2\max\{\mu,\nu\}Id,
\eeq
which is symmetric and piecewise $C^{\alpha}$ in $B_R$.
Therefore, by \cite[Corollary 1.3]{li2000gradient} and Lemma \ref{lem: L inf estimate of jacobi}, for $\alpha<\f14$,
\beq
\begin{split}
&\;\|\tilde{p}_2\|_{C^{1,\alpha}(\overline{B_r})}+\|\tilde{p}_2\|_{C^{1,\alpha}(\overline{B_R\backslash B_r})}\\
\leq &\; C\left(\alpha,\mu,\nu, r, R, \left\|\f{\pa X_k}{\pa x_{2,i}}\f{\pa X_j}{\pa x_{2,i}}\right\|_{C^\alpha(B_R)} \right)\left\|G(\tilde{p}_{2})\chi_{B_r}-\na_{X_k}\f{\pa X_k}{\pa x_{2,i}}\cdot a\f{\pa X_j}{\pa x_{2,i}}\na_{X_j}\tilde{p}_{2}\right\|_{L^\infty}\\
\leq &\; C(\alpha,\mu,\nu, r, R, G) (1+\|\na\tilde{p}_{2}\|_{L^\infty(B_R)})\\
\end{split}
\label{eqn: C 1 alpha estimate for tilde p 2 crude form}
\eeq
We omit the dependence of $C$ on $m_{0,2}+M_{0,2}$ and $m_{\alpha,2}+M_{\alpha,2}$ since they can be bounded by universal constants.
The $\d$-dependence of $C$ is encoded in the $(r,R)$-dependence.
By interpolation inequality, with $\epsilon>0$ to be chosen and $C_{\epsilon}$ depending on $\epsilon$ and $\alpha$,
\beq
\|\na \tilde{p}_2\|_{L^\infty(B_R)}\leq \epsilon \left(\| \tilde{p}_2\|_{C^{1,\alpha}(\overline{B_r})}+\| \tilde{p}_2\|_{C^{1,\alpha}( \overline{B_R \backslash B_r})}\right)+C_\epsilon\|\tilde{p}_2\|_{L^\infty(B_R)}.
\eeq
Taking $\epsilon$ suitably small, we conclude from \eqref{eqn: C 1 alpha estimate for tilde p 2 crude form} that
\beq
\begin{split}
&\;\|\tilde{p}_2\|_{C^{1,\alpha}(\overline{B_r})}+\|\tilde{p}_2\|_{C^{1,\alpha}(\overline{B_R\backslash B_r})}\\
\leq &\; C(\alpha,\mu,\nu, r, R, G) (1+\|\tilde{p}_{2}\|_{L^\infty(B_R)}).
\end{split}
\eeq
Then the desired estimate follows from the fact $p_2\in [0, p_M]$ (see Section \ref{sec: problem formulation}).
\end{proof}
\end{lem}

Now we are ready to prove Lemma \ref{lem: C 1 alpha bounds for two different pressures in reference coordinate}.

\begin{proof}[Proof of Lemma \ref{lem: C 1 alpha bounds for two different pressures in reference coordinate}]
In this proof, we shall use $C_*$ to denote universal constants with the dependence $C_* = C_*(\alpha,\mu,\nu, r, R, G)$.
%
Its precise definition may vary from line to line.
\setcounter{step}{0}
\begin{step}[$L^\infty$-bound]
Rewrite \eqref{eqn: equation for difference between two p tilde solutions} as
\beq
\begin{split}
&\;\na_{X_k}\left(a\f{\pa X_k}{\pa x_{1,i}}\f{\pa X_j}{\pa x_{1,i}}\na_{X_j}( \tilde{p}_{1}-\tilde{p}_{2})\right)\\
&\;\quad +\f{G(\tilde{p}_{1})-G(\tilde{p}_{2})}{\tilde{p}_{1}-\tilde{p}_{2}}\chi_{B_r}\cdot (\tilde{p}_{1}-\tilde{p}_{2})-\na_{X_k}\f{\pa X_k}{\pa x_{1,i}}\cdot a\f{\pa X_j}{\pa x_{1,i}}\na_{X_j}(\tilde{p}_{1}-\tilde{p}_{2})\\
=&\;-\na_{X_k}\left[a\left(\f{\pa X_k}{\pa x_{1,i}}\f{\pa X_j}{\pa x_{1,i}}-\f{\pa X_k}{\pa x_{2,i}}\f{\pa X_j}{\pa x_{2,i}}\right)\na_{X_j}\tilde{p}_{2}\right]\\
&\;+a\left[\na_{X_k}\f{\pa X_k}{\pa x_{1,i}}\cdot \f{\pa X_j}{\pa x_{1,i}}-\na_{X_k}\f{\pa X_k}{\pa x_{2,i}}\cdot \f{\pa X_j}{\pa x_{2,i}}\right]\na_{X_j}\tilde{p}_{2}.
\end{split}
\eeq
Arguing as in the proof of Lemma \ref{lem: C 1 alpha bound for tilde p_2}, we may assume the coefficient matrix satisfies
\beq
\f12 \min\{\mu,\nu\}Id\leq a\f{\pa X_k}{\pa x_{1,i}}\f{\pa X_j}{\pa x_{1,i}} \leq 2\max\{\mu,\nu\}Id,
\label{eqn: ellipticity of the equation of difference of two pressures}
\eeq
and it is symmetric and piecewise $C^\alpha$ in $B_R$.
Moreover,
\beq
\f{G(\tilde{p}_{1})-G(\tilde{p}_{2})}{\tilde{p}_{1}-\tilde{p}_{2}}\chi_{B_r}\leq 0,
\eeq
and
\beq
\left|\f{G(\tilde{p}_{1})-G(\tilde{p}_{2})}{\tilde{p}_{1}-\tilde{p}_{2}}\right|+\left|\na_{X_k}\f{\pa X_k}{\pa x_{1,i}}\cdot a\f{\pa X_j}{\pa x_{1,i}}\right|\leq C(\mu,\nu,r,R, G).
\eeq
Recall that $(\tilde{p}_1-\tilde{p}_2 )|_{\pa B_R}= 0$.
By the $L^\infty$-bound of the weak solution \cite[Theorem 8.16]{gilbarg2015elliptic}, together with Lemma \ref{lem: Holder estimate of jacobi} and Lemma \ref{lem: C 1 alpha bound for tilde p_2},
\beq
\begin{split}
&\;\|\tilde{p}_1-\tilde{p}_2\|_{L^\infty(B_R)}\\
\leq &\; C(\mu,\nu,r,R,G)\left\|a\left(\f{\pa X_k}{\pa x_{1,i}}\f{\pa X_j}{\pa x_{1,i}}-\f{\pa X_k}{\pa x_{2,i}}\f{\pa X_j}{\pa x_{2,i}}\right)\na_{X_j}\tilde{p}_{2}\right\|_{L^4(B_R)}\\
&\;+C(\mu,\nu,r,R,G)\left\|a\left[\na_{X_k}\f{\pa X_k}{\pa x_{1,i}}\cdot \f{\pa X_j}{\pa x_{1,i}}-\na_{X_k}\f{\pa X_k}{\pa x_{2,i}}\cdot \f{\pa X_j}{\pa x_{2,i}}\right]\na_{X_j}\tilde{p}_{2}\right\|_{L^2(B_R)}\\
\leq &\; C_*(\D m_0+\D M_0).
\end{split}
\eeq
This proves \eqref{eqn: L inf difference of two pressures}.
\end{step}

\begin{step}[$C^{1,\alpha}$-bound]
This part of the proof is similar to that of Lemma \ref{lem: C 1 alpha bound for tilde p_2}.

In addition to \eqref{eqn: ellipticity of the equation of difference of two pressures}, we know that
\beq
a\left(\f{\pa X_k}{\pa x_{1,i}}\f{\pa X_j}{\pa x_{1,i}}-\f{\pa X_k}{\pa x_{2,i}}\f{\pa X_j}{\pa x_{2,i}}\right)\na_{X_j}\tilde{p}_{2}
\eeq
is piecewise $C^\alpha$ thanks to Lemma \ref{lem: Holder estimate of jacobi} and Lemma \ref{lem: C 1 alpha bound for tilde p_2}.
Applying \cite[Corollary 1.3]{li2000gradient} to \eqref{eqn: equation for difference between two p tilde solutions}, for $\alpha<\f14$,
\beq
\begin{split}
&\;\|\tilde{p}_1-\tilde{p}_2\|_{C^{1,\alpha}(\overline{B_r})}+\|\tilde{p}_1-\tilde{p}_2\|_{C^{1,\alpha}(\overline{B_R\backslash B_r})}\\
\leq &\;C\|G(\tilde{p}_{1})-G(\tilde{p}_{2})\|_{L^\infty(B_r)}+C\left\|\na_{X_k}\f{\pa X_k}{\pa x_{1,i}}\f{\pa X_j}{\pa x_{1,i}}\right\|_{L^\infty(B_R)}\|\na(\tilde{p}_{1}-\tilde{p}_{2})\|_{L^\infty(B_R)}\\
&\;+C\left\|\na_{X_k}\f{\pa X_k}{\pa x_{1,i}}\cdot \f{\pa X_j}{\pa x_{1,i}}-\na_{X_k}\f{\pa X_k}{\pa x_{2,i}}\cdot \f{\pa X_j}{\pa x_{2,i}}\right\|_{L^\infty(B_R)}\|\na\tilde{p}_{2}\|_{L^\infty(B_R)}\\
&\;+C\left\|\left(\f{\pa X_k}{\pa x_{1,i}}\f{\pa X_j}{\pa x_{1,i}}-\f{\pa X_k}{\pa x_{2,i}}\f{\pa X_j}{\pa x_{2,i}}\right)\na_{X_j}\tilde{p}_{2}\right\|_{C^\alpha(\overline{B_r})}\\
&\;+C\left\|\left(\f{\pa X_k}{\pa x_{1,i}}\f{\pa X_j}{\pa x_{1,i}}-\f{\pa X_k}{\pa x_{2,i}}\f{\pa X_j}{\pa x_{2,i}}\right)\na_{X_j}\tilde{p}_{2}\right\|_{C^\alpha(\overline{B_R\backslash B_r})}.
\end{split}
\label{eqn: C 1 alpha difference of two solutions crude form}
\eeq
Here the constants
\beq
C = C\left(\alpha,\mu,\nu, r, R, \left\|\f{\pa X_k}{\pa x_{1,i}}\f{\pa X_j}{\pa x_{1,i}}\right\|_{C^\alpha(B_R)}\right).
\eeq
By \eqref{eqn: L inf difference of two pressures}, Lemma \ref{lem: Holder estimate of jacobi} and Lemma \ref{lem: C 1 alpha bound for tilde p_2}, we simplify \eqref{eqn: C 1 alpha difference of two solutions crude form} to be
\beq
\begin{split}
&\;\|\tilde{p}_1-\tilde{p}_2\|_{C^{1,\alpha}(\overline{B_r})}+\|\tilde{p}_1-\tilde{p}_2\|_{C^{1,\alpha}(\overline{B_R\backslash B_r})}\\
\leq &\;C_*\|\tilde{p}_{1}-\tilde{p}_{2}\|_{L^\infty(B_r)}+C_*(\d r)^{-1}(m_{0,1}+M_{0,1})\|\na(\tilde{p}_{1}-\tilde{p}_{2})\|_{L^\infty(B_R)}\\
&\;+C_*(\d r)^{-1}(\D m_0+\D M_0)\\
&\;+C_*(\D m_0+\D M_0)+C_*(\d r)^{-\alpha}(\D m_\alpha +\D M_\alpha)\\
\leq &\;C_*\|\na(\tilde{p}_{1}-\tilde{p}_{2})\|_{L^\infty(B_R)}+C_*(\D m_\alpha +\D M_\alpha).
\end{split}
\eeq
By interpolation and arguing as in the proof of Lemma \ref{lem: C 1 alpha bound for tilde p_2},
\beq
\begin{split}
&\;\|\tilde{p}_1-\tilde{p}_2\|_{C^{1,\alpha}(\overline{B_r})}+\|\tilde{p}_1-\tilde{p}_2\|_{C^{1,\alpha}(\overline{B_R\backslash B_r})}\\
\leq &\;C_*(\D m_\alpha +\D M_\alpha+\|\tilde{p}_{1}-\tilde{p}_{2}\|_{L^\infty(B_R)}).
\end{split}
\label{eqn: Holder estimate for difference of two pressures crude form}
\eeq
Now by the $L^\infty$-bound \eqref{eqn: L inf difference of two pressures}, we conclude with \eqref{eqn: Holder estimate for difference of two pressures}.
\end{step}
\end{proof}

Lemma \ref{lem: stability of the interface velocities} follows from Lemma \ref{lem: C 1 alpha bounds for two different pressures in reference coordinate} immediately.

\begin{proof}[Proof of Lemma \ref{lem: stability of the interface velocities}]
%
Back in the physical coordinate, by \eqref{eqn: time derivative of f}, 
\beq
\pa_t h_i = -\f{1}{rf_i}\cdot u_i(\g(\th))\cdot \g'_i(\th)^\perp = -\f{\mu((1+h_i)e_r-h'_ie_\th)_j}{r(1+h_i)}\cdot\left.\left[\f{\pa X_k}{\pa x_{i,j}}\cdot \na_{X_k}\tilde{p}_i\right]\right|_{\pa B_r}.
\label{eqn: inner interface velocity from the reference coordinate}
\eeq
Here $\na_{X_k}$ is taken from the inside of $\pa B_r$.
Similarly,
\beq
\pa_t H_i = -\f{\nu((1+H_i)e_r-H'_ie_\th)_j}{R(1+H_i)}\cdot\left.\left[\f{\pa X_k}{\pa x_{i,j}}\cdot \na_{X_k}\tilde{p}_i\right]\right|_{\pa B_R}.
\label{eqn: outer interface velocity from the reference coordinate}
\eeq

By definition \eqref{eqn: def of zeta i}, $\zeta_i = 1+h_i(\th)$ in a neighborhood of $\pa B_r$, while $\zeta_i = 1+H_i(\th)$ near $\pa B_R$.
So \eqref{eqn: grad zeta} reduces to
\beq
\na \z_i=
\begin{cases}
h_i'(\th)r^{-1} e_\theta&\mbox{ on }\pa B_r,\\
H_i'(\th)R^{-1} e_\theta&\mbox{ on }\pa B_R.
\end{cases}
\eeq
Hence, \eqref{eqn: inverse of jacobi matrix} can be simplified as
\beq
\f{\pa X_k}{\pa x_{i,j}} =
\begin{cases}
(1+h_i(\th))^{-2}[(1+h_i(\th)) \d_{kj} - h_i'(\th) e_{r,k}\otimes e_{\th,j}]&\mbox{ on }\pa B_r,\\
(1+H_i(\th))^{-2}[(1+H_i(\th)) \d_{kj} - H_i'(\th) e_{r,k}\otimes e_{\th,j}]&\mbox{ on }\pa B_R.
\end{cases}
\eeq
Now we calculate by \eqref{eqn: inner interface velocity from the reference coordinate} and \eqref{eqn: outer interface velocity from the reference coordinate} that
\begin{align}
\pa_t h_i=&\; -\f{\mu}{r}\left[\f{(1+h_i)^2+(h'_i)^2}{(1+h_i)^3}e_{r}
-\f{h_i'}{(1+h_i)^2}e_{\th}\right]
\na\tilde{p}_i|_{\pa B_r},\label{eqn: simplified formula for h t}\\
\pa_t H_i=&\;-\f{\nu}{R}\cdot\f{(1+H_i)^2+(H'_i)^2}{(1+H_i)^3}\cdot
e_{r}\cdot \na\tilde{p}_i|_{\pa B_R}.
\label{eqn: simplified formula for H t big}
\end{align}
In \eqref{eqn: simplified formula for H t big}, we used the fact that $\tilde{p}_i|_{\pa B_R} = 0$ and thus $\na \tilde{p}_i|_{\pa B_R}$ is in the $e_r$-direction.

To prove \eqref{eqn: C alpha difference of time derivative of h and H}, we start with the trivial bound
\beq
\|(1+h_i(\th))^{-1}\|_{C^\alpha(\BT)}\leq C
\eeq
due to the smallness of $m_{0,i}$, where $C$ is a universal constant.
Then we simply use $\|fg\|_{C^\alpha(\BT)}\leq 3\|f\|_{C^\alpha(\BT)}\|g\|_{C^\alpha(\BT)}$ to derive that
\beq
\begin{split}
&\;\left\|\f{(1+h_1)^2+(h'_1)^2}{(1+h_1)^3}-\f{(1+h_2)^2+(h'_2)^2}{(1+h_2)^3}\right\|_{C^\alpha(\BT)}\\
\leq &\;\left\|\f{1}{1+h_1}-\f{1}{1+h_2}\right\|_{C^\alpha}
+\left\|\f{(h'_1)^2-(h'_2)^2}{(1+h_1)^3}\right\|_{C^\alpha}
+\left\|(h'_2)^2\f{(1+h_1)^3-(1+h_2)^3}{(1+h_1)^3(1+h_2)^3}\right\|_{C^\alpha}\\
\leq &\;C\|h_1-h_2\|_{C^\alpha}+C\|h_1'+h_2'\|_{C^\alpha}\|h_1'-h_2'\|_{C^\alpha}+C\|h_2'\|_{C^\alpha}^2\|h_1-h_2\|_{C^\alpha}\\
\leq &\;C(\alpha,\d, m_{\alpha,1}+ m_{\alpha,2})\D m_{\alpha}.
\end{split}
\eeq
Similarly,
\beq
\left\|\f{h_1'}{(1+h_1)^2}-\f{h_2'}{(1+h_2)^2}\right\|_{C^\alpha(\BT)}\leq C(\alpha,\d,m_{\alpha,1}+ m_{\alpha,2})\D m_{\alpha}.
\eeq
Setting $h_1 = 0$ or $h_2 = 0$ above yields 
\beq
\left\|\f{(1+h_i)^2+(h'_i)^2}{(1+h_i)^3}\right\|_{C^\alpha(\BT)}
+\left\|\f{h_i'}{(1+h_i)^2}\right\|_{C^\alpha(\BT)}\leq C(\alpha,\d,m_{\alpha,i}).
\eeq
Then it is not difficult to derive from \eqref{eqn: simplified formula for h t} that
\beq
\begin{split}
&\;\|\pa_t h_1-\pa_t h_2\|_{C^\alpha(\BT)}\\
\leq &\;C(\mu,r)\cdot C(\alpha,\d, m_{\alpha,1}+m_{\alpha,2})\D m_{\alpha}\cdot \|\na \tilde{p}_1\|_{C^\alpha(\BT)}\\
&\;+C(\mu,r)\cdot C(\d, m_{\alpha,2}) \|\na (\tilde{p}_1-\tilde{p}_2)\|_{C^\alpha(\BT)}\\
\leq &\;C(\alpha,\mu,r,R, m_{\alpha,1}+m_{\alpha,2})(\D m_{\alpha}\|\na \tilde{p}_1\|_{C^\alpha(\overline{B_r})}+ \|\na (\tilde{p}_1-\tilde{p}_2)\|_{C^\alpha(\overline{B_r})}).
\end{split}
\eeq
By Lemma \ref{lem: C 1 alpha bound for tilde p_2} and Lemma \ref{lem: C 1 alpha bounds for two different pressures in reference coordinate},
\beq
\|\pa_t h_1-\pa_t h_2\|_{C^\alpha(\BT)}
\leq  C_*(\D m_\alpha +\D M_\alpha),
\eeq
where $C_* = C_*(\alpha,\mu,\nu, r, R, G)$.
Once again, the dependence of $C_*$ on $m_{\alpha,i}+M_{\alpha,i}$ is omitted since it is assumed to be small.

Estimates for $(\pa_t H_1-\pa_t H_2)$ can be derived from \eqref{eqn: simplified formula for H t big} in a similar manner.
\end{proof}

\section{Proofs of Lemmas \ref{lem: Holder estimate of singular integral normal component difference}-\ref{lem: W 1p difference from Hilbert transform difference}}
\label{sec: proofs of estimates for difference of singular integral operators}
In this section, we prove Lemmas \ref{lem: Holder estimate of singular integral normal component difference}-\ref{lem: W 1p difference from Hilbert transform difference}.

\begin{proof}[Proof of Lemma \ref{lem: Holder estimate of singular integral normal component difference}]
Let $l_i$ be defined as in \eqref{eqn: def of l} corresponding to $h_i$.
By virtue of \eqref{eqn: crude form of singular integral dot with normal},
%
\beq
\begin{split}
&\;2\pi(\g'_1(\th)^\perp\cdot\mathcal{K}_{\g_1}\p-\g'_2(\th)^\perp\cdot\mathcal{K}_{\g_2}\p)\\
= &\;\f{1}{2}\int_\BT \left(\f{1} {1+l_1}-\f{1} {1+l_2}\right)\p(\th+\xi)\,d\xi\\
&\;+\f{h_2(\th)-h_1(\th)}{1+h_2(\th)}\cdot \f{1}{1+h_1(\th)}\int_\BT\f{\D h_1(\th)-\cos\f{\xi}{2}\cdot h'_1(\th)}{2\sin\f{\xi}{2}} \f{\p(\th+\xi)} {1+l_1}\,d\xi\\
&\;+\f{1}{1+h_2(\th)}\int_\BT\f{\D (h_1-h_2)(\th)-\cos\f{\xi}{2}\cdot (h_1-h_2)'(\th)}{2\sin\f{\xi}{2}} \f{\p(\th+\xi)} {1+l_1}\,d\xi\\
&\;+\f{1}{1+h_2(\th)}\int_\BT\f{\D h_2(\th)-\cos\f{\xi}{2}\cdot h_2'(\th)}{2\sin\f{\xi}{2}}\left(\f{\p(\th+\xi)} {1+l_1}-\f{\p(\th+\xi)} {1+l_2}\right)\,d\xi\\
=:&\; J_1+J_2+J_3+J_4.
\end{split}
\eeq
We start with the integrand of $J_1$.
\beq
\begin{split}
&\;\left\|\left(\f{1} {1+l_1}-\f{1} {1+l_2}\right)\p(\th+\xi)\right\|_{\dot{C}_\th^\beta}\\
\leq &\;\left\|\f{1} {1+l_1}-\f{1} {1+l_2}\right\|_{\dot{C}_\th^\beta}\|\p\|_{L^\infty}+\left\|\f{1} {1+l_1}-\f{1} {1+l_2}\right\|_{L^\infty_\th}\|\p\|_{\dot{C}^\beta}\\
\leq &\;C\|l_1-l_2\|_{\dot{C}_\th^\beta}\|\p\|_{L^\infty}+C\|l_1-l_2\|_{L^\infty_\th} (\|l_1\|_{\dot{C}_\th^\beta}+\|l_2\|_{\dot{C}_\th^\beta})\|\p\|_{L^\infty}\\
&\;+C\|l_1-l_2\|_{L^\infty_\th}\|\p\|_{\dot{C}^\beta}.
\end{split}
\label{eqn: Holder estimate for integrand of first term}
\eeq
We derive that
\beq
\begin{split}
&\;\|l_1-l_2\|_{L^\infty_\th}\\
\leq &\;\left\|\f{(\D h_1)^2-(\D h_2)^2}{(1+h_1(\th))(1+h_1(\th+\xi))}\right\|_{L_\th^\infty}\\
&\;+\left\|(\D h_2)^2\left(\f{1}{(1+h_1(\th))(1+h_1(\th+\xi))}-\f{1}{(1+h_2(\th))(1+h_2(\th+\xi))}\right)\right\|_{L^\infty}\\
\leq &\;C(\|h_1'\|_{L^\infty}+\|h_2'\|_{L^\infty})\|h_1-h_2\|_{W^{1,\infty}},
\end{split}
\label{eqn: L inf difference of building blocks from two configurations}
\eeq
and
\beq
\begin{split}
&\;\|l_1-l_2\|_{\dot{C}_\th^\beta}\\
\leq &\;\left\|\f{(\D h_1)^2-(\D h_2)^2}{(1+h_1(\th))(1+h_1(\th+\xi))}\right\|_{\dot{C}_\th^\beta}\\
&\;+\left\|(\D h_2)^2\left(\f{1}{(1+h_1(\th))(1+h_1(\th+\xi))}-\f{1}{(1+h_2(\th))(1+h_2(\th+\xi))}\right)\right\|_{\dot{C}_\th^\beta}\\
\leq &\;C\|h_1'+h_2'\|_{\dot{C}^\b}\|h_1'-h_2'\|_{L^\infty}+C\|h_1'+h_2'\|_{L^\infty}\|h_1'-h_2'\|_{\dot{C}^\b}\\
&\;+C\|h_1'+h_2'\|_{L^\infty}\|h_1'-h_2'\|_{L^\infty}\|h_1\|_{\dot{C}^\b}\\
&\;+C\|h_2'\|_{\dot{C}^\b}\|h_2'\|_{L^\infty}\|h_1-h_2\|_{L^\infty}+C\|h_2'\|_{L^\infty}^2\|h_1-h_2\|_{\dot{C}^\b}\\
&\;+C\|h_2'\|_{L^\infty}^2\|h_1-h_2\|_{L^\infty}(\|h_1\|_{\dot{C}^\b}+\|h_2\|_{\dot{C}^\b})\\
\leq &\;C(\|h_1'\|_{\dot{C}^\b}+\|h_2'\|_{\dot{C}^\b})\|h_1-h_2\|_{C^{1,\b}}.
\end{split}
\label{eqn: Holder difference of building blocks from two configurations}
\eeq
Taking $h_2 = 0$ in the second last step of \eqref{eqn: Holder difference of building blocks from two configurations} yields that
\beq
\|l_i\|_{\dot{C}_\th^\b}
\leq C\|h'_i\|_{\dot{C}^\b}\|h'_i\|_{L^\infty}.
\label{eqn: Holder estimate of a single building blocks}
\eeq
Combining these estimates with \eqref{eqn: Holder estimate for integrand of first term}, we argue as in \eqref{eqn: Holder bound for L_1 crude form} that
\beq
\begin{split}
\|J_1\|_{\dot{C}^\b}
\leq
&\;C\sup_{\xi\in \BT}\left\|\left(\f{1} {1+l_1}-\f{1} {1+l_2}\right)\p(\th+\xi)\right\|_{\dot{C}_\th^\beta}\\
\leq &\;C\|h_1-h_2\|_{C^{1,\b}} (\|h_1'\|_{\dot{C}^\b}+\|h_2'\|_{\dot{C}^\b})\|\p\|_{C^\beta}.
\end{split}
\label{eqn: estimate for J_1}
\eeq

Next, by taking advantage of \eqref{eqn: difference of difference quotient and the derivative} and \eqref{eqn: estimate finite difference of L2+L3},
\beq
\begin{split}
&\;\|J_2\|_{\dot{C}^\b}\\
\leq &\; \left\|\f{h_2-h_1}{1+h_2}\right\|_{\dot{C}^\b}\left\|\f{1}{1+h_1}\int_\BT \f{\D h_1(\th)-\cos\f{\xi}{2}\cdot h'_1(\th)}{2\sin\f{\xi}{2}}\f{\p(\th+\xi)} {1+l_1}\,d\xi\right\|_{L^\infty}\\
&\;+\left\|\f{h_2-h_1}{1+h_2}\right\|_{L^\infty}\left\|\f{1}{1+h_1}\int_\BT \f{\D h_1(\th)-\cos\f{\xi}{2}\cdot h'_1(\th)}{2\sin\f{\xi}{2}}\f{\p(\th+\xi)} {1+l_1}\,d\xi\right\|_{\dot{C}^\b}\\
\leq &\; C(\|h_2-h_1\|_{\dot{C}^\b}+\|h_2-h_1\|_{L^\infty}\|h_2\|_{\dot{C}^\b})\cdot \|h_1'\|_{\dot{C}^\b}\|\p\|_{L^\infty}\\
&\;+C\|h_2-h_1\|_{L^\infty}\cdot \|h_1'\|_{\dot{C}^\b}(\|\p\|_{C^\b}+\|\p\|_{L^\infty}\|h_1'\|_{\dot{C}^\b}\|h_1'\|_{L^\infty}).
\end{split}
\label{eqn: estimate for J_2}
\eeq

Arguing as in \eqref{eqn: sum of L_2 and L_3}-\eqref{eqn: estimate finite difference of L2+L3},
\beq
\|J_3\|_{\dot{C}^\b}
\leq
C\|(h_1-h_2)'\|_{\dot{C}^\b}(\|\p\|_{C^\b}+\|\p\|_{L^\infty}\|h_1'\|_{\dot{C}^\b}\|h_1'\|_{L^\infty}).
\label{eqn: estimate for J_3}
\eeq
In order to apply the same argument to $J_4$, we need the following estimate.
\beq
\begin{split}
&\;\left|\f{\p(\th+\xi)} {1+l_1}-\f{\p(\th+\xi)} {1+l_2}-\f{\p(\th)} {1+\f{ h_1'(\th)^2}{(1+h_1)^2}}+\f{\p(\th)} {1+\f{ h_2'(\th)^2}{(1+h_2)^2}}\right|\\
\leq &\;C|\p(\th+\xi)-\p(\th)||l_1-l_2|\\
&\;+C|\p(\th)|\left|l_1-l_2-\f{ h_1'(\th)^2}{(1+h_1)^2}+\f{ h_2'(\th)^2}{(1+h_2)^2}\right|\\
&\;+C|\p(\th)|\left|\f{ h_1'(\th)^2}{(1+h_1(\th))^2}-\f{ h_2'(\th)^2}{(1+h_2(\th))^2}\right| \left(\left|l_1-\f{ h_1'(\th)^2}{(1+h_1)^2}\right|+\left|l_2-\f{ h_2'(\th)^2}{(1+h_2)^2}\right|\right).
\end{split}
\label{eqn: a preliminary estimate for J_4 crude form}
\eeq
Since
\beq
\begin{split}
&\;\left|l_1-l_2-\f{ h_1'(\th)^2}{(1+h_1)^2}+\f{ h_2'(\th)^2}{(1+h_2)^2}\right|\\
\leq &\;\left|\f{(\D h_1(\th))^2-h_1'(\th)^2}{(1+h_1)(1+h_1(\th+\xi))}-\f{(\D h_2(\th))^2-h_2'(\th)^2}{(1+h_2)(1+h_2(\th+\xi))}\right|\\
&\;+\left|\f{h_1'(\th)^2(h_1(\th+\xi)-h_1(\th))}{(1+h_1)^2(1+h_1(\th+\xi))} -\f{h_2'(\th)^2(h_2(\th+\xi)-h_2(\th))}{(1+h_2)^2(1+h_2(\th+\xi))}\right|\\
\leq &\;C|\xi|^\b\|h_1-h_2\|_{C^{1,\b}}(\|h_1'\|_{\dot{C}^\b} +\|h_2'\|_{\dot{C}^\b}), 
\end{split}
\label{eqn: double difference of l and its value at theta}
\eeq
we apply this and \eqref{eqn: L inf difference of building blocks from two configurations} to \eqref{eqn: a preliminary estimate for J_4 crude form} to conclude that
\beq
\begin{split}
&\;\left|\f{\p(\th+\xi)} {1+l_1}-\f{\p(\th+\xi)} {1+l_2}-\f{\p(\th)} {1+\f{ h_1'(\th)^2}{(1+h_1)^2}}+\f{\p(\th)} {1+\f{ h_2'(\th)^2}{(1+h_2)^2}}\right|\\
\leq &\;C|\xi|^\b\|\p\|_{\dot{C}^\b}\cdot (\|h_1'\|_{L^\infty}+\|h_2'\|_{L^\infty})\|h_1-h_2\|_{W^{1,\infty}}\\
&\;+C\|\p\|_{L^\infty}\cdot |\xi|^\b\|h_1-h_2\|_{C^{1,\b}}(\|h_1'\|_{\dot{C}^\b} +\|h_2'\|_{\dot{C}^\b})\\
&\;+C\|\p\|_{L^\infty}(\|h_1'\|_{L^\infty}+\|h_2'\|_{L^\infty})\|h_1-h_2\|_{W^{1,\infty}}\cdot |\xi|^\b(\|h_1'\|_{L^\infty}\|h_1'\|_{\dot{C}^\b}+\|h_2'\|_{L^\infty}\|h_2'\|_{\dot{C}^\b})\\
\leq &\;C|\xi|^\b\|\p\|_{C^\b} (\|h_1'\|_{\dot{C}^\b}+\|h_2'\|_{\dot{C}^\b})\|h_1-h_2\|_{C^{1,\b}}.
\end{split}
\label{eqn: a preliminary estimate for J_4}
\eeq
Now we proceed as in \eqref{eqn: sum of L_2 and L_3}-\eqref{eqn: estimate finite difference of L2+L3}.
%
\beq
\begin{split}
&\;|J_4(\th+\e)-J_4(\th)|\\
\leq &\;\left|\f{1}{1+h_2(\th+\e)}-\f{1}{1+h_2(\th)}\right|\int_\BT\left|\f{\D h_2(\th+\e)-\cos\f{\xi}{2}\cdot h_2'(\th+\e)}{2\sin\f{\xi}{2}}\right|\\
&\;\qquad \cdot \left|\f{\p(\th+\e+\xi)} {1+l_1(\th+\e,\th+\e+\xi)}-\f{\p(\th+\e+\xi)} {1+l_2(\th+\e,\th+\e+\xi)}\right|\,d\xi\\
&\;+C\int_\BT\left|\f{\D h_2(\th+\e)-\cos\f{\xi}{2}\cdot h_2'(\th+\e)}{2\sin\f{\xi}{2}}\right|\\
&\;\qquad \cdot |\e|^\b \sup_\xi \left\|\f{\p(\th+\xi)} {1+l_1(\th,\th+\xi)}-\f{\p(\th+\xi)} {1+l_2(\th,\th+\xi)}\right\|_{\dot{C}^\b_\th}\,d\xi\\
&\;+C\int_\BT\left|\f{\D h_2(\th+\e)-\D h_2(\th)-\cos\f{\xi}{2}(h_2'(\th+\e)- h_2'(\th))}{2\sin\f{\xi}{2}}\right|\\
&\;\qquad \cdot \left|\f{\p(\th+\xi)} {1+l_1(\th,\th+\xi)}-\f{\p(\th+\xi)} {1+l_2(\th,\th+\xi)}-\f{\p(\th)} {1+\f{ h_1'(\th)^2}{(1+h_1(\th))^2}}+\f{\p(\th)} {1+\f{ h_2'(\th)^2}{(1+h_2(\th))^2}}\right|\,d\xi\\
&\;+C\left|\int_\BT\f{\D h_2(\th+\e)-\D h_2(\th)-\cos\f{\xi}{2}(h_2'(\th+\e)- h_2'(\th))}{2\sin\f{\xi}{2}}\,d\xi\right|\\
&\;\qquad \cdot \left|\f{\p(\th)} {1+\f{ h_1'(\th)^2}{(1+h_1(\th))^2}}-\f{\p(\th)} {1+\f{ h_2'(\th)^2}{(1+h_2(\th))^2}}\right|.
\end{split}
\eeq
By \eqref{eqn: absolute difference of Taylor expansion type error}, \eqref{eqn: Holder estimate using Hilbert transform}, \eqref{eqn: L inf difference of building blocks from two configurations}-\eqref{eqn: estimate for J_1} and \eqref{eqn: a preliminary estimate for J_4},
\beq
\begin{split}
&\;|J_4(\th+\e)-J_4(\th)|\\
%
\leq &\;C\e^\b  \|\p\|_{C^\b} \|h_2'\|_{\dot{C}^\b} (\|h_1'\|_{\dot{C}^\b}+\|h_2'\|_{\dot{C}^\b})\|h_1-h_2\|_{C^{1,\b}}.
\end{split}
\label{eqn: estimate for J_4}
\eeq

Combining \eqref{eqn: estimate for J_1}-\eqref{eqn: estimate for J_3} and \eqref{eqn: estimate for J_4} yield the desired estimate.
\end{proof}

\begin{proof}[Proof of Lemma \ref{lem: W 1p estimate of normal component of singular integral difference}]
Let $C_*$ and $C_\dag$ be the constants in Lemma \ref{lem: Lp bound for multi-term commutator} and Lemma \ref{lem: W 1p estimate for multi-term commutator}, respectively, both of which only depend on $p$.
Without loss of generality, we may assume $C_\dag\geq C_*\geq 1$.
Following \eqref{eqn: crude form of singular integral dot with normal}, we use $L_{k}^{(i)}$ $(k = 0,1,2,3)$ to denote the corresponding quantities defined by $h_i$ $(i = 1,2)$.
$l_i$ are defined as in \eqref{eqn: def of l} by $h_i$.
Thanks to the smallness of $h_i$, we may assume $|l_i|<1$, and that $C_2>0$ is a universal constant such that $\|(1+h_i)^{-1}\|_{L^\infty}\leq C_2$.

We start with bounding $L_1^{(1)}-L_1^{(2)}$.
Taking their $\th$-derivatives, we use \eqref{eqn: L inf difference of building blocks from two configurations} to derive that 
\beq
\begin{split}
&\;\|L_1^{(1)}-L_1^{(2)}\|_{\dot{W}^{1,p}}\\
\leq &\; \f{1}{2}\left\|\int_\BT \left[\f{\f{2\D h_1 \D h_1'}{(1+h_1(\th))(1+h_1(\th+\xi))}-\f{(\D h_1)^2 (h_1'(\th)+h_1'(\th+\xi)+h_1'(\th)h_1(\th+\xi)+h_1(\th)h_1'(\th+\xi))} {(1+h_1(\th))^2(1+h_1(\th+\xi))^2}} {(1+l_1(\th,\th+\xi))^2} \right.\right.\\
&\;\left.\left.\qquad
-\f{\f{2\D h_2 \D h_2'}{(1+h_2(\th))(1+h_2(\th+\xi))}-\f{(\D h_2)^2 (h_2'(\th)+h_2'(\th+\xi)+h_2'(\th)h_2(\th+\xi)+h_2(\th)h_2'(\th+\xi))} {(1+h_2(\th))^2(1+h_2(\th+\xi))^2}} {(1+l_2(\th,\th+\xi))^2} \right] \p(\th+\xi)\,d\xi\right\|_{L^p}
\\
&\; +\f{1}{2} \left\|\int_\BT \left(\f{1} {1+l_1}
-\f{1} {1+l_2} \right) \p'(\th+\xi)\,d\xi\right\|_{L^p}
\\
\leq &\; C(\|h_1'-h_2'\|_{L^\infty}\|h_1''\|_{L^p}+\|h_2'\|_{L^\infty}\|h_1''-h_2''\|_{L^p}+\|h_2'\|_{L^\infty}\|h_2''\|_{L^p}\|h_1-h_2\|_{W^{1,\infty}})\|\p\|_{L^\infty}
\\
&\; +C (\|h_1'\|_{L^\infty}+\|h_2'\|_{L^\infty}) \|h_1-h_2\|_{W^{1,\infty}}\|\p'\|_{L^p}.
\end{split}
\label{eqn: W1p estimate for L1 difference}
\eeq

As in \eqref{eqn: decomposition of L2}, we Taylor expand $(1+l_i)^{-1}$ and rewrite $L_2^{(i)}$ as
\beq
\begin{split}
L_2^{(i)} = &\;\sum_{j = 0}^\infty(-1)^j(1+h_i(\th))^{-(j+1)}
\mathrm{p.v.}\int_\BT(\D h_i)^{2j+1}(1+h_i(\th+\xi))^{-j}\cdot \f{\p(\th+\xi)}{2\sin\f{\xi}{2}}\,d\xi\\
=:&\;\sum_{j=0}^\infty L_{2,j}^{(i)}.
\end{split}
\label{eqn: decomposition of L2i}
\eeq
We derive
\beq
\begin{split}
&\;(-1)^j(L_{2,j}^{(1)}-L_{2,j}^{(2)})\\
= &\;\left[(1+h_1(\th))^{-(j+1)}-(1+h_2(\th))^{-(j+1)}\right]\\
&\;\quad \cdot \mathrm{p.v.}\int_\BT(\D h_1)^{2j+1}(1+h_1(\th+\xi))^{-j}\cdot \f{\p(\th+\xi)}{2\sin\f{\xi}{2}}\,d\xi\\
&\;+(1+h_2(\th))^{-(j+1)}\\
&\;\quad \cdot \mathrm{p.v.}\int_\BT \D (h_1-h_2)\sum_{l = 0}^{2j}(\D h_1)^{l}(\D h_2)^{2j-l} \cdot (1+h_1(\th+\xi))^{-j}\cdot \f{\p(\th+\xi)}{2\sin\f{\xi}{2}}\,d\xi\\
&\;+(1+h_2(\th))^{-(j+1)}\\
&\;\quad \cdot \mathrm{p.v.}\int_\BT (\D h_2)^{2j+1} \left(\f{1}{1+h_1(\th+\xi)}-\f{1}{1+h_2(\th+\xi)}\right)\\
&\;\quad \qquad\qquad \cdot \sum_{l = 0}^{j-1}(1+h_1(\th+\xi))^{-l}(1+h_2(\th+\xi))^{-(j-1-l)}\cdot \f{\p(\th+\xi)}{2\sin\f{\xi}{2}}\,d\xi.
\end{split}
\label{eqn: decomposition of L2 difference}
\eeq
Note that here in this proof, with abuse of notations, we use $l$ as a summation index, which has nothing to do with \eqref{eqn: def of l}.
%
%
%

By Lemma \ref{lem: Lp bound for multi-term commutator}, for $0\leq l\leq k$,
\beq
\begin{split}
&\;\left\|\mathrm{p.v.}\int_\BT \D (h_1-h_2)(\D h_1)^{l}(\D h_2)^{k-l} \cdot \f{\p(\th+\xi)}{2\sin\f{\xi}{2}}\,d\xi\right\|_{L^p}\\
\leq &\;C C_*^k (\|h_1'\|_{L^\infty}+\|h_2'\|_{L^\infty})^k\|h_1'-h_2'\|_{L^\infty}\|\p\|_{L^p}.
\end{split}
\label{eqn: Lp estimate of a typical term in L2j}
\eeq
Letting $k = 2j$ and replacing $\p$ by $(1+h_1)^{-j}\p$,
\beq
\begin{split}
&\;\left\|\mathrm{p.v.}\int_\BT \D (h_1-h_2)(\D h_1)^{l}(\D h_2)^{2j-l} (1+h_1(\th+\xi))^{-j}\cdot \f{\p(\th+\xi)}{2\sin\f{\xi}{2}}\,d\xi\right\|_{L^p}\\
\leq &\;C (C_*^2 C_2 (\|h_1'\|_{L^\infty}+\|h_2'\|_{L^\infty})^2)^j\|h_1'-h_2'\|_{L^\infty}\|\p\|_{L^p}.
\end{split}
\label{eqn: Lp estimate of L2j second term}
\eeq
Further taking $h_2 = 0$ and $l = 2j$, we find
\beq
\begin{split}
&\;\left\|\mathrm{p.v.}\int_\BT (\D h_1)^{2j+1} (1+h_1(\th+\xi))^{-j}\cdot \f{\p(\th+\xi)}{2\sin\f{\xi}{2}}\,d\xi\right\|_{L^p}\\
\leq &\;C (C_*^{2}C_2\|h_1'\|_{L^\infty}^{2})^j\|h_1'\|_{L^\infty}\|\p\|_{L^p}.
\end{split}
\label{eqn: Lp estimate of L2j first term}
\eeq
Similarly, 
\beq
\begin{split}
&\;\left\|\mathrm{p.v.}\int_\BT (\D h_2)^{2j+1} \left(\f{1}{1+h_1(\th+\xi)}-\f{1}{1+h_2(\th+\xi)}\right)\right.\\
&\;\left.\qquad  \cdot (1+h_1(\th+\xi))^{-l}(1+h_2(\th+\xi))^{-(j-1-l)}\cdot \f{\p(\th+\xi)}{2\sin\f{\xi}{2}}\,d\xi\right\|_{L^p}\\
\leq &\;C (C_*^{2}C_2\|h_2'\|_{L^\infty}^{2})^j\|h_2'\|_{L^\infty}\|h_1-h_2\|_{L^\infty}\|\p\|_{L^p}.
\end{split}
\label{eqn: Lp estimate of L2j third term}
\eeq

On the other hand, by Lemma \ref{lem: W 1p estimate for multi-term commutator}, for $0\leq l\leq k$,
\beq
\begin{split}
&\;\left\|\mathrm{p.v.}\int_\BT \D (h_1-h_2)(\D h_1)^{l}(\D h_2)^{k-l} \cdot \f{\p(\th+\xi)}{2\sin\f{\xi}{2}}\,d\xi\right\|_{\dot{W}^{1,p}}\\
\leq &\;(k+2)C_\dag^{k+2}\|\p'\|_{L^p}\|h_1'-h_2'\|_{L^\infty}(\|h_1'\|_{L^\infty}+\|h_2'\|_{L^\infty})^k\\
&\;+(k+2)C_\dag^{k+2}\|\p\|_{L^\infty}\|h_1''-h_2''\|_{L^p}(\|h_1'\|_{L^\infty}+\|h_2'\|_{L^\infty})^{k}\\
&\;+(k+2)C_\dag^{k+2}\|\p\|_{L^\infty}\|h_1'-h_2'\|_{L^\infty} (\|h_1'\|_{L^\infty}+\|h_2'\|_{L^\infty})^{k-1}\cdot \mathds{1}_{\{k>0\}}(\|h_1''\|_{L^p}+\|h_2''\|_{L^p}).
\end{split}
\label{eqn: W1p estimate of L2j typical term}
\eeq
Taking $k = 2j$ and replacing $\p$ by $(1+h_1)^{-j}\p$,
\beq
\begin{split}
&\;\left\|\mathrm{p.v.}\int_\BT \D (h_1-h_2)(\D h_1)^{l}(\D h_2)^{2j-l} (1+h_1(\th+\xi))^{-j}\cdot \f{\p(\th+\xi)}{2\sin\f{\xi}{2}}\,d\xi\right\|_{\dot{W}^{1,p}}\\
%
\leq &\;(2j+2)C_\dag^{2j+2}(jC_2^{j+1}\|h_1'\|_{L^\infty}\|\p\|_{L^p}+C_2^{j}\|\p'\|_{L^p}) \|h_1'-h_2'\|_{L^\infty}(\|h_1'\|_{L^\infty}+\|h_2'\|_{L^\infty})^{2j}\\
&\;+(2j+2)C_\dag^{2j+2}C_2^j\|\p\|_{L^\infty}\|h_1''-h_2''\|_{L^p}(\|h_1'\|_{L^\infty}+\|h_2'\|_{L^\infty})^{2j}\\
&\;+(2j+2)C_\dag^{2j+2}C_2^j\|\p\|_{L^\infty}\|h_1'-h_2'\|_{L^\infty}\\
&\;\qquad\cdot \mathds{1}_{\{j>0\}}(\|h_1'\|_{L^\infty}+\|h_2'\|_{L^\infty})^{2j-1}(\|h_1''\|_{L^p}+\|h_2''\|_{L^p}).
%
\end{split}
\label{eqn: W1p estimate of L2j second term}
\eeq
Further taking $h_2 = 0$ and $l = 2j$,
\beq
\begin{split}
&\;\left\|\mathrm{p.v.}\int_\BT (\D h_1)^{2j+1} (1+h_1(\th+\xi))^{-j}\cdot \f{\p(\th+\xi)}{2\sin\f{\xi}{2}}\,d\xi\right\|_{\dot{W}^{1,p}}\\
\leq &\;C(j+1)(C_\dag^2 C_2 \|h_1'\|_{L^\infty}^2)^j [
(j\|h_1'\|_{L^\infty}\|\p\|_{L^p}+\|\p'\|_{L^p}) \|h_1'\|_{L^\infty}+\|\p\|_{L^\infty}\|h_1''\|_{L^p}].
\end{split}
\label{eqn: W1p estimate of L2j first term}
\eeq
Similarly,
\beq
\begin{split}
&\;\left\|\mathrm{p.v.}\int_\BT (\D h_2)^{2j+1} \left(\f{1}{1+h_1(\th+\xi)}-\f{1}{1+h_2(\th+\xi)}\right)\right.\\
&\;\left.\qquad  \cdot (1+h_1(\th+\xi))^{-l}(1+h_2(\th+\xi))^{-(j-1-l)}\cdot \f{\p(\th+\xi)}{2\sin\f{\xi}{2}}\,d\xi\right\|_{\dot{W}^{1,p}}\\
\leq &\;(2j+2)C_\dag^{2j+2} \|h_2'\|_{L^\infty}^{2j+1}\\
&\;\quad\cdot \left\|\left(\f{1}{1+h_1}-\f{1}{1+h_2}\right) (1+h_1)^{-l}(1+h_2)^{-(j-1-l)}\p\right\|_{\dot{W}^{1,p}}\\
&\;+C(2j+2)C_\dag^{2j+2} \|h_2'\|_{L^\infty}^{2j}\|h_2''\|_{L^p}\cdot C_2^{j+1}\|h_1-h_2\|_{L^\infty}\|\p\|_{L^\infty}\\
\leq &\;C(j+1)(C_\dag^{2}C_2 \|h_2'\|_{L^\infty}^{2})^j\|h_2'\|_{L^\infty}\cdot \left[\|h_1'-h_2'\|_{L^\infty}\|\p\|_{L^p} \right.\\
&\;\left.\qquad +j\|h_1-h_2\|_{L^\infty}(\|h_1'\|_{L^\infty}+\|h_2'\|_{L^\infty})\|\p\|_{L^p}+\|h_1-h_2\|_{L^\infty}\|\p'\|_{L^p}\right]\\
&\;+C(j+1)(C_\dag^{2}C_2 \|h_2'\|_{L^\infty}^{2})^j\|h_2''\|_{L^p}\cdot\|h_1-h_2\|_{L^\infty}\|\p\|_{L^\infty}.
\end{split}
\label{eqn: W1p estimate of L2j third term}
\eeq

Combining these estimates with \eqref{eqn: decomposition of L2 difference}, we use 
the fact $\|fg\|_{\dot{W}^{1,p}}\leq \|f\|_{\dot{W}^{1,\infty}}\|g\|_{L^p}+\|f\|_{L^\infty}\|g\|_{\dot{W}^{1,p}}$ to derive that
\beq
\begin{split}
&\;\|L_{2,j}^{(1)}-L_{2,j}^{(2)}\|_{\dot{W}^{1,p}}\\
\leq &\;
C(j+1)C_2^j(\|h_1'-h_2'\|_{L^\infty}+(j+2)\|h_1-h_2\|_{L^\infty}\|h_2'\|_{L^\infty}) \cdot (C_*^{2}C_2\|h_1'\|_{L^\infty}^{2})^j\|h_1'\|_{L^\infty}\|\p\|_{L^p}\\
&\;+C(j+1)C_2^j\|h_1-h_2\|_{L^\infty}\\
&\;\quad \cdot (j+1)(C_\dag^2 C_2 \|h_1'\|_{L^\infty}^2)^j [
(j\|h_1'\|_{L^\infty}\|\p\|_{L^p}+\|\p'\|_{L^p}) \|h_1'\|_{L^\infty}+\|\p\|_{L^\infty}\|h_1''\|_{L^p}]\\
&\;+C(j+1)C_2^j\|h_2'\|_{L^\infty} \sum_{l = 0}^{2j} (C_*^{2}C_2(\|h_1'\|_{L^\infty}+\|h_2'\|_{L^\infty})^{2})^j\|h_1'-h_2'\|_{L^\infty}\|\p\|_{L^p}\\
&\;+CC_2^j\sum_{l = 0}^{2j}(j+1)(C_\dag^2 C_2 (\|h_1'\|_{L^\infty}+\|h_2'\|_{L^\infty})^2)^{j}\\
&\;\qquad \cdot [
(j\|h_1'\|_{L^\infty}\|\p\|_{L^p}+\|\p'\|_{L^p}) \|h_1'-h_2'\|_{L^\infty}+\|\p\|_{L^\infty}\|h_1''-h_2''\|_{L^p} \\
&\;\qquad\qquad  +\mathds{1}_{\{j>0\}} \|\p\|_{L^\infty}\|h_1'-h_2'\|_{L^\infty}(\|h_1''\|_{L^p}+\|h_2''\|_{L^p})(\|h_1'\|_{L^\infty}+\|h_2'\|_{L^\infty})^{-1}]
\\
&\;+C(j+1)C_2^j\|h_2'\|_{L^\infty}\sum_{l = 0}^{j-1} (C_*^{2}C_2\|h_2'\|_{L^\infty}^{2})^j\|h_2'\|_{L^\infty}\|h_1-h_2\|_{L^\infty}\|\p\|_{L^p}
\\
&\;+CC_2^j \sum_{l = 0}^{j-1}
(j+1)(C_\dag^{2}C_2 \|h_2'\|_{L^\infty}^{2})^j\|h_2'\|_{L^\infty}\cdot \left[\|h_1'-h_2'\|_{L^\infty}\|\p\|_{L^p} \right.\\
&\;\left.\qquad \qquad +j\|h_1-h_2\|_{L^\infty}(\|h_1'\|_{L^\infty}+\|h_2'\|_{L^\infty})\|\p\|_{L^p}+\|h_1-h_2\|_{L^\infty}\|\p'\|_{L^p}\right]\\
&\;\qquad \qquad +(j+1)(C_\dag^{2}C_2 \|h_2'\|_{L^\infty}^{2})^j\|h_2''\|_{L^p}\cdot\|h_1-h_2\|_{L^\infty}\|\p\|_{L^\infty}.
\end{split}
\eeq
Assuming $\|h_i'\|_{L^\infty}\ll 1$, 
\beq
\begin{split}
\|L_{2}^{(1)}-L_{2}^{(2)}\|_{\dot{W}^{1,p}}
\leq &\;\sum_{j = 0}^\infty\|L_{2,j}^{(1)}-L_{2,j}^{(2)}\|_{\dot{W}^{1,p}}\\
\leq &\;C(\|\p'\|_{L^p} \|h_1-h_2\|_{W^{1,\infty}}+\|\p\|_{L^\infty}\|h_1''-h_2''\|_{L^p}) \\
&\;+ C \|\p\|_{L^\infty}\|h_1-h_2\|_{W^{1,\infty}}(\|h_1''\|_{L^p}+\|h_2''\|_{L^p}).
\end{split}
\label{eqn: W1p estimate for L2 difference}
\eeq

We similarly write 
\beq
\begin{split}
L_3^{(i)}=&\;\sum_{j = 0}^\infty h_i'(\th)(-1-h_i(\th))^{-(j+1)}\mathrm{p.v.}\int_\BT (\D h_i)^{2j}(1+h_i(\th+\xi))^{-j}\cdot \f{\p(\th+\xi)}{2\tan\f{\xi}{2}}\,d\xi\\
=:&\;\sum_{j=0}^\infty L_{3,j}^{(i)}
\end{split}
\label{eqn: decomposition of L3i}
\eeq
and
\beq
\begin{split}
&\;(-1)^{j+1}(L_{3,j}^{(1)}-L_{3,j}^{(2)})\\
=&\;\left[h_1'(\th)(1+h_1(\th))^{-(j+1)}-h_2'(\th)(1+h_2(\th))^{-(j+1)}\right] \cdot \mathrm{p.v.}\int_\BT A_i^{j} \cdot \f{\p(\th+\xi)}{2\tan\f{\xi}{2}}\,d\xi\\
&\;+h_2'(\th)(1+h_2(\th))^{-(j+1)}\cdot \mathrm{p.v.}\int_\BT (A_1^j-A_2^j)\cdot
\f{\p(\th+\xi)}{2\tan\f{\xi}{2}}\,d\xi,
\end{split}
\label{eqn: decomposition of difference of L3j}
\eeq
where
\beq
A_i := \f{(\D h_i)^{2}}{1+h_i(\th+\xi)} = (1+h_i(\th))\cdot l_i(\th,\th+\xi).
\label{eqn: def of A_i}
\eeq

To proceed as before, we need $L^\infty$-bounds for the integrals in \eqref{eqn: decomposition of difference of L3j}.
We additionally define
\beq
B_i = \f{ h_i'(\th)^{2}}{1+h_i(\th)}.
\eeq
It is easy to show that $|A_i|,|B_i|\leq C_1^2 C_2\|h_i'\|_{L^\infty}^2$, where $C_1 = \pi/2$ is introduced in the proof of Lemma \ref{lem: Lp bound for multi-term commutator}, and
\beq
|A_i-B_i|\leq 
C\|h_i'\|_{L^\infty}\|h_i'\|_{\dot{C}^\b}|\xi|^\b.
\label{eqn: difference between A_i and B_i}
\eeq
%
%
%
Hence, by the mean value theorem,
\beq
\begin{split}
&\;\left|\mathrm{p.v.}\int_\BT A_i^{j} \cdot \f{\p(\th+\xi)}{2\tan\f{\xi}{2}}\,d\xi\right|\\
=&\;\left|\int_\BT (A_i^j\p(\th+\xi)
-B_i^j \p(\th))\f{1}{2\tan\f{\xi}{2}}\,d\xi\right|\\
\leq &\;C\int_\BT j(C_1^2C_2\|h_i'\|_{L^\infty}^2)^{j-1}\cdot\|h_i'\|_{L^\infty}\|h_i'\|_{\dot{C}^\b} |\xi|^\b\cdot \|\p\|_{L^\infty}|\xi|^{-1}\,d\xi\\
&\;+C\int_\BT (C_1^2C_2\|h_i'\|_{L^\infty}^{2})^j\cdot |\p(\th+\xi)-\p(\th)| |\xi|^{-1}\,d\xi\\
\leq &\;C(C_1^2C_2)^j (j\|h_i'\|_{L^\infty}^{2j-1}\|h_i'\|_{\dot{C}^\b} \|\p\|_{L^\infty}+\|h_i'\|_{L^\infty}^{2j}\|\p\|_{\dot{C}^\b}).
\end{split}
\label{eqn: L inf estimate of L3j first term}
\eeq

We also derive that
\beq
\begin{split}
&\;\left|\mathrm{p.v.}\int_\BT (A_1^j-A_2^j) \f{\p(\th+\xi)}{2\tan\f{\xi}{2}}\,d\xi\right|\\
\leq &\;\int_\BT |A_1^j-A_2^j-B_1^j+B_2^j| \left|\f{\p(\th+\xi)}{2\tan\f{\xi}{2}}\right|\,d\xi+|B_1^j-B_2^j|\left|\int_\BT \f{\p(\th+\xi)-\p(\th)}{2\tan\f{\xi}{2}}\,d\xi\right|.
\end{split}
\label{eqn: L inf bound for difference of a part of L_3}
\eeq
Write
\beq
\begin{split}
&\;A_1^j-A_2^j-B_1^j+B_2^j \\
=&\; (A_1-A_2-B_1+B_2)\sum_{l = 0}^{j-1}A_1^l A_2^{j-1-l}+(B_1-B_2)\sum_{l = 0}^{j-1}(A_1^l A_2^{j-1-l}-B_1^l B_2^{j-1-l}).
\end{split}
\label{eqn: double difference of A^j and B^j}
\eeq
Since
\beq
\begin{split}
&\;A_1-A_2-B_1+B_2\\
=&\; (1+h_1(\th))\left(l_1 - \f{h_1'(\th)^2}{(1+h_1(\th))^2}\right) - (1+h_2(\th))\left(l_2 - \f{h_2'(\th)^2}{(1+h_2(\th))^2}\right)\\
=&\; \f{h_1-h_2}{1+h_1}(A_1-B_1) \\
&\;+ (1+h_2(\th))\left(l_1-l_2 - \f{h_1'(\th)^2}{(1+h_1(\th))^2}+ \f{h_2'(\th)^2}{(1+h_2(\th))^2}\right),
\end{split}
\eeq
%
we use \eqref{eqn: double difference of l and its value at theta} and \eqref{eqn: difference between A_i and B_i} to derive that
\beq
|A_1-A_2-B_1+B_2|\leq 
C|\xi|^\b\|h_1-h_2\|_{C^{1,\b}}(\|h_1'\|_{\dot{C}^\b} +\|h_2'\|_{\dot{C}^\b})\\
%
\eeq
%
Combining this with \eqref{eqn: difference between A_i and B_i} and \eqref{eqn: double difference of A^j and B^j} yields that
\beq
\begin{split}
&\;|A_1^j-A_2^j-B_1^j+B_2^j |\\
\leq &\; |A_1-A_2-B_1+B_2|  \sum_{l = 0}^{j-1}(C_1^2 C_2\|h_1'\|_{L^\infty}^2)^{l}(C_1^2 C_2\|h_2'\|_{L^\infty}^2)^{j-1-l}\\
&\;+C|B_1-B_2|\sum_{l = 0}^{j-1}l(C_1^2 C_2\|h'_1\|_{L^\infty}^2)^{l-1}\cdot \|h_1'\|_{L^\infty}\|h_1'\|_{\dot{C}^\b}|\xi|^\b
\cdot (C_1^2 C_2\|h_2'\|_{L^\infty}^2)^{j-1-l}\\
&\;+C|B_1-B_2|\sum_{l = 0}^{j-1}(C_1^2 C_2\|h'_1\|_{L^\infty}^2)^{l}\cdot (j-1-l) (C_1^2 C_2\|h_2'\|_{L^\infty}^2)^{j-2-l}\cdot \|h_2'\|_{L^\infty}\|h_2'\|_{\dot{C}^\b}|\xi|^\b\\
%
\leq &\; C(C_1^2 C_2)^{j-1} |\xi|^\b\cdot j\|h_1-h_2\|_{C^{1,\b}}(\|h_1'\|_{\dot{C}^\b}+\|h_2'\|_{\dot{C}^\b}) (\|h_1'\|_{L^\infty}^2+\|h_2'\|_{L^\infty}^2)^{j-1}.
\end{split}
\eeq
Applying this to \eqref{eqn: L inf bound for difference of a part of L_3}, we obtain that
\beq
\begin{split}
&\;\left|\mathrm{p.v.}\int_\BT (A_1^j-A_2^j) \f{\p(\th+\xi)}{2\tan\f{\xi}{2}}\,d\xi\right|\\
\leq &\;C(C_1^2 C_2)^{j-1}\cdot j\|h_1-h_2\|_{C^{1,\b}}(\|h_1'\|_{\dot{C}^\b}+\|h_2'\|_{\dot{C}^\b}) (\|h_1'\|_{L^\infty}^2+\|h_2'\|_{L^\infty}^2)^{j-1}\|\p\|_{C^\b}. 
\end{split}
\label{eqn: L inf estimate of difference of L3j first term}
\eeq

Arguing as in \eqref{eqn: Lp estimate of a typical term in L2j}-\eqref{eqn: Lp estimate of L2j third term}, for $j\geq 1$ and $0\leq l\leq 2j-1$,
%
%
\beq
\begin{split}
&\;\left\|\mathrm{p.v.}\int_\BT \D (h_1-h_2)(\D h_1)^{l}(\D h_2)^{2j-1-l} (1+h_1(\th+\xi))^{-j}\cdot \f{\p(\th+\xi)}{2\tan\f{\xi}{2}}\,d\xi\right\|_{L^p}\\
\leq &\;C C_*^{2j-1}C_2^j(\|h_1'\|_{L^\infty}+\|h_2'\|_{L^\infty})^{2j-1}\|h_1'-h_2'\|_{L^\infty}\|\p\|_{L^p},
\end{split}
\label{eqn: Lp estimate of L3j second term}
\eeq
\beq
\left\|\mathrm{p.v.}\int_\BT A_1^{j}\cdot \f{\p(\th+\xi)}{2\tan\f{\xi}{2}}\,d\xi\right\|_{L^p}
\leq C C_*^{2j-1}C_2^j\|h_1'\|_{L^\infty}^{2j}\|\p\|_{L^p},
\label{eqn: Lp estimate of L3j first term}
\eeq
and
\beq
\begin{split}
&\;\left\|\mathrm{p.v.}\int_\BT (\D h_2)^{2j} \left(\f{1}{1+h_1(\th+\xi)}-\f{1}{1+h_2(\th+\xi)}\right)\right.\\
&\;\left.\qquad  \cdot (1+h_1(\th+\xi))^{-l}(1+h_2(\th+\xi))^{-(j-1-l)}\cdot \f{\p(\th+\xi)}{2\sin\f{\xi}{2}}\,d\xi\right\|_{L^p}\\
\leq &\;C C_*^{2j-1}C_2^j\|h_2'\|_{L^\infty}^{2j}\|h_1-h_2\|_{L^\infty}\|\p\|_{L^p}.
\end{split}
\eeq
Hence,
\beq
\begin{split}
&\;\left\|\mathrm{p.v.}\int_\BT (A_1^{j}-A_2^j)\cdot \f{\p(\th+\xi)}{2\tan\f{\xi}{2}}\,d\xi\right\|_{L^p}\\
\leq &\;\sum_{l = 0}^{2j-1}\left\|\mathrm{p.v.}\int_\BT \D(h_1-h_2)(\D h_1)^l (\D h_2)^{2j-1-l}(1+h_1(\th+\xi))^{-j}\cdot \f{\p(\th+\xi)}{2\tan\f{\xi}{2}}\,d\xi\right\|_{L^p}\\
&\;+\sum_{l = 0}^{j-1}\left\|\mathrm{p.v.}\int_\BT (\D h_2)^{2j}\left(\f{1}{1+h_1(\th+\xi)}-\f{1}{1+h_2(\th+\xi)}\right)\right.\\
&\;\left.\qquad \qquad \cdot (1+h_1(\th+\xi))^{-l}(1+h_1(\th+\xi))^{-(j-1-l)}\cdot \f{\p(\th+\xi)}{2\tan\f{\xi}{2}}\,d\xi\right\|_{L^p}\\
\leq &\;C j C_*^{2j-1}C_2^j(\|h_1'\|_{L^\infty}+\|h_2'\|_{L^\infty})^{2j-1}\|h_1-h_2\|_{W^{1,\infty}}\|\p\|_{L^p}. 
\end{split}
\label{eqn: Lp estimate for difference of A1 A2 integrals in L3 term}
\eeq

Similar to \eqref{eqn: W1p estimate of L2j typical term}-\eqref{eqn: W1p estimate of L2j third term},  for $j \geq 1$,
\beq
\begin{split}
&\;\left\|\mathrm{p.v.}\int_\BT \D (h_1-h_2)(\D h_1)^{l}(\D h_2)^{2j-1-l} (1+h_1(\th+\xi))^{-j}\cdot \f{\p(\th+\xi)}{2\tan\f{\xi}{2}}\,d\xi\right\|_{\dot{W}^{1,p}}\\
\leq &\;(2j+1)C_\dag^{2j+1}(jC_2^{j+1}\|h_1'\|_{L^\infty}\|\p\|_{L^p}+C_2^j\|\p'\|_{L^p})\|h_1'-h_2'\|_{L^\infty}(\|h_1'\|_{L^\infty}+\|h_2'\|_{L^\infty})^{2j-1}\\
&\;+(2j+1)C_\dag^{2j+1}C_2^j\|\p\|_{L^\infty}\|h_1''-h_2''\|_{L^p}(\|h_1'\|_{L^\infty}+\|h_2'\|_{L^\infty})^{2j-1}\\
&\;+(2j+1)C_\dag^{2j+1}C_2^j\|\p\|_{L^\infty}\|h_1'-h_2'\|_{L^\infty} (\|h_1'\|_{L^\infty}+\|h_2'\|_{L^\infty})^{2j-2}\cdot(\|h_1''\|_{L^p}+\|h_2''\|_{L^p}),
\end{split}
\label{eqn: W1p estimate of L3j second term}
\eeq
\beq
\begin{split}
&\;\left\|\mathrm{p.v.}\int_\BT A_1^{j}\cdot \f{\p(\th+\xi)}{2\tan\f{\xi}{2}}\,d\xi\right\|_{\dot{W}^{1,p}}\\
\leq &\;(2j+1)C_\dag^{2j+1}(jC_2^{j+1}\|h_1'\|_{L^\infty}\|\p\|_{L^p}+C_2^j\|\p'\|_{L^p}) \|h_1'\|_{L^\infty}^{2j}\\
&\;+C(2j+1)C_\dag^{2j+1}C_2^j\|\p\|_{L^\infty}\|h_1''\|_{L^p}\|h_1'\|_{L^\infty}^{2j-1},
\end{split}
\label{eqn: W1p estimate of L3j first term}
\eeq
and
\beq
\begin{split}
&\;\left\|\mathrm{p.v.}\int_\BT (\D h_2)^{2j} \left(\f{1}{1+h_1(\th+\xi)}-\f{1}{1+h_2(\th+\xi)}\right)\right.\\
&\;\left.\qquad  \cdot (1+h_1(\th+\xi))^{-l}(1+h_2(\th+\xi))^{-(j-1-l)}\cdot \f{\p(\th+\xi)}{2\sin\f{\xi}{2}}\,d\xi\right\|_{\dot{W}^{1,p}}\\
\leq &\;(2j+1)C_\dag^{2j+1}\left\|\left(\f{1}{1+h_1}-\f{1}{1+h_2}\right) (1+h_1)^{-l}(1+h_2)^{-(j-1-l)}\p\right\|_{\dot{W}^{1,p}}\|h_2'\|_{L^\infty}^{2j}\\
&\;+C(2j+1)C_\dag^{2j+1}C_2^{j+1}\|h_1-h_2\|_{L^\infty}\|\p\|_{L^\infty}\|h_2''\|_{L^p}\|h_2'\|_{L^\infty}^{2j-1}\\
\leq &\;C(2j+1)C_\dag^{2j+1}C_2^j\|h_2'\|_{L^\infty}^{2j}\cdot [\|h_1'-h_2'\|_{L^\infty}\|\p\|_{L^p} \\
&\;\qquad +j\|h_1-h_2\|_{L^\infty} (\|h_1'\|_{L^\infty}+\|h_2'\|_{L^\infty})\|\p\|_{L^p} +\|h_1-h_2\|_{L^\infty}\|\p'\|_{L^p}]\\
&\;+C(2j+1)C_\dag^{2j+1}C_2^{j+1}\|h_2'\|_{L^\infty}^{2j-1}\|h_1-h_2\|_{L^\infty}\|\p\|_{L^\infty}\|h_2''\|_{L^p}.
\end{split}
\eeq
Hence,
\beq
\begin{split}
&\;\left\|\mathrm{p.v.}\int_\BT (A_1^{j}-A_2^j)\cdot \f{\p(\th+\xi)}{2\tan\f{\xi}{2}}\,d\xi\right\|_{\dot{W}^{1,p}}\\
\leq &\;\sum_{l = 0}^{2j-1}\left\|\mathrm{p.v.}\int_\BT \D(h_1-h_2)(\D h_1)^l (\D h_2)^{2j-1-l}(1+h_1(\th+\xi))^{-j}\cdot \f{\p(\th+\xi)}{2\tan\f{\xi}{2}}\,d\xi\right\|_{\dot{W}^{1,p}}\\
&\;+\sum_{l = 0}^{j-1}\left\|\mathrm{p.v.}\int_\BT (\D h_2)^{2j}\left(\f{1}{1+h_1(\th+\xi)}-\f{1}{1+h_2(\th+\xi)}\right)\right.\\
&\;\left.\qquad \qquad \cdot (1+h_1(\th+\xi))^{-l}(1+h_2(\th+\xi))^{-(j-1-l)}\cdot \f{\p(\th+\xi)}{2\tan\f{\xi}{2}}\,d\xi\right\|_{\dot{W}^{1,p}}\\
\leq &\;Cj^2(2j+1)C_\dag^{2j+1}C_2^j \|\p\|_{L^p}\|h_1-h_2\|_{W^{1,\infty}}(\|h_1'\|_{L^\infty}+\|h_2'\|_{L^\infty})^{2j}\\
&\;+Cj(2j+1)C_\dag^{2j+1}C_2^j \|\p'\|_{L^p}\|h_1-h_2\|_{W^{1,\infty}}(\|h_1'\|_{L^\infty}+\|h_2'\|_{L^\infty})^{2j-1}\\
&\;+Cj(2j+1)C_\dag^{2j+1}C_2^j\|\p\|_{L^\infty}\|h_1''-h_2''\|_{L^p}(\|h_1'\|_{L^\infty}+\|h_2'\|_{L^\infty})^{2j-1}\\
&\;+Cj(2j+1)C_\dag^{2j+1}C_2^j\|\p\|_{L^\infty}\|h_1-h_2\|_{W^{1,\infty} } (\|h_1'\|_{L^\infty}+\|h_2'\|_{L^\infty})^{2j-2}(\|h_1''\|_{L^p}+\|h_2''\|_{L^p}).
%
\end{split}
\label{eqn: W1p estimate for difference of A1 A2 integrals in L3 term}
\eeq

To this end, by \eqref{eqn: decomposition of difference of L3j}, 
\beq
\begin{split}
&\;\|L_{3,j}^{(1)}-L_{3,j}^{(2)}\|_{\dot{W}^{1,p}}\\
\leq &\;\|h_1''(1+h_1)^{-(j+1)}-h_2''(1+h_2)^{-(j+1)}\|_{L^p} \left\|\mathrm{p.v.}\int_\BT A_1^j \f{\p(\th+\xi)}{2\tan\f{\xi}{2}}\,d\xi\right\|_{L^\infty}\\
&\;+(j+1)\|(h_1')^2(1+h_1)^{-(j+2)}-(h_2')^2(1+h_2)^{-(j+2)}\|_{L^\infty}\left\|\mathrm{p.v.}\int_\BT A_1^j \f{\p(\th+\xi)}{2\tan\f{\xi}{2}}\,d\xi\right\|_{L^p}\\
&\;+\|h_1'(1+h_1)^{-(j+1)}-h_2'(1+h_2)^{-(j+1)}\|_{L^\infty}\left\|\mathrm{p.v.}\int_\BT A_1^j \f{\p(\th+\xi)}{2\tan\f{\xi}{2}}\,d\xi\right\|_{\dot{W}^{1,p}}\\
&\;+\|h_2''(1+h_2)^{-(j+1)}\|_{L^p} \left\|\mathrm{p.v.}\int_\BT (A_1^j-A_2^j)
\f{\p(\th+\xi)}{2\tan\f{\xi}{2}}\,d\xi\right\|_{L^\infty}\\
&\;+(j+1)\|(h_2')^2(1+h_2)^{-(j+2)}\|_{L^\infty}\left\| \mathrm{p.v.}\int_\BT (A_1^j-A_2^j)
\f{\p(\th+\xi)}{2\tan\f{\xi}{2}}\,d\xi\right\|_{L^p}\\
&\;+\|h_2'(1+h_2)^{-(j+1)}\|_{L^\infty} \left\|\mathrm{p.v.}\int_\BT (A_1^j-A_2^j)
\f{\p(\th+\xi)}{2\tan\f{\xi}{2}}\,d\xi\right\|_{\dot{W}^{1,p}}.
\end{split}
\label{eqn: W 1 p norm of the difference of L3j crude form}
\eeq
For $j = 0$, this can be simplified as
\beq
\begin{split}
&\;\|L_{3,0}^{(1)}-L_{3,0}^{(2)}\|_{\dot{W}^{1,p}}\\
%
\leq &\;C(\|h_1''-h_2''\|_{L^p}+\|h_2''\|_{L^p}\|h_1-h_2\|_{L^\infty}) \| \p\|_{\dot{C}^\b}\\
&\;+C\|h_1-h_2\|_{W^{1,\infty}} (\|h_1'\|_{L^\infty}+\|h_2'\|_{L^\infty})
\|\p\|_{L^p}\\
&\;+C\|h_1-h_2\|_{W^{1,\infty}}\|\p'\|_{L^p}.
\end{split}
\label{eqn: differece of L 3 0}
\eeq
For $j\geq 1$, by applying \eqref{eqn: L inf estimate of L3j first term}, \eqref{eqn: L inf estimate of difference of L3j first term}, \eqref{eqn: Lp estimate of L3j first term}, \eqref{eqn: Lp estimate for difference of A1 A2 integrals in L3 term}, \eqref{eqn: W1p estimate of L3j first term} and \eqref{eqn: W1p estimate for difference of A1 A2 integrals in L3 term} to \eqref{eqn: W 1 p norm of the difference of L3j crude form}, we derive that 
\beq
\begin{split}
&\;\|L_{3,j}^{(1)}-L_{3,j}^{(2)}\|_{\dot{W}^{1,p}}\\
\leq &\; C(C_2^{j+1}\|h_1''-h_2''\|_{L^p}+\|h_2''\|_{L^p}(j+1)C_2^{j+2}\|h_1-h_2\|_{L^\infty})\\
&\;\quad \cdot (C_1^2C_2)^j \cdot j\|h_1'\|_{L^\infty}^{2j-1}\|h_1'\|_{\dot{C}^\b} \|\p\|_{C^\b}\\
&\;+C(j+1)(j+2)C_2^{j+2}\|h_1-h_2\|_{W^{1,\infty}}(\|h_1'\|_{L^\infty}+\|h_2'\|_{L^\infty})\\
&\;\quad \cdot C_*^{2j-1}C_2^j\|h_1'\|_{L^\infty}^{2j}\|\p\|_{L^p}\\
&\;+C(j+1)C_2^{j+1}\|h_1-h_2\|_{W^{1,\infty}}\\
&\;\quad \cdot \left[(2j+1)C_\dag^{2j+1}(jC_2^{j+1}\|h_1'\|_{L^\infty}\|\p\|_{L^p}+C_2^j\|\p'\|_{L^p}) \|h_1'\|_{L^\infty}^{2j}\right.\\
&\; \left.\qquad +(2j+1)C_\dag^{2j+1}C_2^j\|\p\|_{L^\infty}\|h_1''\|_{L^p}\|h_1'\|_{L^\infty}^{2j-1}\right]\\
&\;+CC_2^{j+1}\|h_2''\|_{L^p}\\
&\;\quad \cdot (C_1^2 C_2)^{j-1}\cdot j\|h_1-h_2\|_{C^{1,\b}}(\|h_1'\|_{\dot{C}^\b}+\|h_2'\|_{\dot{C}^\b}) (\|h_1'\|_{L^\infty}^2+\|h_2'\|_{L^\infty}^2)^{j-1}\|\p\|_{C^\b}\\
&\;+CC_2^{j+2}(j+1)\|h_2'\|_{L^\infty}^2 \\
&\;\quad \cdot  j C_*^{2j-1}C_2^j(\|h_1'\|_{L^\infty}+\|h_2'\|_{L^\infty})^{2j-1}\|h_1-h_2\|_{W^{1,\infty}}\|\p\|_{L^p}\\
&\;+CC_2^{j+1}\|h_2'\|_{L^\infty}\cdot j(2j+1)C_\dag^{2j+1}C_2^j (\|h_1'\|_{L^\infty}+\|h_2'\|_{L^\infty})^{2j-2}\\
&\;\quad \cdot \left[ j \|\p\|_{L^p}\|h_1-h_2\|_{W^{1,\infty}}(\|h_1'\|_{L^\infty}+\|h_2'\|_{L^\infty})^{2}\right.\\
&\;\qquad + \|\p'\|_{L^p}\|h_1-h_2\|_{W^{1,\infty}}(\|h_1'\|_{L^\infty}+\|h_2'\|_{L^\infty})\\
&\;\qquad +\|\p\|_{L^\infty}\|h_1''-h_2''\|_{L^p}(\|h_1'\|_{L^\infty}+\|h_2'\|_{L^\infty})\\
&\;\left.\qquad +\|\p\|_{L^\infty}\|h_1-h_2\|_{W^{1,\infty} } (\|h_1''\|_{L^p}+\|h_2''\|_{L^p})\right].
%
\end{split}
\eeq
This together with \eqref{eqn: differece of L 3 0} and the smallness of $h_i$ implies
\beq
\begin{split}
&\;\|L_{3}^{(1)}-L_{3}^{(2)}\|_{\dot{W}^{1,p}}\\
\leq &\;\|L_{3,0}^{(1)}-L_{3,0}^{(2)}\|_{\dot{W}^{1,p}}+\sum_{j = 1}^\infty\|L_{3,j}^{(1)}-L_{3,j}^{(2)}\|_{\dot{W}^{1,p}}\\
\leq &\; C \|h_1''-h_2''\|_{L^p} (1+\|h_1'\|_{\dot{C}^\b} +\|h_2'\|_{\dot{C}^\b}) \|\p\|_{C^\b}\\
&\;+C(\|h_1''\|_{L^p}+\|h_2''\|_{L^p} ) \|h_1-h_2\|_{C^{1,\b}}(1+\|h_1'\|_{\dot{C}^\b}+\|h_2'\|_{\dot{C}^\b}) \|\p\|_{C^\b}\\
%
&\;+C\|h_1-h_2\|_{W^{1,\infty}}\|\p'\|_{L^p}.
\end{split}
\label{eqn: W1p estimate for L3 difference}
\eeq

Then the desired estimate follows from \eqref{eqn: W1p estimate for L1 difference}, \eqref{eqn: W1p estimate for L2 difference} and \eqref{eqn: W1p estimate for L3 difference}.
\end{proof}

\begin{proof}[Proof of Lemma \ref{lem: W 1p difference from Hilbert transform difference}]
Following \eqref{eqn: decomposition of singular integral tangent}, we use $\tilde{L}_{k}^{(i)}$ $(k = 1,2,3)$ to denote the corresponding quantities defined by $h_i$ $(i = 1,2)$.

Using \eqref{eqn: relation between L_1 and tilde L_1}, we find that
\beq
\begin{split}
&\;\|\tilde{L}_1^{(1)}-\tilde{L}_1^{(2)}\|_{\dot{W}^{1,p}}\\
\leq &\; \left\|\left(\f{h_1''}{1+h_1} - \f{(h_1')^2}{(1+h_1)^2}- \f{h_2''}{1+h_2} + \f{(h_2')^2}{(1+h_2)^2}\right)
\left(\f12\int_\BT \p\,d\xi+ L_1^{(1)}\right)\right\|_{L^p}
\\
&\;+\left\|\left(\f{h_2''}{1+h_2} - \f{(h_2')^2}{(1+h_2)^2}\right)
(L_1^{(1)}-L_1^{(2)})\right\|_{L^p}
\\
&\;+\left\|\left(\f{h_1'}{1+h_1}-\f{h_2'}{1+h_2}\right)(L_1^{(1)})'\right\|_{L^p}
+\left\|\f{h_2'}{1+h_2}(L_1^{(1)}-L_1^{(2)})'\right\|_{L^p}\\
\leq &\;C(\|h_1''-h_2''\|_{L^p}+\|h_2''\|_{L^p}\|h_1-h_2\|_{L^\infty})
\left(\left|\int_\BT \p\,d\xi\right|+\|L_1^{(1)}\|_{L^\infty}\right)\\
&\;+C\|h_2''\|_{L^p}\|L_1^{(1)}-L_1^{(2)}\|_{L^\infty}\\
&\;+C\|h_1-h_2\|_{W^{1,\infty}}\|L_1^{(1)}\|_{\dot{W}^{1,p}}
+C\|h_2'\|_{L^\infty}\|L_1^{(1)}-L_1^{(2)}\|_{\dot{W}^{1,p}}.
\end{split}
\label{eqn: W1p estimate for tilde L1 difference crude form}
\eeq
%
It is not difficult to show by \eqref{eqn: L inf difference of building blocks from two configurations} that
\beq
\begin{split}
\|L_1^{(1)}-L_1^{(2)}\|_{L^\infty}
\leq &\;C \|\p\|_{L^\infty}\int_\BT \|l_1-l_2\|_{L^\infty_\th}\,d\xi \\
\leq &\;C \|\p\|_{L^\infty}(\|h_1'\|_{L^\infty}+\|h_2'\|_{L^\infty})\|h_1-h_2\|_{W^{1,\infty}}.
\end{split}
\eeq
Taking $h_2 =0$ yields $\|L_1^{(1)}\|_{L^\infty}\leq C \|\p\|_{L^\infty}\|h_1'\|_{L^\infty}$; here we used the fact $m_{0,i}\ll 1$.
Substituting these estimates as well as \eqref{eqn: W1p estimate for L1} and \eqref{eqn: W1p estimate for L1 difference} into \eqref{eqn: W1p estimate for tilde L1 difference crude form}, we obtain that
\beq
\begin{split}
&\;\|\tilde{L}_1^{(1)}-\tilde{L}_1^{(2)}\|_{\dot{W}^{1,p}}\\
\leq &\;C(\|h_1''-h_2''\|_{L^p} +(\|h_1''\|_{L^p}+\|h_2''\|_{L^p})\|h_1-h_2\|_{W^{1,\infty}})\\
&\;\qquad \cdot
\left(\left|\int_\BT \p\,d\xi\right|+\|\p\|_{L^\infty}(\|h_1'\|_{L^\infty}+\|h_2'\|_{L^\infty})\right)\\
&\; +C (\|h_1'\|_{L^\infty}+\|h_2'\|_{L^\infty})^2 \|h_1-h_2\|_{W^{1,\infty}}\|\p'\|_{L^p}.
\end{split}
\label{eqn: W1p estimate for tilde L1 difference}
\eeq

To bound $\tilde{L}_2^{(1)}-\tilde{L}_2^{(2)}$, we are going to make use of the estimates for $L_2^{(1)}-L_2^{(2)}$ in Lemma \ref{lem: W 1p estimate of normal component of singular integral difference}, since $\tilde{L}_2^{(i)}$ coincides with $-h_i'(\th)L_2^{(i)}$ if $\p$ in the definition of $L_2^{(i)}$ is replaced by $\p/(1+h_i)$.
For this purpose, an $L^\infty$-estimate for $L_2^{(1)}-L_2^{(2)}$ is needed.
We start with
\beq
\begin{split}
&\;|L_2^{(1)}-L_2^{(2)}|\\
\leq &\;\left|\mathrm{p.v.}\int_\BT \left(\f{\f{\D h_1-h_1'(\th)}{1+h_1(\th)}}{1+l_1}
-\f{\f{\D h_2-h_2'(\th)}{1+h_2(\th)}}{1+l_2}\right)\f{\p(\th+\xi)}{2\sin\f{\xi}{2}}\,d\xi\right|\\
&\;+\left|\mathrm{p.v.}\int_\BT \left(\f{\f{h_1'(\th)}{1+h_1(\th)}}{1+l_1}
-\f{\f{h_2'(\th)}{1+h_2(\th)}}{1+l_2}\right)\f{\p(\th+\xi)}{2\sin\f{\xi}{2}}\,d\xi\right|.
\end{split}
\label{eqn: L inf bound for difference of L_2 crude form}
\eeq
It is straightforward to bound the first term.
\beq
\begin{split}
&\;\left|\mathrm{p.v.}\int_\BT \left(\f{\f{\D h_1-h_1'(\th)}{1+h_1(\th)}}{1+l_1}
-\f{\f{\D h_2-h_2'(\th)}{1+h_2(\th)}}{1+l_2}\right)\f{\p(\th+\xi)}{2\sin\f{\xi}{2}}\,d\xi\right|\\
\leq &\;C\int_\BT |\xi|^\b( \|h_1'-h_2'\|_{\dot{C}^\b}+\|h_2'\|_{\dot{C}^\b}(\|h_1-h_2\|_{L^\infty}+|l_1-l_2|))\|\p\|_{L^\infty}|\xi|^{-1}\,d\xi\\
\leq &\;C( \|h_1'-h_2'\|_{\dot{C}^\b}+\|h_2'\|_{\dot{C}^\b}\|h_1-h_2\|_{W^{1,\infty}})\|\p\|_{L^\infty}.
\end{split}
\label{eqn: L inf bound of L_2 difference first part}
\eeq
To bound the second term in \eqref{eqn: L inf bound for difference of L_2 crude form}, we first note that \eqref{eqn: L inf estimate of L3j first term} and \eqref{eqn: L inf estimate of difference of L3j first term} still hold if $2\tan\f{\xi}{2}$ in their denominators are replaced by $2\sin\f{\xi}{2}$.
Hence, we argue as in the proof of Lemma \ref{lem: W 1p estimate of normal component of singular integral difference} by Taylor expanding $(1+l_i)^{-1}$  that
\beq
\begin{split}
&\;\left|\mathrm{p.v.}\int_\BT \left(\f{1}{1+l_1}
-\f{1}{1+l_2}\right)\f{\p(\th+\xi)}{2\sin\f{\xi}{2}}\,d\xi\right|\\
\leq &\;\sum_{j = 1}^\infty\left|\mathrm{p.v.}\int_\BT \left(\f{A_1^j}{(1+h_1(\th))^j}-\f{A_2^j}{(1+h_2(\th))^j}\right)\f{\p(\th+\xi)}{2\sin\f{\xi}{2}}\,d\xi\right|\\
\leq &\;\sum_{j = 1}^\infty C_2^j \left|\mathrm{p.v.}\int_\BT (A_1^j-A_2^j) \f{\p(\th+\xi)}{2\sin\f{\xi}{2}}\,d\xi\right|+\left|\f{1}{(1+h_1)^j}-\f{1}{(1+h_2)^j}\right|\left|\mathrm{p.v.}\int_\BT A_2^j \f{\p(\th+\xi)}{2\sin\f{\xi}{2}}\,d\xi\right|\\
\leq &\;C\|\p\|_{C^\b} \|h_1-h_2\|_{C^{1,\b}}(\|h_1'\|_{\dot{C}^\b} +\|h_2'\|_{\dot{C}^\b}).
\end{split}
\eeq
Taking $h_2 = 0$ here yields
\beq
\left|\mathrm{p.v.}\int_\BT \left(\f{1}{1+l_1}
-1\right)\f{\p(\th+\xi)}{2\sin\f{\xi}{2}}\,d\xi\right|
\leq C\|\p\|_{C^\b}\|h_1\|_{C^{1,\b}}\|h_1'\|_{\dot{C}^\b},
\eeq
which further implies
\beq
\left|\mathrm{p.v.}\int_\BT \f{1}{1+l_1}\f{\p(\th+\xi)}{2\sin\f{\xi}{2}}\,d\xi\right|
\leq C\|\p\|_{C^\b}(1+\|h_1\|_{C^{1,\b}})^2.
\eeq
To this end, we may bound the second term in \eqref{eqn: L inf bound for difference of L_2 crude form} as follows
\beq
\begin{split}
&\;\left|\mathrm{p.v.}\int_\BT \left(\f{\f{h_1'(\th)}{1+h_1(\th)}}{1+l_1}
-\f{\f{h_2'(\th)}{1+h_2(\th)}}{1+l_2}\right)\f{\p(\th+\xi)}{2\sin\f{\xi}{2}}\,d\xi\right|\\
\leq &\;\left|\f{h_1'}{1+h_1}-\f{h_2'}{1+h_2}\right| \left|\mathrm{p.v.}\int_\BT \f{1}{1+l_1}\f{\p(\th+\xi)}{2\sin\f{\xi}{2}}\,d\xi\right|\\
&\;+\left|\f{h_2'}{1+h_2}\right| \left|\mathrm{p.v.}\int_\BT \left(\f{1}{1+l_1}
-\f{1}{1+l_2}\right)\f{\p(\th+\xi)}{2\sin\f{\xi}{2}}\,d\xi\right|\\
\leq &\; C\|\p\|_{C^\b} \|h_1-h_2\|_{C^{1,\b}}(1+\|h_1\|_{C^{1,\b} } +\|h_2\|_{C^{1,\b}})^2. 
\end{split}
\eeq
Combining this with \eqref{eqn: L inf bound for difference of L_2 crude form} and \eqref{eqn: L inf bound of L_2 difference first part},
\beq
\|L_2^{(1)}-L_2^{(2)}\|_{L^\infty}
\leq C\|\p\|_{C^\b} \|h_1-h_2\|_{C^{1,\b}}(1+\|h_1\|_{C^{1,\b} } +\|h_2\|_{C^{1,\b}})^2.
\label{eqn: L inf bound for difference of L_2}
\eeq
Setting $h_1 = 0$ (or $h_2 = 0$) provides
\beq
\|L_2^{(i)}\|_{L^\infty}
\leq C\|\p\|_{C^\b} \|h_i\|_{C^{1,\b}}(1+\|h_i\|_{C^{1,\b}})^2.
\label{eqn: L inf bound for L_2}
\eeq

To emphasize the $\p$-dependence of $L_2^{(i)}$, we shall rewrite $L_{2}^{(i)}$ as $L_{2,\p}^{(i)}$.
Since $\tilde{L}_2^{(i)}=-h_i'(\th)L_{2,\p/(1+h_i)}^{(i)}$, we derive with \eqref{eqn: W1p estimate for L2 difference}, \eqref{eqn: L inf bound for difference of L_2} and \eqref{eqn: L inf bound for L_2} that
\beq
\begin{split}
&\;\|\tilde{L}_2^{(1)}-\tilde{L}_2^{(2)}\|_{\dot{W}^{1,p}}\\
\leq &\;\|h_1''-h_2''\|_{L^p}\|L_{2,\p/(1+h_1)}^{(1)}\|_{L^\infty} +\|h_2''\|_{L^p}\|L_{2,\p/(1+h_1)}^{(1)}-L_{2,\p/(1+h_1)}^{(2)}\|_{L^\infty} \\
&\;+\|h_2''\|_{L^p}\|L_{2,\p/(1+h_1)-\p/(1+h_2)}^{(2)}\|_{L^\infty}\\
&\;+\|h_1'-h_2'\|_{L^\infty}\|L_{2,\p/(1+h_1)}^{(1)}\|_{\dot{W}^{1,p}} +\|h_2'\|_{L^\infty}\|L_{2,\p/(1+h_1)}^{(1)}-L_{2,\p/(1+h_1)}^{(2)}\|_{\dot{W}^{1,p}} \\
&\;+\|h_2'\|_{L^\infty}\|L_{2,\p/(1+h_1)-\p/(1+h_2)}^{(2)}\|_{\dot{W}^{1,p}}\\
\leq &\;C\|h_1''-h_2''\|_{L^p}\left\|\f{\p}{1+h_1}\right\|_{C^\b} \|h_1\|_{C^{1,\b}}(1+\|h_1\|_{C^{1,\b}})^2\\
&\;+C\|h_2''\|_{L^p}\left\|\f{\p}{1+h_1}\right\|_{C^\b} \|h_1-h_2\|_{C^{1,\b}}(1+\|h_1\|_{C^{1,\b} } +\|h_2\|_{C^{1,\b}})^2\\
&\;+C\|h_2''\|_{L^p} \left\|\f{\p}{1+h_1}-\f{\p}{1+h_2}\right\|_{C^\b} \|h_2\|_{C^{1,\b}}(1+\|h_2\|_{C^{1,\b}})^2\\
&\;+C\|h_1'-h_2'\|_{L^\infty}
\left(\left\|\f{\p}{1+h_1}\right\|_{\dot{W}^{1,p}} \|h_1\|_{W^{1,\infty}}+\|\p\|_{L^\infty}\|h_1''\|_{L^p}\right)
\\
&\;+C\|h_2'\|_{L^\infty}\left(\left\|\f{\p}{1+h_1}\right\|_{\dot{W}^{1,p} } \|h_1-h_2\|_{W^{1,\infty}}+\|\p\|_{L^\infty}\|h_1''-h_2''\|_{L^p}\right. \\
&\;\qquad + \|\p\|_{L^\infty}\|h_1-h_2\|_{W^{1,\infty}}(\|h_1''\|_{L^p}+\|h_2''\|_{L^p})\bigg)
\\
&\;+C\|h_2'\|_{L^\infty}\left(\left\|\f{\p}{1+h_1}-\f{\p}{1+h_2}\right\|_{\dot{W}^{1,p}} \|h_2\|_{W^{1,\infty}}+\left\|\f{\p}{1+h_1}-\f{\p}{1+h_2}\right\|_{L^\infty}\|h_2''\|_{L^p}\right).
\end{split}
\eeq
This gives
\beq
\begin{split}
&\;\|\tilde{L}_2^{(1)}-\tilde{L}_2^{(2)}\|_{\dot{W}^{1,p}}\\
\leq &\;C\|h_1''-h_2''\|_{L^p}\|\p\|_{C^\b} (\|h_1\|_{C^{1,\b}}+\|h_2\|_{C^{1,\b}})(1+\|h_1\|_{C^{1,\b}}+\|h_2\|_{C^{1,\b}})^2\\
&\;+C(\|h_1''\|_{L^p}+\|h_2''\|_{L^p})\|\p\|_{C^\b} \|h_1-h_2\|_{C^{1,\b}}(1+\|h_1\|_{C^{1,\b} } +\|h_2\|_{C^{1,\b}})^3\\
&\;+C\|h_1-h_2\|_{W^{1,\infty}}
\|\p'\|_{L^p} (\|h_1\|_{W^{1,\infty}}+\|h_2\|_{W^{1,\infty}}).
\end{split}
\label{eqn: W1p estimate for tilde L2 difference}
\eeq

For $\tilde{L}_3^{(i)}$, we rewrite
\beq
\tilde{L}_3^{(i)} =
\sum_{j = 1}^\infty (-1)^{j+1}(1+h_i(\th))^{-j}\mathrm{p.v.}\int_\BT A_i^j \cdot  \f{\p(\th+\xi)}{2\tan\f{\xi}{2}}\,d\xi.
\eeq
Thanks to \eqref{eqn: Lp estimate of L3j first term}, \eqref{eqn: Lp estimate for difference of A1 A2 integrals in L3 term}, \eqref{eqn: W1p estimate of L3j first term} and \eqref{eqn: W1p estimate for difference of A1 A2 integrals in L3 term}, we derive as in the proof of Lemma \ref{lem: W 1p estimate of normal component of singular integral difference} that
\beq
\begin{split}
&\;\|\tilde{L}_{3}^{(1)}-\tilde{L}_{3}^{(2)}\|_{\dot{W}^{1,p}}\\
\leq 
&\;C  \|h_1-h_2\|_{W^{1,\infty} } (\|\p'\|_{L^p}(\|h_1'\|_{L^\infty}+\|h_2'\|_{L^\infty})+\|\p\|_{L^\infty} (\|h_1''\|_{L^p}+\|h_2''\|_{L^p}))\\
&\; +C\|h_1''-h_2''\|_{L^p}\|\p\|_{L^\infty}(\|h_1'\|_{L^\infty}+\|h_2'\|_{L^\infty}).
\end{split}
\label{eqn: W1p estimate for tilde L3 difference}
\eeq

Combining \eqref{eqn: W1p estimate for tilde L1 difference}, \eqref{eqn: W1p estimate for tilde L2 difference} and \eqref{eqn: W1p estimate for tilde L3 difference}, we obtain \eqref{eqn: W1p difference estimate of the singular integral tangent component}.
\end{proof}

\bibliographystyle{unsrt}

\end{document}